\documentclass[11pt,reqno]{amsart}
\usepackage[tt=false]{libertine}
\usepackage{amssymb,mathrsfs}
\usepackage[varbb]{newpxmath}

\let\savedbigtimes\bigtimes
\let\bigtimes\relax
\usepackage{mathabx}
\let\bigtimes\savedbigtimes

\usepackage[margin=1in, bottom=1in]{geometry}
\usepackage{enumerate}
\usepackage{graphicx}
\usepackage{subcaption}
\usepackage{enumitem}

\usepackage[usenames,dvipsnames]{xcolor}
\usepackage[colorlinks=true,
	linkcolor=teal!60!black,
	citecolor=PineGreen,
	urlcolor=RedViolet]{hyperref}
\hypersetup{bookmarksopen=true}

\usepackage[yyyymmdd,hhmmss]{datetime}

\newtheorem{thm}{Theorem}[section]
\newtheorem{lem}[thm]{Lemma}
\newtheorem{ppn}[thm]{Proposition}
\newtheorem{cor}[thm]{Corollary}
\newtheorem{fac}[thm]{Fact}
\newtheorem{clm}[thm]{Claim}

\theoremstyle{definition}
\newtheorem{dfn}[thm]{Definition}
\newtheorem{rmk}[thm]{Remark}
\newtheorem*{rmk*}{Remark}

\newtheorem{con}[thm]{Condition}

\def\fr{\frac}
\def\lt{\left}
\def\rt{\right}
\def\la{\langle}
\def\ra{\rangle}
\def\eps{{\varepsilon}}
\def\ups{{\upsilon}}

\def\bbR{{\mathbb{R}}}
\def\bbE{{\mathbb{E}}}
\def\bbN{{\mathbb{N}}}
\def\bbP{{\mathbb{P}}}

\def\bbW{{\mathbb{W}}}
\def\bbZ{{\mathbb{Z}}}

\def\cE{{\mathcal{E}}}
\def\cF{{\mathcal{F}}}
\def\cG{{\mathcal{G}}}
\def\cH{{\mathcal{H}}}
\def\cI{{\mathcal{I}}}
\def\cJ{{\mathcal{J}}}

\def\cN{{\mathcal{N}}}
\def\cL{{\mathcal{L}}}
\def\cP{{\mathcal{P}}}

\def\cS{{\mathcal{S}}}
\def\cT{{\mathcal{T}}}

\def\sE{{\mathscr{E}}}
\def\sG{{\mathscr{G}}}
\def\sL{{\mathscr{L}}}
\def\sK{{\mathscr{K}}}
\def\sO{{\mathscr{O}}}
\def\sS{{\mathscr{S}}}

\def\osK{{\overline \sK}}
\def\osS{{\overline \sS}}

\def\bA{{\boldsymbol A}}
\def\bB{{\boldsymbol B}}
\def\bD{{\boldsymbol D}}
\def\bG{{\boldsymbol G}}
\def\bH{{\boldsymbol H}}
\def\bI{{\boldsymbol I}}
\def\bK{{\boldsymbol K}}
\def\bM{{\boldsymbol M}}
\def\bN{{\boldsymbol N}}
\def\bR{{\boldsymbol R}}

\def\bT{{\boldsymbol T}}
\def\bU{{\boldsymbol U}}
\def\bV{{\boldsymbol V}}
\def\bW{{\boldsymbol W}}
\def\bX{{\boldsymbol X}}
\def\bY{{\boldsymbol Y}}
\def\be{{\boldsymbol e}}

\def\bg{{\boldsymbol g}}
\def\bh{{\boldsymbol h}}
\def\bm{{\boldsymbol m}}
\def\bn{{\boldsymbol n}}
\def\bu{{\boldsymbol u}}
\def\bv{{\boldsymbol v}}
\def\bw{{\boldsymbol w}}
\def\bx{{\boldsymbol x}}
\def\by{{\boldsymbol y}}
\def\bz{{\boldsymbol z}}

\def\bxi{{\boldsymbol \xi}}
\def\bzero{{\boldsymbol 0}}
\def\bone{{\boldsymbol 1}}
\def\bDel{{\boldsymbol \Delta}}
\def\bLam{{\boldsymbol \Lambda}}
\def\tbD{{\widetilde \bD}}
\def\tbLam{{\widetilde \bLam}}

\def\tq{{\widetilde q}}
\def\tpsi{{\widetilde \psi}}
\def\tZ{{\widetilde Z}}
\def\tdet{{\widetilde \det}}

\def\tbG{{\widetilde \bG}}
\def\hbG{{\widehat \bG}}
\def\tbg{{\widetilde \bg}}
\def\hbG{{\widehat \bG}}
\def\hbg{{\widehat \bg}}
\def\hbn{{\widehat \bn}}

\def\tbm{{\widetilde \bm}}
\def\tbn{{\widetilde \bn}}

\def\oF{{\overline F}}

\def\oq{{\overline q}}
\def\opsi{{\overline \psi}}

\def\obG{{\overline \bG}}
\def\tobG{{\widetilde \obG}}

\def\df{{\dot f}}
\def\hf{{\widehat f}}
\def\dg{{\dot g}}
\def\hg{{\widehat g}}
\def\dh{{\dot h}}
\def\hh{{\widehat h}}

\def\dH{{\dot H}}
\def\hH{{\widehat H}}
\def\dXi{{\dot \Xi}}
\def\hXi{{\widehat \Xi}}

\def\dmu{{\dot \mu}}
\def\hmu{{\widehat \mu}}

\def\dxi{{\dot \xi}}
\def\hxi{{\widehat \xi}}

\def\ah{{\acute h}}
\def\dbe{{\dot \be}}
\def\hbe{{\widehat \be}}

\def\dbg{{\dot \bg}}
\def\hbg{{\widehat \bg}}
\def\dobg{{\dot {\overline \bg}}}
\def\hobg{{\widehat {\overline \bg}}}
\def\dbh{{\dot \bh}}
\def\hbh{{\widehat \bh}}
\def\dbu{{\dot \bu}}
\def\hbu{{\widehat \bu}}
\def\dbv{{\dot \bv}}
\def\hbv{{\widehat \bv}}
\def\dbw{{\dot \bw}}
\def\hbw{{\widehat \bw}}
\def\dSig{{\dot \Sigma}}
\def\hSig{{\widehat \Sigma}}
\def\hmu{{\widehat \mu}}
\def\abh{{\acute \bh}}
\def\dbxi{{\dot \bxi}}
\def\hbxi{{\widehat \bxi}}

\def\vv{\vec v}
\def\dvv{{\dot \vv}}

\def\obm{{\overline \bm}}
\def\obn{{\overline \bn}}
\def\obe{{\overline \bh}}
\def\obh{{\overline \bh}}
\def\dobe{{\dot \obe}}
\def\hobe{{\widehat \obe}}
\def\dobh{{\dot \obh}}
\def\hobh{{\widehat \obh}}

\def\dT{{\dot T}}
\def\hT{{\hat T}}
\def\dbH{{\dot \bH}}
\def\hbH{{\hat \bH}}
\def\dbT{{\dot \bT}}
\def\hbT{{\hat \bT}}

\def\bah{{\bar h}}
\def\brh{{\breve h}}
\def\babh{{\bar \bh}}
\def\brbh{{\breve \bh}}
\def\babH{{\bar \bH}}
\def\brbH{{\breve \bH}}

\def\va{{\vec a}}
\def\vp{{\vec p}}
\def\vr{{\vec r}}

\def\indep{{\perp\!\!\!\perp}}

\def\AMP{{\mathrm{AMP}}}
\def\TAP{{\mathrm{TAP}}}

\def\ch{{\mathrm{ch}}}
\def\sh{{\mathrm{sh}}}
\def\th{{\mathrm{th}}}

\def\tF{{\widetilde F}}
\def\tth{{\widetilde \th}}

\def\all{{\mathrm{all}}}
\def\elt{{\mathrm{elt}}}
\def\lb{{\mathrm{lb}}}
\def\ub{{\mathrm{ub}}}

\def\tr{{\mathrm{tr}}}

\def\bd{{\mathsf{bd}}}
\def\Ball{{\mathsf{B}}}

\def\Crt{{\mathsf{Crt}}}
\def\cvx{{\mathsf{cvx}}}
\def\DATA{{\mathsf{DATA}}}
\def\de{{\mathsf{d}}}
\def\diag{{\mathrm{diag}}}
\def\ent{{\mathsf{ent}}}
\def\err{{\mathsf{err}}}

\def\GOE{{\mathsf{GOE}}}
\def\mesh{{\mathsf{mesh}}}

\def\op{{\mathsf{op}}}
\def\Pl{{\mathsf{Pl}}}
\def\reg{{\mathsf{reg}}}
\def\spn{{\mathsf{span}}}
\def\unif{{\mathsf{unif}}}

\newcommand{\norm}[1]{{\|#1\|}}
\newcommand{\tnorm}[1]{\|#1\|}

\def\beq#1\eeq{%
    \begin{equation}%
    #1%
    \end{equation}%
}

\def\baln#1\ealn{%
    \begin{align*}%
    #1%
    \end{align*}%
}
\def\balnn#1\ealnn{%
    \begin{align}%
    #1%
    \end{align}%
}
\DeclareMathOperator*{\EE}{\bbE}
\DeclareMathOperator*{\PP}{\bbP}
\DeclareMathOperator*{\argmin}{\arg\,\min}
\DeclareMathOperator*{\plim}{p-lim}
\DeclareMathOperator*{\spec}{{\mathsf{spec}}}

\def\frt{{\mathfrak{t}}}
\def\frs{{\mathfrak{s}}}
\def\hz{{\widehat z}}
\def\tg{{\widetilde g}}
\def\hg{{\widehat g}}

\title{Capacity Threshold for the Ising Perceptron}

\author{Brice Huang}
\thanks{Department of Electrical Engineering and Computer Science, Massachusetts Institute of Technology}

\begin{document}

\maketitle

\setcounter{tocdepth}{1}

\begin{abstract}
    We show that the capacity of the Ising perceptron is with high probability upper bounded by the constant $\alpha_\star \approx 0.833$ conjectured by Krauth and M\'ezard, under the condition that an explicit two-variable function $\sS_\star(\lambda_1,\lambda_2)$ is maximized at $(1,0)$. The earlier work of Ding and Sun \cite{ding2018capacity} proves the matching lower bound subject to a similar numerical condition, and together these results give a conditional proof of the conjecture of Krauth and M\'ezard.
\end{abstract}

\tableofcontents

\section{Introduction}
\label{sec:intro}

The Ising perceptron was introduced in \cite{wendel1962problem, cover1965geometrical} as a simple model of a neural network.
Mathematically, it is an intersection of a high-dimensional discrete cube with random half-spaces, defined as follows.
Fix any $\kappa \in \bbR$ (our main result is for $\kappa = 0$).
For $N\ge 1$, let $\Sigma_N = \{\pm 1\}^N$, and let $\bg^1,\bg^2,\ldots$ be a sequence of i.i.d. samples from $\cN(\bzero,\bI_N)$.
For $M\ge 1$, the Ising perceptron is the random set
\beq
    \label{eq:solution-set}
    S_N^M =
    \lt\{
        \bx \in \Sigma_N : \fr{\la \bg^a, \bx \ra}{\sqrt{N}} \ge \kappa
        \quad \forall 1\le a\le M
    \rt\}.
\eeq
As explained in \cite{gardner1987maximum}, $S_N^M$ models the set of configurations of synaptic weights in a single-layer neural network that memorize all $M$ patterns $\bg^1,\ldots,\bg^M$.
Define the random variable $M_N = M_N(\kappa)$ as the largest $M$ such that $S_N^M \neq \emptyset$.
Then, the \textbf{capacity} of this model is defined as the ratio $M_N/N$, and models the maximum number of patterns this network can memorize per synapse.

Krauth and M\'ezard \cite{krauth1989storage} analyzed this model using the (non-rigorous) replica method from statistical physics.
They conjectured that as $N\to\infty$, the capacity concentrates around an explicit constant $\alpha_\star = \alpha_\star(\kappa)$, which is approximately $0.833$ for $\kappa = 0$ and is formally defined in Proposition~\ref{ppn:km-well-defd-kappa0} below.\footnote{\cite{krauth1989storage} studied a model with Bernoulli disorder, i.e. where the $g^a_i$ are i.i.d. samples from $\unif(\pm 1)$ rather than $\cN(0,1)$. As \cite{nakajima2023sharp} shows this model's sharp threshold sequence is universal with respect to any subgaussian disorder, we may work with gaussian disorder for convenience.}
This was part of a series of works in the statistical physics literature \cite{gardner1987maximum,gardner1988optimal,gardner1988space,krauth1989storage,mezard1989space} which analyzed various perceptron models using the replica or cavity methods and put forward detailed predictions for their behavior.
In particular, \cite{krauth1989storage} provided a conjecture for the limiting capacity of the Ising perceptron, while \cite{gardner1988optimal} gave an analogous conjecture for the spherical perceptron, where the spins $\bx$ belong to the sphere $\{\bx \in \bbR^N : \tnorm{\bx} = \sqrt{N}\}$ instead of $\Sigma_N$.

Ding and Sun \cite{ding2018capacity} proved that $\alpha_\star$ is a rigorous lower bound for the capacity, subject to a numerical condition that an explicit univariate function is maximized at $0$.
\begin{thm}{\cite[Theorem 1.1]{ding2018capacity}}
    \label{thm:ds-lb}
    Under Condition 1.2 therein, the following holds for the $\kappa = 0$ Ising perceptron.
    For any $\alpha < \alpha_\star$, $\liminf_{N\to\infty} \PP(M_N/N \ge \alpha) > 0$.
\end{thm}
Furthermore, \cite{xu2021sharp,nakajima2023sharp} showed that the capacity has a sharp threshold sequence, thereby improving the positive probability guarantee of Theorem~\ref{thm:ds-lb} to high probability.
Our main result is a matching upper bound for the capacity, subject to a similar numerical condition.
\begin{thm}
    \label{thm:main-kappa0}
    Under Condition~\ref{con:2varfn} below, the following holds for the $\kappa = 0$ Ising perceptron.
    For any $\alpha > \alpha_\star$, $\lim_{N\to\infty} \PP(M_N/N \ge \alpha) = 0$.
\end{thm}
\begin{con}
    \label{con:2varfn}
    The function $\sS_\star(\lambda_1,\lambda_2)$ defined in \eqref{eq:def-sS-star} satisfies $\sS_\star(\lambda_1,\lambda_2) \le 0$ for all $\lambda_1,\lambda_2 \in \bbR$.
\end{con}
See \S\ref{subsec:heuristic-1mt} for a discussion of this condition.
In particular $\sS_\star(1,0) = 0$ is a local maximum, and numerical plots suggest it is the unique global maximum.

Theorem~\ref{thm:main-kappa0} is a consequence of the more general Theorem~\ref{thm:main}, which states that $\alpha_\star(\kappa)$ upper bounds the capacity for general $\kappa$, under a number of numerical conditions depending on $\kappa$.
The most complicated of these is Condition~\ref{con:2varfn}, and we derive Theorem~\ref{thm:main-kappa0} by verifying the remaining conditions when $\kappa = 0$.
This computer-assisted verification is described in Appendix~\ref{app:numerics} and carried out in the attached Python 3 file using \verb|python-flint|, a rigorous library for interval arithmetic.

\subsection{Related work}
\label{subsec:related-work}

For the spherical perceptron, the capacity threshold of \cite{gardner1988optimal} has been proved rigorously for all $\kappa \ge 0$ \cite{shcherbina2003rigorous, stojnic2013another}.
(See also \cite{stojnic2013negative} for some work on the $\kappa < 0$ case.)
These works exploit the fact that the spherical perceptron with $\kappa \ge 0$ is a convex optimization problem.
The Ising perceptron does not have this property, and our understanding of it is comparatively less complete.
The replica heuristic also gives a prediction for the free energy of a positive-temperature version of this model \cite{gardner1988optimal, krauth1989storage}, which was verified by \cite{talagrand2000intersecting} at sufficiently high temperature using a rigorous version of the cavity method.
The works \cite{kim1998covering,talagrand1999intersecting} showed that for the $\kappa = 0$ perceptron, there exists $\eps>0$ such that $\eps \le M_N/N \le 1-\eps$ with high probability.
The breakthrough work of Ding and Sun \cite{ding2018capacity} showed that $\alpha_\star$ lower bounds the capacity for the $\kappa=0$ perceptron, conditional on a numerical assumption.
Very recently, \cite{altschuler2024capacity} showed that $0.847$ is a rigorous upper bound for the capacity in this model.
Recent works have also shown the replica-symmetric formula for the free energy at low constraint density in generalized perceptron models \cite{bolthausen2022gardner}, existence of a sharp threshold sequence \cite{xu2021sharp, nakajima2023sharp}, and universality in the disorder \cite{nakajima2023sharp}.
We also mention the works \cite{alaoui2022algorithmic, montanari2024tractability} on algorithms for the negative spherical perceptron.

Another recent line of work originating with \cite{aubin2019storage} studied the \textbf{symmetric binary perceptron}, where the constraints in \eqref{eq:solution-set} are replaced by $|\la \bg^a, \bx \ra| / \sqrt{N} \le \kappa$.
Symmetry makes this model significantly more tractable (see \S\ref{subsec:amp-conditioned-mt} for more discussion); a series of remarkable works have established the limiting capacity \cite{perkins2021frozen, abbe2022proof}, ``frozen 1-RSB" structure \cite{perkins2021frozen}, lognormal limit of partition function \cite{abbe2022proof}, and critical window \cite{altschuler2023critical, sah2023distribution}, and shed light on the performance of algorithms \cite{abbe2022binary, gamarnik2022algorithms, gamarnik2023geometric, barbier2023atypical}.

\subsection{Notation}

While we introduce other parameters over the course of the proof, unless stated otherwise we send $N\to\infty$ first, treating the remaining parameters as small or large constants.
Thus, we use $o_N(1)$ to denote a quantity vanishing with $N$, while notations like $o_\eps(1)$ denote quantities independent of $N$ tending to zero as the subscripted parameter tends to $0$ or $\infty$ (which will be clear from context).
We say an event occurs with high probability if it occurs with probability $1-o_N(1)$.
Further notations will be introduced in \S\ref{subsec:params-list}, before the main body of proofs.

\subsection*{Acknowledgements}

I would like to thank Mehtaab Sawhney for pointing me to the reference \cite{guionnet2000concentration}, and Will Perkins, Mehtaab Sawhney, Mark Sellke, and Nike Sun for helpful feedback on the manuscript.
I am also grateful to Andrea Montanari and Huy Tuan Pham for a collaboration that inspired parts of this work.
Thanks to Saba Lepsveridze for a helpful and motivating conversation.
This work was supported by a Google PhD Fellowship, NSF CAREER grant DMS-1940092, and the Solomon Buchsbaum Research Fund at MIT.

\section{Further background and proof outline}
\label{sec:technical-overview}

This section contains a technical overview of the paper, and is organized as follows.
In \S\ref{subsec:amp-conditioned-mt}, we review the AMP-conditioned moment method used in \cite{ding2018capacity} to prove the capacity lower bound and discuss the main difficulties of proving the upper bound.
In \S\ref{subsec:approx-contiguity-planted}, we outline a new approach based on reducing to a planted model and argue that if three primary inputs \ref{itm:reduction-top-triv}, \ref{itm:reduction-exist-crit}, \ref{itm:reduction-det-conc} hold, then the upper bound reduces to a tractable moment computation.
\S\ref{subsec:top-triv} discusses the most difficult input \ref{itm:reduction-top-triv}, and \S\ref{subsec:crit-near-late-amp} discusses the more straightforward inputs \ref{itm:reduction-exist-crit} and \ref{itm:reduction-det-conc}.
\S\ref{subsec:overview-planted-models} discusses related work involving planted models.
Finally, \S\ref{subsec:heuristic-1mt} heuristically carries out the aforementioned moment computation, explains how Condition~\ref{con:2varfn} emerges from it, and gives numerical evidence for Condition~\ref{con:2varfn} when $\kappa = 0$.

\subsection{AMP-conditioned moment method}
\label{subsec:amp-conditioned-mt}

A natural approach to studying the limiting capacity is the moment method.
Let $M = \alpha N$, and let $\bG \in \bbR^{M\times N}$ have rows $\bg^1,\ldots,\bg^M$.
Then let $S_N(\bG) = S_N^M$ (recall \eqref{eq:solution-set}) and $Z_N(\bG) = |S_N(\bG)|$.
If $\EE [Z_N(\bG)] \ll 1$, then $S_N(\bG)$ is w.h.p. empty, and if $\EE [Z_N(\bG)^2] / \EE[Z_N(\bG)]^2$ is bounded, then $S_N(\bG)$ is nonempty with positive probability.
If these two estimates hold for (respectively) $\alpha = \alpha_\star + \eps$ and $\alpha = \alpha_\star - \eps$, for any $\eps > 0$, this shows the limiting capacity is $\alpha_\star$.

Let $\bm_\star(\bG) = \fr{1}{|S_N(\bG)|} \sum_{\bx \in S_N(\bG)} \bx$ denote the barycenter of the solution set $S_N(\bG)$.
For models where $\bm_\star(\bG) = \bzero$, such as the symmetric binary perceptron \cite{aubin2019storage, perkins2021frozen, abbe2022proof}, this two-moment analysis often suffices to determine the limiting capacity.
However, due to the asymmetry of the activation in the present model, $\bm_\star(\bG)$ is typically macroscopic and random.
It is expected that for any $\alpha > 0$, large-deviations events in the location of $\bm_\star(\bG)$ dominate the first and second moments.
Thus $Z_N(\bG)$ is typically exponentially smaller than $\EE [Z_N(\bG)]$, and $\EE [Z_N(\bG)]^2$ exponentially smaller than $\EE [Z_N(\bG)^2]$, which causes the moment method to fail.
For example, for the $\kappa = 0$ perceptron, $\fr1N \log \EE [Z_N(\bG)]$ crosses zero at $\alpha = 1$, larger than $\alpha_\star(0) \approx 0.833$.

To overcome this difficulty, \cite{ding2018capacity} and \cite{bolthausen2019morita} (the latter for the Sherrington--Kirkpatrick model) concurrently developed a conditional moment method, in which one conditions on a suitable proxy for $\bm_\star(\bG)$ before computing moments.
The conditioning step effectively recenters spins around $\bm_\star(\bG)$, after which the moment method can potentially succeed.

The choice of conditioning is motivated by the TAP heuristic \cite{thouless1977solution} from statistical physics, which provides a powerful but non-rigorous framework to study this and other mean-field models.
The central object in this framework is a \textbf{TAP free energy} $\cF_\TAP(\bm,\bn)$, which is defined in \eqref{eq:TAP-fe} and can be thought of as a mean-field (dense graph) limit of the Bethe free energy of an appropriate message-passing system.
It is expected that $\cF_\TAP$ has a unique stationary point $(\bm,\bn) \in [-1,1]^N \times \bbR^M$, with the following interpretation: $\bm$ approximates the barycenter $\bm_\star(\bG)$ of $S_N(\bG)$, and for each $a\in [M]$, $n_a$ approximates a function of the average slack of the constraint $\la \bg^a, \bx \ra / \sqrt{N} \ge \kappa$ over solutions $\bx \in S_N(\bG)$.\footnote{More generally, the statistical physics literature predicts that the Gibbs measure --- here, the uniform measure on $S_N(\bG)$ --- decomposes as a convex combination of well-concentrated ``pure states," whose barycenters each approximate a stationary point of the TAP free energy \cite{mezard1987spin}. The present model is expected to be replica symmetric, meaning the entire Gibbs measure is one pure state.}
It is also predicted that $\bm$ and $\bn$ have specific coordinate profiles: for $(q_\star,\psi_\star)$ defined as the fixed point of a scalar recursion (see Condition~\ref{con:km-well-defd}) and $F = F_{1-q_\star}$ as in \eqref{eq:def-cE-F}, the prediction is that the coordinates of $\dbh = \th^{-1}(\bm)$ and $\hbh = F^{-1}(\hbh)$ have empirical distribution approximating $\cN(0,\psi_\star)$ and $\cN(0,q_\star)$.\footnote{Here and throughout, nonlinearities such as $\th^{-1}$ and $F^{-1}$ are applied coordinate-wise.}

An important fact we will exploit is that for fixed $(\bm,\bn)$, the stationarity condition $\nabla \cF_\TAP(\bm,\bn) = \bzero$ can be written as two \textbf{linear} equations in $\bG$.
These are the \textbf{TAP equations}, defined in \eqref{eq:TAP-eqn}.
Using this fact, we can define a \textbf{planted model}, which plays an important motivational role in \cite{ding2018capacity,bolthausen2019morita}: we first chooose $(\bm,\bn)$ with aforementioned coordinate profile, and then sample $\bG$ conditional on $\nabla \cF_\TAP(\bm,\bn) = \bzero$.
(This is different from the more well-known notion of planted model introduced in \cite{achlioptas2008algorithmic}, in that we are planting a TAP fixed point rather than a satisfying assignment; see \S\ref{subsec:overview-planted-models} for further discussion.)

If we imagine for a moment that $\bG$ were sampled from this planted model, then the moment method becomes tractable.
In this model, the law of $\bG$ conditional on $(\bm,\bn)$ remains gaussian because the TAP equations are linear in $\bG$, and the conditional first and second moments of $Z_N(\bG)$ can be computed.
They amount to tractable $O(1)$-dimensional optimization problems: for example, computing $\EE[Z_N(\bG) | \bm,\bn]$ amounts to optimizing the exponential-order contribution to the first moment from subsets of $\Sigma_N$ defined by their inner products with $\bm$ and $\dbh$ (see \S\ref{subsec:heuristic-1mt} for details).
The planted model removes the main difficulty of the macroscopically-fluctuating barycenter, giving the moment method a chance to succeed.

However, this planted model is different from the true model, in which the TAP solution $(\bm,\bn)$ depends on $\bG$ in a complicated way.
It is a priori unclear that these can be rigorously linked, because in the true model both existence and uniqueness of the TAP solution are not known.
To carry out this approach, \cite{ding2018capacity, bolthausen2019morita} instead condition on a sequence of \textbf{approximate message passing} (AMP) iterates $(\bm^0,\bn^0,\ldots,\bm^k,\bn^k)$ whose dependence on $\bG$ is explicit.
The AMP iteration was introduced in \cite{bolthausen2014iterative, bayati2011dynamics}, and is defined (roughly speaking, see \eqref{eq:amp-iteration-zero}) by iterating the TAP equations.
Its behavior can be understood through the powerful state evolution description of \cite{bolthausen2014iterative, bayati2011dynamics, javanmard2013state, berthier2020state}: for any $k$ not growing with $N$, state evolution exactly characterizes the limiting overlap structure of $(\bm^0,\ldots,\bm^k)$ and $(\bn^0,\ldots,\bn^k)$.
Using this description, it can be shown that the AMP iterates converge to an approximate stationary point of $\cF_\TAP$:
\beq
    \label{eq:overview-amp-converges}
    \lim_{k_1,k_2\to\infty}
    \plim_{N\to\infty}
    N^{-1/2} \tnorm{(\bm^{k_1},\bn^{k_1}) - (\bm^{k_2},\bn^{k_2})}
    =
    \lim_{k\to\infty}
    \plim_{N\to\infty}
    N^{-1/2} \tnorm{\nabla \cF_\TAP(\bm^k,\bn^k)}
    = 0.
\eeq
Here $\plim$ denotes limit in probability.
It is in this sense that the AMP iterates are a proxy for $(\bm,\bn)$.

While the main advantages of conditioning on the AMP filtration are explicit dependence on $\bG$ and state evolution, the main disadvantage is the greater complexity of the resulting moment calculation.
Although the law of $\bG$ conditional on $(\bm^0,\bn^0,\ldots,\bm^k,\bn^k)$ remains gaussian, the conditional first and second moments of $Z_N(\bG)$ are now $O(k)$-dimensional optimization problems, in which one optimizes over subsets of $\Sigma_N$ defined by their inner products with $\bm^0,\ldots,\bm^k$ and related vectors.
These problems are not in general tractable.
We note that \cite{bolthausen2019morita, bolthausen2022gardner} successfully carry out this optimization in their respective settings, but only at sufficiently high temperature or low constraint density.

An important insight of \cite{ding2018capacity} is that this approach still gives a tractable proof of the capacity lower bound, because --- to show a lower bound for $Z_N(\bG)$ --- one may truncate $Z_N(\bG)$ before computing moments.
They construct a truncation $\tZ_N(\bG)$ of $Z_N(\bG)$, restricting (among other conditions) to $\bx \in \Sigma_N$ with prescribed inner products with $\bm^0,\ldots,\bm^k$.
The conditional first moment of $\tZ_N(\bG)$ is then explicit, while the conditional second moment becomes a $1$-dimensional optimization.
\cite{ding2018capacity} shows that (under the aforementioned numerical condition) $\EE[\tZ_N(\bG)^2] / \EE[\tZ_N(\bG)]^2$ is bounded for any $\alpha < \alpha_\star$, which implies the capacity lower bound.

We mention that \cite{brennecke2022replica, bolthausen2022gardner} carry out similar truncated second moment arguments in their respective settings, and the former improves the parameter regime where the method of \cite{bolthausen2019morita} obtains the replica symmetric free energy lower bound for the Sherrington--Kirkpatrick model.

The main difficulty of the capacity upper bound is that truncation is no longer available.
Without it, proving the capacity upper bound within the AMP-conditioned moment method would require solving the above $O(k)$-dimensional optimization problem, which does not appear to be tractable.

\subsection{Approximate contiguity with planted model}
\label{subsec:approx-contiguity-planted}

Our proof revisits and justifies the planted model heuristic described above, where we select $(\bm,\bn)$ with appropriate coordinate profile and generate $\bG$ conditional on $\nabla \cF_\TAP(\bm,\bn) = \bzero$.
We will show that the true model is approximately contiguous to the planted model, in the sense of \eqref{eq:approximate-contiguity} below.
So, rather than conditioning on the AMP filtration, we can condition directly on $(\bm,\bn)$ after all.
The conditional first moment of $Z_N(\bG)$ then reverts to a simple optimization in two, rather than $O(k)$, dimensions.
This makes the capacity upper bound tractable.

The idea of passing by contiguity to a model with a planted TAP solution is also used in simultaneous joint work with A. Montanari and H. T. Pham \cite{huang2024sampling}, on sampling from the Gibbs measure of a spherical mixed $p$-spin glass in total variation by an algorithmic implementation of stochastic localization \cite{eldan2020taming, alaoui2022sampling}.
A similar inequality to \eqref{eq:approximate-contiguity} appears as Proposition 4.4(d) therein.
However, these two papers differ in both how this reduction is used, and how it is proved.
While \cite{huang2024sampling} develops a reduction similar to \eqref{eq:approximate-contiguity}, its main focus is to compute a high-precision estimate for the mean of a Gibbs measure, and the reduction to a planted model arises as a step in the analysis of this estimator.
In the present paper, the reduction \eqref{eq:approximate-contiguity} is itself the main technical step, but the proof of it is also more challenging.
Most notably, a key ingredient in the proof of \eqref{eq:approximate-contiguity}, in both the present paper and \cite{huang2024sampling}, is the uniqueness of the TAP fixed point in a certain region, see \ref{itm:reduction-top-triv} below.
Whereas this ingredient is available in the spin glass setting of \cite{huang2024sampling} from known results, showing it in our setting requires new ideas, described in detail in \S\ref{subsec:top-triv}.

We now state the approximate contiguity estimate.
For small $\ups>0$, let $\cS_\ups$ denote the set of $(\bm,\bn)$ whose coordinate profile is $\ups$-close (in a suitable metric, see \eqref{eq:Supsilon}) to that predicted by the TAP heuristic.
We will show, roughly speaking, that there exists $C = O(1)$ such that for any $\bG$-measurable event $\sE$,
\beq
    \label{eq:approximate-contiguity}
    \PP(\sE) \le
    C \sup_{(\bm,\bn) \in \cS_\ups}
    \PP(\sE | \nabla \cF_\TAP(\bm,\bn) = \bzero)^{1/2} + o_N(1).
\eeq
\begin{rmk}
	For reasons described below, we actually prove \eqref{eq:approximate-contiguity} for perturbations $\cF^\eps_\TAP$, $\cS_{\eps,\ups}$ of $\cF_\TAP$, $\cS_\ups$, and this qualification holds for the entire discussion below, even where not stated.
	These perturbations are defined in \eqref{eq:TAP-fe-eps} and \eqref{eq:Supsilon}, and the formal version of \eqref{eq:approximate-contiguity} is given in Lemma~\ref{lem:reduce-to-planted}.
\end{rmk}
We then take $\sE = \{S_N(\bG) \neq \emptyset\}$.
The first moment bound will show that (under Condition~\ref{con:2varfn}) this event has vanishing probability in the planted model for any $\alpha > \alpha_\star$.
Then \eqref{eq:approximate-contiguity} implies the conclusion.

Next, we discuss the proof of \eqref{eq:approximate-contiguity}.
The following two central ingredients establish uniqueness and existence of the critical point of $\cF_\TAP$ within the set $\cS_\ups$, with high probability in the true model.
\begin{enumerate}[label=(R\arabic*),ref=(R\arabic*)]
	\item \label{itm:reduction-top-triv} The expected number of critical points of $\cF_\TAP$ in $\cS_\ups$ is $1+o(1)$.
    \item \label{itm:reduction-exist-crit} With high probability, there exists a critical point of $\cF_\TAP$ in $\cS_\ups$.
\end{enumerate}
\begin{rmk}
	Although the TAP perspective predicts $\cF_\TAP$ has a unique critical point in the full input space, uniqueness in $\cS_\ups$ (and for the perturbed $\cF^\eps_\TAP$) suffices for our proof.
\end{rmk}
A short argument based on the Kac--Rice formula \cite{kac1948average, rice1944mathematical} (see \cite[Theorem 11.2.1]{adler2009random} for a textbook treatment) shows that \eqref{eq:approximate-contiguity} follows from \ref{itm:reduction-top-triv}, \ref{itm:reduction-exist-crit}, and the following additional input, which is a concentration condition on the change of volume term $|\det \nabla^2 \cF_\TAP(\bm,\bn)|$ in the Kac--Rice formula.
This argument is carried out in the proof of Lemma~\ref{lem:reduce-to-planted}, see \eqref{eq:kac-rice-for-reduction}.
\begin{enumerate}[label=(R\arabic*),ref=(R\arabic*)]
	\setcounter{enumi}{2}
    \item \label{itm:reduction-det-conc} There exists $C'=O(1)$ such that uniformly over $(\bm,\bn) \in \cS_\ups$,
    \[
        \EE [|\det \nabla^2 \cF_\TAP(\bm,\bn)|^2 \big| \nabla \cF_\TAP(\bm,\bn) = \bzero]^{1/2}
        \le C'
        \EE [|\det \nabla^2 \cF_\TAP(\bm,\bn)| \big| \nabla \cF_\TAP(\bm,\bn) = \bzero].
    \]
\end{enumerate}
\begin{rmk}
    \label{rmk:subexp-suffices}
    Since the probability in \eqref{eq:approximate-contiguity} is exponentially small, the proof can be carried out with $e^{o(N)}$ in place of $C$ in \eqref{eq:approximate-contiguity}.
    Consequently, showing \ref{itm:reduction-top-triv} and \ref{itm:reduction-det-conc} with $e^{o(N)}$ in place of $1+o(1)$, $O(1)$ also suffices.
\end{rmk}
\noindent Input \ref{itm:reduction-exist-crit} is proved constructively, by showing that AMP finds a critical point in the following sense.
\begin{enumerate}[label=(R\arabic*),ref=(R\arabic*)]
    \setcounter{enumi}{3}
    \item \label{itm:reduction-amp-select-crit} There exists $r_k = o_k(1)$ such that with high probability, $\cF_\TAP$ has a unique critical point in a $r_k \sqrt{N}$-neighborhood of the AMP iterate $(\bm^k,\bn^k)$ (which lies in $\cS_\ups$ by state evolution), for each sufficiently large $k$.
\end{enumerate}
Input \ref{itm:reduction-det-conc} will follow from a classic spectral concentration argument of \cite{guionnet2000concentration}.
We next discuss the proofs of \ref{itm:reduction-top-triv}, \ref{itm:reduction-amp-select-crit} and \ref{itm:reduction-det-conc}, in that order.

\subsection{Topological trivialization of TAP free energy}
\label{subsec:top-triv}

Condition \ref{itm:reduction-top-triv} is the most important input to the proof of \eqref{eq:approximate-contiguity}.
It is related to a remarkable line of work pioneered by \cite{fyodorov2004complexity, auffinger2013random}, on the landscapes of random high-dimensional functions.
This line of work has obtained expected critical point counts in a variety of settings, including spherical $p$-spin glasses \cite{auffinger2013complexity,auffinger2013random} (see \cite{subag2017complexity,auffinger2020number,subag2021concentration, arous2020geometry, huang2023constructive} for matching second moment estimates in certain cases) spiked tensor models \cite{arous2019landscape, auffinger2022sharp}, the TAP free energy for $\bbZ_2$-synchronization \cite{fan2021tap,celentano2021local}, bipartite spin glasses \cite{kivimae2023ground, mckenna2024complexity}, the elastic manifold \cite{arous2024landscape}, and generalized linear models \cite{maillard2020landscape}.
We also refer the reader to earlier non-rigorous work on this topic from the statistical physics literature \cite{bray1980metastable, parisi1995mean, crisanti2005complexity}.

One phenomenon studied in these works is \textbf{topological trivialization} \cite{fyodorov2014topology, fyodorov2015high, belius2022triviality, huang2023strong}, a phase transition where the number of critical points drops from $e^{cN}$ to $e^{o(N)}$, or often $O(1)$.
Proving \ref{itm:reduction-top-triv} amounts to showing \textbf{annealed topological trivialization} for $\cF^\eps_\TAP$ on $\cS_{\eps,\ups}$.

The strategy of these works is to calculate the expected number of critical points using the Kac--Rice formula, evaluating the integrand using random matrix theory.
Usually, the most complicated term in the integrand is the expected absolute value of the determinant of a random matrix.
The most well-understood application is where the landscape is a spherical mixed $p$-spin glass, in which case this random matrix is a GOE shifted by a scalar multiple of the identity.
For this case, an exact formula for this expected absolute determinant is known, see \cite[Lemma 3.3]{auffinger2013random}.
This makes the Kac--Rice calculation explicit and tractable.
In particular, \cite{fyodorov2015high, belius2022triviality} use this approach to determine the topologically trivial phase of spherical mixed $p$-spin glasses, and \cite{huang2024sampling} uses these results to establish \ref{itm:reduction-top-triv} for its application.
However, for other models, results on topological trivialization are not as readily available.

It may still be possible to show \ref{itm:reduction-top-triv} for our model in this way, by evaluating the more general random determinant that appears in the Kac--Rice formula.
This is the approach taken by \cite{fan2021tap} which, for $\bbZ_2$-synchronization at sufficiently large signal, shows annealed trivialization of suitably low-energy TAP solutions.
Their method bounds the random determinant in the Kac--Rice formula using free probability \cite{voiculescu1991limit}.
Furthermore, \cite{arous2022exponential} introduced a general tool for studying random determinants, showing that under mild conditions, their exponential order is the integral of $\log |\lambda|$ against the random matrix's limiting spectral measure.
The spectral measure can then be studied using free probability.

Using this approach, one can often express the exponential order of the expected number of critical points as a variational formula, in which one term is an implicitly-defined function arising from free probability \cite{kivimae2023ground, huang2023strong, arous2024landscape, mckenna2024complexity}.
This yields a plausible way to show \ref{itm:reduction-top-triv}: if we can show the variational formula for our model has value zero, annealed trivialization follows (in the sense of $e^{o(N)}$ expected critical points, which suffices by Remark~\ref{rmk:subexp-suffices}).
Recently, \cite{huang2023strong} showed that this method can be carried out for multi-species spherical spin glasses, and it in fact characterizes the topologically trivial phase.
Nonetheless, the variational formula is highly model-dependent --- the proof in \cite{huang2023strong} relies on a detailed understanding of a vector Dyson equation --- and it is unclear if this method can be carried out for our model.

We instead show annealed topological trivialization by a different, and arguably more conceptual, approach.
We will show that \ref{itm:reduction-top-triv} follows from the following variant of \ref{itm:reduction-amp-select-crit}:
\begin{enumerate}[label=(R\arabic*),ref=(R\arabic*)]
    \setcounter{enumi}{4}
    \item \label{itm:reduction-amp-returns-home} In a model where we plant a stationary point $(\bm,\bn) \in \cS_{\eps,\ups}$ of $\cF^\eps_\TAP$ (i.e. condition on $\nabla \cF^\eps_\TAP(\bm,\bn) = \bzero$), the same AMP iteration finds $(\bm,\bn)$, in the sense of \ref{itm:reduction-amp-select-crit}, with high probability.
\end{enumerate}
This implication is proved in Lemma~\ref{lem:top-triv}.
Heuristically, the reason \ref{itm:reduction-amp-returns-home} implies \ref{itm:reduction-top-triv} is that any realization of the disorder where $\cF^\eps_\TAP$ has $T>1$ stationary points in $\cS_{\eps,\ups}$ can arise in $T$ different planted models, and the event in \ref{itm:reduction-amp-returns-home} can hold in only one of these $T$ realizations.
If the expected number of critical points is too large, \ref{itm:reduction-amp-returns-home} cannot occur with the stated probability.
The input \ref{itm:reduction-amp-returns-home} can be proved by similar methods as \ref{itm:reduction-amp-select-crit}, as described in the next subsection.
This method yields the first proof of topological trivialization that does not directly evaluate the Kac--Rice formula.
We believe this is interesting in its own right.

\subsection{Critical point near late AMP iterates and determinant concentration}
\label{subsec:crit-near-late-amp}

This subsection discusses inputs \ref{itm:reduction-amp-select-crit}, \ref{itm:reduction-amp-returns-home}, and \ref{itm:reduction-det-conc}, in that order.
As state evolution ensures $\tnorm{\nabla \cF_\TAP(\bm^k,\bn^k)} = o_k(1) \sqrt{N}$ (recall \eqref{eq:overview-amp-converges}), \ref{itm:reduction-amp-select-crit} holds if, for example, $\cF_\TAP$ is $C$-strongly concave in a neighborhood of late AMP iterates for $C>0$ independent of $k$.
Recent works in the variational inference literature \cite{celentano2021local,celentano2023mean,celentano2024sudakov} develop tools to establish this local concavity, and using them prove analogs of \ref{itm:reduction-amp-select-crit} in several models.

In our setting, the fact that $\cF_\TAP$ is \textbf{not} strongly concave near late AMP iterates introduces some complications.
In fact, $\cF_\TAP$ is strongly concave in $\bm$, but convex --- and problematically, not strongly convex --- in $\bn$.
This issue is one reason we carry out the argument on a perturbation $\cF^\eps_\TAP$ of $\cF_\TAP$, and a similarly perturbed AMP iteration and set $\cS_{\eps,\ups}$.
(This perturbation serves several other purposes as well, described in Remark~\ref{rmk:why-perturb}.)
We will show that near late AMP iterates, $\cF^\eps_\TAP$ is strongly convex in $\bn$ and $\cG^\eps_\TAP(\bm) \equiv \inf_\bn \cF^\eps_\TAP(\bm,\bn)$ is strongly concave, which is enough to imply \ref{itm:reduction-amp-select-crit}.
Strong convexity of $\cF^\eps_\TAP$ in $\bn$ holds (deterministically) essentially by construction.

Our proof of local strong concavity of $\cG^\eps_\TAP$ uses an idea introduced in \cite{celentano2024sudakov}, to bound the Hessian at a late AMP iterate by applying a gaussian comparison inequality conditionally on the AMP iterates.
\cite{celentano2024sudakov} considers a setting where AMP is performed on disorder $\bW \sim \GOE(N)$ and the relevant Hessian is of the form $\bA + \bW$, where $\bA$ is a function of a late AMP iterate.
He develops a method to upper bound the top eigenvalue of this matrix by applying the Sudakov--Fernique inequality \cite{sudakov1971gaussian, fernique1975regularite, sudakov1979geometric} to the part of $\bW$ that remains random after observing the AMP iterates.
For us, the Hessian takes the form
\beq
    \label{eq:overview-nabla2}
    \nabla^2 \cG^\eps_\TAP(\bm,\bn)
    = \bA_1 + \fr1N \bG^\top \bA_2 \bG + \bDel,
\eeq
where $\bA_1, \bA_2$ are functions of $(\bm,\bn)$, and $\bDel$ is a low-rank term depending on both $\bG$ and $(\bm,\bn)$.
We can arrange $\cF^\eps_\TAP$ so that $\bDel$ does not contribute to the top eigenvalue.
However, the post-AMP Sudakov--Fernique inequality does not apply to the remaining part, because --- unlike for a GOE matrix --- the quadratic form induced by $\bG^\top \bA_2 \bG$ is not a gaussian process.
We instead recast the top eigenvalue as a minimax program, via the identity (for $\bA_2 \prec 0$)
\[
    \lambda_{\max}\lt(\bA_1 + \fr1N \bG^\top \bA_2 \bG\rt)
    = \sup_{\tnorm{\dbv} = 1}
    \inf_{\hbv \in \bbR^M}
    \lt\{
        \la \dbv, \bA_1 \dbv \ra
        - \la \hbv, \bA_2^{-1} \hbv \ra
        + \fr{2}{\sqrt{N}} \la \hbv, \bG \dbv \ra
    \rt\}.
\]
This can be bounded by Gordon's inequality \cite{gordon1985some, gordon1988milman} conditional on the AMP iterates.
Interestingly, the bound obtained in this way is sharp, matching a lower bound for the top eigenvalue obtained by free probability (see Remark~\ref{rmk:freeprob}).

The input \ref{itm:reduction-amp-returns-home} follows similarly to \ref{itm:reduction-amp-select-crit}.
We will show that with high probability over the planted model, late AMP iterates are approximate critical points of $\cF^\eps_\TAP$, near which $\cF^\eps_\TAP(\bm,\cdot)$ is strongly convex and $\cG^\eps_\TAP$ is strongly concave.
While the law of the disorder is different under the planted model, it remains gaussian and a similar analysis can be carried out.

We turn to \ref{itm:reduction-det-conc}.
An argument of \cite{guionnet2000concentration} implies that if a symmetric $\bX \in \bbR^{N\times N}$ has independent (not necessarily centered or identically distributed) entries on and above the diagonal with uniformly bounded log-Sobolev constant, then $\fr{1}{\sqrt{N}} \bX$ enjoys a strong spectral concentration property: any $1$-Lipschitz spectral trace has $O(1)$-scale subgaussian fluctuations.
We will see that conditional on $\nabla \cF^\eps_\TAP(\bm,\bn) = \bzero$, $\det \nabla^2 \cF^\eps_\TAP(\bm,\bn)$ is a nonrandom multiple of $\det \nabla^2 \cG^\eps_\TAP(\bm,\bn)$, which has form \eqref{eq:overview-nabla2}.
The entries of this matrix are not independent, but we can rewrite it via the classical trick
\balnn
    \label{eq:overview-hermitianization}
    \det \lt(\bA_1 + \fr1N \bG^\top \bA_2 \bG \rt)
    &= \det \bX, &
    \bX = \begin{bmatrix}
        \bA_1 & \fr{1}{\sqrt{N}} \bG^\top \\
        \fr{1}{\sqrt{N}} \bG & -\bA_2^{-1}
    \end{bmatrix}.
\ealnn
Conditional on $\nabla \cF^\eps_\TAP(\bm,\bn) = \bzero$, the matrices $\bA_1, \bA_2$ are nonrandom while $\bG$ has a (noncentered) gaussian law.
Thus the result of \cite{guionnet2000concentration} applies to $\bX$.
(A slightly more elaborate version of \eqref{eq:overview-hermitianization} also accounts for the random low-rank spike $\bDel$ in \eqref{eq:overview-nabla2}, see \eqref{eq:hermitianization}.)

From the above discussion, conditional on $\nabla \cF^\eps_\TAP(\bm,\bn) = \bzero$, $\cF^\eps_\TAP(\bm,\cdot)$ is strongly convex near $\bn$ and $\cG^\eps_\TAP$ is w.h.p. strongly concave near $\bm$.
This implies that the spectrum of $\nabla^2 \cF^\eps_\TAP(\bm,\bn)$, and thus $\bX$, is bounded away from zero, and provides the final ingredient to prove \ref{itm:reduction-det-conc}: since $x\mapsto \log |x|$ is $O(1)$-Lipschitz away from zero, $\log |\det \bX|$ is approximately a $O(1)$-Lipschitz spectral trace, which has $O(1)$-scale subgaussian fluctuations by \cite{guionnet2000concentration}.

\begin{rmk}
    The fact that this log determinant has $O(1)$-scale fluctuations is only possible because the spectrum is bounded away from zero.
    For Wigner or Ginibre matrices, two examples of random matrices whose limiting bulk spectrum does include zero, the log determinant is known to have $\Theta(\sqrt{\log N})$ fluctuations \cite{tao2012central, nguyen2014random}, which diverges with $N$.
\end{rmk}

\subsection{On planted models}
\label{subsec:overview-planted-models}

Reducing to a planted model is a powerful tool in the analysis of random functions.
This technique was introduced in the seminal work \cite{achlioptas2008algorithmic} and has seen a wide range of applications in the past decade.
The underlying idea is to show contiguity of the original model with a planted version, defined as the null model conditioned on having a particular (randomly chosen) solution.
If this holds, properties of the null model can be deduced from the planted version, which is often easier to analyze.

A frequent application of this method is to probe the local landscape around a typical solution.
This is the original application of \cite{achlioptas2008algorithmic}: contiguity implies that the landscape around a typical solution to the null model can be approximated by the landscape around the planted solution in the planted model.
From this, \cite{achlioptas2008algorithmic} shows the existence of a shattering transition in several random constraint satisfaction problems.
This approach has since also been used to show ``frozen 1RSB" structure in the symmetric binary perceptron \cite{perkins2021frozen, abbe2022proof} and shattering in the Gibbs measures of spherical spin glasses \cite{alaoui2023shattering}.
In a similar spirit, \cite{huang2024sampling} passes to a model with a planted TAP solution to obtain a high-precision estimate of the magnetization of a spherical spin glass.

In other applications, including the present work, the object of interest is not the local landscape, but the planted model is nonetheless simpler to analyze than the null model.
Such applications include the RS free energy of random constraint satisfaction problems \cite{bapst2015condensation, bapst2016condensation, coja2017information, coja2018charting, coja2020replica}, the 1RSB free energy of random regular NAE-SAT \cite{sly2022number}, and the Parisi formula for spherical spin glasses in the RS and zero-temperature 1RSB phases \cite{huang2023constructive}.
Passage to a planted model is also a crucial tool in the analysis of sampling algorithms based on stochastic localization  \cite{alaoui2022sampling, alaoui2023sampling}.

\subsection{First moment in planted model}
\label{subsec:heuristic-1mt}

In this subsection, we give a heuristic calculation of the first moment of $Z_N(\bG)$ in the planted model.
The function $\sS_\star(\lambda_1,\lambda_2)$ appearing in Condition~\ref{con:2varfn} arises from this calculation, and under this condition the first moment method succeeds.
At the end of this subsection, we also give numerical evidence for Condition~\ref{con:2varfn} when $\kappa = 0$.

We work at constraint density $\alpha_\star$, setting $M = \lfloor \alpha_\star N\rfloor$ and $\bG, S_N(\bG), Z_N(\bG)$ as above with this $M$.
Let $\PP^{\bm,\bn}_\Pl$ and $\EE^{\bm,\bn}_\Pl$ denote probability and expectation w.r.t. the model conditional on $\nabla \cF_\TAP(\bm,\bn) = \bzero$.
We will argue that under Condition~\ref{con:2varfn}, $\EE^{\bm,\bn}_\Pl Z_N(\bG) = e^{o(N)}$.
Then, at any constraint density $\alpha > \alpha_\star$, the $(\alpha - \alpha_\star)N$ additional constraints will make this moment exponentially small.

This argument will be made rigorous in \S\ref{sec:planted-mt}.
Per the above discussion, the rigorous version of this argument will plant a critical point of $\cF^\eps_\TAP$ rather than $\cF_\TAP$.

We first define the function $\sS_\star$.
Let $(q_0,\psi_0) = (q_\star(\alpha_\star,\kappa),\psi_\star(\alpha_\star,\kappa))$ be defined by Condition~\ref{con:km-well-defd}.
As discussed in \S\ref{subsec:amp-conditioned-mt}, these are the variances of the (gaussian) coordinate empirical measures of $\hbh$, $\dbh$ predicted by the TAP heuristic, at constraint density $\alpha_\star$.
Let $\dbH \sim \cN(0,\psi_0)$ and $\hbH \sim \cN(0,q_0)$.
These two random variables may be defined on different probability spaces, as all expectations in the below formulas will involve random variables from only one space.
Let $\bM = \th(\dbH)$ and $\bN = F_{1-q_0}(\hbH)$.
For any measurable $\bLam : \bbR \to [-1,1]$, define
\beq
    \label{eq:def-ent}
    \ent(\bLam) = \EE \cH\lt(\fr{1 + \bLam(\dbH)}{2}\rt),
\eeq
where $\cH(x) = -x\log x - (1-x) \log(1-x)$ is the binary entropy function.
Let $\Psi$ be the complementary gaussian cumulative density function defined in \eqref{eq:gaussian-density-cdf}.
For $s \ge 0$, define
\balnn
    \label{eq:def-sS-star-bLam-nu}
    \sS_\star(\bLam,s)
    &= \fr12 s^2 \psi_0
    + \ent(\bLam)
    + \alpha_\star \EE \log \Psi \lt\{
        \fr{
            \kappa
            - \fr{\EE[\bM \bLam(\dbH)]}{q_0} \hbH
            - \fr{\EE[\dbH \bLam(\dbH)]}{\psi_0} \bN
        }{
            \sqrt{1 - \fr{\EE[\bM \bLam(\dbH)]^2}{q_0}}
        }
        + s \bN
    \rt\}.
\ealnn
Finally, let $\bLam_{\lambda_1,\lambda_2}(x) = \th(\lambda_1 x + \lambda_2 \th(x))$ and define
\balnn
    \label{eq:def-sS-star}
    \sS_\star(\bLam) &= \inf_{s \ge 0} \sS_\star(\bLam,s), &
    \sS_\star(\lambda_1,\lambda_2) &= \sS_\star(\bLam_{\lambda_1,\lambda_2}).
\ealnn
These quantities have the following physical meanings.
$\dbH, \hbH, \bM, \bN$ are the coordinate distributions of $\dbh, \hbh, \bm, \bn$.
$\bLam$ specifies a set $\Sigma_N(\bLam) \subseteq \Sigma_N$ of points $\bx$ where $x_i$ has ``conditional average" $\bLam(\dh_i)$, in the sense that (informally, see \eqref{eq:SigmaN-subset-with-profile})
\beq
	\label{eq:SigmaN-subset-with-profile-heuristic}
    \fr{1}{\#\{i\in [N]: \dh_i \approx \dh\}}
    \sum_{i \in [N] : \dh_i \approx \dh} x_i
    \approx \bLam(\dh),
    \qquad \forall \dh \in \bbR.
\eeq
Note that $\ent(\bLam)$ is the entropy of this set, that is (see Lemma~\ref{lem:planted-mt-enumeration})
\beq
    \label{eq:heuristic-volume}
    \fr1N \log |\Sigma_N(\bLam)|
    \simeq
    \ent(\bLam).
\eeq
Here and throughout, $\simeq$ denotes equality up to additive $o_N(1)$.

Let $Z_N(\bG,\bLam) = |S_N(\bG) \cap \Sigma_N(\bLam)|$ denote the number of solutions with profile $\bLam$.
We will see that for all $s \ge 0$, $\sS_\star(\bLam,s)$ upper bounds the exponential order of $\EE^{\bm,\bn}_\Pl Z_N(\bG,\bLam)$.
Thus $\sS_\star(\bLam)$ also upper bounds this quantity, and $\EE^{\bm,\bn}_\Pl Z_N(\bG)$ is bounded (heuristically) by Laplace's principle:
\[
    \fr1N \log \bbE^{\bm,\bn}_\Pl Z_N(\bG)
    \simeq \sup_\bLam \lt\{\fr1N \log \bbE^{\bm,\bn}_\Pl Z_N(\bG,\bLam)\rt\}
    \le \sup_\bLam \sS_\star(\bLam) + o_N(1).
\]
While this supremum is a priori an infinite-dimensional optimization problem, the following observation reduces it to two dimensions.
For any $a_1,a_2$, a Lagrange multipliers calculation (see Lemma~\ref{lem:planted-mt-lm}) shows that the maximum of $\ent(\bLam)$ subject to $\EE[\dbH \bLam(\dbH)] = a_1$, $\EE [\bM \bLam(\dbH)] = a_2$ is attained by $\bLam$ of the form $\bLam_{\lambda_1,\lambda_2}$.
As the remaining terms in $\sS_\star(\bLam,s)$ depend on $\bLam$ only through $\EE[\dbH \bLam(\dbH)]$ and $\EE [\bM \bLam(\dbH)]$, we may restrict attention to $\bLam$ of this form.
Thus
\[
    \fr1N \log \bbE^{\bm,\bn}_\Pl Z_N(\bG)
    \le \sup_{\lambda_1,\lambda_2} \sS_\star(\lambda_1,\lambda_2) + o_N(1).
\]
This implies $\bbE^{\bm,\bn}_\Pl Z_N(\bG) = e^{o(N)}$ under Condition~\ref{con:2varfn}.

We next argue that $\sS_\star(\bLam,s)$ upper bounds the exponential order of $\EE^{\bm,\bn}_\Pl Z_N(\bG,\bLam)$, as claimed above.
Due to \eqref{eq:heuristic-volume}, it suffices to bound the probability that some $\bx \in \Sigma_N(\bLam)$ satisfies all constraints.
The planted model has the following law.
Let $\dbh \in \bbR^N$, $\hbh \in \bbR^M$ have coordinate distributions approximating $\cN(0,\psi_0)$, $\cN(0,q_0)$, and let $\bm = \th(\dbh)$, $\bn = F_{1-q_0}(\hbh)$.
A gaussian conditioning calculation (see Corollary~\ref{cor:conditional-law-correct-profile}) shows that conditional on $\nabla \cF_\TAP(\bm,\bn) = \bzero$,
\[
    \fr{\bG}{\sqrt{N}} \stackrel{d}{=}
    \fr{\hbh \bm^\top}{Nq_0}
    + \fr{\bn \dbh^\top}{N\psi_0}
    + \fr{P_\bn^\perp \tbG P_\bm^\perp}{\sqrt{N}}
    + o_N(1).
\]
Here $\tbG$ is an i.i.d. copy of $\bG$, $P_\bm^\perp$ denotes the projection operator to the orthogonal complement of $\bm$, and $o_N(1)$ is a matrix of operator norm $o_N(1)$.
For any $\bx \in \Sigma_N(\bLam)$, we have $\fr1N \la \bm, \bx \ra \simeq \EE[\bM \bLam(\dbH)]$ and $\fr1N \la \dbh, \bx \ra \simeq \EE[\dbH \bLam(\dbH)]$.
So,
\[
    \fr{\bG \bx}{\sqrt{N}}
    \stackrel{d}{=}
    \fr{\EE[\bM \bLam(\dbH)]}{q_0} \hbh
    + \fr{\EE[\dbH \bLam(\dbH)]}{\psi_0} \bn
    + \sqrt{1 - \fr{\EE[\bM \bLam(\dbH)]^2}{q_0}} \tbg + o(\sqrt{N}),
\]
where $\tbg \sim \cN(0,P_\bn^\perp)$ and $o(\sqrt{N})$ denotes a vector of norm $o(\sqrt{N})$.
Thus
\beq
    \label{eq:heuristic-probability-prelim}
    \fr1N \log \bbP^{\bm,\bn}_\Pl \lt(
        \fr{\bG \bx}{\sqrt{N}} \ge \kappa \bone
    \rt)
    \simeq \fr1N \log \PP \lt\{
        \tbg \ge \fr{
            \kappa \bone
            - \fr{\EE[\bM \bLam(\dbH)]}{q_0} \hbh
            - \fr{\EE[\dbH \bLam(\dbH)]}{\psi_0} \bn
        }{
            \sqrt{1 - \fr{\EE[\bM \bLam(\dbH)]^2}{q_0}}
        }
    \rt\}.
\eeq
This can be bounded by a change of measure calculation also used in \cite{ding2018capacity}.
Let $\hbg \sim \cN(s \bn, \bI_N)$ for any $s \ge 0$.
Note that conditional on $\la \hbg, \bn \ra = 0$, we have $\hbg =_d \tbg$.
So, if $S$ denotes the event in \eqref{eq:heuristic-probability-prelim}, then
\[
    \PP(\tbg \in S)
    \le \fr{\PP(\hbg \in S)}{\PP(\la \hbg, \bn \ra \approx 0)}
    \approx \exp\lt(\fr12 s^2 \psi_0 N\rt) \PP(\hbg \in S).
\]
Since $\hbh$ has coordinate distribution $\hbH$, this implies (see Lemma~\ref{lem:planted-mt-probability} for formal statement) that \eqref{eq:heuristic-probability-prelim} is bounded by
\[
    \fr12 s^2 \psi_0
    + \alpha_\star \EE \log \Psi \lt(
        \fr{
            \kappa
            - \fr{\EE[\bM \bLam(\dbH)]}{q_0} \hbH
            - \fr{\EE[\dbH \bLam(\dbH)]}{\psi_0} \bN
        }{
            \sqrt{1 - \fr{\EE[\bM \bLam(\dbH)]^2}{q_0}}
        }
        + s \bN
    \rt).
\]
Combining with \eqref{eq:heuristic-volume} shows that $\fr1N \log \EE^{\bm,\bn}_\Pl Z_N(\bG,\bLam) \le \sS_\star(\bLam,s) + o_N(1)$.

We conclude this subsection with a discussion of Condition~\ref{con:2varfn}.
We expect $\bm$ to approximate the barycenter of $S_N(\bG)$, and therefore that $\sS_\star(\lambda_1,\lambda_2)$ is maximized by $(\lambda_1,\lambda_2) = (1,0)$, corresponding to $\bLam_{\lambda_1,\lambda_2}(\dbH) = \th(\dbH) = \bM$.
Let
\[
    \osS_\star(\lambda_1,\lambda_2) = \sS_\star(\bLam_{\lambda_1,\lambda_2},\sqrt{1-q_0}),
\]
which is an upper bound for $\sS_\star$.
\begin{lem}[Proved in \S\ref{sec:planted-mt}]
    \label{lem:sS-zero}
    The following holds.
    \begin{enumerate}[label=(\alph*),ref=(\alph*)]
        \item \label{itm:sS-zero-attain-sup} The function $\sS_\star(\lambda_1,\lambda_2)$ attains its supremum on $\bbR^2$.
        \item \label{itm:sS-zero-val} $\sS_\star(1,0) = \osS_\star(1,0) = 0$.
        \item \label{itm:sS-zero-1deriv} $\nabla \sS_\star(1,0) = \nabla \osS_\star(1,0) = 0$.
        \item \label{itm:sS-zero-2deriv-bd} $\nabla^2 \sS_\star(1,0) \preceq \nabla^2 \osS_\star(1,0)$
    \end{enumerate}
\end{lem}
\begin{clm}[Proved in Appendix~\ref{app:numerics}]
    \label{clm:sS-zero-2deriv}
    For $\kappa = 0$, there exists $C>0$ such that $\nabla^2 \osS_\star(1,0) \preceq -CI$.
\end{clm}
Lemma~\ref{lem:sS-zero} is proved for all $\kappa$, while Claim~\ref{clm:sS-zero-2deriv} is verified numerically for $\kappa = 0$ using rigorous interval arithmetic.
Together, they imply that for $\kappa = 0$, $(1,0)$ is a local maximum of $\sS_\star$ and $\osS_\star$.
In Figure~\ref{fig:plot}, we provide a plot of $\osS_\star$ for the case $\kappa = 0$.
This gives numerical evidence that $\osS_\star$, and therefore $\sS_\star$, has global maximum $(1,0)$.
\begin{figure}
    \begin{subfigure}[b]{.45\textwidth}
        \centering
        \includegraphics[width=.9\linewidth]{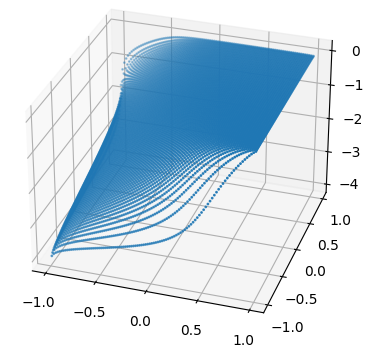}
        \caption{$x,y \in [-1,1]^2$}
        \label{subfig:plot-main}
    \end{subfigure}
    \begin{subfigure}[b]{.45\textwidth}
        \centering
        \includegraphics[width=.9\linewidth]{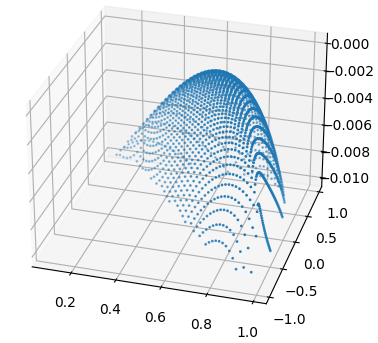}
        \caption{$\osS_\star(\th^{-1}(x),\th^{-1}(y)) \ge -0.01$}
        \label{subfig:plot-zoom}
    \end{subfigure}
    \caption{Plots of $(x,y) \mapsto \osS_\star(\th^{-1}(x),\th^{-1}(y))$ for $\kappa = 0$.
    Figure~\ref{subfig:plot-main} plots over $x,y \in [-1,1]^2$, while Figure~\ref{subfig:plot-zoom} restricts to inputs with $\osS_\star(\th^{-1}(x),\th^{-1}(y)) \ge -0.01$. The plots lie below $0$, and from Figure~\ref{subfig:plot-zoom} it appears the unique maximizer is $(x,y) = (\th(1),0)$, corresponding to $(\lambda_1,\lambda_2) = (1,0)$.}
    \label{fig:plot}
\end{figure}

\section{Formal statement of results}
\label{sec:results}

In this section we state our main result for general $\kappa$, Theorem~\ref{thm:main}.
We also reduce Theorem~\ref{thm:main} to two primary inputs: approximate contiguity with a planted model (Lemma~\ref{lem:reduce-to-planted}) and the upper bound for the first moment in the planted model (Proposition~\ref{ppn:planted-mt}), which are proved in \S\ref{sec:planted-reduction}--\ref{sec:local-concavity} and \S\ref{sec:planted-mt}.

\subsection{Krauth--M\'ezard threshold}
\label{subsec:km-threshold}

We first define the threshold $\alpha_\star$ conjectured by \cite{krauth1989storage}, following the presentation of \cite{ding2018capacity}.
Define the standard gaussian density and complementary CDF by
\balnn
    \label{eq:gaussian-density-cdf}
    \varphi(x) &= \fr{\exp(-x^2/2)}{(2\pi)^{1/2}}, &
    \Psi(x) &= \int_x^\infty \phi(u)~\de u.
\ealnn
Fix once and for all $\kappa \in \bbR$.
For $q\in [0,1)$, define\footnote{The function $F_{1-q}$ is denoted $F_q$ in \cite{ding2018capacity}. We change this notation to be consistent with the meaning of $F_{\eps,\varrho}$ \eqref{eq:perturbed-nonlinearities} appearing in our proofs.}
\balnn
    \label{eq:def-cE-F}
    \cE(x) &= \fr{\varphi(x)}{\Psi(x)}, &
    F_{1-q}(x) &= \fr{\cE}{(1-q)^{1/2}}\lt(\fr{\kappa - x}{(1-q)^{1/2}}\rt).
\ealnn
For $\psi \ge 0$ and $Z \sim \cN(0,1)$, further define
\baln
    P(\psi) &= \EE [\th(\psi^{1/2} Z)^2], &
    R_\alpha(q) &= \alpha \EE [F_{1-q}(q^{1/2} Z)^2],
\ealn
and define the Gardner free energy (or Gardner volume formula) by
\beq
    \label{eq:RS-fe}
    \sG(\alpha,q,\psi)
    = - \fr{(1-q)\psi}{2}
    + \EE \log (2 \ch (\psi^{1/2} Z))
    - \alpha \EE \log \Psi\lt(\fr{\kappa - q^{1/2} Z}{(1-q)^{1/2}}\rt).
\eeq
The physical meanings of these formulas are best understood in terms of a heuristic derivation of the TAP free energy $\cF_\TAP(\bm,\bn)$ and TAP equations, which we explain next.
(These quantities will be formally defined in \eqref{eq:TAP-fe}, \eqref{eq:TAP-eqn}.)
If we regard $\bG$ as a complete bipartite factor graph on $N$ variables and $M$ constraints, we can study the perceptron model by the standard \textbf{belief propagation} (BP) equations \cite[Chapter 14]{mezard2009information}.
In the mean-field (dense graph) limit, these equations simplify considerably.
First, because the influence of any particular message is small, all the messages emanating from a particular variable $i\in [M]$ (resp. constraint $a\in [M]$) can be consolidated into a single message $m_i$ (resp. $n_a$).
The TAP variables $(\bm,\bn)$ thus represent these consolidated messages.
The BP equations then become the TAP equations, and the \textbf{Bethe free energy} of this BP system becomes the TAP free energy.
See \cite{mezard2017mean} for an example of this derivation in a related model.

Moreover, by central limit theorem considerations, we expect that the coordinates of $\dbh = \th^{-1}(\bm)$ and $\hbh = F_{1-\tnorm{\bm}^2/N}^{-1}(\bn)$ have gaussian empirical measure.
Let these gaussians have variance $\psi$ and $q$, respectively; this is the physical meaning of these parameters.
Then the BP consistency relations require that $\psi,q$ satisfy the fixed-point equation $q = P(\psi)$, $\psi = R_\alpha(q)$, and the corresponding Bethe free energy is precisely $\sG(\alpha,q,\psi)$.
Finally, we expect $\alpha_\star$ to be the constraint density where this Bethe free energy crosses zero.
Under the following condition, which was verified in \cite{ding2018capacity} for $\kappa = 0$, this heuristic picture can be formalized into a definition of $\alpha_\star$.
\begin{con}
    \label{con:km-well-defd}
    There exist $0 < \alpha_\lb < \alpha_\ub$ and $0< q_\lb < q_\ub < 1$ (depending on $\kappa$) such that the following holds.
    For any $\alpha \in (\alpha_\lb, \alpha_\ub)$,
    \[
        \sup_{q\in (q_\lb, q_\ub)}
        (P \circ R_\alpha)'(q)
        < 1,
    \]
    and there is a unique $q_\star = q_\star(\alpha,\kappa) \in (q_\lb,q_\ub)$ such that $q_\star = P(R_\alpha(q_\star))$.
    Let $\psi_\star = \psi_\star(\alpha,\kappa) = R_\alpha(q_\star)$.
    For $\alpha \in (\alpha_\lb,\alpha_\ub)$, the function $\sG_\star(\alpha) = \sG(\alpha,q_\star(\alpha,\kappa),\psi_\star(\alpha,\kappa))$ is strictly decreasing, with a unique root $\alpha_\star = \alpha_\star(\kappa)$.
\end{con}
\begin{ppn}[{\cite[Proposition 1.3]{ding2018capacity}}]
    \label{ppn:km-well-defd-kappa0}
    For $\kappa = 0$, Condition~\ref{con:km-well-defd} holds for $\alpha_\lb = 0.833078599$, $\alpha_\ub = 0.833078600$, $q_\lb = 0.56394907949$, $q_\ub = 0.56394908030$.
\end{ppn}

\subsection{Main result}
\label{subsec:main-result}

Throughout, let $\alpha_\star = \alpha_\star(\kappa)$ and $(q_0,\psi_0) = (q_\star(\alpha_\star,\kappa),\psi_\star(\alpha_\star,\kappa))$ be given by Condition~\ref{con:km-well-defd}.
We now introduce two more numerical conditions needed for our main result, which will be verified for $\kappa = 0$ in Appendix~\ref{app:numerics} using rigorous interval arithmetic.
In the below formulas, let $Z \sim \cN(0,1)$.
\begin{con}
    \label{con:amp-works}
    We have $\alpha_\star \EE[\th'(\psi_0^{1/2}Z)^2] \EE [F'_{1-q_0}(q_0^{1/2} Z)^2] < 1$.
\end{con}
\begin{con}
    \label{con:local-concavity}
    Define the functions $m : (-1,+\infty) \to (0,+\infty)$ and $\hf_0 : \bbR \to (0,+\infty)$ by
    \baln
        m(z) &= \EE[(z + \ch^2(\psi_0^{1/2} Z))^{-1}], \\
        \hf_0(x) &= - \fr{F'_{1-q_0}(x)}{1 + (1-q_0) F'_{1-q_0}(x)}
        = \fr{\cE'((\kappa-x)/(1-q_0)^{1/2}))}{(1-q_0)(1 - \cE'((\kappa-x)/(1-q_0)^{1/2}))}.
    \ealn
    (By Lemma~\ref{lem:E-derivative-bds}\ref{itm:cEpr} below, $\cE'$ has image in $(0,1)$, and thus $\hf_0(x) > 0$.)
    Then, for $d_0 = \alpha_\star \EE[F'_{1-q_0}(q_0^{1/2} Z)]$ and $\lambda : (-1,+\infty) \to \bbR$ defined by
    \[
        \lambda(z) = z - \alpha_\star \EE\lt[
            \fr{\hf_0(q_0^{1/2} Z)}{1 + m(z) \hf_0(q_0^{1/2}Z)}
        \rt] - d_0,
    \]
    we have $\lambda_0 \equiv \inf_{z>-1} \lambda(z) < 0$.
\end{con}
\noindent The following lemma shows that minimizer of $\lambda$ exists and is the unique root of a decreasing function, and it suffices to check Condition~\ref{con:local-concavity} at the value $\lambda(z_0)$.
\begin{lem}[Proved in \S\ref{sec:local-concavity}]
    \label{lem:freeprob-well-defd}
    The function $\lambda$ is differerentiable with $\lambda'(z) = 1 - \alpha_\star \theta(z)$, where $\theta : (-1,+\infty) \to (0,+\infty)$ is defined by
    \[
        \theta(z) = \EE[(z + \ch^2(\psi_0^{1/2} Z))^{-2}]
        \EE\lt[\lt(
            \fr{\hf_0(q_0^{1/2}Z)}{1 + m(z) \hf_0(q_0^{1/2} Z)}
        \rt)^2\rt].
    \]
    Moreover $\theta$ is continuous and strictly decreasing, with
    \baln
        \lim_{z \downarrow -1} \theta(z)
        &= +\infty, &
        \lim_{z \uparrow +\infty} \theta(z)
        &= 0.
    \ealn
    In particular $\theta$ has a well-defined inverse $\theta^{-1} : (0,+\infty) \to (-1,+\infty)$, and $\lambda$ is strictly convex on $(-1,+\infty)$ with minimizer $z_0 = \theta^{-1}(\alpha_\star^{-1})$.
    Thus $\lambda_0$ defined in Condition~\ref{con:local-concavity} satisfies $\lambda_0 = \lambda(z_0)$.
\end{lem}
\begin{thm}[Main result, general $\kappa$]
    \label{thm:main}
    For any $\kappa \in \bbR$, under Conditions~\ref{con:2varfn}, \ref{con:km-well-defd}, \ref{con:amp-works}, and \ref{con:local-concavity} the following holds.
    For any $\alpha > \alpha_\star(\kappa)$, we have $\lim_{N\to\infty} \PP(M_N(\kappa) / N \ge \alpha) = 0$.
\end{thm}
\begin{rmk}
    The conditions in Theorem~\ref{thm:main} serve the following purposes.
    \begin{itemize}
        \item Condition~\ref{con:2varfn} controls the first moment of the partition function in the planted model.
        \item Condition~\ref{con:km-well-defd} makes the threshold $\alpha_\star(\kappa)$ well-defined.
        \item Condition~\ref{con:amp-works} ensures that the AMP iterates converge in the sense of \eqref{eq:overview-amp-converges}.
        \item Condition~\ref{con:local-concavity} ensures that $\cG^\eps_\TAP$ (see \S\ref{subsec:crit-near-late-amp}) is locally concave near late AMP iterates.
    \end{itemize}
\end{rmk}
\noindent With the exception of Appendix~\ref{app:numerics}, we will assume all conditions in Theorem~\ref{thm:main} without further notice.

\subsection{Proof of Theorem~\ref{thm:main}}
\label{subsec:main-proof}

We will carry out nearly the entire proof at constraint density $\alpha_\star$.
Thus, we set $M = \lfloor \alpha_\star N \rfloor$ and define $\bG \in \bbR^{M\times N}$ and $Z_N(\bG)$ as above.

The main step of the proof is a reduction to a planted model, formalized by Lemma~\ref{lem:reduce-to-planted} below.
Let $\PP$ denote the law of $\bG$ with i.i.d. $\cN(0,1)$ entries, and let $\bbP^{\bm,\bn}_{\eps,\Pl}$ be the planted law defined in Definition~\ref{dfn:planted-model}.
This is the law of $\bG$ conditional on $\nabla \cF^\eps_\TAP(\bm,\bn) = \bzero$, for a perturbation $\cF^\eps_\TAP$ of $\cF_\TAP$ defined in \eqref{eq:TAP-fe-eps}.
(These will actually be probability measures over $(\bG,\dbg,\hbg)$ for auxiliary disorder $\dbg, \hbg$ defined below.)
Let $\cS_{\eps,\ups}$ be a similar perturbation of $\cS_\ups$ defined in \eqref{eq:Supsilon}.
\begin{lem}[Proved in \S\ref{sec:planted-reduction}--\ref{sec:local-concavity}]
    \label{lem:reduce-to-planted}
    For any $(\bG,\dbg,\hbg)$-measurable event $\sE$ and any $\eps,\ups > 0$, there exists $C = C(\eps,\ups)$ such that
    \[
        \PP(\sE) \le C \sup_{(\bm,\bn) \in \cS_{\eps,\ups}} \bbP^{\bm,\bn}_{\eps,\Pl}(\sE)^{1/2} + o_N(1).
    \]
\end{lem}
\noindent The following proposition controls the first moment of $Z_N(\bG)$ in the planted model, formalizing the heuristic calculation in \S\ref{subsec:heuristic-1mt}.
Here $\bbE^{\bm,\bn}_{\eps,\Pl}$ denotes expectation with respect to $\bbP^{\bm,\bn}_{\eps,\Pl}$.
\begin{ppn}[Proved in \S\ref{sec:planted-mt}]
    \label{ppn:planted-mt}
    For any $\delta > 0$, there exists $\eps, \ups > 0$ such that
    \[
        \sup_{(\bm,\bn) \in \cS_{\eps,\ups}}
        \bbE^{\bm,\bn}_{\eps,\Pl}[Z_N(\bG)]
        \le e^{\delta N}.
    \]
\end{ppn}
\noindent From these two results, Theorem~\ref{thm:main} follows by a short argument.
\begin{ppn}
    \label{ppn:main1}
    For any $\delta > 0$,
    \[
        \PP[Z_N(\bG) \le e^{\delta N}]
        = 1-o_N(1).
    \]
\end{ppn}
\begin{proof}
    Let $\sE = \{Z_N(\bG) \le e^{\delta N}\}$.
    By Lemma~\ref{lem:reduce-to-planted} and Markov's inequality,
    \[
        \PP(\sE^c)
        \le C
        \sup_{(\bm,\bn) \in \cS_{\eps,\ups}}
        \bbP^{\bm,\bn}_{\eps,\Pl}(\sE^c)^{1/2}
        + o_N(1)
        \le C e^{-\delta N / 2}
        \sup_{(\bm,\bn) \in \cS_{\eps,\ups}}
        \bbE^{\bm,\bn}_{\eps,\Pl}[Z_N(\bG)]^{1/2}
        + o_N(1).
    \]
    By Proposition~\ref{ppn:planted-mt}, we may choose $\eps, \ups$ so this supremum is at most $e^{\delta N /4}$.
\end{proof}
\begin{proof}[Proof of Theorem~\ref{thm:main}]
    Let $M_{\all} = \lfloor \alpha N \rfloor$, and let $\bG_{\all} = (\begin{smallmatrix}\bG \\ \hbG\end{smallmatrix}) \in \bbR^{M_\all \times N}$, where $\hbG \in \bbR^{(M_\all - M) \times N}$ has i.i.d. $\cN(0,1)$ entries.
    Set $\delta < \fr12 (\alpha - \alpha_\star) \log \fr{1}{\Phi(\kappa)}$.
    Let $\sE = \{Z_N(\bG) \le e^{\delta N}\}$, which satisfies $\PP(\sE) = 1-o_N(1)$ by Proposition~\ref{ppn:main1}.
    Then
    \[
        \PP(M_N(\kappa) / N \ge \alpha)
        = \PP(Z_N(\bG_{\all}) > 0)
        \le \PP(\sE^c) + \EE[Z_N(\bG_{\all}) \bone\{\sE\}].
    \]
    Since the rows of $\hbG$ are i.i.d. samples from $\cN(\bzero,\bI_N)$ independent of $\bG$, for any $\bx \in \Sigma_N$,
    \[
        \EE[Z_N(\bG_{\all}) \bone\{\sE\}]
        \le e^{\delta N} \PP_{\bg \sim \cN(\bzero,\bI_N)}\lt(\fr{\la \bg, \bx \ra}{\sqrt{N}} \ge \kappa \rt)^{M_\all - M}
        = e^{\delta N} \Phi(\kappa)^{M_\all - M}
        = o_N(1). \qedhere
    \]
\end{proof}

\subsection{TAP and AMP formulas}
\label{subsec:tap-amp-formulas}

In this subsection we provide the formulas for the TAP free energy, TAP equations, and AMP iteration mentioned above.
The heuristic derivation of the former two were discussed below \eqref{eq:RS-fe}, and the latter is obtained by iterating the TAP equations in a suitable way.

The contents of this subsection play no formal role in the following proofs.
We include these formulas for the reader's convenience, to allow a comparison with the $\eps$-perturbed TAP free energy and AMP iteration defined in \S\ref{subsec:perturbed-tap-amp} below.
(See also \eqref{eq:TAP-stationarity-m}, \eqref{eq:TAP-stationarity-n} for the $\eps$-perturbed TAP equations.)
For $(\bm,\bn) \in \bbR^N \times \bbR^M$, let $q(\bm) = \tnorm{\bm}^2 / N$ and $\psi(\bn) = \tnorm{\bn}^2/N$.
The TAP free energy for this model is
\beq
    \label{eq:TAP-fe}
    \cF_\TAP(\bm,\bn)
    = \sum_{i=1}^N \cH\lt(\fr{1+m_i}{2}\rt)
    + \sum_{a=1}^M \log \Psi\lt(
        \fr{\kappa - \fr{\la \bg^a, \bm \ra}{\sqrt{N}} + (1-q(\bm)) n_a}{(1-q(\bm))^{1/2}}
    \rt)
    + \fr{N}{2} (1-q(\bm)) \psi(\bn).
\eeq
(Recall $\cH(x) = - x \log(x) - (1-x) \log (1-x)$ is the binary entropy function.)
The TAP equations are the stationarity conditions of $\cF_\TAP$, and are
\balnn
    \label{eq:TAP-eqn}
    \bn &= F_{1-q(\bm)}(\hbh) \equiv F_{1-q(\bm)} \lt(
        \fr{\bG \bm}{\sqrt{N}} - b(\bm) \bn
    \rt), &
    \bm &= \th(\dbh) \equiv \th \lt(
        \fr{\bG^\top \bn}{\sqrt{N}} - d(\bm,\bn) \bm
    \rt),
\ealnn
where
\baln
    b(\bm) &= 1 - q(\bm), &
    d(\bm,\bn) &= \fr1N \sum_{a=1}^M F'_{1-q(\bm)}(n_a).
\ealn
Recall that these are the mean-field limit of the BP equations for this model.
The terms $b(\bm) \bn$ and $d(\bm,\bn) \bm$ compensate for backtracking and are known as the \textbf{Onsager correction} terms.

Let $q_0, \psi_0$ be as in Condition~\ref{con:km-well-defd}, and define
\baln
    b_0 &= \EE [\th'(\psi_0^{1/2} Z)] = 1-q_0, &
    d_0 &= \alpha_\star \EE [F'_{1-q_0}(q_0^{1/2} Z)].
\ealn
The AMP iteration associated to $\cF_\TAP$ is given by $\bn^{-1} = \bzero \in \bbR^M$, $\bm^0 = q_0^{1/2} \bone \in \bbR^N$, and
\balnn
    \label{eq:amp-iteration-zero}
    \bn^k &= F_{1-q_0}(\hbh^k) = F_{1-q_0}\lt(\fr{\bG \bm^k}{\sqrt{N}} - b_0 \bn^{k-1}\rt), &
    \bm^{k+1} &= \th(\dbh^{k+1}) = \th\lt(\fr{\bG^\top \bn^k}{\sqrt{N}} - d_0 \bm^k\rt).
\ealnn

\section{Reduction to planted model}
\label{sec:planted-reduction}

In this section we prove the central Lemma~\ref{lem:reduce-to-planted}, using inputs from \S\ref{sec:amp}--\ref{sec:local-concavity} as described below.
\S\ref{subsec:params-list}--\ref{subsec:planted-reduction-proof} are devoted to this proof.
\S\ref{subsec:conditional-law} derives the law of the planted model $\PP^{\bm,\bn}_{\eps,\Pl}$, which will be useful for calculations in the rest of the paper.
To maintain a smooth presentation, we defer some proofs to \S\ref{subsec:planted-reduction-deferred-proofs}, and routine but technical arguments to Appendix~\ref{app:amp}.

\subsection{Parameter list and notations}
\label{subsec:params-list}

For convenience, we record here the order in which several parameters used in the proof of Lemma~\ref{lem:reduce-to-planted} are set.
Each item in this list can be set sufficiently small or large depending on any preceding item.
\begin{itemize}
    \item $\eps$, size of the perturbation to the AMP iteration and TAP free energy.
    \item $C_\cvx$ and $C_\bd$, estimates for $\rho_\eps$ (defined below, see \eqref{eq:rho-tau}) and its derivatives.
    \item $\eta$, bound on strong convexity of $\cF^\eps_\TAP(\bm,\bn)$ in $\bn$, and $C_\reg$, bound on regularity of $\nabla^2 \cF^\eps_\TAP$.
    \item $r_0$, radius around late AMP iterates where there is a unique critical point of $\cF^\eps_\TAP$.
    \item $\ups_0$, accuracy of AMP iterate under which there is a unique critical point of $\cF^\eps_\TAP$ nearby.
    \item $k$, index of AMP iterate $(\bm^k,\bn^k)$ with accuracy $\ups_0$.
    \item $\ups$, tolerance in $\cS_{\eps,\ups}$.
    \item $\ups_1$, accuracy of AMP iterate under which, by convex-concavity considerations, the nearby unique critical point lies in $\cS_{\eps,\ups}$.
    \item $\ell$, index of AMP iterate $(\bm^\ell,\bn^\ell)$ with accuracy $\ups_1$.
    \item $N$, problem dimension.
\end{itemize}
This information will be reviewed when these parameters are introduced.
Notations such as $o_k(1)$ will denote quantities that tend to zero as the subscripted parameter tends to zero or infinity, which may depend arbitrarily on preceding items in this list but do not depend on subsequent items.
We will use the term ``absolute constant" to mean a constant depending on none of these parameters (but possibly depending on $\kappa,\alpha_\star,q_0,\psi_0$, which are fixed at the outset).
Note that the statement of Lemma~\ref{lem:reduce-to-planted} is monotone in $\ups$, and thus $\ups$ can be set small depending on the parameters preceding it in this list.

We also define more notations appearing in the proofs.
Throughout, $Z,Z',Z''$ denote i.i.d. standard gaussians.
We use $\cP_2(\bbR^k)$ to denote the space of probability measures on $\bbR^k$ with bounded second moment and $\bbW_2$ to denote $2$-Wasserstein distance.
$\plim$ denotes limit in probability.

We often consider functions $\cF : \bbR^N \times \bbR^M \to \bbR$, with input $(\bm,\bn) \in \bbR^N \times \bbR^M$.
We will write $\nabla_\bm \cF \in \bbR^N$, $\nabla_\bn \cF \in \bbR^M$ for the restriction of $\nabla \cF$ to the coordinates corresponding to $\bm$ and $\bn$.
The Hessian restrictions $\nabla^2_{\bm,\bm} \cF \in \bbR^{N\times N}$, $\nabla^2_{\bm,\bn} \cF \in \bbR^{N\times M}$, and $\nabla^2_{\bn,\bn} \cF \in \bbR^{M\times M}$ are defined similarly.
$P_\bm = \bm\bm^\top / \tnorm{\bm}^2 \in \bbR^{N\times N}$ denotes the projection operator onto the span of $\bm$, and $P_\bm^\perp = \bI_N - P_\bm$ denotes the projection operator onto its orthogonal complement.

\subsection{Perturbed nonlinearities, AMP iteration, and TAP free energy}
\label{subsec:perturbed-tap-amp}

We next introduce perturbed versions of the AMP iteration \eqref{eq:amp-iteration-zero} and TAP free energy \eqref{eq:TAP-fe}.
The purpose of the various perturbations is discussed in Remark~\ref{rmk:why-perturb} below.
Let $\eps > 0$ be small.
For $\varrho \ge 0$, define
\baln
    \oF_{\eps,\varrho}(x) &= \log \EE \chi_\eps (x + \varrho^{1/2} Z), &
    \chi_\eps(x) &= \exp\lt(-\fr12 \eps x^2\rt)
    \PP (x + \eps^{1/2} Z' \ge \kappa).
\ealn
Then, define the perturbed nonlinearities
\balnn
    \label{eq:perturbed-nonlinearities}
    \th_\eps(x) &= \th(x) + \eps x, &
    F_{\eps,\varrho}(x) &= \oF'_{\eps,\varrho}(x).
\ealnn
An elementary calculation shows that explicitly,
\balnn
    \notag
    \oF_{\eps,\varrho}(x)
    &= -\fr12 \log(1+\eps \varrho)
    - \fr{\eps x^2}{2(1+\eps \varrho)}
    + \log \Psi \lt(\fr{
        \kappa (1+\eps \varrho) - x
    }{
        \sqrt{(\varrho + \eps(1 + \eps \varrho))(1+\eps \varrho)}
    }\rt) \\
    \label{eq:F-explicit}
    F_{\eps,\varrho}(x)
    &= - \fr{\eps x}{1 + \eps \varrho}
    + \fr{1}{\sqrt{(\varrho + \eps(1 + \eps \varrho))(1+\eps \varrho)}}
    \cE\lt(\fr{
        \kappa (1+\eps \varrho) - x
    }{
        \sqrt{(\varrho + \eps(1 + \eps \varrho))(1+\eps \varrho)}
    }\rt).
\ealnn
Let
\[
    \varrho_\eps (q,\psi) =
    \fr{1-q+\eps-\eps^2 (\psi+\eps)}{1-2\eps (\psi+\eps)}.
\]
Define perturbed variants of the functions $P,R_{\alpha_\star}$ by
\baln
    P^\eps(\psi) &= \EE [\th_\eps((\psi+\eps)^{1/2} Z)^2], &
    R^\eps(q,\psi) &= \alpha_\star \EE [F_{\eps,\varrho_\eps(q,\psi)}((q+\eps)^{1/2} Z)^2],
\ealn
and let $\zeta_\eps(\psi) = R^\eps(P^\eps(\psi),\psi)$.
\begin{ppn}[Proved in Appendix~\ref{app:amp}]
    \label{ppn:eps-perturb-fixed-point}
    There exists $\iota > 0$ such that for all sufficiently small $\eps > 0$,
    \[
        \sup_{\psi \in [\psi_0-\iota,\psi_0+\iota]}
        \zeta'_\eps(\psi) < 1,
    \]
    and there is a unique solution $\psi_\eps \in [\psi_0-\iota,\psi_0+\iota]$ to $\psi_\eps = \zeta_\eps(\psi_\eps)$.
    Let $q_\eps = P^\eps(\psi_\eps)$ and $\varrho_\eps = \varrho_\eps(q_\eps,\psi_\eps)$.
    We further have $(q_\eps,\psi_\eps,\varrho_\eps) \to (q_0,\psi_0,1-q_0)$ as $\eps \downarrow 0$.
\end{ppn}
\begin{lem}[Proved in \S\ref{subsec:planted-reduction-deferred-proofs}]
    \label{lem:b-varrho}
    We have $\varrho_\eps = \EE [\th'_\eps((\psi_\eps + \eps)^{1/2} Z)]$.
\end{lem}
\noindent Let $d_\eps = \alpha_\star \EE [F'_{\eps,\varrho_\eps}((q_\eps + \eps)^{1/2} Z)]$.
Further, let $\dbg \sim \cN(0, \bI_N)$, $\hbg \sim \cN(0, \bI_M)$ be independent of $\bG$.
The perturbed AMP iteration is defined by $\bn^{-1} = \bzero \in \bbR^M$, $\bm^0 = q_\eps^{1/2} \bone \in \bbR^N$, and
\balnn
    \label{eq:amp-iterates-n}
    \bn^k &= F_{\eps,\varrho_\eps}(\hbh^k)
    = F_{\eps,\varrho_\eps}\lt(
        \fr{\bG \bm^k}{\sqrt{N}} + \eps^{1/2} \hbg - \varrho_\eps \bn^{k-1}
    \rt), \\
    \label{eq:amp-iterates-m}
    \bm^{k+1} &= \th_\eps(\dbh^{k+1})
    = \th_\eps\lt(
        \fr{\bG^\top \bn^k}{\sqrt{N}} + \eps^{1/2} \dbg - d_\eps \bm^k
    \rt).
\ealnn
Define the convex function $V_\eps : \bbR \to \bbR$ and its dual
\baln
    V_\eps(\dh) &= \log (2\ch (\dh)) + \fr12 \eps \dh^2, &
    V^\ast_\eps(m) &= \inf_\dh \lt\{-m\dh + V_\eps(\dh)\rt\}.
\ealn
Let $C_\cvx, C_\bd > 0$ be large in $\eps$.
Let $\rho_\eps : \bbR \to \bbR$ be an (unspecified) thrice-differentiable function satisfying
\balnn
    \label{eq:rho-tau}
    \rho_\eps(q_\eps) &= \varrho_\eps, &
    \rho_\eps'(q_\eps) &= -1, &
    \rho_\eps''(q_\eps) &= C_\cvx,
\ealnn
such that the image of $\rho_\eps$ and its derivatives satisfies
\balnn
    \label{eq:rho-tau-regularity}
    \rho_\eps &\in [C_\bd^{-1},C_\bd], &
    |\rho^{(p)}_\eps| &\le C_\bd ~\text{for}~p\in \{1,2,3\}.
\ealnn
(For every $C_\cvx$, there exists $C_\bd$ such that this is possible.)
Recall that for $(\bm, \bn) \in \bbR^N \times \bbR^M$, we defined $q(\bm) = \tnorm{\bm}^2/N$ and $\psi(\bn) = \tnorm{\bn}^2/N$.
The perturbed TAP free energy is
\balnn
    \notag
    \cF^\eps_\TAP(\bm,\bn)
    &=
    \sum_{i=1}^N V_\eps^\ast(m_i)
    + \eps^{1/2} \la \dbg, \bm \ra
    + \sum_{a=1}^M \oF_{\eps,\rho_\eps(q(\bm))} \lt(\fr{\la \bg^a, \bm \ra}{\sqrt{N}} + \eps^{1/2} \hg_a - \rho_\eps(q(\bm)) n_a\rt) \\
    \label{eq:TAP-fe-eps}
    &+ \fr{N}{2} \rho_\eps (q(\bm)) \psi(\bn).
\ealnn
We are now ready to define the planted model.
\begin{dfn}
    \label{dfn:planted-model}
    For $(\bm,\bn) \in \bbR^N \times \bbR^M$, let $\PP^{\bm,\bn}_{\eps,\Pl}$ denote the law of $(\bG,\dbg,\hbg)$ conditional on $\nabla \cF^\eps_\TAP(\bm,\bn) = \bzero$, and $\EE^{\bm,\bn}_{\eps,\Pl}$ denote the corresponding expectation.
    ($\PP$ and $\EE$ continue to refer to the law of $(\bG,\dbg,\hbg)$ with i.i.d. standard gaussian entries.)
\end{dfn}
\begin{rmk}
    As shown in Lemma~\ref{lem:tap-1deriv} below, for any fixed $(\bm,\bn)$, $\nabla \cF^\eps_\TAP(\bm,\bn) = \bzero$ is equivalent to two linear equations \eqref{eq:TAP-stationarity-m}, \eqref{eq:TAP-stationarity-n} in $(\bG,\dbg,\hbg)$, and thus in the planted model $(\bG,\dbg,\hbg)$ remains gaussian.
\end{rmk}
\begin{rmk}
    \label{rmk:why-perturb}
    The above perturbations serve the following purposes.
    \begin{itemize}
        \item $V^\ast_\eps(m_i)$ regularizes the term $\cH(\fr{1+m_i}{2})$ in the original $\cF_\TAP$, avoiding the singular behavior of $\cF_\TAP$ near the boundary of $[-1,1]^N$.
        \item $\oF_{\eps,\varrho_\eps}$ is chosen so that $\cF^\eps_\TAP$ is strongly convex in $\bn$.
        As a consequence, if we define
        \baln
            \cG_\TAP(\bm) &= \inf_\bn \cF_\TAP(\bm,\bn), &
            \cG^\eps_\TAP(\bm) &= \inf_\bn \cF^\eps_\TAP(\bm,\bn),
        \ealn
        then $\cG^\eps_\TAP(\bm)$ also regularizes $\cG_\TAP(\bm)$, avoiding a singular behavior near the boundary of $\fr{1}{\sqrt{N}} \bG \bm \ge \kappa$.
        Indeed, $\cG_\TAP(\bm) = -\infty$ if this inequality fails in any coordinate.
        \item The nonlinearities $\th_\eps$ and $F_{\eps,\varrho_\eps}$ have Lipschitz inverses, so that Euclidean distances in $(\bm,\bn)$ and $(\dbh,\hbh)$ are comparable.
        \item The perturbations $\eps^{1/2} \hbg$ and $\eps^{1/2} \dbg$ are for technical convenience, as solutions to the original TAP equation \eqref{eq:TAP-eqn} must lie on the codimension-one manifold
        \[
            \la \dbh + d(\bm,\bn) \bm, \bm \ra
            = \fr{1}{\sqrt{N}} \la \bn, \bG \bm \ra
            = \la \bn, \hbh + b(\bm) \bn \ra.
        \]
        With this perturbation, Kac--Rice arguments can take place on full space.
        \item We will see in \S\ref{sec:local-concavity} that the Hessian of $\cG^\eps_\TAP(\bm)$ is the sum of an anisotropic sample covariance matrix, a full-rank diagonal matrix, and a low-rank spike (recall \eqref{eq:overview-nabla2}).
        The condition $\rho_\eps''(q_\eps) = C_\cvx$ ensures this spike cannot contribute to the top eigenvalue by adding a large negative spike to the Hessian.
        This simplifies the proof of strong concavity of $\cG^\eps_\TAP$ near late AMP iterates.
    \end{itemize}
\end{rmk}

\subsection{Inputs to reduction}

We next state several inputs needed to prove Lemma~\ref{lem:reduce-to-planted}.
As anticipated in \S\ref{subsec:approx-contiguity-planted}, the main input is Proposition~\ref{ppn:amp-guarantees}, which formalizes criteria \ref{itm:reduction-amp-select-crit} and \ref{itm:reduction-amp-returns-home}.
First, we record that $\cF^\eps_\TAP$ is (deterministically) strongly convex in $\bn$.
\begin{ppn}[Proved in \S\ref{subsec:planted-reduction-deferred-proofs}]
    \label{ppn:n-convexity}
    There exists $\eta = \eta(\eps,C_\cvx,C_\bd) > 0$ such that $\nabla^2_{\bn,\bn} \cF^\eps_\TAP(\bm,\bn) \succeq \eta \bI_M$ for any $(\bm,\bn) \in \bbR^N \times \bbR^M$.
\end{ppn}
\noindent We next record a basic regularity estimate.
Define
\beq
    \label{eq:nabla2-diamond}
    \nabla^2_\diamond \cF^\eps_\TAP(\bm,\bn)
    = \nabla^2_{\bm,\bm} \cF^\eps_\TAP(\bm,\bn)
    - (\nabla^2_{\bm,\bn} \cF^\eps_\TAP(\bm,\bn))
    (\nabla^2_{\bn,\bn} \cF^\eps_\TAP(\bm,\bn))^{-1}
    (\nabla^2_{\bm,\bn} \cF^\eps_\TAP(\bm,\bn))^\top.
\eeq
This arises as the Hessian of $\cG^\eps_\TAP$, as shown in Lemma~\ref{lem:implicit-fn-thm} below.
\begin{ppn}[Proved in Appendix~\ref{app:amp}]
    \label{ppn:regularity}
    For any $D > 0$, there exists $C_\reg = C_\reg(\eps,C_\cvx,C_\bd,D)$ such that over both $\PP$ and $\PP^{\bm',\bn'}_{\eps,\Pl}$ for any $\tnorm{\bm'}^2, \tnorm{\bn'}^2 \le DN$, with high probability the following holds.
    For all $\tnorm{\bm}^2,\tnorm{\bn}^2 \le DN$, we have $\|\nabla^2 \cF^\eps_\TAP(\bm,\bn)\|_\op \le C_\reg$.
\end{ppn}
\noindent For $\dbh \in \bbR^N$, $\hbh \in \bbR^M$, define the coordinate empirical measures
\balnn
    \label{eq:coord-dists}
    \mu_\dbh &= \fr1N \sum_{i=1}^N \delta(\dh_i), &
    \mu_\hbh &= \fr1M \sum_{a=1}^M \delta(\hh_i).
\ealnn
In words, these are probability measures on $\bbR$ with mass $1/N$ on each $\dh_i$ (resp. $1/M$, $\hh_i$).
For $\ups > 0$, let
\balnn
    \notag
    \cT_{\eps,\ups} &= \lt\{
        (\dbh,\hbh) \in \bbR^N \times \bbR^M :
        \bbW_2(\mu_\dbh,\cN(0,\psi_\eps + \eps)),
        \bbW_2(\mu_\hbh,\cN(0,q_\eps + \eps)) \le \ups
    \rt\}, \\
    \label{eq:Supsilon}
    \cS_{\eps,\ups} &= \lt\{
        (\th_\eps(\dbh), F_{\eps,\varrho_\eps}(\hbh)) :
        (\dbh,\hbh) \in \cT_{\eps,\ups}
    \rt\}.
\ealnn
Let $(\bm^k,\bn^k)$ be as in \eqref{eq:amp-iterates-n}, \eqref{eq:amp-iterates-m}.
\begin{ppn}[Proved in \S\ref{sec:amp} and \S\ref{sec:local-concavity}]
    \label{ppn:amp-guarantees}
    There exist $r_0>0$, $k_0 : \bbR_+ \to \bbN$, $\ups : \bbR_+ \times \bbN \to \bbR_+$, depending on $\eps,C_\cvx,C_\bd,\eta,C_\reg$, and an absolute constant $C_{\spec}>0$ such that the following holds.
    For any $\ups_0>0$ and $k\ge k_0(\ups_0)$, with high probability under $\PP$:
    \begin{enumerate}[label=(\alph*),ref=(\alph*)]
        \item \label{itm:amp-guarantee-profile} $(\bm^k,\bn^k) \in \cS_{\eps,\ups_0}$,
        \item \label{itm:amp-guarantee-stationary} $\tnorm{\nabla \cF^\eps_\TAP(\bm^k,\bn^k)} \le \ups_0 \sqrt{N}$,
        \item \label{itm:amp-guarantee-concave} $\nabla^2_\diamond \cF^\eps_\TAP(\bm,\bn) \preceq -C_{\spec} \bI_N$ for all $(\bm,\bn)$ such that $\|(\bm,\bn) - (\bm^k,\bn^k)\| \le r_0 \sqrt{N}$.
    \end{enumerate}
    Moreover, let $\ups = \ups(\ups_0,k)$.
    For any $(\bm',\bn') \in \cS_{\eps,\ups}$, with high probability under $\PP^{\bm',\bn'}_{\eps,\Pl}$, the above three conclusions hold and:
    \begin{enumerate}[label=(\alph*),ref=(\alph*)]
        \setcounter{enumi}{3}
        \item \label{itm:amp-guarantee-planted} $\tnorm{(\bm^k,\bn^k) - (\bm',\bn')} \le \ups_0 \sqrt{N}$.
    \end{enumerate}
\end{ppn}
\noindent The following concentration estimate follows by adapting an argument of \cite{guionnet2000concentration} and provides input \ref{itm:reduction-det-conc}.
\begin{lem}[Proved in \S\ref{sec:local-concavity}]
    \label{lem:det-concentration}
    There exists $C$ depending on $\eps,C_\cvx$ such that for sufficiently small $\ups$, uniformly over $(\bm,\bn) \in \cS_{\eps,\ups}$,
    \[
        \bbE^{\bm,\bn}_{\eps,\Pl} \lt[
            |\det \nabla^2 \cF^\eps_\TAP(\bm,\bn)|^2
        \rt]^{1/2}
        \le C
        \bbE^{\bm,\bn}_{\eps,\Pl} \lt[
            |\det \nabla^2 \cF^\eps_\TAP(\bm,\bn)|
        \rt].
    \]
\end{lem}

\subsection{Unique nearby critical point and conditioning lemma}

Lemma~\ref{lem:concave-crit} below provides a criterion under which a function has a unique critical point near a given approximate critical point.
Lemma~\ref{lem:conditioning} is a lemma about conditioning a random function on a random vector with a unique critical point nearby, which is an adaptation of the Kac--Rice formula.
This important technical tool also appears as \cite[Lemma 3.6]{huang2024sampling}, where it is used in conjunction with known results on topological trivialization to condition on the TAP fixed point selected by AMP.
Here, we use it with properties of the planted model provided by Proposition~\ref{ppn:amp-guarantees} to prove topological trivialization itself.
\begin{lem}
    \label{lem:implicit-fn-thm}
    Let $U_1 \subseteq \bbR^N$, $U_2 \subseteq \bbR^M$ be open and convex.
    Suppose $\cF : U_1 \times U_2 \to \bbR$ is twice differentiable and satisfies $\nabla^2_{\bn,\bn} \cF(\bm,\bn) \succeq \eta \bI_M$ for all $(\bm,\bn) \in U_1 \times U_2$ for some $\eta > 0$, and $\cG(\bm) \equiv \min_{\bn \in U_2} \cF(\bm,\bn)$ exists for all $\bm \in U_1$.
    Then $\bn(\bm) = \argmin_{\bn \in U_2} \cF(\bm,\bn)$ is unique and differentiable, with
    \beq
        \label{eq:nabla-bn}
        \nabla \bn(\bm)
        = (\nabla^2_{\bn,\bn} \cF(\bm,\bn(\bm)))^{-1}
        (\nabla^2_{\bm,\bn} \cF(\bm,\bn(\bm)))^\top.
    \eeq
    Moreover $\cG$ is twice differentiable, with
    \balnn
        \label{eq:nabla-cG}
        \nabla \cG(\bm) &= \nabla_\bm \cF(\bm,\bn), &
        \nabla^2 \cG(\bm) &= \nabla^2_\diamond \cF(\bm,\bn).
    \ealnn
\end{lem}
\begin{proof}
    Strong convexity of $\cF$ in $\bn$ implies that $\bn(\bm)$ is unique, and can be defined as the solution to $\nabla_\bm \cF(\bm,\bn) = \bzero$.
    Then \eqref{eq:nabla-bn} follows from the implicit function theorem, while \eqref{eq:nabla-cG} follows from \eqref{eq:nabla-bn} and the chain rule.
\end{proof}
\begin{lem}
    \label{lem:concave-crit}
    Let $\cF : \bbR^N \times \bbR^M \to \bbR$ be twice differentiable and $(\bm_0,\bn_0) \in \bbR^N \times \bbR^M$.
    Let $\eta, C_\reg, \ups_0 > 0$, $r_0 = 2\eta^{-1} (1+C_\reg\eta^{-1})^2 \ups_0$, and $U = \Ball((\bm_0,\bn_0), r_0\sqrt{N})$.
    Suppose that:
    \begin{enumerate}[label=(C\arabic*),ref=(C\arabic*)]
        \item \label{itm:conditioning-stationary} $\tnorm{\nabla \cF(\bm_0,\bn_0)} \le \ups_0 \sqrt{N}$,
        \item \label{itm:conditioning-regular} $\tnorm{\nabla^2 \cF(\bm,\bn)}_\op \le C_\reg$ for all $(\bm,\bn) \in U$,
        \item \label{itm:conditioning-convex} $\nabla^2_{\bn,\bn} \cF(\bm,\bn) \succeq \eta \bI_M$ for all $(\bm,\bn) \in \bbR^N \times \bbR^M$,
        \item \label{itm:conditioning-concave} $\nabla^2_\diamond \cF(\bm,\bn) \preceq -\eta \bI_N$ for all $(\bm,\bn) \in U$.
    \end{enumerate}
    Then, there is a unique $(\bm_\ast,\bn_\ast) \in U$ such that $\nabla \cF(\bm_\ast,\bn_\ast) = \bzero$.
    Moreover, for sufficiently small (possibly in $N$) $\iota > 0$, the image of $U$ under the map $\nabla \cF$ contains $\Ball(\bzero,\iota) \subseteq \bbR^N \times \bbR^M$ and is one-to-one on this set.
\end{lem}
\begin{proof}
    Let $U_1 = \Ball(\bm_0, r_0\sqrt{N}) \subseteq \bbR^N$ and $U_2 = \bbR^M$.
    Item \ref{itm:conditioning-convex} implies that the hypotheses of Lemma~\ref{lem:implicit-fn-thm} hold for $\cF^\eps_\TAP$ with this $(U_1,U_2)$.
    Thus, for $\bm \in U_1$, $\bn(\bm)$ and $\cG(\bm)$ from Lemma~\ref{lem:implicit-fn-thm} are well-defined, with derivatives given therein.
    If $(\bm_\ast,\bn_\ast)$ is a critical point of $\cF$, then $\bm_\ast$ must be a critical point of $\cG$.
    Item \ref{itm:conditioning-concave} and equation \eqref{eq:nabla-cG} imply that $\nabla^2 \cG(\bm) \preceq -\eta \bI_N$ for all $\bm \in U_1$.
    Thus $\cG$ has at most one critical point in $U_1$, and $\cF$ has at most one critical point in $U_1 \times U_2 \supseteq U$.

    We now show that such a point exists.
    By strong concavity of $\cF(\bm_0,\cdot)$ and \ref{itm:conditioning-stationary},
    \[
        \tnorm{\bn_0 - \bn(\bm_0)}
        \le \eta^{-1} \tnorm{\nabla_\bn \cF(\bm_0,\bm_0)}
        \le \eta^{-1} \ups_0 \sqrt{N}.
    \]
    Because $\tnorm{\nabla^2 \cF(\bm,\bn)}_\op \le C_\reg$, the map $(\bm,\bn) \mapsto \nabla \cF(\bm,\bn)$ is $C_\reg$-Lipschitz.
    Thus
    \[
        \tnorm{\nabla \cG(\bm_0)}
        = \tnorm{\nabla \cF(\bm_0,\bn(\bm_0))}
        \le \tnorm{\nabla \cF(\bm_0,\bn_0)}
        + C_\reg \tnorm{\bn_0 - \bn(\bm_0)}
        \le (1 + C_\reg\eta^{-1}) \ups_0 \sqrt{N}.
    \]
    By strong concavity of $\cG$, there exists a critical point $\bm_\ast$ of $\cG$ with
    \[
        \tnorm{\bm_0 - \bm_\ast}
        \le \eta^{-1} \tnorm{\nabla \cG(\bm_0)}
        \le \eta^{-1} (1 + C_\reg\eta^{-1}) \ups_0 \sqrt{N}.
    \]
    Then, with $\bn_\ast = \bn(\bm_\ast)$, $(\bm_\ast,\bn_\ast)$ is a critical point of $\cF$.
    By conditions \ref{itm:conditioning-regular}, \ref{itm:conditioning-convex} and equation \eqref{eq:nabla-bn}, $\bn(\cdot)$ is $C_\reg\eta^{-1}$-Lipschitz.
    So,
    \[
        \tnorm{\bn_0 - \bn_\ast}
        \le \tnorm{\bn_0 - \bn(\bm_0)}
        + C_\reg\eta^{-1} \tnorm{\bm_0 - \bm_\ast}
        \le \eta^{-1} (1 + C_\reg\eta^{-1})^2 \ups_0 \sqrt{N}.
    \]
    This shows that $(\bm_\ast,\bn_\ast) \in U$, proving the first claim, and furthermore $(\bm_\ast,\bn_\ast)$ lies in the interior of $U$.
    To show the second claim, we first prove that any $(\bm,\bn) \in U$ such that $\tnorm{\nabla \cF(\bm,\bn)} \le \iota$ lies in a neighborhood of $(\bm^\ast,\bn^\ast)$.
    First,
    \[
        \tnorm{\bn - \bn(\bm)}
        \le \eta^{-1} \tnorm{\nabla_\bn \cF(\bm,\bn)}
        \le \eta^{-1} \iota.
    \]
    Similarly to above, $\tnorm{\nabla \cG(\bm)} \le (1 + C_\reg \eta^{-1}) \iota$, so we conclude
    \baln
        \tnorm{\bm - \bm_\ast} &\le \eta^{-1} (1 + C_\reg \eta^{-1}) \iota, &
        \tnorm{\bn - \bn_\ast} &\le \eta^{-1} (1 + C_\reg \eta^{-1})^2 \iota.
    \ealn
    Thus $(\bm,\bn)$ lies in a neighborhood of $(\bm_\ast,\bn_\ast)$, which is contained in $U$ because $(\bm_\ast,\bn_\ast)$ lies in the interior of $U$.
    However, by Schur's lemma,
    \[
        \det \nabla^2 \cF(\bm_\ast,\bn_\ast)
        = \det \nabla^2_{\bn,\bn} \cF(\bm_\ast,\bn_\ast)
        \det \nabla^2_\diamond \cF(\bm_\ast,\bn_\ast)
        \neq 0.
    \]
    By the inverse function theorem, $\nabla \cF$ is invertible in a neighborhood of $(\bm_\ast,\bn_\ast)$, mapping it bijectively to a neighborhood of $\bzero$.
    This concludes the proof.
\end{proof}
\begin{lem}
    \label{lem:conditioning}
    Let $\cF : \bbR^N \times \bbR^M \to \bbR$ be a twice differentiable random function and $(\bm_0,\bn_0) \in \bbR^N \times \bbR^M$ be a random vector in the same probability space.
    Let $\eta, C_\reg, \ups_0, r_0$ be as in Lemma~\ref{lem:concave-crit}, and $U = \Ball((\bm_0,\bn_0), r_0\sqrt{N})$ (which is now a random set).
    Let $D>0$ be arbitrary and $\sE_0$ be the event that \ref{itm:conditioning-stationary} through \ref{itm:conditioning-concave} hold and $\tnorm{\bm_0}^2, \tnorm{\bn_0}^2 \le DN$.

    Let $\varphi_{\nabla \cF(\bm,\bn)}$ denote the probability density of $\nabla \cF(\bm,\bn)$ w.r.t. Lebesgue measure on $\bbR^N \times \bbR^M$.
    Suppose $\varphi_{\nabla \cF(\bm,\bn)}(\bz)$ is bounded for $(\bm,\bn) \in \bbR^N \times \bbR^M$ and $\bz$ in a neighborhood of $\bzero$, and continuous in $\bz$ in this neighborhood uniformly over $(\bm,\bn)$.
    Then, for any event $\sE \subseteq \sE_0$ in the same probability space,
    \[
        \PP(\sE)
        = \int_{\bbR^N \times \bbR^M} \bbE \lt[
            |\det \nabla^2 \cF(\bm,\bn)|
            \bone\{\sE \cap \{(\bm,\bn) \in U\}\}
            \big| \nabla \cF(\bm,\bn) = \bzero
        \rt]
        \varphi_{\nabla \cF(\bm,\bn)}(\bzero)
        ~\de (\bm,\bn).
    \]
\end{lem}
\begin{proof}
    On $\sE_0$, Lemma~\ref{lem:concave-crit} implies there is a unique critical point $(\bm_\ast,\bn_\ast)$ of $\cF$ in $U$.
    Moreover the image of $U$ under $\nabla \cF$ contains $\Ball(\bzero,\iota)$ for small $\iota$ and is one-to-one on this set.
    By the area formula, on $\sE_0$,
    \[
        1 =
        \fr{1}{|\Ball(\bzero,\iota)|}
        \int_U
        |\det \nabla^2 \cF(\bm,\bn)|
        \bone\{\tnorm{\nabla \cF(\bm,\bn)} \le \iota \}
        ~\de(\bm,\bn).
    \]
    Since $\sE\subseteq \sE_0$, multiplying both sides by $\bone\{\sE\}$ and taking expectations (via Fubini's theorem) yields
    \baln
        \PP(\sE)
        &= \fr{1}{|\Ball(\bzero,\iota)|}
        \EE \int_{\bbR^N \times \bbR^M}
        |\det \nabla^2 \cF(\bm,\bn)|
        \bone\{\tnorm{\nabla \cF(\bm,\bn)} \le \iota \}
        \bone\{\bm \in U\}
        ~\de(\bm,\bn) \\
        &=\int_{\bbR^N \times \bbR^M}
        \EE \lt[
            |\det \nabla^2 \cF(\bm,\bn)|
            \bone\{\sE \cap \{\bm \in U\}\}
            \big|
            \tnorm{\nabla \cF(\bm,\bn)} \le \iota
        \rt]
        \fr{\PP\{\tnorm{\nabla \cF(\bm,\bn)} \le \iota \}}{|\Ball(\bzero,\iota)|}
        ~\de(\bm,\bn).
    \ealn
    We now take the limit as $\iota \to 0$.
    On $\sE_0$, $|\det \nabla^2 \cF(\bm,\bn)| \le C_\reg^{M+N}$.
    Since $\bm_0,\bn_0$ are bounded on $\sE_0$, $\bone\{\bm \in U\} = 0$ almost surely for $\bm$ outside a compact set.
    Since $\varphi_{\nabla \cF(\bm,\bn)}(\bz)$ is bounded and continuous in $\bz$, $\PP\{\tnorm{\nabla \cF(\bm,\bn)} \le \iota \} / |\Ball(\bzero,\iota)|$ is bounded, and limits to $\varphi_{\nabla \cF(\bm,\bn)}(\bz)$ as $\iota \to 0$.
    Taking $\iota \to 0$ gives the result by dominated convergence.
\end{proof}

\subsection{Proof of planted reduction}
\label{subsec:planted-reduction-proof}

We are now ready to prove Lemma~\ref{lem:reduce-to-planted}.
As anticipated in \S\ref{subsec:approx-contiguity-planted},
Lemma~\ref{lem:crit-whp} deduces \ref{itm:reduction-exist-crit} from \ref{itm:reduction-amp-select-crit}, and Lemma~\ref{lem:top-triv} deduces \ref{itm:reduction-top-triv}
from \ref{itm:reduction-amp-returns-home}.
Then, Lemma~\ref{lem:reduce-to-planted} follows readily from the Kac--Rice formula.
\begin{lem}
    \label{lem:crit-whp}
    For any $\ups>0$, $\cS_{\eps,\ups}$ contains a critical point of $\cF^\eps_\TAP$ with high probability under $\PP$.
\end{lem}
\begin{proof}
    Let $\eta = \min(\eta(\eps,C_\cvx,C_\bd), C_{\spec})$, where these terms are given by Propositions~\ref{ppn:n-convexity} and \ref{ppn:amp-guarantees}.
    Then, let $D = 2\max(q_\eps,\psi_\eps)$ and $C_\reg = C_\reg(\eps,C_\cvx,C_\bd,D)$ be given by Proposition~\ref{ppn:regularity}.
    Let $r_0$ be given by Proposition~\ref{ppn:amp-guarantees}.
    Let $\ups_1$ be small enough in $\ups$ that, with $r_1 = 2\eta^{-1} (1+C_\reg\eta^{-1})^2 \ups_1$, we have $r_1 \le r_0$ and
    \beq
        \label{eq:contain-in-Supsilon}
        \bigcup_{(\tbm,\tbn) \in \cS_{\eps,\ups_1}}
        \Ball((\tbm,\tbn), r_1 \sqrt{N})
        \subseteq \cS_{\eps,\ups}.
    \eeq
    (Since $\cS_{\eps,\ups}$ is the image of a product of two Wasserstein-balls under $(\th_\eps,F_{\eps,\varrho_\eps})$, and $\th_\eps^{-1},F_{\eps,\varrho_\eps}^{-1}$ have Lipschitz constant depending only on $\eps$, there exists $\ups_1$ such that this holds.)
    Let $\ell = k_0(\ups_1)$ be given by Proposition~\ref{ppn:amp-guarantees}.
    By Propositions~\ref{ppn:regularity} and \ref{ppn:amp-guarantees}, with high probability under $\PP$,
    \begin{itemize}
        \item $\|\nabla^2 \cF^\eps_\TAP(\bm,\bn)\|_\op \le C_\reg$ for all $\tnorm{\bm}^2,\tnorm{\bn}^2 \le DN$,
        \item $(\bm^\ell,\bn^\ell) \in \cS_{\eps,\ups_1}$,
        \item $\tnorm{\nabla \cF^\eps_\TAP(\bm^\ell,\bn^\ell)} \le \ups_1 \sqrt{N}$,
        \item $\nabla^2_\diamond \cF^\eps_\TAP(\bm,\bn) \preceq -C_{\spec} \bI_N$ for all $\tnorm{(\bm,\bn) - (\bm^\ell,\bn^\ell)} \le r_0 \sqrt{N}$.
    \end{itemize}
    We now apply Lemma~\ref{lem:concave-crit} with $(\cF^\eps_\TAP,\bm^\ell,\bn^\ell,\ups_1,r_1)$ in place of $(\cF,\bm_0,\bn_0,\ups_0,r_0)$.
    The above points imply that conditions \ref{itm:conditioning-stationary}, \ref{itm:conditioning-regular}, \ref{itm:conditioning-concave} hold, and condition \ref{itm:conditioning-convex} holds by Proposition~\ref{ppn:n-convexity}.
    By Lemma~\ref{lem:concave-crit}, $\cF^\eps_\TAP$ has a critical point in $\Ball((\bm^\ell,\bn^\ell), r_1 \sqrt{N})$.
    This lies in $\cS_{\eps,\ups}$ by \eqref{eq:contain-in-Supsilon}.
\end{proof}
\noindent The following lemma shows that the condition in Lemma~\ref{lem:conditioning} regarding $\varphi_{\nabla \cF}$ holds for $\cF = \cF^\eps_\TAP$.
\begin{lem}[Proved in \S\ref{subsec:planted-reduction-deferred-proofs}]
    \label{lem:conditioning-lem-condition-satisfied}
    The density $\varphi_{\nabla \cF^\eps_\TAP(\bm,\bn)}(\bz)$ under $\PP$ is bounded for $(\bm,\bn) \in \bbR^N \times \bbR^M$ and $\bz$ in a neighborhood of $\bzero$, and continuous in $\bz$ in this neighborhood uniformly over $(\bm,\bn)$.
\end{lem}
\begin{lem}
    \label{lem:top-triv}
    Let $\Crt_\ups$ denote the set of critical points of $\cF^\eps_\TAP$ in $\cS_{\eps,\ups}$.
    For small $\ups > 0$, $\EE |\Crt_\ups| \le 1+o_N(1)$.
\end{lem}
\begin{proof}
    By the Kac--Rice formula,
    \beq
        \label{eq:kac-rice}
        \EE |\Crt_\ups| =
        \int_{\cS_{\eps,\ups}}
        \bbE^{\bm,\bn}_{\eps,\Pl} \lt[
            |\det \nabla^2 \cF^\eps_\TAP(\bm,\bn)|
        \rt]
        \varphi_{\nabla \cF^\eps_\TAP(\bm,\bn)}(\bzero)
        ~\de (\bm,\bn).
    \eeq
    As above, let $\eta = \min(\eta(\eps,C_\cvx,C_\bd),C_{\spec})$, $D = 2\max(q_\eps,\psi_\eps)$, and $C_\reg = C_\reg(\eps,C_\cvx,C_\bd,D)$.
    Let $r_0$ be given by Proposition~\ref{ppn:amp-guarantees}, and
    \[
        \ups_0 = \fr{\eta r_0}{2(1 + C_\reg \eta^{-1})^2}.
    \]
    Then set $k = k_0(\ups_0)$, where $k_0(\cdot)$ is as in Proposition~\ref{ppn:amp-guarantees}.
    Let $\sE$ be the event that:
    \begin{itemize}
        \item $\tnorm{\bm^k}^2, \tnorm{\bn^k}^2 \le DN$,
        \item $\|\nabla^2 \cF^\eps_\TAP(\bm,\bn)\|_\op \le C_\reg$ for all $\tnorm{\bm}^2,\tnorm{\bn}^2 \le DN$,
        \item $\tnorm{\nabla \cF^\eps_\TAP(\bm^k,\bn^k)} \le \ups_0 \sqrt{N}$,
        \item $\nabla^2_\diamond \cF^\eps_\TAP(\bm,\bn) \preceq -C_{\spec} \bI_N$ for all $\tnorm{(\bm,\bn) - (\bm^k,\bn^k)} \le r_0 \sqrt{N}$.
    \end{itemize}
    We claim that $\sE \subseteq \sE_0$, where $\sE_0$ is the event defined in Lemma~\ref{lem:conditioning} with $(\cF^\eps_\TAP,\bm^k,\bn^k)$ for $(\cF,\bm_0,\bn_0)$ (and thus $U = \Ball((\bm^k,\bn^k),r_0\sqrt{N})$).
    The above points imply conditions \ref{itm:conditioning-stationary}, \ref{itm:conditioning-regular}, \ref{itm:conditioning-concave}, and condition \ref{itm:conditioning-convex} follows from Proposition~\ref{ppn:n-convexity}.
    By Lemma~\ref{lem:conditioning-lem-condition-satisfied}, Lemma~\ref{lem:conditioning} applies.
    Thus,
    \beq
        \label{eq:top-triv-step}
        1 \ge \PP(\sE)
        = \int_{\bbR^N \times \bbR^M}
        \bbE^{\bm,\bn}_{\eps,\Pl} \lt[
            |\det \nabla^2 \cF^\eps_\TAP(\bm,\bn)|
            \bone\{\sE \cap \{(\bm,\bn) \in U\}\}
        \rt]
        \varphi_{\nabla \cF^\eps_\TAP(\bm,\bn)}(\bzero)
        ~\de (\bm,\bn).
    \eeq
    Let $\ups \le \min(\ups(\ups_0,k),\ups(r_0,k))$, for $\ups(\cdot,\cdot)$ as in Proposition~\ref{ppn:amp-guarantees}.
    Define (compare with \eqref{eq:kac-rice})
    \[
        I_1 = \int_{\cS_{\eps,\ups}}
        \bbE^{\bm,\bn}_{\eps,\Pl} \lt[
            |\det \nabla^2 \cF^\eps_\TAP(\bm,\bn)|
            \bone\{\sE \cap \{(\bm,\bn) \in U\}\}
        \rt]
        \varphi_{\nabla \cF^\eps_\TAP(\bm,\bn)}(\bzero)
        ~\de (\bm,\bn)
    \]
    and $I_2 = \EE |\Crt_\ups| - I_1$.
    By Propositions~\ref{ppn:regularity} and \ref{ppn:amp-guarantees}, for any $(\bm,\bn) \in \cS_{\eps,\ups}$, we have $\bbP^{\bm,\bn}_{\eps,\Pl}(\sE \cap \{(\bm,\bn) \in U\}) \ge 1-\iota$ for some $\iota = o_N(1)$.
    By Cauchy--Schwarz and Lemma~\ref{lem:det-concentration},
    \baln
        I_2 &= \int_{\cS_{\eps,\ups}}
        \bbE^{\bm,\bn}_{\eps,\Pl} \lt[
            |\det \nabla^2 \cF^\eps_\TAP(\bm,\bn)|
            \bone\{(\sE \cap \{(\bm,\bn) \in U\})^c\}
        \rt]
        \varphi_{\nabla \cF^\eps_\TAP(\bm,\bn)}(\bzero)
        ~\de (\bm,\bn) \\
        &\le \int_{\cS_{\eps,\ups}}
        \bbE^{\bm,\bn}_{\eps,\Pl} \lt[
            |\det \nabla^2 \cF^\eps_\TAP(\bm,\bn)|^2
        \rt]^{1/2}
        \bbP^{\bm,\bn}_{\eps,\Pl} \lt[(\sE \cap \{(\bm,\bn) \in U\})^c\rt]^{1/2}
        \varphi_{\nabla \cF^\eps_\TAP(\bm,\bn)}(\bzero)
        ~\de (\bm,\bn) \\
        &\le C\iota^{1/2}
        \int_{\cS_{\eps,\ups}}
        \bbE^{\bm,\bn}_{\eps,\Pl} \lt[
            |\det \nabla^2 \cF^\eps_\TAP(\bm,\bn)|
        \rt]
        \varphi_{\nabla \cF^\eps_\TAP(\bm,\bn)}(\bzero)
        ~\de (\bm,\bn)
        \stackrel{\eqref{eq:kac-rice}}{=} C\iota^{1/2} \EE |\Crt_\ups|.
    \ealn
    So, $I_1 \ge (1-C\iota^{1/2}) \EE |\Crt_\ups|$.
    Since \eqref{eq:top-triv-step} implies $I_1 \le 1$, and $\iota = o_N(1)$, the conclusion follows.
\end{proof}
\begin{proof}[Proof of Lemma~\ref{lem:reduce-to-planted}]
    Set $\ups > 0$ small enough that Lemma~\ref{lem:top-triv} holds.
    Let $\sE_1$ be the event that $\cF^\eps_\TAP$ has a critical point in $\cS_{\eps,\ups}$.
    By the Kac--Rice formula,
    \baln
        \PP(\sE \cap \sE_1)
        &\le \EE[\bone\{\sE \cap \sE_1\} |\Crt_\ups|] \\
        &= \int_{\cS_{\eps,\ups}}
        \bbE^{\bm,\bn}_{\eps,\Pl} \lt[
            |\det \nabla^2 \cF^\eps_\TAP(\bm,\bn)|
            \bone\{\sE \cap \sE_1\}
        \rt]
        \varphi_{\nabla \cF^\eps_\TAP(\bm,\bn)}(\bzero)
        ~\de (\bm,\bn).
    \ealn
    This is bounded by
    \balnn
        \notag
        &\int_{\cS_{\eps,\ups}}
        \bbE^{\bm,\bn}_{\eps,\Pl} \lt[
            |\det \nabla^2 \cF^\eps_\TAP(\bm,\bn)|^2
        \rt]^{1/2}
        \bbP^{\bm,\bn}_{\eps,\Pl} (\sE)^{1/2}
        \varphi_{\nabla \cF^\eps_\TAP(\bm,\bn)}(\bzero)
        ~\de (\bm,\bn) \\
        \notag
        &\le C\sup_{(\bm,\bn) \in \cS_{\eps,\ups}}
        \bbP^{\bm,\bn}_{\eps,\Pl} (\sE)^{1/2}
        \times
        \int_{\cS_{\eps,\ups}}
        \bbE^{\bm,\bn}_{\eps,\Pl} \lt[
            |\det \nabla^2 \cF^\eps_\TAP(\bm,\bn)|
        \rt]
        \varphi_{\nabla \cF^\eps_\TAP(\bm,\bn)}(\bzero)
        ~\de (\bm,\bn) \\
        \label{eq:kac-rice-for-reduction}
        &\le C\sup_{(\bm,\bn) \in \cS_{\eps,\ups}}
        \bbP^{\bm,\bn}_{\eps,\Pl} (\sE)^{1/2}
        \cdot \EE |\Crt_\ups|
        \stackrel{Lem.~\ref{lem:top-triv}}{\le}
        (1+o_N(1)) C
        \sup_{(\bm,\bn) \in \cS_{\eps,\ups}}
        \bbP^{\bm,\bn}_{\eps,\Pl} (\sE)^{1/2}.
    \ealnn
    The result follows because $\PP(\sE) \le \PP(\sE \cap \sE_1) + \PP(\sE_1^c)$, and $\PP(\sE_1^c) = o_N(1)$ by Lemma~\ref{lem:crit-whp}.
\end{proof}

\subsection{Conditional law in planted model}
\label{subsec:conditional-law}

Having proved the reduction to the planted model $\PP^{\bm,\bn}_{\eps,\Pl}$, we now calculate the law of the disorder in it.
This is stated in Lemma~\ref{lem:conditional-law} for general $(\bm,\bn)$, and Corollary~\ref{cor:conditional-law-correct-profile} for $(\bm,\bn) \in \cS_{\eps,\ups}$.
The following lemma is proved by direct differentiation of $\cF^\eps_\TAP$.
\begin{lem}[Proved in Appendix~\ref{app:amp}]
    \label{lem:tap-1deriv}
    Let $\bm \in \bbR^N$, $\bn \in \bbR^M$, and
    \baln
        \abh &= \fr{\bG \bm}{\sqrt{N}} + \eps^{1/2} \hbg - \rho_\eps(q(\bm)) \bn, &
        d_\eps(\bm,\bn) &= \fr1N \sum_{a=1}^M (n_a - F_{\eps,\rho_\eps(q(\bm))}(\ah_a))^2 + F'_{\rho_\eps(q(\bm))}(\ah_a).
    \ealn
    Then,
    \balnn
        \label{eq:tap-deriv-m}
        \nabla_\bm \cF^\eps_\TAP(\bm,\bn)
        &= - \th^{-1}_\eps(\bm)
        + \fr{\bG^\top F_{\eps,\rho_\eps(q(\bm))}(\abh)}{\sqrt{N}}
        + \eps^{1/2} \dbg
        + \rho'_\eps(q(\bm)) d_\eps(\bm,\bn) \bm, \\
        \label{eq:tap-deriv-n}
        \nabla_\bn \cF^\eps_\TAP(\bm,\bn)
        &= \rho_\eps(q(\bm)) \lt(\bn - F_{\eps,\rho_\eps(q(\bm))}(\abh)\rt).
    \ealnn
    In particular $\nabla \cF^\eps_\TAP(\bm,\bn) = \bzero$ if and only if, with $\dbh = \th_\eps^{-1}(\bm)$ and $\hbh = F_{\eps,\rho_\eps(q(\bm))}^{-1}(\bn)$,
    \balnn
        \label{eq:TAP-stationarity-m}
        \fr{\bG \bm}{\sqrt{N}} + \eps^{1/2} \hbg &= \hbh + \rho_\eps(q(\bm)) \bn, \\
        \label{eq:TAP-stationarity-n}
        \fr{\bG^\top \bn}{\sqrt{N}} + \eps^{1/2} \dbg &= \dbh - \rho'_\eps(q(\bm)) d_\eps(\bm,\bn) \bm.
    \ealnn
    (Note that \eqref{eq:TAP-stationarity-m} is equivalent to $\hbh = \abh$.)
\end{lem}

\begin{lem}
    \label{lem:conditional-law}
    Under $\PP^{\bm,\bn}_{\eps,\Pl}$, $\bG$ has law
    \balnn
        \label{eq:conditional-law}
        \fr{\bG}{\sqrt{N}}
        &\stackrel{d}{=}
        \fr{\hbh \bm^\top}{N(q(\bm) + \eps)}
        + \fr{\bn \dbh^\top}{N(\psi(\bn) + \eps)}
        + \fr{\Delta(\bm,\bn) \bn \bm^\top}{N(q(\bm) + \psi(\bn) + \eps)}
        + \fr{\tbG}{\sqrt{N}},
        \qquad \text{where} \\
        \label{eq:conditional-law-Delta}
        \Delta(\bm,\bn)
        &=
        \rho_\eps(q(\bm)) - \rho'_\eps(q(\bm)) d_\eps(\bm,\bn)
        - \fr{\la \bn, \hbh \ra}{N(q(\bm) + \eps)}
        - \fr{\la \bm, \dbh \ra}{N(\psi(\bn) + \eps)}
    \ealnn
    and where $\tbG$ has the following law.
    Let $\dbe_1,\ldots,\dbe_N$ and $\hbe_1,\ldots,\hbe_M$ be orthonormal bases of $\bbR^N$ and $\bbR^M$ with $\dbe_1 = \bm / \tnorm{\bm}$ and $\hbe_1 = \bn / \tnorm{\bn}$, and abbreviate $\tbG(i,j) = \la \hbe_j, \tbG \dbe_i \ra$.
    Then the $\tbG(i,j)$ are independent centered gaussians with variance
    \beq
        \label{eq:residual-variances}
        \EE \tbG(i,j)^2 = \begin{cases}
            \eps / (q(\bm) + \psi(\bn) + \eps) & i=j=1, \\
            \eps / (q(\bm) + \eps) & i=1, j\neq 1, \\
            \eps / (\psi(\bn) + \eps) & i\neq 1, j=1, \\
            1 & i\neq 1, j\neq 1.
        \end{cases}
    \eeq
\end{lem}
\begin{proof}
    This is a standard gaussian conditioning calculation, which we present briefly.
    For fixed $\dbv \in \bbR^N$, $\hbv \in \bbR^M$ and
    \baln
        \hbw &= \fr{\la \bm, \dbv \ra}{N(q(\bm) + \eps)} \hbv - \fr{\la \bm, \dbv \ra \la \bn, \hbv \ra}{N^2 (q(\bm) + \eps)(q(\bm) + \psi(\bn) + \eps)} \bn, \\
        \dbw &= \fr{\la \bn, \hbv \ra}{N(\psi(\bn) + \eps)} \dbv - \fr{\la \bm, \dbv \ra \la \bn, \hbv \ra}{N^2 (\psi(\bn) + \eps)(q(\bm) + \psi(\bn) + \eps)} \bm,
    \ealn
    we may verify the independence
    \[
        \fr{\la \hbv, \bG \dbv \ra}{\sqrt{N}}
        - \lt\la \hbw, \fr{\bG \bm}{\sqrt{N}} + \eps^{1/2} \hbg \rt\ra
        - \lt\la \dbw, \fr{\bG^\top \bn}{\sqrt{N}} + \eps^{1/2} \dbg \rt\ra
        \indep \lt\{
            \fr{\bG \bm}{\sqrt{N}} + \eps^{1/2} \hbg,
            \fr{\bG^\top \bn}{\sqrt{N}} + \eps^{1/2} \dbg
        \rt\}.
    \]
    By Lemma~\ref{lem:tap-1deriv}, $\nabla \cF^\eps_\TAP(\bm,\bn) = \bzero$ if and only if \eqref{eq:TAP-stationarity-m} and \eqref{eq:TAP-stationarity-n} hold.
    Let $\hbu, \dbu$ denote the right-hand sides of \eqref{eq:TAP-stationarity-m}, \eqref{eq:TAP-stationarity-n}, respectively.
    Then, for all $\dbv,\hbv$,
    \[
        \EE\lt[\fr{\la \hbv, \bG \dbv \ra}{\sqrt{N}} \bigg| \eqref{eq:TAP-stationarity-m}, \eqref{eq:TAP-stationarity-n} \rt]
        = \la \hbw, \hbu \ra + \la \dbw, \dbu \ra.
    \]
    Expanding shows $\bG$ has the conditional mean given in \eqref{eq:conditional-law}.
    The law \eqref{eq:residual-variances} of $\tbG$ follows from computing the covariance of the gaussian process
    \[
        (\dbv, \hbv)
        \mapsto
        \fr{\la \hbv, \tbG \dbv \ra}{\sqrt{N}}
        \equiv
        \fr{\la \hbv, \bG \dbv \ra}{\sqrt{N}}
        - \lt\la \hbw, \fr{\bG \bm}{\sqrt{N}} + \eps^{1/2} \hbg \rt\ra
        - \lt\la \dbw, \fr{\bG^\top \bn}{\sqrt{N}} + \eps^{1/2} \dbg \rt\ra.
    \]
    (Note that if $\hbv \in \{\hbe_2,\ldots,\hbe_M\}$, then $\la \bn, \hbv \ra = 0$ and thus $\dbw = \bzero$.
    Similarly if $\dbv \in \{\dbe_2,\ldots,\dbe_N\}$, then $\hbw = \bzero$.
    So in most cases the above formulas simplify considerably.)
\end{proof}

\begin{cor}
    \label{cor:conditional-law-correct-profile}
    If $(\bm,\bn) \in \cS_{\eps,\ups}$, then under $\PP^{\bm,\bn}_{\eps,\Pl}$, $\bG$ has law
    \beq
        \label{eq:conditional-law-correct-profile}
        \fr{\bG}{\sqrt{N}}
        \stackrel{d}{=}
        \fr{(1+o_\ups(1)) \hbh \bm^\top}{N(q_\eps + \eps)}
        + \fr{(1+o_\ups(1)) \bn \dbh^\top}{N(\psi_\eps + \eps)}
        + \fr{o_\ups(1) \bn \bm^\top}{N}
        + \fr{\tbG}{\sqrt{N}},
    \eeq
    where $o_\ups(1)$ denotes a term vanishing as $\ups \to 0$ and $\tbG$ is as in Lemma~\ref{lem:conditional-law}.
\end{cor}
\noindent This corollary is proved by a standard approximation argument, which we record as Fact~\ref{fac:pseudo-lipschitz} below.
\begin{dfn}
    \label{dfn:pseudo-lipschitz}
    A function $f : \bbR \to \bbR$ is $(2,L)$-pseudo-Lipschitz if $|f(x) - f(y)| \le L|x-y|(|x|+|y|+1)$.
\end{dfn}
\begin{fac}[Proved in Appendix~\ref{app:amp}]
    \label{fac:pseudo-lipschitz}
    Suppose $\mu,\mu' \in \cP_2(\bbR)$ and let $\mu_2 = \bbE_{x\sim \mu} [x^2]$.
    If $f$ is $(2,L)$-pseudo-Lipschitz, then
    \[
        |\bbE_\mu[f] - \bbE_{\mu'}[f]|
        \le 3L \bbW_2(\mu,\mu') (\mu_2 + \bbW_2(\mu,\mu') + 1).
    \]
\end{fac}

\begin{proof}[Proof of Corollary~\ref{cor:conditional-law-correct-profile}]
    Let $\dbh = \th_\eps^{-1}(\bm)$, $\hbh = F_{\eps,\varrho_\eps}^{-1}(\bn)$, so $(\dbh,\hbh) \in \cT_{\eps,\ups}$.
    Recall $\mu_\dbh$ defined in \eqref{eq:coord-dists}.
    Since $\dh \mapsto \th_\eps(\dh)^2$ is $(2,O(1))$-pseudo-Lipschitz, by Fact~\ref{fac:pseudo-lipschitz},
    \[
        |q(\bm) - q_\eps|
        = \lt|\bbE_{\dh \sim \mu_\dbh} [\th_\eps(\dh)^2]
        - \bbE_{\dh \sim \cN(0,\psi_\eps + \eps)} [\th_\eps(\dh)^2]\rt|
        = o_\ups(1).
    \]
    Similarly $\psi(\bn) = \psi_\eps + o_\ups(1)$ and $d_\eps (\bm,\bn) = d_\eps + o_\ups(1)$.
    Also, by gaussian integration by parts and Lemma~\ref{lem:b-varrho},
    \[
        \bbE_{\dh \sim \cN(0,\psi_\eps + \eps)}
        [\dh \th_\eps(\dh)]
        = (\psi_\eps + \eps) \varrho_\eps.
    \]
    Thus
    \[
        \lt|
            \fr{\la \bm, \dbh \ra}{N(\psi_\eps + \eps)}
            - \varrho_\eps
        \rt|
        = \lt|
            \bbE_{\dh \sim \mu_\dbh} [\dh \th_\eps(\dh)]
            - \bbE_{\dh \sim \cN(0,\psi_\eps + \eps)} [\dh \th_\eps(\dh)]
        \rt|
        = o_\ups(1).
    \]
    Similarly $\fr{\la \bn, \hbh \ra}{N(q_\eps + \eps)} = d_\eps + o_\ups(1)$.
    Finally, equation \eqref{eq:rho-tau} and regularity of $\rho_\eps, \rho'_\eps$ (recall \eqref{eq:rho-tau-regularity}) imply
    \baln
        \rho_\eps(q(\bm)) &= \varrho_\eps + o_\ups(1), &
        \rho'_\eps(q(\bm)) &= -1 + o_\ups(1).
    \ealn
    Combining these estimates shows the conditional mean of $\bG$ in \eqref{eq:conditional-law} simplifies to the form \eqref{eq:conditional-law-correct-profile}.
    In particular note that $\Delta(\bm,\bn) = o_\ups(1)$.
\end{proof}

\subsection{Deferred proofs}
\label{subsec:planted-reduction-deferred-proofs}

We now prove various results deferred from the above proof.
\begin{lem}[{\cite[Lemma 10.1]{ding2018capacity}}]
    \label{lem:E-derivative-bds}
    The function $\cE$ satisfies the following for all $x\in \bbR$.
    \begin{enumerate}[label=(\alph*),ref=(\alph*)]
        \item \label{itm:cE-bd} $0\le \cE(x) \le |x| + 1$.
        \item \label{itm:cEpr} $\cE'(x) = \cE(x)(\cE(x)-x) \in (0,1)$.
        \item \label{itm:cE2} $\cE''(x) \in (0,1)$.
        \item \label{itm:cE3} $\cE^{(3)} \in (-1/2,13)$.
    \end{enumerate}
\end{lem}
\begin{proof}[Proof of Lemma~\ref{lem:b-varrho}]
    We calculate
    \baln
        q_\eps &= \EE [\th_\eps ((\psi_\eps + \eps)^{1/2} Z)^2] \\
        &= \eps^2 (\psi_\eps + \eps)
        + 2\eps \EE[(\psi_\eps + \eps)^{1/2} Z \th ((\psi_\eps + \eps)^{1/2} Z)]
        + \EE[\th^2 ((\psi_\eps + \eps)^{1/2} Z)^2] \\
        &= \eps^2 (\psi_\eps + \eps)
        + 2\eps (\psi_\eps + \eps) \EE[1 - \th^2 ((\psi_\eps + \eps)^{1/2} Z)]
        + \EE[\th^2 ((\psi_\eps + \eps)^{1/2} Z)^2].
    \ealn
    Thus
    \[
        \EE[\th^2 ((\psi_\eps + \eps)^{1/2} Z)^2]
        = \fr{q_\eps - 2\eps (\psi_\eps + \eps) - \eps^2 (\psi_\eps + \eps)}{1 - 2\eps (\psi_\eps + \eps)},
    \]
    and
    \[
        \EE [\th'_\eps((\psi_\eps + \eps)^{1/2} Z)]
        = 1+\eps-\EE [\th^2((\psi_\eps + \eps)^{1/2} Z)]
        = \fr{1-q_\eps+\eps-\eps^2 (\psi_\eps+\eps)}{1-2\eps (\psi_\eps+\eps)}
        = \varrho_\eps.
    \]
\end{proof}
\noindent Differentiating \eqref{eq:F-explicit} and applying Lemma~\ref{lem:E-derivative-bds}\ref{itm:cEpr} shows the following fact.
\begin{fac}
    \label{fac:Fp-bounded}
    For $\eps,\varrho > 0$ and any $x\in \bbR$,
    \[
        - \fr{1+\eps^2}{\varrho + \eps(1+\eps \varrho)}
        \le F'_{\eps,\varrho}(x)
        \le -\fr{\eps}{1+\eps \varrho}.
    \]
    Thus
    \beq
        \label{eq:d2n-bound}
        1 + \varrho F'_{\eps,\varrho}(x) \ge \fr{\eps}{\varrho + \eps(1+\eps\varrho)}.
    \eeq
    For $\varrho$ in any compact set away from $0$, $|F'_{\eps,\varrho}|$, $|F''_{\eps,\varrho}|$ and $|F^{(3)}_{\eps,\varrho}|$ are uniformly bounded independently of $\eps$.
\end{fac}
\begin{proof}[Proof of Proposition~\ref{ppn:n-convexity}]
    It is clear that $\nabla^2_{\bn,\bn} \cF^\eps_\TAP(\bm,\bn)$ is diagonal, so it suffices to check $\partial^2_{n_a} \cF^\eps_\TAP(\bm,\bn) \ge \eta$ for all $a \in [M]$.
    We calculate
    \baln
        \partial^2_{n_a} \cF^\eps_\TAP(\bm,\bn)
        &= \rho_\eps(q(\bm)) \lt(
            1 + \rho_\eps(q(\bm)) F'_{\eps,\varrho}\lt(
                \fr{\la \bg^a, \bm \ra}{\sqrt{N}}
                + \eps^{1/2} \hg_a
                - \rho_\eps(q(\bm)) n_a
            \rt)
        \rt) \\
        &\stackrel{\eqref{eq:d2n-bound}}{\ge} \fr{\eps \rho_\eps(q(\bm))}{\rho_\eps(q(\bm)) + \eps(1+\eps \rho_\eps(q(\bm)))}.
    \ealn
    Since $\rho_\eps \in [C_\bd^{-1}, C_\bd]$ the result follows.
\end{proof}

\begin{proof}[Proof of Lemma~\ref{lem:conditioning-lem-condition-satisfied}]
    The function $x \mapsto \rho_\eps(q(\bm)) F_{\eps,\rho_\eps(q(\bm))}(x)$ is uniformly Lipschitz over $\bm \in \bbR^N$, because $\rho_\eps(q(\bm)) \in [C_\bd^{-1},C_\bd]$.
    Note that $\hbg$ appears in \eqref{eq:tap-deriv-n} through the term $\eps^{1/2} \hbg$ in $\abh$ and is independent of all other terms apeparing in \eqref{eq:tap-deriv-n}.
    Thus $\varphi_{\nabla_\bn \cF^\eps_\TAP(\bm,\bn)}(\bz)$ is bounded, and continuous for $\bz$ in an neighborhood of $\bzero$, uniformly in $\bm,\bn$.
    Similarly, $\dbg$ appears in \eqref{eq:tap-deriv-m}, \eqref{eq:tap-deriv-n} only as the term $\eps^{1/2}\dbg$ in \eqref{eq:tap-deriv-m}.
    This implies the conclusion.
\end{proof}

\section{Analysis of AMP}
\label{sec:amp}

In this section, we prove items \ref{itm:amp-guarantee-profile}, \ref{itm:amp-guarantee-stationary}, and \ref{itm:amp-guarantee-planted} of Proposition~\ref{ppn:amp-guarantees}.
Item~\ref{itm:amp-guarantee-concave} will be proved in \S\ref{sec:local-concavity}.

\subsection{Scalar recursions}
For $q \in [0,q_\eps]$, $\psi \in [0,\psi_\eps]$, define
\baln
    P_\AMP (\psi) &= \EE[
        \th_\eps((\psi + \eps)^{1/2} Z + (\psi_\eps - \psi)^{1/2} Z')
        \th_\eps((\psi + \eps)^{1/2} Z + (\psi_\eps - \psi)^{1/2} Z'')
    ], \\
    R_\AMP (q) &= \alpha_\star \EE[
        F_{\eps,\varrho_\eps}((q + \eps)^{1/2} Z + (q_\eps - q)^{1/2} Z')
        F_{\eps,\varrho_\eps}((q + \eps)^{1/2} Z + (q_\eps - q)^{1/2} Z'')
    ],
\ealn
Define the sequences $(\oq_k)_{k\ge 0}$ and $(\opsi_k)_{k\ge 1}$ by $\oq_0 = 0$ and the recursion
\baln
    \opsi_{k+1} &= R_\AMP(\oq_k), &
    \oq_k &= P_\AMP(\opsi_k).
\ealn
\begin{lem}
    \label{lem:amp-convergence}
    The sequences $(\oq_k)_{k\ge 0}$, $(\opsi_k)_{k\ge 1}$ are increasing, and for small $\eps$, we have $\oq_k \uparrow q_\eps$ and $\opsi_k \uparrow \psi_\eps$.
\end{lem}
\begin{proof}
    Let the functions
    \baln
        \tth_\eps(x) &= \th_\eps((\psi_\eps + \eps)^{1/2} x), &
        \tF_\eps(x) &= F_{\eps,\varrho_\eps}((q_\eps + \eps)^{1/2} x)
    \ealn
    have Hermite expansions
    \baln
        \tth_\eps(x) &= \sum_{p\ge 0} a_p H_p(x), &
        \tF_\eps(x) &= \sum_{p\ge 0} b_p H_p(x),
    \ealn
    where $H_p(x)$ is the $p$-th Hermite polynomial, normalized to $\EE H_p(Z)^2 = 1$.
    Then
    \baln
        P_\AMP(\psi)
        &= \sum_{p\ge 0} a_p^2 \lt(\fr{\psi + \eps}{\psi_\eps + \eps}\rt)^p, &
        R_\AMP(q)
        &= \alpha_\star \sum_{p\ge 0} b_p^2 \lt(\fr{q + \eps}{q_\eps + \eps}\rt)^p.
    \ealn
    So, $P_\AMP$ and $R_\AMP$ are increasing and convex.
    Thus $(\oq_k)_{k\ge 0}$, $(\opsi_k)_{k\ge 1}$ are increasing, and their limit is the smallest fixed point of $P_\AMP \circ R_\AMP$.
    It remains to show this fixed point is $(q_\eps,\psi_\eps)$.
    By definition of $q_\eps, \psi_\eps$, $(q_\eps,\psi_\eps)$ is a fixed point.
    Since $P_\AMP \circ R_\AMP$ is convex, it suffices to show $(P_\AMP \circ R_\AMP)'(q_\eps) < 1$.
    Note that
    \[
        (P_\AMP \circ R_\AMP)'(q_\eps)
        = P'_\AMP(\psi_\eps)
        R'_\AMP(q_\eps).
    \]
    By gaussian integration by parts,
    \baln
        P'_\AMP (\psi) &= \EE[
            \th'_\eps((\psi + \eps)^{1/2} Z + (\psi_\eps - \psi)^{1/2} Z')
            \th'_\eps((\psi + \eps)^{1/2} Z + (\psi_\eps - \psi)^{1/2} Z'')
        ], \\
        R'_\AMP (q) &= \alpha_\star \EE[
            F'_{\eps,\varrho_\eps}((q + \eps)^{1/2} Z + (q_\eps - q)^{1/2} Z')
            F'_{\eps,\varrho_\eps}((q + \eps)^{1/2} Z + (q_\eps - q)^{1/2} Z'')
        ],
    \ealn
    and in particular
    \baln
        P'_\AMP (\psi_\eps) &= \EE[\th'_\eps((\psi_\eps + \eps)^{1/2} Z)^2], &
        R'_\AMP (q_\eps) &= \alpha_\star \EE[F'_{\eps,\varrho_\eps}((q_\eps + \eps)^{1/2} Z)^2].
    \ealn
    In light of Proposition~\ref{ppn:eps-perturb-fixed-point}, a simple continuity argument shows
    \baln
        \EE[\th'_\eps((\psi_\eps + \eps)^{1/2} Z)^2]
        &\stackrel{\eps \downarrow 0}{\to}
        \EE[\th'(\psi_0^{1/2} Z)^2], &
        \EE[F'_{\eps,\varrho_\eps}((q_\eps + \eps)^{1/2} Z)^2]
        &\stackrel{\eps \downarrow 0}{\to}
        \EE[F'_{1-q_0}(q_0^{1/2} Z)^2].
    \ealn
    Thus,
    \baln
        (P_\AMP \circ R_\AMP)'(q_\eps)
        &= \alpha_\star
        \EE[\th'_\eps((\psi_\eps + \eps)^{1/2} Z)^2]
        \EE[F'_{\eps,\varrho_\eps}((q_\eps + \eps)^{1/2} Z)^2] \\
        &\stackrel{\eps \downarrow 0}{\to}
        \alpha_\star
        \EE[\th'(\psi_0^{1/2}Z)^2]
        \EE [F'_{1-q_0}(q_0^{1/2} Z)^2]
        \stackrel{Cond.~\ref{con:amp-works}}{<} 1.
    \ealn
    Thus, $(R_\AMP \circ P_\AMP)'(q_\eps) < 1$ for sufficiently small $\eps$.
    Hence $\oq_k \uparrow q_\eps$ and $\opsi_k \uparrow \psi_\eps$.
\end{proof}

\subsection{State evolution}

The limiting overlap structure of the AMP iterates in the null model follows directly from the state evolution of \cite{bolthausen2014iterative, bayati2011dynamics, javanmard2013state, berthier2020state}.
Define the infinite arrays $(\dSig_{i,j} : i,j\ge 1)$ and $(\hSig_{i,j} : i,j\ge 0)$ by
\baln
    \dSig_{i,j} &=
    \begin{cases}
        \psi_\eps & i=j, \\
        \opsi_{i\wedge j} & i\neq j,
    \end{cases} &
    \hSig_{i,j} &=
    \begin{cases}
        q_\eps & i=j, \\
        \oq_{i\wedge j} & i\neq j.
    \end{cases}
\ealn
For any $k\ge 0$, let $\dSig_{\le k} \in \bbR^{k\times k}$ and $\hSig^+_{\le k} \in \bbR^{(k+1)\times (k+1)}$ denote the sub-arrays indexed by $i,j\le k$.
\begin{ppn}
    \label{ppn:state-evolution}
    For any $k\ge 0$, as $N\to\infty$ the empirical coordinate distribution of the AMP iterates converges in $\bbW_2$ in probability under $\PP$, to
    \balnn
        \label{eq:state-evolution}
        \fr1N \sum_{i=1}^N \delta(\dh^1_i,\ldots,\dh^k_i)
        &\stackrel{\bbW_2}{\to}
        \cN(0, \dSig_{\le k} + \eps \bone \bone^\top), &
        \fr1M \sum_{a=1}^M \delta(\hh^0_a,\ldots,\hh^k_a)
        &\stackrel{\bbW_2}{\to}
        \cN(0, \hSig_{\le k} + \eps \bone \bone^\top).
    \ealnn
\end{ppn}
\begin{proof}
    The state evolution \cite[Theorem 1]{berthier2020state} implies that (in probability)
    \baln
        \fr1N \sum_{i=1}^N \delta(\dh^1_i,\ldots,\dh^k_i)
        &\stackrel{\bbW_2}{\to}
        \cN(0, \dSig^{(0)}_{\le k} + \eps \bone \bone^\top), &
        \fr1M \sum_{a=1}^M \delta(\hh^0_a,\ldots,\hh^k_a)
        &\stackrel{\bbW_2}{\to}
        \cN(0, \hSig^{(0)}_{\le k} + \eps \bone \bone^\top).
    \ealn
    holds for arrays $\dSig^{(0)}$, $\hSig^{(0)}$ defined as follows.
    As initialization, $\hSig^{(0)}_{0,i} = \hSig^{(0)}_{i,0} = \hSig_{0,i}$ for all $i\ge 0$.
    Then, for $(\hH_0,\ldots,\hH_k) \sim \cN(0, \hSig^{(0)}_{\le k} + \eps \bone \bone^\top)$ and $0\le i\le k$, define recursively
    \[
        \dSig^{(0)}_{k+1,i+1} = \dSig^{(0)}_{i+1,k+1} = \alpha_\star \EE[F_{\eps,\varrho_\eps}(\hH_i)F_{\eps,\varrho_\eps}(\hH_k)].
    \]
    For $(\dH_0,\ldots,\dH_{k+1}) \sim \cN(0, \dSig^{(0)}_{\le k+1} + \eps \bone \bone^\top)$ and $1\le i\le k+1$, let
    \[
        \hSig^{(0)}_{k+1,i} = \hSig^{(0)}_{i,k+1} = \EE[\th_\eps(\dH_i)\th_\eps(\dH_{k+1})].
    \]
    It remains to show $\dSig^{(0)}, \hSig^{(0)}$ coincide with $\dSig,\hSig$.
    Since $\hSig_{0,0} = q_\eps$, induction shows the diagonal entries are
    \baln
        \dSig^{(0)}_{k,k} &= \psi_\eps = \dSig_{k,k}, &
        \hSig^{(0)}_{k,k} &= q_\eps = \hSig_{k,k}.
    \ealn
    Then, the above recursion gives $\dSig^{(0)}_{i+1,j+1} = R_\AMP(\hSig^{(0)}_{i,j})$, $\hSig^{(0)}_{i,j} = P_\AMP(\dSig^{(0)}_{i,j})$.
    By induction, for $i\neq j$,
    \baln
        \dSig^{(0)}_{i,j} &= \opsi_{i\wedge j} = \dSig_{i,j}, &
        \hSig^{(0)}_{i,j} &= \oq_{i\wedge j} = \hSig_{i,j}.
    \ealn
    Thus $\dSig^{(0)} = \dSig$ and $\hSig^{(0)} = \hSig$.
\end{proof}
The following proposition characterizes the limiting overlap structure in the planted model.
To conserve notation, we will denote the planted solution by $(\bm,\bn)$, rather than $(\bm',\bn')$ as in Proposition~\ref{ppn:amp-guarantees}.
\begin{ppn}
    \label{ppn:planted-state-evolution}
    Let $(\bm,\bn) \in \cS_{\eps,o_N(1)}$, $\dbh = \th^{-1}_\eps(\bm)$, $\hbh = F^{-1}_{\eps,\varrho_\eps}(\bn)$, and $(\bG,\dbg,\hbg) \sim \PP^{\bm,\bn}_{\eps,\Pl}$.
    For any $k\ge 0$, as $N\to\infty$ the empirical coordinate distribution of $(\dbh,\hbh)$ and the AMP iterates converges in $\bbW_2$ in probability under $\bbP^{\bm,\bn}_{\eps,\Pl}$, to
    \baln
        \fr1N \sum_{i=1}^N \delta(\dh^1_i,\ldots,\dh^k_i,\dh_i)
        &\stackrel{\bbW_2}{\to}
        \cN(0, \dSig_{\le k+1} + \eps \bone \bone^\top), &
        \fr1M \sum_{a=1}^M \delta(\hh^0_a,\ldots,\hh^k_a,\hh_a)
        &\stackrel{\bbW_2}{\to}
        \cN(0, \hSig_{\le k+1} + \eps \bone \bone^\top).
    \ealn
\end{ppn}
We prove this proposition by introducing an auxiliary AMP iteration.
We fix $\bm,\bn,\dbh,\hbh$ as in Proposition~\ref{ppn:planted-state-evolution}.
Let $\tbG \in \bbR^{M\times N}$ be given by \eqref{eq:residual-variances} and $\hbG \in \bbR^{M\times N}$ have i.i.d. $\cN(0,1)$ entries, and couple these matrices so that a.s.
\beq
    \label{eq:amp-def-obG}
    P_\bn^\perp \tbG P_\bm^\perp = P_\bn^\perp \hbG P_\bm^\perp,
\eeq
and, with $\obG$ denoting this common value, $\tbG - \obG$ and $\hbG - \obG$ are independent.
Further, let $Z \sim \cN(0,1)$, $\dbxi \sim \cN(0,\bI_N)$, $\hbxi \sim \cN(0,\bI_M)$ be coupled to $\tbG$ such that
\balnn
    \label{eq:tbG-decomp}
    \tbG + \bDel
    &= \obG
    - \sqrt{\fr{\eps}{q(\bm) + \eps}} \cdot \fr{\hbxi \bm^\top}{\tnorm{\bm}}
    - \sqrt{\fr{\eps}{\psi(\bn) + \eps}} \cdot \fr{\bn \dbxi^\top}{\tnorm{\bn}}, \qquad \text{where} \\
    \label{eq:def-bDel-amp-approximation}
    \bDel &= \sqrt{\fr{\eps}{q(\bm) + \eps} + \fr{\eps}{\psi(\bn) + \eps} - \fr{\eps}{q(\bm) + \psi(\bn) + \eps}} \fr{\bn\bm^\top}{\tnorm{\bn}\tnorm{\bm}} Z.
\ealnn
(Such a coupling exists by \eqref{eq:residual-variances}.)
The auxiliary AMP iteration has initialization $\bn^{(1),-1} = \bzero$, $\bm^{(1),0} = q_\eps^{1/2} \bone$, and iteration
\baln
    \bm^{(1),k} &= \th_\eps(\dbh^{(1),k}), &
    \bn^{(1),k} &= F_{\eps,\varrho_\eps}(\hbh^{(1),k}),
\ealn
for $\dbh^{(1),k}, \hbh^{(1),k}$ as follows.
Let $\opsi_0 = 0$, and
\balnn
    \label{eq:hbh1-iteration}
    \hbh^{(1),k}
    &=
    \fr{1}{\sqrt{N}} \hbG \lt( \bm^{(1),k} - \fr{\oq_k}{q_\eps} \bm \rt)
    + \fr{\sqrt{\eps} (q_\eps - \oq_k)}{\sqrt{q_\eps(q_\eps + \eps)}} \hbxi
    + \fr{\oq_k + \eps}{q_\eps + \eps} \hbh
    - \varrho_\eps \lt( \bn^{(1),k-1} - \fr{\opsi_k}{\psi_\eps} \bn \rt) \\
    \notag
    \dbh^{(1),k+1}
    &= \fr{1}{\sqrt{N}} \hbG^\top \lt( \bn^{(1),k} - \fr{\opsi_{k+1}}{\psi_\eps} \bn \rt)
    + \fr{\sqrt{\eps} (\psi_\eps - \psi_{k+1})}{\sqrt{\psi_\eps(\psi_\eps + \eps)}} \dbxi
    + \fr{\opsi_{k+1} + \eps}{\psi_\eps + \eps} \dbh
    - d_\eps \lt( \bm^{(1),k} - \fr{\oq_k}{q_\eps} \bm \rt).
\ealnn
Define augmented arrays $(\dSig^+_{i,j} : i,j \in \{\diamond,\bowtie\} \cup \bbZ_{\ge 1})$ and $(\hSig^+_{i,j} : i,j \in \{\diamond,\bowtie\} \cup \bbZ_{\ge 0})$ by
\baln
    \dSig^+_{i,j} &=
    \begin{cases}
        \psi_\eps + \eps & i=j \ge 1 ~\text{or}~ i=j=\diamond, \\
        \opsi_j + \eps & i > j \ge 1, \\
        \opsi_i + \eps & i \ge 1, j = \diamond, \\
        \fr{\sqrt{\eps} (\psi_\eps - \opsi_i)}{\sqrt{\psi_\eps (\psi_\eps + \eps)}} & i \ge 1, j = \bowtie, \\
        1 & i = j = \bowtie, \\
        0 & i = \diamond, j = \bowtie,
    \end{cases} &
    \hSig^+_{i,j} &=
    \begin{cases}
        q_\eps + \eps & i=j \ge 0 ~\text{or}~ i=j=\diamond, \\
        \oq_j + \eps & i > j \ge 0, \\
        \oq_i + \eps & i \ge 0, j = \diamond, \\
        \fr{\sqrt{\eps} (q_\eps - \oq_i)}{\sqrt{q_\eps (q_\eps + \eps)}} & i \ge 0, j = \bowtie, \\
        1 & i = j = \bowtie, \\
        0 & i = \diamond, j = \bowtie,
    \end{cases}
\ealn
with the remaining entries defined by symmetry over the diagonal.
Note that on indices $(i,j)$ where $\{i,j\} \cap \{\diamond,\bowtie\} = \emptyset$, these arrays coincide with $\dSig + \eps \bone \bone^\top$ and $\hSig + \eps \bone \bone^\top$.
Let $\dSig^+_{\le k} \in \bbR^{(k+2)\times (k+2)}$ and $\hSig^+_{\le k} \in \bbR^{(k+3)\times (k+3)}$ denote the sub-arrays indexed by $\{\diamond,\bowtie\}$ and $\{1,\ldots,k\}$ (resp. $\{0,\ldots,k\}$).
\begin{ppn}[Proved in Appendix~\ref{app:amp}]
    \label{ppn:planted-state-evolution-aux}
    For any $k\ge 0$, as $N\to\infty$ we have the convergence in $\bbW_2$ in probability under $\bbP^{\bm,\bn}_{\eps,\Pl}$
    \baln
        \fr1N \sum_{i=1}^N \delta(\dh_i,\dxi_i,\dh^{(1),1}_i,\ldots,\dh^{(1),k}_i)
        &\stackrel{\bbW_2}{\to}
        \cN(0, \dSig^+_{\le k}), &
        \fr1M \sum_{a=1}^M \delta(\hh_a,\hxi_a,\hh^{(1),0}_a,\ldots,\hh^{(1),k}_a)
        &\stackrel{\bbW_2}{\to}
        \cN(0, \hSig^+_{\le k}).
    \ealn
\end{ppn}
This is proved by applying state evolution, analogously to Proposition~\ref{ppn:state-evolution}.
We next show that this AMP iteration approximates the original one, in the following sense.
\begin{ppn}[Proved in Appendix~\ref{app:amp}]
    \label{ppn:amp-approximation-main}
    For any $k\ge 0$, as $N\to\infty$ we have $\tnorm{\hbh^{(1),k} - \hbh^k} / \sqrt{N} \to 0$ in probability under $\bbP^{\bm,\bn}_{\eps,\Pl}$ and if $k\ge 1$, $\tnorm{\dbh^{(1),k} - \dbh^k} / \sqrt{N} \to 0$ in probability under $\bbP^{\bm,\bn}_{\eps,\Pl}$.
\end{ppn}
\begin{proof}[Proof of Proposition~\ref{ppn:planted-state-evolution}]
    If we identify index $\diamond$ with $k+1$, the array $\{\dSig^+_{i,j} : i,j \in \{\diamond\} \cup \{1,\ldots,k\}\}$ coincides with $\dSig_{\le k+1} + \eps \bone \bone^\top$, and similarly $\{\hSig^+_{i,j} : i,j \in \{\diamond\} \cup \{0,\ldots,k\}\}$ coincides with $\hSig_{\le k+1} + \eps \bone \bone^\top$.
    By Proposition~\ref{ppn:planted-state-evolution-aux},
    \baln
        \fr1N \sum_{i=1}^N \delta(\dh^{(1),1}_i,\ldots,\dh^{(1),k}_i,\dh_i)
        &\stackrel{\bbW_2}{\to}
        \cN(0, \dSig^+_{\le k+1} + \eps \bone \bone^\top ), \\
        \fr1M \sum_{a=1}^M \delta(\hh^{(1),0}_a,\ldots,\hh^{(1),k}_a,\hh_a)
        &\stackrel{\bbW_2}{\to}
        \cN(0, \hSig^+_{\le k+1} + \eps \bone \bone^\top )
    \ealn
    in probability under $\bbP^{\bm,\bn}_{\eps,\Pl}$.
    Proposition~\ref{ppn:amp-approximation-main} implies the conclusion.
\end{proof}

\subsection{Completion of the proof}

We separately prove Proposition~\ref{ppn:amp-guarantees} under $\PP$ and $\PP^{\bm,\bn}_{\eps,\Pl}$.

\begin{proof}[Proof of Proposition~\ref{ppn:amp-guarantees}\ref{itm:amp-guarantee-profile}\ref{itm:amp-guarantee-stationary}, under $\PP$]
    By Proposition~\ref{ppn:state-evolution}, for \emph{any} $k$,
    \baln
        \mu_{\dbh^k} &\stackrel{\bbW_2}{\to} \cN(0, \psi_\eps + \eps), &
        \mu_{\hbh^k} &\stackrel{\bbW_2}{\to} \cN(0, q_\eps + \eps)
    \ealn
    in probability.
    So, with high probability, $(\dbh^k,\hbh^k) \in \cT_{\eps,\ups_0}$ and thus item \ref{itm:amp-guarantee-profile} holds.
    Approximation arguments similar to the proof of Corollary~\ref{cor:conditional-law-correct-profile} using Fact~\ref{fac:pseudo-lipschitz} yield
    \[
        q(\bm^k) \to \EE [\th_\eps((\psi_\eps + \eps)^{1/2} Z)^2] = q_\eps
    \]
    in probability.
    Regularity of $\rho_\eps$ and its derivatives then implies
    \baln
        \rho_\eps(q(\bm^k)) &\to \varrho_\eps, &
        \rho'_\eps(q(\bm^k)) &\to -1
    \ealn
    in probability.
    Proposition~\ref{ppn:state-evolution} also implies
    \[
        \lim_{k\to\infty} \plim_{N\to\infty} \fr1N \tnorm{\dbh^{k+1} - \dbh^k}^2
        = \lim_{k\to\infty} \plim_{N\to\infty} \fr1N \tnorm{\hbh^{k+1} - \hbh^k}^2
        = 0.
    \]
    Below, let $o_{k,P}(\sqrt{N})$ denote a random vector $\bv$ such that $\lim_{k\to\infty} \plim_{N\to\infty} \fr{1}{\sqrt{N}} \tnorm{\bv} = 0$, and $o_{k,P}(1)$ denote a random scalar $\iota$ such that $\lim_{k\to\infty} \plim_{N\to\infty} |\iota| = 0$.
    Let
    \[
        \abh^k = \fr{\bG \bm^k}{\sqrt{N}} + \eps^{1/2} \hbg - \rho_\eps(q(\bm^k)) \bn^k.
    \]
    By Lemma~\ref{lem:b-varrho},
    \[
        \hbh^k = \fr{\bG \bm^k}{\sqrt{N}} + \eps^{1/2} \hbg - \varrho_\eps \bn^{k-1}.
    \]
    The above discussion implies $\hbh^k - \abh^k = o_{k,P}(\sqrt{N})$, and thus $\bn^k - F_{\eps,\rho_\eps(q(\bm))}(\abh^k) = o_{k,P}(\sqrt{N})$.
    By \eqref{eq:tap-deriv-n},
    \[
        \nabla_\bn \cF^\eps_\TAP(\bm^k,\bn^k) = o_{k,P}(\sqrt{N}).
    \]
    Moreover,
    \[
        d_\eps(\bm^k,\bn^k)
        = \fr1N \sum_{a=1}^M F'_{\eps,\varrho_\eps}(\hh^k) + o_{k,P}(1)
        = d_\eps + o_{k,P}(1),
    \]
    for $d_\eps$ defined below Lemma~\ref{lem:b-varrho}.
    So
    \[
        \nabla_\bm \cF^\eps_\TAP(\bm^k,\bn^k)
        = -\th^{-1}_\eps(\bm^k) + \fr{\bG^\top \bn^k}{\sqrt{N}} + \eps^{1/2} \dbg - d_\eps \bm^k + \lt(1 + \fr{\tnorm{\bG}_\op}{\sqrt{N}}\rt) o_{k,P}(\sqrt{N}).
    \]
    Since $\tnorm{\bG}_\op \le C\sqrt{N}$ w.h.p.,
    \baln
        \nabla_\bm \cF^\eps_\TAP(\bm^k,\bn^k)
        &= -\dbh^k + \fr{\bG^\top \bn^k}{\sqrt{N}} + \eps^{1/2} \dbg - d_\eps \bm^k + o_{k,P}(\sqrt{N}) \\
        &= \dbh^{k+1} - \dbh^k + o_{k,P}(\sqrt{N})
        = o_{k,P}(\sqrt{N}),
    \ealn
    proving item \ref{itm:amp-guarantee-stationary}.
\end{proof}
\begin{proof}[Proof of Proposition~\ref{ppn:amp-guarantees}\ref{itm:amp-guarantee-profile}\ref{itm:amp-guarantee-stationary}\ref{itm:amp-guarantee-planted}, under $\PP^{\bm,\bn}_{\eps,\Pl}$]
    Suppose first $(\bm,\bn) \in \cS_{\eps,o_N(1)}$, and let $\dbh = \th^{-1}_\eps(\bm)$, $\hbh = F^{-1}_{\eps,\varrho_\eps}(\bn)$.
    The above argument, using Proposition~\ref{ppn:planted-state-evolution} in place of Proposition~\ref{ppn:state-evolution}, shows items \ref{itm:amp-guarantee-profile} and \ref{itm:amp-guarantee-stationary} hold with high probability under $\PP^{\bm,\bn}_{\eps,\Pl}$.
    Proposition~\ref{ppn:planted-state-evolution} also yields
    \[
        \lim_{k\to\infty} \plim_{N\to\infty} \fr1N \tnorm{\dbh^k - \dbh}^2
        = \lim_{k\to\infty} \plim_{N\to\infty} \fr1N \tnorm{\hbh^k - \hbh}^2
        = 0.
    \]
    Thus item \ref{itm:amp-guarantee-planted} holds with high probability under $\PP^{\bm,\bn}_{\eps,\Pl}$.
    Finally, we show this remains true for $(\bm,\bn) \in \cS_{\eps,\ups}$, for suitably small $\ups$.
    Let $(\obm,\obn) \in \cS_{\eps,o_N(1)}$ be such that $\fr1N \tnorm{\bm-\obm}^2, \fr1N \tnorm{\bn-\obn}^2 = o_\ups(1)$.
    We will show there is a coupling of $(\bG,\dbg,\hbg) \sim \PP^{\bm,\bn}_{\eps,\Pl}$ and $(\obG,\dobg,\hobg) \sim \PP^{\obm,\obn}_{\eps,\Pl}$ such that
    \beq
        \label{eq:planted-lipschitz-goal}
        \tnorm{\bG-\obG}_\op,
        \tnorm{\dbg-\dobg},
        \tnorm{\hbg-\hobg} \le o_\ups(1) \sqrt{N}.
    \eeq
    If $(\bm^k,\bn^k)$ are the AMP iterates under $\PP^{\bm,\bn}_{\eps,\Pl}$ and $(\obm^k,\obn^k)$ are the AMP iterates under $\PP^{\obm,\obn}_{\eps,\Pl}$, this implies $\tnorm{\bm^k - \obm^k}, \tnorm{\bn^k - \obn^k} \le o_\ups(1) \sqrt{N}$ (this uses crucially that $\ups$ is set small depending on $k$).
    This implies \ref{itm:amp-guarantee-profile} and \ref{itm:amp-guarantee-planted} continue to hold, and similar approximation arguments to above show \ref{itm:amp-guarantee-stationary} continues to hold.

    We now prove \eqref{eq:planted-lipschitz-goal}.
    Let $\dobh = \th_\eps^{-1}(\obm)$ and $\hobh = F^{-1}_{\eps,\rho_\eps(q(\obm))}(\obn)$.
    Another approximation argument shows $\tnorm{\dbh - \dobh}, \tnorm{\hbh - \hobh} \le o_\ups(1) \sqrt{N}$.
    The conditional means of $\bG,\obG$ are given by \eqref{eq:conditional-law}, and an approximation argument shows
    \[
        \norm{
            \bbE^{\bm,\bn}_{\eps,\Pl}[\bG] -
            \bbE^{\obm,\obn}_{\eps,\Pl}[\obG]
        }_\op \le o_\ups(1) \sqrt{N}.
    \]
    We couple the random parts $\tbG,\tobG$ as follows.
    Let $\dbe_1,\hbe_1$ (resp. $\dobe_1,\hobe_1$) be the the unit vectors parallel to $\bm,\bn$ (resp. $\obm,\obn$).
    Let $\dT, \hT$ be rotation operators on $\bbR^N, \bbR^M$ with $\dT \dbe_1 = \dobe_1$ and $\hT \hbe_1 = \hobe_1$.
    These can be set so $\tnorm{\dT - \bI_N}_\op, \tnorm{\hT - \bI_M}_\op \le o_\ups(1)$.
    By \eqref{eq:residual-variances}, we can couple $\tbG,\tobG$ such that $\tobG = \hT \tbG \dT^{-1}$.
    Since, for some absolute constant $C$,
    $\tnorm{\tbG}_\op \le C\sqrt{N}$ with high probability,
    on this event
    \[
        \tnorm{\tbG - \tobG}_\op
        \le \tnorm{\tbG}_\op (\tnorm{\dT - \bI_N}_\op + \tnorm{\hT - \bI_M}_\op)
        = o_\ups(1) \sqrt{N}.
    \]
    Thus $\tnorm{\bG - \obG}_\op \le o_\ups(1)\sqrt{N}$.
    The stationary equations \eqref{eq:TAP-stationarity-m}, \eqref{eq:TAP-stationarity-n} then imply $\tnorm{\dbg - \dobg}_\op, \tnorm{\hbg - \hobg}_\op \le o_\ups(1) \sqrt{N}$.
    This proves \eqref{eq:planted-lipschitz-goal}.
\end{proof}

\section{Local concavity of perturbed TAP free energy}
\label{sec:local-concavity}

In this section, we prove Lemmas~\ref{lem:freeprob-well-defd} and \ref{lem:det-concentration} and Proposition~\ref{ppn:amp-guarantees}\ref{itm:amp-guarantee-concave}.

\subsection{Description of spectral gap bound}

We first define a quantity $\lambda_\eps$, which is a perturbed analog of the value $\lambda_0 = \inf_{z>-1} \lambda(z)$ defined in Condition~\ref{con:local-concavity}.
We will see that $\lambda_\eps$ upper bounds the maximum eigenvalue of $\nabla^2_\diamond \cF^\eps_\TAP$ near late AMP iterates.
To define $\lambda_\eps$, we introduce $\eps$-perturbed variants of quantities appearing in Condition~\ref{con:local-concavity} and Lemma~\ref{lem:freeprob-well-defd}.
Let
\baln
    \df_\eps(x) &= \fr{\ch^2 x}{1 + \eps \ch^2(x)}, &
    \hf_\eps(x) &= -\fr{F'_{\eps,\varrho_\eps}(x)}{1 + \varrho_\eps F'_{\eps,\varrho_\eps}(x)}.
\ealn
We extend these definitions to $\eps = 0$ by defining $\df_0(x) = \ch^2(x)$ and $\hf_0$ as in Condition~\ref{con:local-concavity}; this extension will be used solely in Lemma~\ref{lem:freeprob-well-defd-eps} and the proof of Lemma~\ref{lem:freeprob-well-defd} below.

Note that $\df_\eps$ and $\hf_\eps$ are positive, the latter because Fact~\ref{fac:Fp-bounded} implies $F'_{\eps,\varrho_\eps}(x) < 0$ and $1 + \varrho_\eps F'_{\eps,\varrho_\eps}(x) > 0$, and $\df_\eps(x)$ has minimum $\df_\eps(0) = \fr{1}{1+\eps}$.
The function $\hf_0$ is also positive, as Lemma~\ref{lem:E-derivative-bds}\ref{itm:cEpr} implies $F'_{1-q_0}(x) < 0$ and $1 + (1-q_0) F'_{1-q_0}(x) > 0$.
In the below, it will be convenient to abbreviate $\tq_\eps = q_\eps + \eps$, $\tpsi_\eps = \psi_\eps + \eps$.
\begin{lem}
    \label{lem:freeprob-well-defd-eps}
    For any $\eps \ge 0$ (including $\eps=0$), the functions $m_\eps, \theta_\eps : (-\fr{1}{1+\eps},+\infty) \to (0,+\infty)$ defined by
    \baln
        m_\eps(z) &= \EE[(z + \df_\eps(\tpsi_\eps^{1/2} Z))^{-1}], \\
        \theta_\eps(z) &= \EE[(z + \df_\eps(\tpsi_\eps^{1/2} Z))^{-2}]
        \EE \lt[\lt(\fr{\hf_\eps(\tq_\eps^{1/2} Z)}{1 + m_\eps(z) \hf_\eps(\tq_\eps^{1/2} Z)}\rt)^2\rt]
    \ealn
    are continuous and strictly decreasing, with
    \baln
        \lim_{z\downarrow -(1+\eps)^{-1}} m_\eps(z)
        &= \lim_{z\downarrow -(1+\eps)^{-1}} \theta_\eps(z)
        = +\infty, &
        \lim_{z\uparrow +\infty} m_\eps(z)
        &= \lim_{z\uparrow +\infty} \theta_\eps(z)
        =0.
    \ealn
    In particular $\theta_\eps$ has a well-defined inverse $\theta_\eps^{-1} : (0,+\infty) \to (-\fr{1}{1+\eps},+\infty)$.
\end{lem}
\begin{proof}[Proof of Lemma~\ref{lem:freeprob-well-defd-eps}]
    Note that $m_\eps(z)$ is clearly decreasing on $(-\fr{1}{1+\eps},+\infty)$ with $\lim_{z\uparrow +\infty} m_\eps(z) = 0$.
    To show the other limit, let
    \[
        \dg_\eps(x) = \df_\eps(x) - \fr{1}{1+\eps} = \fr{\sh^2(x)}{(1+\eps)(1 + \eps \ch^2(x))}.
    \]
    For $z = -\fr{1}{1+\eps} + \iota$, with $\iota>0$ small,
    \[
        m_\eps(z)
        = \EE[(\iota + \dg_\eps(\tpsi_\eps^{1/2} Z))^{-1}]
        \ge \EE[\bone\{|Z| \le \iota^{1/2}\} (\iota + \dg_\eps(\tpsi_\eps^{1/2} Z))^{-1}]
        \ge \Omega(\iota^{-1/2}).
    \]
    Thus $\lim_{z\downarrow -(1+\eps)^{-1}} m_\eps(z) = +\infty$.
    We can write $\theta_\eps(z)$ as
    \beq
        \label{eq:theta-eps-alt}
        \theta_\eps(z)
        = \fr{\EE[(z + \df_\eps(\tpsi_\eps^{1/2} Z))^{-2}]}{\EE[(z + \df_\eps(\tpsi_\eps^{1/2} Z))^{-1}]^2}
        \EE \lt[\fr{(m_\eps(z) \hf_\eps(\tq_\eps^{1/2} Z))^2}{(1 + m_\eps(z) \hf_\eps\tq_\eps^{1/2} Z))^2}\rt].
    \eeq
    Since $m_\eps(z)$ is decreasing and $\hf_\eps$ is positive, the second factor of \eqref{eq:theta-eps-alt} is manifestly decreasing.
    The $z$-derivative of the first is
    \[
        \fr{
            - \EE[(z + \df_\eps(\tpsi_\eps^{1/2} Z))^{-1}] \EE[(z + \df_\eps(\tpsi_\eps^{1/2} Z))^{-3}]
            + \EE[(z + \df_\eps(\tpsi_\eps^{1/2} Z))^{-2}]^2
        }{
            \EE[(z + \df_\eps(\tpsi_\eps^{1/2} Z))^{-1}]^3
        }
        < 0
    \]
    by Cauchy--Schwarz.
    Thus $\theta_\eps$ is decreasing on $(-\fr{1}{1+\eps},+\infty)$.
    We now calculate its limits as $z\downarrow -\fr{1}{1+\eps}$ and $z\uparrow +\infty$.
    Consider first $z = -\fr{1}{1+\eps} + \iota$ for $\iota$ small.
    Then the first factor of \eqref{eq:theta-eps-alt} is
    \[
        \fr{
            \EE [(\iota + \dg_\eps(\tpsi_\eps^{1/2} Z))^{-2}]
        }{
            \EE [(\iota + \dg_\eps(\tpsi_\eps^{1/2} Z))^{-1}]^2
        }
        \ge \fr{
            \EE [\bone\{|Z| \le \iota^{1/2}\} (\iota + \dg_\eps(\tpsi_\eps^{1/2} Z))^{-2}]
        }{
            \EE [\bone\{|Z| \le \iota^{1/3}\} (\iota + \dg_\eps(\tpsi_\eps^{1/2} Z))^{-1} + O(\iota^{-2/3})]^2
        }
        = \fr{\Omega(\iota^{-3/2})}{O(\iota^{-4/3})},
    \]
    which diverges as $\iota \downarrow 0$.
    The second factor of \eqref{eq:theta-eps-alt} tends to $1$ in this limit by dominated convergence.
    Thus $\lim_{z\downarrow -(1+\eps)^{-1}} \theta_\eps(z) = +\infty$.
    We can write the first factor of \eqref{eq:theta-eps-alt} as
    \[
        \fr{\EE [(1 + z^{-1} \df_\eps(\tpsi_\eps^{1/2} Z))^{-2}]}{\EE [(1 + z^{-1} \df_\eps(\tpsi_\eps^{1/2} Z))^{-1}]^2},
    \]
    which tends to $1$ as $z\uparrow +\infty$ by dominated convergence.
    In this limit, the second factor of \eqref{eq:theta-eps-alt} tends to $0$ by dominated convergence, so $\lim_{z\uparrow +\infty} \theta_\eps(z) = 0$.
    This completes the proof.
\end{proof}
\begin{proof}[Proof of Lemma~\ref{lem:freeprob-well-defd}]
    Note that
    \[
        m'(z) = -\EE[(z + \ch^2(\psi_0^{1/2} Z))^{-2}].
    \]
    Thus, differentiating $\lambda$ yields
    \[
        \lambda'(z) = 1 + \alpha_\star m'(z) \EE\lt[\lt(
            \fr{\hf_0(q_0^{1/2}Z)}{1 + m(z) \hf_0(q_0^{1/2} Z)}
        \rt)^2\rt]
        = 1 - \alpha_\star \theta(z).
    \]
    The assertions about $\theta$ follow from Lemma~\ref{lem:freeprob-well-defd-eps}, with $\eps = 0$.
    Since $\theta$ is strictly decreasing on $(-1,+\infty)$, $\lambda'$ is strictly increasing on this interval, and therefore $\lambda$ is strictly convex on this interval.
    Since $\theta^{-1} : (0,+\infty) \to (-1,+\infty)$ is well-defined, we may define $z_0 = \theta^{-1}(\alpha_\star^{-1})$.
    This point satisfies the stationarity condition $\lambda'(z_0) = 0$ and is thus the unique minimizer of $\lambda$ on $(-1,+\infty)$.
\end{proof}
Recall from below Lemma~\ref{lem:b-varrho} that $d_\eps = \alpha_\star \EE [F'_{\eps,\varrho_\eps}(\tq_\eps^{1/2} Z)]$.
We now define the threshold $\lambda_\eps$.
\begin{dfn}
    \label{dfn:lambda-eps}
    Let $z_\eps = \theta_\eps^{-1}(\alpha_\star^{-1})$ and
    \beq
        \label{eq:def-lambda-eps}
        \lambda_\eps
        \equiv z_\eps
        - \alpha_\star \EE \lt[
            \fr{\hf_\eps(\tq_\eps^{1/2} Z)}{1 + m_\eps(z_\eps) \hf_\eps(\tq_\eps^{1/2} Z)}
        \rt]
        - d_\eps.
    \eeq
\end{dfn}
\begin{lem}
    \label{lem:lambda-eps-to-0}
    As $\eps \downarrow 0$, $\lambda_\eps \to \lambda_0$ (defined in Condition~\ref{con:local-concavity}).
\end{lem}
\begin{proof}
    By Proposition~\ref{ppn:eps-perturb-fixed-point}, as $\eps \downarrow 0$, $(\tq_\eps,\tpsi_\eps) \to (q_0,\psi_0)$.
    Thus, for $\df_0(x) = \ch^2(x)$, the push-forwards $(\df_\eps)_\# \cN(0,\tpsi_\eps)$ and $(\hf_\eps)_\# \cN(0,\tq_\eps)$ converge weakly to $(\df_0)_\# \cN(0,\psi_0)$ and $(\hf_0)_\# \cN(0,q_0)$.

    For any $z > -1$ and small $\eps$, the integrand of $m_\eps(z)$ is bounded independently of $\eps$, and thus $\lim_{\eps \downarrow 0} m_\eps(z) = m(z)$ by dominated convergence.
    Similarly, all three integrands in \eqref{eq:theta-eps-alt} are bounded, so $\lim_{\eps \downarrow 0} \theta_\eps(z) = \theta(z)$.
    Moreover, one easily checks that on any compact subset of $(-1,+\infty)$, the derivatives of $m_\eps, \theta_\eps$ are bounded independently of $\eps$.
    Thus $m_\eps \to m$, $\theta_\eps \to \theta$ uniformly on compact subsets of $(-1,+\infty)$.

    By Lemma~\ref{lem:freeprob-well-defd}, $\lim_{z\downarrow -1} \theta(z) = +\infty$, so $z_0 = \theta^{-1}(\alpha_\star^{-1})$ is bounded away from $-1$.
    The above uniform convergence then implies $z_\eps \to z_0$ and $m_\eps(z_\eps) \to m(z_0)$.
    Since the below integrands are bounded,
    \[
        \EE \lt[
            \fr{\hf_\eps(\tq_\eps^{1/2} Z)}{1 + m_\eps(z_\eps) \hf_\eps(\tq_\eps^{1/2} Z)}
        \rt]
        \to \EE \lt[
            \fr{\hf_0(q_0^{1/2} Z)}{1 + m(z_0) \hf_0(q_0^{1/2} Z)}
        \rt].
    \]
    Finally, as $F'_{\eps,\varrho_\eps}$ is bounded (by Fact~\ref{fac:Fp-bounded}) and limits to the bounded function $F'_{1-q_0}$, we have $d_\eps \to d_0$.
\end{proof}

\subsection{Hessian estimate}

We next prove the following upper bound on $\nabla^2_\diamond \cF^\eps_\TAP$.
\begin{lem}
    \label{lem:hessian-estimate}
    Suppose $(\bm,\bn) \in \cS_{\eps,r_0}$, and $\tnorm{\bG}_\op, \tnorm{\hbg} \le C \sqrt{N}$ for some absolute constant $C$ (i.e. independent of all parameters in \S\ref{subsec:params-list}).
    Let $\dbh \in \bbR^N$, $\abh \in \bbR^M$ be defined (as in Lemma~\ref{lem:tap-1deriv}) by
    \baln
        \dbh &= \th_\eps^{-1}(\bm), &
        \abh &= \fr{\bG \bm}{\sqrt{N}} + \eps^{1/2} \hbg - \rho_\eps(q(\bm)) \bn,
    \ealn
    and
    \baln
        \bD_1 &= \diag(\df_\eps(\dbh)), &
        \bD_2 &= \diag(\hf_\eps(\abh)).
    \ealn
    Then,
    \[
        \nabla^2_\diamond \cF^\eps_\TAP(\bm,\bn) \preceq
        P_\bm^\perp\lt(-\bD_1 - \fr1N \bG^\top \bD_2 \bG - d_\eps \bI_N \rt) P_\bm^\perp
        + \fr{\lambda_\eps \bm\bm^\top}{\tnorm{\bm}^2}
        + (o_{C_\cvx}(1) + o_{r_0}(1)) \bI_N.
    \]
    (Recall the meaning of $o_{C_\cvx}(1), o_{r_0}(1)$ discussed in \S\ref{subsec:params-list}.)
\end{lem}

\begin{fac}[Proved in Appendix~\ref{app:amp}]
    \label{fac:tap-2deriv}
    Let $\bm \in \bbR^N$, $\bn \in \bbR^M$, and let $\abh, \dbh$ be as above.
    Let $F = F_{\eps,\rho_\eps(q(\bm))}$ and
    \baln
        \bD_3 &= \diag\lt(F'(\abh)\rt), &
        \bD_4 &= \bI_M + \rho_\eps(q(\bm)) \bD_3.
    \ealn
    Then,
    \baln
        \nabla^2_{\bm,\bm} \cF^\eps_\TAP(\bm,\bn)
        &= - \bD_1
        + \fr{\bG^\top \bD_3 \bG}{N}
        + \rho'_\eps(q(\bm)) d_\eps(\bm,\bn) \bI_N \\
        &+ \rho'_\eps(q(\bm)) \cdot \fr{
            \bG^\top (F''(\abh) + 2\bD_3 (F(\abh) - \bn)) \bm^\top +
            \bm (F''(\abh) + 2\bD_3 (F(\abh) - \bn))^\top \bG
        }{N^{3/2}} \\
        &+ \lt\{
            \rho''_\eps(q(\bm)) d_\eps(\bm,\bn)
            + \fr{\rho'_\eps(q(\bm))^2}{N} \sum_{a=1}^M \lt(
                2F'(\ah_a)^2 + F^{(3)}(\ah_a)
            \rt)
        \rt\}
        \fr{\bm\bm^\top}{N} \\
        \nabla^2_{\bm,\bn} \cF^\eps_\TAP(\bm,\bn)
        &= - \fr{\rho_\eps(q(\bm))}{\sqrt{N}} \bG^\top \bD_3
        - \rho'_\eps(q(\bm)) \fr{\bm (\rho_\eps(q(\bm)) F''(\abh) + 2 \bD_4 (F(\abh) - \bn))^\top}{N} \\
        \nabla^2_{\bn,\bn} \cF^\eps_\TAP(\bm,\bn) &= \rho_\eps(q(\bm)) \bD_4,
    \ealn
    Furthermore, for
    \[
        \tbD_2 = - \bD_3 + \rho_\eps(q(\bm)) \bD_3^2 \bD_4^{-1}
        = \diag\lt(-\fr{F'(\abh)}{1 + \rho_\eps(q(\bm)) F'(\abh)}\rt),
    \]
    we have
    \baln
        \nabla^2_\diamond \cF^\eps_\TAP(\bm,\bn)
        &= - \bD_1
        - \fr{\bG^\top \tbD_2 \bG}{N}
        + \rho'_\eps(q(\bm)) d_\eps(\bm,\bn) \bI_N \\
        &+ \rho'_\eps(q(\bm)) \cdot \fr{
            \bG^\top \bD_4^{-1} F''(\abh) \bm^\top +
            \bm F''(\abh)^\top \bD_4^{-1} \bG
        }{N^{3/2}} \\
        &+ \bigg\{
            \rho''_\eps(q(\bm)) d_\eps(\bm,\bn)
            + \fr{\rho'_\eps(q(\bm))^2}{N} \sum_{a=1}^M \bigg(
                2F'(\ah_a)^2 + F^{(3)}(\ah_a) \\
                &\qquad - \fr{(\rho_\eps(q(\bm)) F''(\ah_a) + 2(F(\ah_a)-n_a)(1+\rho_\eps(q(\bm))F'(\ah_a)))^2}{\rho_\eps(q(\bm)) (1 + \rho_\eps(q(\bm)) F'(\ah_a))}
            \bigg)
        \bigg\}
        \fr{\bm\bm^\top}{N}.
    \ealn
\end{fac}
\begin{lem}[Proved in Appendix~\ref{app:amp}]
    \label{lem:hessian-crude-estimates}
    Suppose $(\bm,\bn) \in \cS_{\eps,r_0}$ and $\tnorm{\bG}_\op, \tnorm{\hbg} \le C\sqrt{N}$ for an absolute constant $C$.
    The following estimates hold for sufficiently small $r_0$ (depending on $\eps, C_\cvx, C_\bd, \eta$).
    \begin{enumerate}[label=(\alph*)]
        \item \label{itm:hessian-estimates-approx} Up to additive $o_{r_0}(1)$ error, $q(\bm) \approx q_\eps$, $\psi(\bn) \approx \psi_\eps$, and $d_\eps(\bm,\bn) \approx d_\eps$, $\rho_\eps(q(\bm)) \approx \varrho_\eps$, $\rho'_\eps(q(\bm)) \approx -1$, $\rho''_\eps(q(\bm)) \approx C_\cvx$.
        \item \label{itm:hessian-estimates-approx-bD2} $\tnorm{\tbD_2 - \bD_2}_\op = o_{r_0}(1)$.
        \item \label{itm:hessian-estimates-approx-mm} $\fr1N \sum_{a=1}^M (2F'(\ah_a)^2 + F^{(3)}(\ah_a))$ is bounded by an absolute constant.
        \item \label{itm:hessian-estimates-approx-cross} $\fr{1}{\sqrt{N}} \tnorm{\bD_4^{-1} F''(\abh)}$ is bounded, with bound depending only on $\eps$.
    \end{enumerate}
\end{lem}

\begin{proof}[Proof of Lemma~\ref{lem:hessian-estimate}]
    By Fact~\ref{fac:tap-2deriv} and Lemma~\ref{lem:hessian-crude-estimates},
    \[
        \nabla^2_\diamond \cF^\eps_\TAP(\bm,\bn)
        \preceq - \bD_1
        - \fr{\bG^\top \tbD_2 \bG}{N}
        - d_\eps \bI_N
        + \fr{\bG^\top \bv_1 \bm^\top + \bm \bv_1^\top \bG}{N}
        + \lt(C_{\cvx} d_\eps + C_1\rt) \fr{\bm\bm^\top}{N}
        + o_{r_0}(1) \bI_N,
    \]
    for $C_1 \in \bbR$, $\bv_1 \in \bbR^N$ with $|C_1|$, $\tnorm{\bv_1}$ bounded depending only on $\eps$.
    By the assumption on $\tnorm{\bG}_\op$, $\fr{1}{\sqrt{N}} \tnorm{\bG^\top \bv_1}$ is also bounded depending only on $\eps$.
    Note that
    \[
        - \bD_1
        \preceq
        - P_\bm^\perp \bD_1 P_\bm^\perp
        - (P_\bm^\perp \bD_1 P_\bm + P_\bm \bD_1 P_\bm^\perp)
        = - P_\bm^\perp \bD_1 P_\bm^\perp
        - \fr{(P_\bm^\perp \bD_1 \bm) \bm^\top + \bm (P_\bm^\perp \bD_1 \bm)}{q(\bm) N}
    \]
    and similarly
    \[
        - \fr1N \bG^\top \bD_2 \bG
        \preceq - P_\bm^\perp \bG^\top \bD_2 \bG P_\bm^\perp
        - \fr{(P_\bm^\perp \bG^\top \bD_2 \bG \bm) \bm^\top + \bm (P_\bm^\perp \bG^\top \bD_2 \bG \bm)^\top }{q(\bm) N^2}.
    \]
    Moreover $\tnorm{\bD_1}_\op, \tnorm{\bD_2}_\op \le O(\eps^{-1})$, the latter by \eqref{eq:d2n-bound}.
    So, there exists $C_2\in \bbR$, $\bv_2 \in \bbR^N$ with $|C_2|$, $\tnorm{\bv_2}$ bounded depending only on $\eps$, such that
    \[
        \nabla^2_\diamond \cF^\eps_\TAP(\bm,\bn)
        \preceq P_\bm^\perp \lt(- \bD_1
        - \fr{\bG^\top \tbD_2 \bG}{N} \rt) P_\bm^\perp
        - d_\eps \bI_N
        + \fr{\bv_2 \bm^\top + \bm \bv_2^\top}{N^{1/2}}
        + \lt(C_{\cvx} d_\eps + C_2\rt) \fr{\bm\bm^\top}{N}
        + o_{r_0}(1) \bI_N.
    \]
    Note that $d_\eps < 0$, because $F'_{\eps,\varrho_\eps} < 0$ by Fact~\ref{fac:Fp-bounded}.
    So, for large $C_\cvx$,
    \[
         (C_\cvx d_\eps + C_2) \fr{\bm\bm^\top}{N}
         + \fr{\bv_2 \bm^\top + \bm \bv_2^\top}{N^{1/2}}
         \preceq \fr{(\lambda_\eps + d_\eps) \bm\bm^\top}{\tnorm{\bm}^2} + \fr{\bv_2\bv_2^\top}{C_\cvx |d_\eps| - C_2 + (\lambda_\eps + d_\eps) / q(\bm)}.
    \]
    The final term has operator norm $o_{C_\cvx}(1)$.
\end{proof}

\subsection{Null model: post-AMP Gordon's inequality}

We turn to the proof of Proposition~\ref{ppn:amp-guarantees}\ref{itm:amp-guarantee-concave}, first under the measure $\bbP$.
In light of Lemma~\ref{lem:hessian-estimate}, we define
\beq
    \label{eq:def-bR}
    \bR(\bm,\bn) = P_\bm^\perp \lt(
        -\bD_1 - \fr{1}{N} \bG^\top \bD_2 \bG
    \rt) P_\bm^\perp,
\eeq
where, as in that lemma, $\bD_1 = \diag(\df_\eps(\dbh))$, $\bD_2 = \diag(\hf_\eps(\abh(\bm,\bn,\bG)))$ for $\dbh = \th_\eps^{-1}(\bm)$ and
\[
    \abh(\bm,\bn,\bG) = \fr{\bG \bm}{\sqrt{N}} + \eps^{1/2} \hbg - \rho_\eps(q(\bm)) \bn.
\]
\begin{ppn}
    \label{ppn:null-concavity}
    With high probability under $\PP$, $\bR(\bm,\bn) \preceq (\lambda_\eps + d_\eps + o_{r_0}(1) + o_k(1)) P_\bm^\perp$ for all $\tnorm{(\bm,\bn) - (\bm^k,\bn^k)} \le r_0 \sqrt{N}$.
\end{ppn}
\noindent For $z_\eps$ defined in Definition~\ref{dfn:lambda-eps}, let
\[
    r_\eps^2 = \EE[(z_\eps + \df_\eps(\tpsi_\eps^{1/2} Z))^{-2}]^{-1}.
\]
Define the AMP iterates $\bm^0,\bn^0,\ldots,\bm^k,\bn^k$ and $\hbh^0,\dbh^1,\hbh^1,\ldots,\dbh^k,\hbh^k$ as in \eqref{eq:amp-iterates-n}, \eqref{eq:amp-iterates-m}, and
\[
    \DATA = (\dbg,\dbh^1,\ldots,\dbh^k,\hbg,\hbh^0,\ldots,\hbh^k).
\]
Let $U(r_0) = \{(\bm,\bn) : \tnorm{(\bm,\bn) - (\bm^k,\bn^k)} \le r_0 \sqrt{N}\}$.
Let $\abh^k \equiv \abh(\bm^k,\bn^k,\bG)$, and note that
\beq
    \label{eq:def-abh-k}
    \abh^k = \hbh^k + \varrho_\eps \bn^{k-1} - \rho_\eps(q(\bm^k)) \bn^k
\eeq
is $\DATA$-measurable.
Let $U'(r_0) = \{\abh : \tnorm{\abh - \abh^k} \le C r_0 \sqrt{N}\}$, for a suitably large absolute constant $C$.
Since $\tnorm{\bG}_\op = O(\sqrt{N})$ with high probability, on this event $\abh(\bm,\bn,\bG) \in U'(r_0)$ for all $(\bm,\bn) \in U(r_0)$.

Below, we will write $\bD_2(\abh) = \diag(\hf_\eps(\abh))$ for a varying $\abh$ which is not necessarily $\abh(\bm,\bn,\bG)$.
On the other hand $\bD_1$ always refers to the function of $\bm$ defined above.
The starting point of our proof of Proposition~\ref{ppn:null-concavity} is to recast the maximum eigenvalue as a minimax program, as follows:
\baln
    &\sup_{(\bm,\bn) \in U(r_0)}
    \sup_{\substack{\tnorm{\dbv} = 1 \\ \dbv \perp \bm}}
    \dbv^\top \lt(-\bD_1 - \fr1N \bG^\top \bD_2(\abh(\bm,\bn,\bG)) \bG\rt) \dbv \\
    &=
    \sup_{(\bm,\bn) \in U(r_0)}
    \sup_{\substack{\tnorm{\dbv} = 1 \\ \dbv \perp \bm}}
    \inf_{\hbv \in \bbR^M} \lt\{
        - \la \bD_1 \dbv, \dbv \ra
        + \la \bD_2(\abh(\bm,\bn,\bG))^{-1} \hbv, \hbv \ra
        + \fr{2}{\sqrt{N}} \la \bG \dbv, \hbv \ra
    \rt\}.
\ealn
Here we used that $\bD_1, \bD_2$ are positive definite, by positivity of $\df_\eps$, $\hf_\eps$.
On the high probability event that $\tnorm{\bG}_\op = O(\sqrt{N})$, this is bounded by
\beq
    \label{eq:null-concavity-start}
    \sup_{\substack{(\bm,\bn) \in U(r_0) \\ \abh \in U'(r_0)}}
    \sup_{\substack{\tnorm{\dbv} = 1 \\ \dbv \perp \bm}}
    \inf_{\substack{\tnorm{\hbv} = r_\eps \\ \hbv \perp \bn}} \lt\{
        - \la \bD_1 \dbv, \dbv \ra
        + \la \bD_2(\abh)^{-1} \hbv, \hbv \ra
        + \fr{2}{\sqrt{N}} \la \bG \dbv, \hbv \ra
    \rt\}.
\eeq
We will control \eqref{eq:null-concavity-start} by applying Gordon's minimax inequality conditional on the AMP iterates; we explain this next.
Let
\baln
    \dmu_\AMP &= \fr1N \sum_{i=1}^N \delta(\eps^{1/2} \dg,\dh^1_i,\ldots,\dh^k_i), &
    \hmu_\AMP &= \fr1M \sum_{a=1}^M \delta(\eps^{1/2} \hg,\hh^0_a,\ldots,\hh^k_a).
\ealn
Further let $(\dSig^+_{i,j})_{i,j\ge 0}$ and $(\hSig^+_{i,j})_{i,j\ge -1}$ be augmented versions of $(\dSig_{i,j})_{i,j\ge 1}, (\hSig_{i,j})_{i,j\ge 0}$ where we add a row and column of zeros, i.e. $\dSig^+_{0,i} = \dSig^+_{i,0} = \hSig^+_{-1,i} = \hSig^+_{i,-1} = 0$.
\begin{lem}
    \label{lem:amp-event}
    For any $\ups>0$, with high probability, 
    \beq
        \label{eq:amp-event}
        \bbW_2(\dmu_\AMP, \cN(0, \dSig^+_{\le k} + \eps \bone \bone^\top)),
        \bbW_2(\hmu_\AMP, \cN(0, \hSig^+_{\le k} + \eps \bone \bone^\top))
        \le \ups.
    \eeq
\end{lem}
\begin{proof}
    Follows from AMP state evolution, identically to Proposition~\ref{ppn:state-evolution}.
\end{proof}
We now let $\ups$ be sufficiently small depending on $r_0, k$ and condition on a realization of $\DATA$ such that \eqref{eq:amp-event} holds.
(Note that \eqref{eq:amp-event} is $\DATA$-measurable.)
Define $\babh^i = \dbh^i - \eps^{1/2} \dbg$, $\brbh^i = \hbh^i - \eps^{1/2} \hbg$, and
\baln
    \bM_{(k)} &= (\bm^0,\ldots,\bm^k) \in \bbR^{N\times (k+1)}, &
    \bN_{(k)} &= (\bn^0,\ldots,\bn^{k-1}) \in \bbR^{M\times k}, \\
    \babH_{(k)} &= (\babh^1,\ldots,\babh^k) \in \bbR^{N\times k}, &
    \brbH_{(k)} &= (\brbh^0,\ldots,\brbh^k) \in \bbR^{M\times (k+1)}.
\ealn
Note that on event \eqref{eq:amp-event},
\balnn
    \label{eq:bM-bN-overlap-structure}
    \fr1N \bM_{(k)}^\top \bM_{(k)} &= \hSig_{\le k} + o_\ups(1), &
    \fr1N \bN_{(k)}^\top \bN_{(k)} &= \dSig_{\le k} + o_\ups(1),
    \\
    \label{eq:babH-brbH-overlap-structure}
    \fr1N \babH_{(k)}^\top \babH_{(k)} &= \dSig_{\le k} + o_\ups(1), &
    \fr1M \brbH_{(k)}^\top \brbH_{(k)} &= \hSig_{\le k} + o_\ups(1),
\ealnn
where $o_\ups(1)$ denotes an additive error of operator norm $o_\ups(1)$.
That is, $\{\bn^0,\ldots,\bn^{k-1}\}$ and $\{\babh^1,\ldots,\babh^k\}$ span $k$-dimensional subspaces of $\bbR^M$ and $\bbR^N$, and the linear mapping between them that sends $\bn^i$ to $\babh^{i+1}$ is an approximate isometry.
The same is true, after scaling by a factor $\alpha_\star$, for $\{\bm^0,\ldots,\bm^k\}$ and $\{\brbh^0,\ldots,\brbh^k\}$.
Define the linear maps
\baln
    \dbT &= \babH_{(k)} (\bN^\top_{(k)} \bN_{(k)})^{-1} \bN^\top_{(k)}, &
    \hbT &= \brbH_{(k)} (\bM^\top_{(k)} \bM_{(k)})^{-1} \bM^\top_{(k)}.
\ealn
(The inverses are well-defined because the matrices are full-rank, by \eqref{eq:bM-bN-overlap-structure}.)
That is, $\dbT$ (resp. $\hbT$) projects onto the span of $\{\bn^0,\ldots,\bn^{k-1}\}$ (resp. $\{\bm^0,\ldots,\bm^k\}$) and then applies the linear map that sends $\bn^i$ to $\dbh^{i+1}$ (resp. $\bm^i$ to $\hbh^i$).
\begin{lem}[Post-AMP Gordon's inequality]
    \label{lem:gordon-post-amp}
    Conditional on any realization of $\DATA$ satisfying event \eqref{eq:amp-event}, the following holds.
    Let $\dbxi \sim \cN(0, \bI_N)$, $\hbxi \sim \cN(\bzero,\bI_M)$, $Z \sim \cN(0,1)$ be independent of everything else and
    \baln
        \dbg'_\AMP(\hbv) &= \sqrt{N} \dbT \hbv
        + \tnorm{P_{\bN_{(k)}}^\perp \hbv} P_{\bM_{(k)}}^\perp \dbxi, &
        \hbg'_\AMP(\dbv) &= \sqrt{N} \hbT \dbv
        + \tnorm{P_{\bM_{(k)}}^\perp \dbv} P_{\bN_{(k)}}^\perp \hbxi.
    \ealn
    For any continuous $f : \bbR^N \times \bbR^M \times \bbR^N \times (\bbR^M)^2 \times \bbR^{N \times (k+1)} \times \bbR^{M\times (k+2)} \to \bbR$,
    \[
        \sup_{\substack{(\bm,\bn) \in U(r_0) \\ \abh \in U'(r_0)}}
        \sup_{\substack{\tnorm{\dbv} = 1 \\ \dbv \perp \bm}}
        \inf_{\substack{\tnorm{\hbv} = r_\eps, \\ \hbv \perp \bn}} \lt\{
            f(\dbv,\hbv;\bm,\bn,\abh,\DATA)
            + \fr{2}{\sqrt{N}} \la \bG \dbv, \hbv \ra
            + \fr{2 \tnorm{P_{\bN_{(k)}}^\perp \hbv} \tnorm{P_{\bM_{(k)}}^\perp \dbv}}{\sqrt{N}} Z
        \rt\}
    \]
    is stochastically dominated by
    \[
        \sup_{\substack{(\bm,\bn) \in U(r_0) \\ \abh \in U'(r_0)}}
        \sup_{\substack{\tnorm{\dbv} = 1 \\ \dbv \perp \bm^k}}
        \inf_{\substack{\tnorm{\hbv} = r_\eps, \\ \hbv \perp \bn^k}} \lt\{
            f(\dbv,\hbv;\bm,\bn,\abh,\DATA)
            + \fr{2}{\sqrt{N}} \la \dbv, \dbg'_\AMP(\hbv) \ra
            + \fr{2}{\sqrt{N}} \la \hbv, \hbg'_\AMP(\dbv) \ra
        \rt\} + o_\ups(1).
    \]
\end{lem}
\begin{proof}
    We will first show that conditional on $\DATA$,
    \beq
        \label{eq:bG-conditional-on-data}
        \fr{1}{\sqrt{N}} \bG
        \stackrel{d}{=}
        \dbT^\top + \hbT + o_\ups(1)
        + \fr{P_{\bN_{(k)}^\perp} \obG P_{\bM_{(k)}}^\perp}{\sqrt{N}},
    \eeq
    where $o_\ups(1)$ is a deterministic error of operator norm $o_\ups(1)$ and $\obG$ is an i.i.d. copy of $\bG$.
    Conditioning on $\DATA$ amounts to conditioning on the linear relations
    \balnn
        \label{eq:condition-data-linear-relations}
        \fr{1}{\sqrt{N}} \bG \bm^i &= \brbh^i + \varrho_\eps \bn^{i-1}, &
        \fr{1}{\sqrt{N}} \bG^\top \bn^i &= \babh^{i+1} + d_\eps \bm^i
    \ealnn
    for $0\le i\le k$ and $0\le i\le k-1$.
    So, $P_{\bN_{(k)}}^\perp \bG P_{\bM_{(k)}}^\perp$ is independent of $\DATA$ and $\bG - P_{\bN_{(k)}}^\perp \bG P_{\bM_{(k)}}^\perp$ is $\DATA$-measurable.
    It suffies to show the latter part is $\dbT^\top + \hbT$, up to $o_\ups(1)$ additive operator norm error.
    Recall from \eqref{eq:bM-bN-overlap-structure} that the condition number of $\fr1N \bM_{(k)}^\top \bM_{(k)}$ and $\fr1N \bN_{(k)}^\top \bN_{(k)}$ is bounded depending on $k$.
    So it suffices to show
    \balnn
        \label{eq:bG-conditional-on-data-goal}
        \lt\|\fr{1}{\sqrt{N}} \bG \bM_{(k)} - (\dbT^\top + \hbT) \bM_{(k)} \rt\|_\op
        &= o_\ups(1) \sqrt{N}, &
        \lt\|\fr{1}{\sqrt{N}} \bG^\top \bN_{(k)} - (\dbT + \hbT^\top) \bN_{(k)} \rt\|_\op
        &= o_\ups(1) \sqrt{N}.
    \ealnn
    By \eqref{eq:condition-data-linear-relations} and the definition of $\dbT, \hbT$,
    \baln
        \fr{1}{\sqrt{N}} \bG \bM_{(k)}
        &= \brbH_{(k)} + \varrho_\eps [\bzero,\bN_{(k)}], &
        \fr{1}{\sqrt{N}} \bG^\top \bN_{(k)}
        &= \babH_{(k)} + d_\eps \bM_{(k-1)}, \\
        \hbT \bM_{(k)} &= \brbH_{(k)}, &
        \dbT \bN_{(k)} &= \babH_{(k)}.
    \ealn
    For all $i,j\ge 1$, we have by gaussian integration by parts
    \baln
        \fr1N \la \babh^i, \bm^j \ra
        &= \fr1N \la \babh^i, \th_\eps(\babh^j + \eps^{1/2} \dbg) \ra \\
        &= \EE[
            (\opsi_{i\wedge j}^{1/2} Z + (\psi_\eps + \eps - \opsi_{i\wedge j}) Z')
            \th_\eps(\opsi_{i\wedge j}^{1/2} Z + (\psi_\eps + \eps - \opsi_{i\wedge j})^{1/2} Z'')
        ] + o_\ups(1) \\
        &= \varrho_\eps \opsi_{i\wedge j}  + o_\ups(1).
    \ealn
    Moreover $\fr1N \la \babh^i, \bm^0 \ra = o_\ups(1)$.
    Thus,
    \baln
        \dbT^\top \bM_{(k)}
        &= \bN_{(k)}
        \lt(\fr1N \bN^\top_{(k)} \bN_{(k)} \rt)^{-1}
        \lt(\fr1N \babH^\top_{(k)} \bM_{(k)} \rt) \\
        &= \bN_{(k)}
        \lt(\dSig_{\le k} + o_\ups(1) \rt)^{-1}
        \lt([0,\varrho_\eps \dSig_{\le k}] + o_\ups(1) \rt)
        = \varrho_\eps [\bzero, \bN_{(k)}] + o_\ups(1) \sqrt{N},
    \ealn
    where the errors are all in operator norm.
    A similar calculation shows
    \[
        \hbT^\top \bN_{(k)} = d_\eps \bM_{(k-1)} + o_\ups(1) \sqrt{N}.
    \]
    Combining proves \eqref{eq:bG-conditional-on-data-goal} and thus \eqref{eq:bG-conditional-on-data}.
    So, conditional on $\DATA$,
    \[
        \fr{1}{\sqrt{N}} \la \bG \dbv, \hbv \ra
        \stackrel{d}{=} \la \dbv, \dbT \hbv \ra + \la \hbv, \hbT \dbv \ra
        + o_\ups(1) + \fr{1}{\sqrt{N}} \la \obG P_{\bM_{(k)}}^\perp \dbv, P_{\bN_{(k)}}^\perp \hbv \ra
    \]
    By Gordon's inequality applied to $\obG$,
    \baln
        \sup_{\substack{(\bm,\bn) \in U(r_0) \\ \abh \in U'(r_0)}}
        \sup_{\substack{\tnorm{\dbv} = 1 \\ \dbv \perp \bm}}
        &\inf_{\substack{\tnorm{\hbv} = r_\eps, \\ \hbv \perp \bn}} \bigg\{
            f(\dbv,\hbv;\bm,\bn,\abh,\DATA)
            + 2\la \dbv, \dbT \hbv \ra
            + 2\la \hbv, \hbT \dbv \ra \\
            &+ \fr{2}{\sqrt{N}} \la \obG P_{\bM_{(k)}}^\perp \dbv, P_{\bN_{(k)}}^\perp \hbv \ra
            + \fr{2 \tnorm{P_{\bN_{(k)}}^\perp \hbv} \tnorm{P_{\bM_{(k)}}^\perp \dbv}}{\sqrt{N}} Z
        \bigg\}
    \ealn
    is stochastically dominated by
    \baln
        \sup_{\substack{(\bm,\bn) \in U(r_0) \\ \abh \in U'(r_0)}}
        \sup_{\substack{\tnorm{\dbv} = 1 \\ \dbv \perp \bm}}
        &\inf_{\substack{\tnorm{\hbv} = r_\eps, \\ \hbv \perp \bn}} \bigg\{
            f(\dbv,\hbv;\bm,\bn,\abh,\DATA)
            + 2\la \dbv, \dbT \hbv \ra
            + 2\la \hbv, \hbT \dbv \ra \\
            &+ \fr{2\tnorm{P_{\bN_{(k)}}^\perp \hbv}}{\sqrt{N}} \la \dbv, P_{\bM_{(k)}}^\perp \dbxi\ra
            + \fr{2\tnorm{P_{\bM_{(k)}}^\perp \dbv}}{\sqrt{N}} \la \hbv, P_{\bN_{(k)}}^\perp \hbxi\ra
        \bigg\}.
    \ealn
    The quantity inside the sup-inf is precisely $f(\dbv,\hbv,\DATA) + \fr{2}{\sqrt{N}} \la \dbv, \dbg'_\AMP(\hbv) \ra + \fr{2}{\sqrt{N}} \la \hbv, \hbg'_\AMP(\dbv) \ra$.
\end{proof}
\noindent Define
\baln
    \dbg_\AMP(\hbv) &= \sqrt{N} \dbT \hbv
    + \tnorm{P_{\bN_{(k)}}^\perp \hbv} \dbxi, &
    \hbg_\AMP(\dbv) &= \sqrt{N} \hbT \dbv
    + \tnorm{P_{\bM_{(k)}}^\perp \dbv} \hbxi.
\ealn
Note that
\baln
    \fr{1}{\sqrt{N}} \tnorm{\dbg_\AMP(\hbv) - \dbg'_\AMP(\hbv)} &\le \fr{r_\eps}{\sqrt{N}} \tnorm{P_{\bM_{(k)}}
     \dbxi}, &
    \fr{1}{\sqrt{N}} \tnorm{\hbg_\AMP(\dbv) - \dbg'_\AMP(\dbv)} &\le \fr{1}{\sqrt{N}} \tnorm{P_{\bN_{(k)}} \hbxi},
\ealn
are both bounded by $\ups$ with high probability, and similarly $|Z|/\sqrt{N} \le \ups$ with high probability.
Below, let $\err$ denote an error term of order $o_{r_0}(1) + o_k(1) + o_\ups(1)$.
By \eqref{eq:null-concavity-start}, Lemma~\ref{lem:gordon-post-amp}, and these observations, it suffices to show that with high probability,
\balnn
    \notag
    \sup_{\substack{(\bm,\bn) \in U(r_0) \\ \abh \in U'(r_0)}}
    \sup_{\substack{\tnorm{\dbv} = 1 \\ \dbv \perp \bm}}
    &\inf_{\substack{\tnorm{\hbv} = r_\eps, \\ \hbv \perp \bn}} \bigg\{
        - \la \bD_1 \dbv, \dbv \ra
        + \la \bD_2(\abh)^{-1} \hbv, \hbv \ra \\
        \label{eq:null-concavity-goal2}
        &+ \fr{2}{\sqrt{N}} \la \dbv, \dbg_\AMP(\hbv) \ra
        + \fr{2}{\sqrt{N}} \la \hbv, \hbg_\AMP(\dbv) \ra
    \bigg\}
    \le \lambda_\eps + d_\eps + \err.
\ealnn
\begin{lem}
    \label{lem:amp-event-2}
    Let
    \baln
        \dmu'_\AMP &= \fr1N \sum_{i=1}^N \delta(\dxi_i,\bah^1_i,\ldots,\bah^k_i), &
        \hmu'_\AMP &= \fr1M \sum_{a=1}^M \delta(\hxi_a,\brh^0_a,\ldots,\brh^k_a).
    \ealn
    Conditional on a realization of $\DATA$ such that \eqref{eq:amp-event} holds, with high probability,
    \beq
        \label{eq:amp-event-2}
        \bbW_2(\dmu'_\AMP, \cN(0,1) \times \cN(0,\dSig_{\le k})),
        \bbW_2(\hmu'_\AMP, \cN(0,1) \times \cN(0,\hSig_{\le k}))
        \le 2\ups.
    \eeq
\end{lem}
\begin{proof}
    Under event \eqref{eq:amp-event}, the $\bbW_2$-distance of the marginal of $\dmu'_\AMP$ on all but the first coordinate to $\cN(0,\dSig_{\le k}))$ is deterministically at most $\ups$.
    Since $\dbxi$ is independent of $\DATA$, it follows that $\bbW_2(\dmu'_\AMP, \cN(0,1) \times \cN(0,\dSig_{\le k})) \le 2\ups$ with high probability.
    The estimate for $\hmu'_\AMP$ is analogous.
\end{proof}
\begin{fac}[Proved in Appendix~\ref{app:amp}]
    \label{fac:lipschitz-3prod-bound}
    Let $\mu,\mu' \in \cP_2(\bbR^3)$, and suppose the marginals of $\mu$ have fourth moments.
    Suppose $f_1,f_2,f_3$ are $L$-Lipschitz functions, and $f_3$ is bounded by $L$.
    Then there exists $C = C(\mu,L)$ such that
    \beq
        \label{eq:lipschitz-3prod-bound}
        |\bbE_{(x,y,z) \sim \mu} f_1(x)f_2(y)f_3(z)
        - \bbE_{(x',y',z') \sim \mu'} f_1(x')f_2(y')f_3(z')|
        \le C\max(\bbW_2(\mu,\mu'), \bbW_2(\mu,\mu')^2).
    \eeq
\end{fac}
\begin{lem}
    \label{lem:amp-emp-msr-approximations}
    Suppose \eqref{eq:amp-event-2} holds.
    Uniformly over $(\bm,\bn) \in U(r_0)$, $\abh \in U'(r_0)$, $\dbv \in \{\tnorm{\dbv} = 1, \dbv \perp \bm\}$,
    \beq
        \label{eq:amp-emp-msr-approximation-clause}
        \bbW_2\lt(
            \fr1M \sum_{a=1}^M
            \delta(\hh^k_a, \ah_a, n_a, \hg_\AMP(\dbv)_a),
            (\tq_\eps^{1/2} Z, \tq_\eps^{1/2} Z, F_{\eps,\varrho_\eps}(\tq_\eps^{1/2} Z), Z')
        \rt) \le \err.
    \eeq
    Similarly, uniformly over $(\bm,\bn) \in U(r_0)$, $\hbv \in \{\tnorm{\hbv} = r_\eps, \hbv \perp \bn\}$,
    \beq
        \label{eq:amp-emp-msr-approximation-var}
        \bbW_2\lt(
            \fr1N \sum_{i=1}^N
            \delta(\dh^k_i, m_i, \dg_\AMP(\hbv)_i),
            (\tpsi_\eps^{1/2} Z, \th_\eps(\tpsi_\eps^{1/2} Z), r_\eps Z')
        \rt) \le \err.
    \eeq
\end{lem}
\begin{proof}
    We first show that for any $\dbv' \in \{\tnorm{\dbv'} = 1, \dbv' \perp \bm\}$,
    \beq
        \label{eq:hg-amp-dbvpr-goal}
        \bbW_2\lt(
            \fr1M \sum_{a=1}^M
            \delta(\hh^k_a, \hg_\AMP(\dbv')_a),
            (\tq_\eps^{1/2} Z, Z')
        \rt) = o_\ups(1).
    \eeq
    Indeed, let $\dbv' = \fr{1}{\sqrt{N}} \bM_{(k)} \dvv + P_{\bM_{(k)}}^\perp \dbv'$ for some $\dvv \in \bbR^{k+1}$, so that $\hbg_\AMP(\dbv') = \brbH_{(k)} \dvv + \tnorm{P_{\bM_{(k)}}^\perp \dbv'} \hbxi$.
    By the approximate isometry \eqref{eq:bM-bN-overlap-structure}, \eqref{eq:babH-brbH-overlap-structure}, since $\fr{1}{\sqrt{N}} \bM_{(k)} \dvv \perp \bm^k$, we have $\fr1N \la \brbh^k, \brbH_{(k)} \dvv \ra = o_\ups(1)$.
    (Since $\ups$ is small depending on $k$, we may take it much smaller than the condition number of $\hSig_{\le k}$.)
    By this isometry,
    \[
        \bbW_2\lt(
            \fr1M \sum_{a=1}^M
            \delta(\hh^k_a, (\brbH_{(k)} \dvv)_a),
            (\tq_\eps^{1/2} Z, \tnorm{P_{\bM_{(k)}} \dbv'} Z')
        \rt)
        = o_\ups(1).
    \]
    Then \eqref{eq:amp-event-2} implies \eqref{eq:hg-amp-dbvpr-goal}.
    Now consider $(\bm,\bn) \in U(r_0)$ and let $T$ be a rotation operator mapping $\bm / \tnorm{\bm}$ to $\bm^k / \tnorm{\bm^k}$.
    Note that $\tnorm{T - I}_\op = o_{r_0}(1)$.
    Consider any $\dbv \in \{\tnorm{\dbv} = 1, \dbv \perp \bm\}$, and let $\dbv' = T \dbv$.
    Then,
    \[
        \tnorm{\hbg_\AMP(\dbv') - \hbg_\AMP(\dbv)}
        \le (\sqrt{N} \tnorm{\hbT}_\op + \tnorm{\hbxi}) \tnorm{\dbv' - \dbv}
        \le \sqrt{N} (\tnorm{\hbT}_\op + O(1)) o_{r_0}(1).
    \]
    Note that
    \[
        \tnorm{\hbT}_\op
        = \sup_{\dvv \in \bbR^{k+1}}
        \fr{\tnorm{\hbT \bM_{(k)} \dvv}}{\tnorm{\bM_{(k)} \dvv}}
        = \sup_{\dvv \in \bbR^{k+1}}
        \fr{\tnorm{\brbH \dvv}}{\tnorm{\bM_{(k)} \dvv}}
        = \sup_{\dvv \in \bbR^{k+1}}
        \sqrt{
            \fr{\la \fr1N \brbH^\top \brbH, \dvv^{\otimes 2} \ra}{\la \fr1N \bM^\top \bM, \dvv^{\otimes 2} \ra}
        }
    \]
    is bounded by an absolute constant by \eqref{eq:bM-bN-overlap-structure}, \eqref{eq:babH-brbH-overlap-structure}.
    Thus $\tnorm{\hbg_\AMP(\dbv') - \hbg_\AMP(\dbv)} \le o_{r_0}(1) \sqrt{N}$.
    By \eqref{eq:def-abh-k} and definition of $U'(r_0)$,
    \beq
        \label{eq:hbhk-to-abh-bd}
        \tnorm{\hbh^k - \abh}
        \le \tnorm{\hbh^k - \abh^k}
        + \tnorm{\abh^k - \abh}
        \le (o_k(1) + o_{r_0}(1)) \sqrt{N}.
    \eeq
    Similarly,
    \beq
        \label{eq:bnk-to-bn-bd}
        \tnorm{F_{\eps,\varrho_\eps}(\hbh^k) - \bn}
        = \tnorm{\bn^k - \bn}
        \le o_{r_0}(1) \sqrt{N}.
    \eeq
    Combining these bounds with \eqref{eq:hg-amp-dbvpr-goal} proves \eqref{eq:amp-emp-msr-approximation-clause}.
    \eqref{eq:amp-emp-msr-approximation-var} is proved similarly.
\end{proof}
\begin{ppn}
    \label{ppn:hbv-terms}
    If \eqref{eq:amp-event-2} holds, uniformly over $(\bm,\bn) \in U(r_0)$, $\abh \in U'(r_0)$, $\dbv \in \{\tnorm{\dbv} = 1, \dbv \perp \bm\}$,
    \[
        \inf_{\substack{\tnorm{\hbv} = r_\eps, \\ \hbv \perp \bn}}
        \la \bD_2(\abh)^{-1} \hbv, \hbv \ra
        + \fr{2}{\sqrt{N}} \la \hbv, \hbg_\AMP(\dbv) \ra
        \le
        - \alpha_\star \EE \lt[
            \fr{\hf_\eps(\tq_\eps^{1/2} Z)}{1 + m_\eps(z_\eps) \hf_\eps(\tq_\eps^{1/2} Z)}
        \rt]
        - m_\eps(z_\eps) r_\eps^2
        + \err.
    \]
\end{ppn}
\begin{proof}
    Let
    \[
        \hbv' = - \fr{1}{\sqrt{N}} \lt(\bD_2(\abh)^{-1} + m_\eps(z_\eps) I\rt)^{-1} \hbg_\AMP(\dbv).
    \]
    Note the identity
    \beq
        \label{eq:hbv-terms-identity}
        \alpha_\star \EE\lt[\lt(
            \fr{\hf_\eps(\tq_\eps^{1/2} Z)}{1 + m_\eps(z_\eps) \hf_\eps(\tq_\eps^{1/2} Z)}
        \rt)^2\rt]
        = \fr{\alpha_\star \theta_\eps(z_\eps)}{\EE[(z_\eps + \df_\eps(\tpsi_\eps^{1/2} Z))^{-2}]}
        = r_\eps^2.
    \eeq
    Then,
    \baln
        \tnorm{\hbv'}^2
        &= \fr1N
        \hbg_\AMP(\dbv)^\top
        \lt(\tbD_2(\abh)^{-1} + m_\eps(z_\eps) I\rt)^{-2}
        \hbg_\AMP(\dbv) \\
        &= \fr{\alpha_\star}{M}
        \sum_{a=1}^M
        \lt(\fr{\hf_\eps(\ah_a)}{1 + m_\eps(z_\eps) \hf_\eps(\ah_a)}\rt)^2
        \hg_\AMP(\dbv)_a^2 \\
        &= \alpha_\star \EE\lt[\lt(
            \fr{\hf_\eps(\tq_\eps^{1/2} Z)}{1 + m_\eps(z_\eps) \hf_\eps(\tq_\eps^{1/2} Z)}
        \rt)^2 (Z')^2\rt] + \err
        = r_\eps^2 + \err.
    \ealn
    In the last line we used Lemma~\ref{lem:amp-emp-msr-approximations} and Fact~\ref{fac:lipschitz-3prod-bound}, with $f_1(x)=f_2(x)=x$, $f_3(x) = (\fr{\hf_\eps(x)}{1+m_\eps(z_\eps)\hf_\eps(x)})^2$.
    (Note that we have not shown the coordinate empirical measure in \eqref{eq:amp-emp-msr-approximation-clause} has bounded fourth moments, but it suffices for Fact~\ref{fac:lipschitz-3prod-bound} that the gaussian approximating it does.)
    Similarly,
    \baln
        \fr{1}{\sqrt{N}} \la \hbv', \bn \ra
        &= -\fr{\alpha_\star}{M}
        \sum_{a=1}^M
        \lt(\fr{\hf_\eps(\ah_a)}{1 + m_\eps(z_\eps) \hf_\eps(\ah_a)}\rt)
        n_a
        \hbg_\AMP(\dbv)_a \\
        &= -\alpha_\star \EE \lt[
            \lt(\fr{\hf_\eps(\tq_\eps^{1/2} Z)}{1 + m_\eps(z_\eps) \hf_\eps(\tq_\eps^{1/2} Z)}\rt)
            F_{\eps,\varrho_\eps}(\tq_\eps^{1/2} Z)
            Z'
        \rt]
        + \err
        = \err.
    \ealn
    Likewise,
    \baln
        \la (\bD_2(\abh)^{-1} + m_\eps(z_\eps) \bI_M) \hbv', \hbv' \ra
        = - \fr{1}{\sqrt{N}} \la \hbv', \hbg_\AMP(\dbv) \ra
        &= \fr{\alpha_\star}{M} \sum_{a=1}^M
        \lt(\fr{\hf_\eps(\ah_a)}{1 + m_\eps(z_\eps) \hf_\eps(\ah_a)}\rt)
        \hbg_\AMP(\dbv)_a^2 \\
        &= \alpha_\star \EE \lt[
            \fr{\hf_\eps(\tq_\eps^{1/2} Z)}{1 + m_\eps(z_\eps) \hf_\eps(\tq_\eps^{1/2} Z)}
        \rt] + \err.
    \ealn
    From this, it follows that
    \[
        \la \bD_2(\abh)^{-1} \hbv', \hbv' \ra
        + \fr{2}{\sqrt{N}} \la \hbv', \hbg_\AMP(\dbv) \ra
        = - \alpha_\star \EE \lt[
            \fr{\hf_\eps(\tq_\eps^{1/2} Z)}{1 + m_\eps(z_\eps) \hf_\eps(\tq_\eps^{1/2} Z)}
        \rt] - m_\eps(z_\eps) r_\eps^2 + \err.
    \]
    By the above estimates on $\tnorm{\hbv'}^2$ and $\fr{1}{\sqrt{N}} \la \hbv', \bn \ra$, we can find $\hbv$ such that $\tnorm{\hbv} = r_\eps$, $\hbv \perp \bn$, and $\tnorm{\hbv - \hbv'} \le \err$.
    Since $\bD_2(\abh)^{-1}$ has operator norm bounded independently of $r_0,k,\ups$,
    \[
        |\la \bD_2(\abh)^{-1} \hbv, \hbv \ra - \la \bD_2^{-1} \hbv', \hbv' \ra|
        \le 2 \tnorm{\bD_2^{-1}(\abh)}_\op \tnorm{\hbv - \hbv'}
        \le \err.
    \]
    By Cauchy--Schwarz,
    \[
        \fr{2}{\sqrt{N}} |\la \hbv, \hbg_\AMP(\dbv) \ra - \la \hbv', \hbg_\AMP(\dbv) \ra| \le
        \fr{2}{\sqrt{N}} \tnorm{\hbg_\AMP(\dbv)} \tnorm{\hbv - \hbv'}
        \le \err.
    \]
    This completes the proof.
\end{proof}

\begin{ppn}
    \label{ppn:dbv-terms}
    If \eqref{eq:amp-event-2} holds, uniformly over $(\bm,\bn) \in U(r_0)$, $\hbv \in \{\tnorm{\hbv} = r_\eps, \hbv \perp \bn\}$, we have
    \[
        \sup_{\substack{\tnorm{\dbv} = 1 \\ \dbv \perp \bm}}
        -\la \bD_1 \dbv, \dbv \ra + \fr{2}{\sqrt{N}} \la \dbv, \dbg_\AMP(\hbv) \ra
        \le z_\eps + m_\eps(z_\eps) r_\eps^2 + \err.
    \]
\end{ppn}
\begin{proof}
    Fix any $(\bm,\bn)$ and $\hbv$ satisfying the stated conditions.
    We estimate
    \beq
        \label{eq:dbv-terms-start}
        \sup_{\substack{\tnorm{\dbv} = 1 \\ \dbv \perp \bm}}
        -\la \bD_1 \dbv, \dbv \ra + \fr{2}{\sqrt{N}} \la \dbv, \dbg_\AMP(\hbv) \ra
        \le
        \sup_{\dbv \perp \bm}
        -\la \bD_1 \dbv, \dbv \ra + \fr{2}{\sqrt{N}} \la \dbv, \dbg_\AMP(\hbv) \ra - z_\eps \lt(\tnorm{\dbv}^2 - 1\rt).
    \eeq
    Note that $-\bD_1 - z_\eps \bI_N$ is negative definite, as $z_\eps > -\fr{1}{1+\eps} = \max_{x\in \bbR} \{-\df(x)\}$.
    So, the supremum on the right-hand side of \eqref{eq:dbv-terms-start} is maximized by $\dbv$ solving the stationarity condition (in $\spn(\bm)^\perp$):
    \[
        \dbv = \fr{1}{\sqrt{N}} P_{\bm}^\perp (\bD_1 + z_\eps \bI_N)^{-1} P_{\bm}^\perp \dbg_\AMP(\hbv).
    \]
    Let
    \[
        \dbv' = \fr{1}{\sqrt{N}} (\bD_1 + z_\eps \bI_N)^{-1} \dbg_\AMP(\hbv).
    \]
    Note that, by Fact~\ref{fac:lipschitz-3prod-bound} and Lemma~\ref{lem:amp-emp-msr-approximations},
    \baln
        \la (\bD_1 + z_\eps \bI_N) \dbv', \dbv' \ra
        = \fr{1}{\sqrt{N}} \la \dbv', \dbg_\AMP(\hbv) \ra
        &= \fr1N \sum_{i=1}^N \dg_\AMP(\hbv)_i^2 (\df_\eps(\dh_i) + z_\eps)^{-1} \\
        &= r_\eps^2 \EE \lt[(\df_\eps(\tpsi_\eps Z) + z_\eps)^{-1}\rt] + \err \\
        &= m_\eps(z_\eps) r_\eps^2 + \err.
    \ealn
    Thus
    \[
        -\la \bD_1 \dbv', \dbv' \ra + \fr{2}{\sqrt{N}} \la \dbv', \dbg_\AMP(\hbv) \ra - z_\eps \lt(\tnorm{\dbv'}^2 - 1\rt)
        = z_\eps + m_\eps(z_\eps) r_\eps^2 + \err.
    \]
    We now estimate $\tnorm{\dbv - \dbv'}$. Note that
    \[
        \tnorm{\dbv - \dbv'}
        \le \tnorm{(\bD_1 + z_\eps \bI_N)^{-1}}_\op \tnorm{P_{\bm} \dbg_\AMP(\hbv)}
        + \tnorm{P_{\bm} (\bD_1 + z_\eps \bI_N)^{-1} \dbg_\AMP(\hbv)},
    \]
    and by Fact~\ref{fac:lipschitz-3prod-bound} and Lemma~\ref{lem:amp-emp-msr-approximations}, both terms on the right-hand side are bounded by $\err$.
    Since $\bD_1 + z_\eps \bI_N$ has bounded operator norm,
    \[
        |\la (\bD_1 + z_\eps \bI_N) \dbv, \dbv \ra - \la (\bD_1 + z_\eps \bI_N) \dbv', \dbv' \ra|
        \le 2\tnorm{\bD_1 + z_\eps \bI_N}_\op \tnorm{\dbv - \dbv'}
        \le \err.
    \]
    By Cauchy--Schwarz,
    \[
        \fr{2}{\sqrt{N}} |\la \dbv', \dbg_\AMP(\hbv) \ra - \la \dbv, \dbg_\AMP(\hbv) \ra|
        \le \fr{2}{\sqrt{N}} \tnorm{\dbg_\AMP(\hbv)} \tnorm{\dbv - \dbv'}
        \le \err.
    \]
    Combining completes the proof.
\end{proof}

\begin{proof}[Proof of Proposition~\ref{ppn:null-concavity}]
    By Propositions~\ref{ppn:hbv-terms} and \ref{ppn:dbv-terms}, on the high probability event \eqref{eq:amp-event-2}, the left-hand side of \eqref{eq:null-concavity-goal2} is bounded by
    \[
        z_\eps - \alpha_\star \EE \lt[
            \fr{\hf_\eps(\tq_\eps^{1/2} Z)}{1 + m_\eps(z_\eps) \hf_\eps(\tq_\eps^{1/2} Z)}
        \rt] + \err
        = \lambda_\eps + d_\eps + \err.
    \]
    This proves \eqref{eq:null-concavity-goal2}, and by the discussion leading to \eqref{eq:null-concavity-goal2} the proposition follows.
\end{proof}

\begin{proof}[Proof of Proposition~\ref{ppn:amp-guarantees}\ref{itm:amp-guarantee-concave}, under $\PP$]
    By Proposition~\ref{ppn:amp-guarantees}\ref{itm:amp-guarantee-profile}, with high probability, $(\bm^k,\bn^k) \in \cS_{\eps,\ups_0}$.
    Recall that $\th_\eps, F_{\eps,\varrho_\eps}$ are $O(1)$-Lipschitz, with $O_\eps(1)$-Lipschitz inverses (i.e. Lipschitz constant depending only on $\eps$).
    On this event, for $\ups_0$ small depending on $r_0$ and some $C_\eps = O_\eps(1)$,
    \beq
        \label{eq:Ur0-in-cSeps}
        U(r_0)
        \subseteq \cS_{\eps, \ups_0 + C_\eps r_0}
        \subseteq \cS_{\eps, 2C_\eps r_0}.
    \eeq
    Since $\tnorm{\bG}_\op, \tnorm{\hbg} \le C\sqrt{N}$ holds with high probability under $\PP$, Lemma~\ref{lem:hessian-estimate} applies.
    Applying this lemma with $2C_\eps r_0$ in place of $r_0$ shows that for all $(\bm,\bn) \in U(r_0)$,
    \[
        \nabla^2_\diamond \cF^\eps_\TAP(\bm,\bn) \preceq
        \bR(\bm,\bn)
        + \lambda_\eps P_\bm
        + (o_{C_\cvx}(1) + o_{r_0}(1)) \bI_N.
    \]
    Combined with Proposition~\ref{ppn:null-concavity}, this gives that with high probability,
    \[
        \nabla^2_\diamond \cF(\bm,\bn)
        \preceq (\lambda_\eps + o_{C_\cvx}(1) + o_{r_0}(1) + o_k(1)) \bI_N.
    \]
    By Lemma~\ref{lem:lambda-eps-to-0},
    \[
        \nabla^2_\diamond \cF(\bm,\bn)
        \preceq (\lambda_0 + o_\eps(1) + o_{C_\cvx}(1) + o_{r_0}(1) + o_k(1)) \bI_N.
    \]
    Under Condition~\ref{con:local-concavity}, $\lambda_0 < 0$.
    The conclusion follows by setting the parameters so the error term in the last display is bounded by $|\lambda_0|/2$.
\end{proof}

\begin{rmk}
    \label{rmk:freeprob}
    The bound $\lambda_\eps + d_\eps$ in Proposition~\ref{ppn:null-concavity} is tight.
    One way to see this is to calculate the upper edge of the limiting spectral measure of
    \baln
        \bA &= P_{\bM_{(k)}}^\perp \lt(
            - \bD_1 - \bW
        \rt) P_{\bM_{(k)}}^\perp, &
        &\text{where} &
        \bW &= \fr1N \bG^\top P_{\bN^{(k)}}^\perp \bD_2 P_{\bN^{(k)}}^\perp \bG,
    \ealn
    using free probability \cite{voiculescu1991limit}.
    We now outline this calculation.
    Note that conditional on $\DATA$, $-\bD_1$ and $-\bW$ are orthogonally invariant as quadratic forms on $\spn(\bm^0,\ldots,\bm^k)^\perp$.
    The inverse Cauchy transform of $-\bD_1$ is approximated within $\err$ by $m^{-1}_\eps(t)$.
    By e.g. \cite[Equation 1.2]{bai1998no}, the inverse Cauchy transform of $-\bW$ is approximated within $\err$ by
    \[
        \fr{1}{t} - \alpha_\star \EE \lt[\fr{\hf_\eps(\tq_\eps^{1/2} Z)}{1 + t \hf_\eps(\tq_\eps^{1/2} Z)}\rt],
    \]
    Since R-transforms add under free additive convolution, $\bA$ has limiting inverse Cauchy transform
    \[
        \vartheta_\eps(t) = m^{-1}_\eps(t) - \alpha_\star \EE \lt[\fr{\hf_\eps(\tq_\eps^{1/2} Z)}{1 + t \hf_\eps(\tq_\eps^{1/2} Z)}\rt].
    \]
    One calculates that
    \[
        \vartheta'_\eps(t) = - \EE[(m^{-1}_\eps(t) + \df_\eps(\tpsi_\eps^{1/2} Z))^{-2}]^{-1}
        + \EE \lt[\lt(\fr{\hf_\eps(\tq_\eps^{1/2} Z)}{1 + t \hf_\eps(\tq_\eps^{1/2} Z)}\rt)^2\rt]
    \]
    has the same sign as $\theta_\eps(m^{-1}_\eps(t)) - \alpha_\star^{-1}$.
    Thus $\vartheta_\eps(t)$ is decreasing on $(0,m_\eps(z_\eps)]$ and increasing $[m_\eps(z_\eps),+\infty)$.
    It follows that the limiting spectral measure of $\bA$ has upper edge $\vartheta_\eps(m_\eps(z_\eps)) = \lambda_\eps + d_\eps$.
    By the Weyl inequalities the same is true for $\bR(\bm,\bn)$, so Proposition~\ref{ppn:null-concavity} is tight.
\end{rmk}

\subsection{Planted model}

The proof of Proposition~\ref{ppn:amp-guarantees}\ref{itm:amp-guarantee-concave} in the planted model is only simpler, as we will be able to apply Gordon's inequality directly rather than conditional on AMP iterates.
The main step is the following proposition.
Let $\ups$ be sufficiently small depending on $r_0, k$.
\begin{ppn}
    \label{ppn:planted-concavity}
    Suppose $(\bm',\bn') \in \cS_{\eps,\ups}$.
    With high probability under $\PP^{\bm',\bn'}_{\eps,\Pl}$, $\bR(\bm,\bn) \preceq (\lambda_\eps + d_\eps + \err) P_{\bm}^\perp$ for all $\tnorm{(\bm,\bn) - (\bm',\bn')} \le 2r_0 \sqrt{N}$.
\end{ppn}
\noindent Let $\dbh' = \th_\eps^{-1}(\bm')$, $\hbh' = F_{\eps,\rho_\eps(q(\bm))}^{-1}(\bn')$.
By Lemma~\ref{lem:tap-1deriv}, under $\PP^{\bm,\bn'}_{\eps,\Pl}$ we have $\abh(\bm',\bn',\bG) = \hbh'$.

For this subsection, let $U(r_0) = \{(\bm,\bn) : \tnorm{(\bm,\bn) - (\bm',\bn')} \le 2r_0 \sqrt{N}\}$
and $U'(r_0) = \{\abh : \tnorm{\abh - \hbh'} \le Cr_0 \sqrt{N}\}$, for suitably large constant $C$.
Identically to the discussion above \eqref{eq:null-concavity-start}, to prove Proposition~\ref{ppn:planted-concavity} it suffices to show, with high probability,
\[
    \sup_{(\bm,\bn) \in U(r_0)}
    \sup_{\substack{\tnorm{\dbv} = 1 \\ \dbv \perp \bm}}
    \inf_{\substack{\tnorm{\hbv} = r_\eps, \\ \hbv \perp \bn}} \lt\{
        - \la \bD_1 \dbv, \dbv \ra
        + \la \bD_2(\abh)^{-1} \hbv, \hbv \ra
        + \fr{2}{\sqrt{N}} \la \bG \dbv, \hbv \ra
    \rt\}
    \le \lambda_\eps + d_\eps + \err.
\]
\begin{lem}
    Let $\dbxi,\dbxi' \sim \cN(0, \bI_N)$, $\hbxi,\hbxi' \sim \cN(\bzero,\bI_M)$, $Z,Z' \sim \cN(0,1)$ be independent of everything else and
    \baln
        \dbg'_\Pl(\hbv) &=
        \fr{\tnorm{P_{\bn'} \hbv} (\dbh' + \eps^{1/2} P_{\bm'}^\perp \dbxi')}{\tpsi_\eps^{1/2}}
        + \tnorm{P_{\bn'}^\perp \hbv} P_{\bm'}^\perp \dbxi, &
        \hbg'_\Pl(\dbv) &=
        \fr{\tnorm{P_{\bm'} \dbv} (\hbh' + \eps^{1/2} P_{\bn'}^\perp \hbxi')}{\tq_\eps^{1/2}}
        + \tnorm{P_{\bm'}^\perp \dbv} P_{\bn'}^\perp \hbxi.
    \ealn
    For any continuous $f : \bbR^N \times \bbR^M \times (\bbR^N)^2 \times (\bbR^M)^3 \to \bbR$,
    \[
        \sup_{\substack{(\bm,\bn) \in U(r_0) \\ \abh \in U'(r_0)}}
        \sup_{\substack{\tnorm{\dbv} = 1 \\ \dbv \perp \bm}}
        \inf_{\substack{\tnorm{\hbv} = r_\eps, \\ \hbv \perp \bn}} \lt\{
            f(\dbv,\hbv;\bm',\bm,\bn',\bn,\abh)
            + \fr{2}{\sqrt{N}} \la \bG \dbv, \hbv \ra
            + \fr{2 \tnorm{P_{\bn'}^\perp \hbv} \tnorm{P_{\bm'}^\perp \dbv}}{\sqrt{N}} Z
        \rt\}
    \]
    is stochastically dominated by
    \baln
        \sup_{\substack{(\bm,\bn) \in U(r_0) \\ \abh \in U'(r_0)}}
        \sup_{\substack{\tnorm{\dbv} = 1 \\ \dbv \perp \bm}}
        &\inf_{\substack{\tnorm{\hbv} = r_\eps, \\ \hbv \perp \bn}} \bigg\{
            f(\dbv,\hbv;\bm',\bm,\bn',\bn,\abh)
            + \fr{2}{\sqrt{N}} \la \dbv, \dbg'_\Pl(\hbv) \ra
            + \fr{2}{\sqrt{N}} \la \hbv, \hbg'_\Pl(\dbv) \ra \\
            &+ \fr{2 \eps^{1/2} \tnorm{P_{\bn'} \hbv} \tnorm{P_{\bm'} \dbv}}{(q_\eps + \psi_\eps + \eps)^{1/2} \sqrt{N}} Z'
        \bigg\} + o_\ups(1).
    \ealn
\end{lem}
\begin{proof}
    By Corollary~\ref{cor:conditional-law-correct-profile}, the gaussian process $(\dbv,\hbv) \mapsto \fr{1}{\sqrt{N}} \la \bG \dbv, \hbv \ra$ has the form
    \baln
        \fr{1}{\sqrt{N}} \la \bG \dbv, \hbv \ra
        &\stackrel{d}{=}
        \fr{\la \dbh', \dbv \ra \la \bn', \hbv \ra}{N\tpsi_\eps}
        + \fr{\la \bm', \dbv \ra \la \hbh', \hbv \ra}{N\tq_\eps}
        + o_\ups(1)
        + \fr{1}{\sqrt{N}} \la \tbG \dbv, \hbv \ra \\
        &=
        \fr{\tnorm{P_{\bn'}\hbv} \la \dbh', \dbv \ra}{\tpsi_\eps^{1/2} \sqrt{N}}
        + \fr{\tnorm{P_{\bm'}\dbv} \la \hbh', \hbv \ra}{\tq_\eps^{1/2} \sqrt{N}}
        + o_\ups(1)
        + \fr{1}{\sqrt{N}} \la \tbG \dbv, \hbv \ra.
    \ealn
    Here the $o_\ups(1)$ is uniform over bounded $\tnorm{\dbv}, \tnorm{\hbv}$.
    Moreover, by \eqref{eq:residual-variances}, the random part $\la \tbG \dbv, \hbv \ra$ expands as
    \baln
        \la \tbG \dbv, \hbv \ra
        &=\la \tbG P_{\bm'}^\perp \dbv, P_{\bn'}^\perp \hbv \ra
        + \la \tbG P_{\bm'}^\perp \dbv, P_{\bn'} \hbv \ra
        + \la \tbG P_{\bm'} \dbv, P_{\bn'}^\perp \hbv \ra
        + \la \tbG P_{\bm'} \dbv, P_{\bn'} \hbv \ra \\
        &\stackrel{d}{=}
        \la \tbG P_{\bm'}^\perp \dbv, P_{\bn'}^\perp \hbv \ra
        + \fr{\eps^{1/2}}{\tpsi_\eps^{1/2}} \tnorm{P_{\bn'} \hbv} \la P_{\bm'}^\perp \dbxi', \dbv \ra
        + \fr{\eps^{1/2}}{\tq_\eps^{1/2}} \tnorm{P_{\bm'} \dbv} \la P_{\bn'}^\perp \hbxi', \hbv \ra
        + \fr{\eps^{1/2} \tnorm{P_{\bn'} \hbv} \tnorm{P_{\bm'} \dbv}}{(q_\eps+\psi_\eps+\eps)^{1/2}} Z'.
    \ealn
    Thus, (as processes)
    \baln
        \fr{1}{\sqrt{N}} \la \bG \dbv, \hbv \ra
        + \fr{\tnorm{P_{\bn'}^\perp \hbv} \tnorm{P_{\bm'}^\perp \dbv}}{\sqrt{N}} Z
        &\stackrel{d}{=}
        \fr{1}{\sqrt{N}} \la \tbG P_{\bm'}^\perp \dbv, P_{\bn'}^\perp \hbv \ra
        + \fr{\tnorm{P_{\bn'}^\perp \hbv} \tnorm{P_{\bm'}^\perp \dbv}}{\sqrt{N}} Z \\
        &+ \fr{\tnorm{P_{\bn'} \hbv} \la \dbh' + \eps^{1/2} P_{\bm'}^\perp \dbxi', \dbv \ra}{\tpsi_\eps^{1/2} \sqrt{N}}
        + \fr{\tnorm{P_{\bm'} \dbv} \la \hbh' + \eps^{1/2} P_{\bn'}^\perp \hbxi', \hbv \ra}{\tq_\eps^{1/2} \sqrt{N}} \\
        &+ \fr{\eps^{1/2} \tnorm{P_{\bn'} \hbv} \tnorm{P_{\bm'} \dbv}}{(q_\eps+\psi_\eps+\eps)^{1/2} \sqrt{N}} Z' + o_\ups(1).
    \ealn
    The result now follows by using Gordon's inequality to compare $\fr{1}{\sqrt{N}} \la \tbG P_{\bm'}^\perp \dbv, P_{\bn'}^\perp \hbv \ra + \fr{\tnorm{P_{\bn'}^\perp \hbv} \tnorm{P_{\bm'}^\perp \dbv}}{\sqrt{N}} Z$ to $\fr{1}{\sqrt{N}} \tnorm{P_{\bn'}^\perp \hbv} \la \dbv, P_{\bm'}^\perp \dbxi \ra + \fr{1}{\sqrt{N}} \tnorm{P_{\bm'}^\perp \dbv} \la \hbv, P_{\bn'}^\perp \hbxi \ra$.
\end{proof}
\noindent Let
\baln
    \dbg_\Pl(\hbv) &=
    \fr{\tnorm{P_{\bn'} \hbv} (\dbh' + \eps^{1/2} \dbxi')}{\tpsi_\eps^{1/2}}
    + \tnorm{P_{\bn'}^\perp \hbv} \dbxi, &
    \hbg_\Pl(\dbv) &=
    \fr{\tnorm{P_{\bm'} \dbv} (\hbh' + \eps^{1/2} \hbxi')}{\tq_\eps^{1/2}}
    + \tnorm{P_{\bm'}^\perp \dbv} \hbxi.
\ealn
As argued above \eqref{eq:null-concavity-goal2}, with high probability,
\[
    \fr{1}{\sqrt{N}} |Z|,
    \fr{1}{\sqrt{N}} |Z'|,
    \fr{1}{\sqrt{N}} \sup_{\tnorm{\hbv}=r_\eps} \tnorm{\dbg_\Pl(\hbv) - \dbg'_\Pl(\hbv)},
    \fr{1}{\sqrt{N}} \sup_{\tnorm{\dbv}=1} \tnorm{\hbg_\Pl(\dbv) - \hbg'_\Pl(\dbv)} \le \ups.
\]
So it suffices to show that with high probability,
\balnn
    \notag
    \sup_{\substack{(\bm,\bn) \in U(r_0) \\ \abh \in U'(r_0)}}
    \sup_{\substack{\tnorm{\dbv} = 1 \\ \dbv \perp \bm}}
    &\inf_{\substack{\tnorm{\hbv} = r_\eps, \\ \hbv \perp \bn}} \bigg\{
        - \la \bD_1 \dbv, \dbv \ra
        + \la \bD_2(\abh)^{-1} \hbv, \hbv \ra \\
        \label{eq:planted-concavity-goal2}
        &+ \fr{2}{\sqrt{N}} \la \dbv, \dbg_\Pl(\hbv) \ra
        + \fr{2}{\sqrt{N}} \la \hbv, \hbg_\Pl(\dbv) \ra
    \bigg\}
    \le \lambda_\eps + d_\eps + \err.
\ealnn
\begin{lem}
    \label{lem:pl-emp-msr-approximations}
    For all $(\bm',\bn') \in \cS_{\eps,\ups}$, the following holds with high probability.
    Uniformly over $(\bm,\bn) \in U(r_0)$, $\abh \in U'(r_0)$, $\dbv \in \{\tnorm{\dbv} = 1, \dbv \perp \bm\}$,
    \beq
        \label{eq:pl-emp-msr-approximation-clause}
        \bbW_2\lt(
            \fr1M \sum_{a=1}^M
            \delta(\hh'_a, \ah_a, n'_a, \hg_\Pl(\dbv)_a),
            (\tq_\eps^{1/2} Z, \tq_\eps^{1/2} Z, F_{\eps,\varrho_\eps}(\tq_\eps^{1/2} Z), Z')
        \rt) \le \err.
    \eeq
    Similarly, uniformly over $(\bm',\bn') \in \cS_{\eps,\ups}$, $(\bm,\bn) \in U(r_0)$, $\hbv \in \{\tnorm{\hbv} = r_\eps, \hbv \perp \bn\}$,
    \beq
        \label{eq:pl-emp-msr-approximation-var}
        \bbW_2\lt(
            \fr1N \sum_{i=1}^N
            \delta(\dh'_i, m'_i, \dg_\Pl(\hbv)_i),
            (\tpsi_\eps^{1/2} Z, \th_\eps(\tpsi_\eps^{1/2} Z), r_\eps Z')
        \rt) \le \err.
    \eeq
\end{lem}
\begin{proof}
    Let $\hbh'' = F_{\eps,\varrho_\eps}^{-1}(\bn')$.
    Consider first $\dbv' \in \{\tnorm{\dbv'} = 1, \dbv' \perp \bm\}$,
    Then $\hbg_\Pl(\dbv') = \hbxi$, so clearly
    \[
        \bbW_2\lt(
            \fr1M \sum_{a=1}^M
            \delta(\hh''_a, \hg_\Pl(\dbv')_a),
            (\tq_\eps^{1/2} Z, Z')
        \rt)
        = o_\ups(1).
    \]
    For $(\bm,\bn) \in U(r_0)$, let $T$ be a rotation operator mapping $\bm / \tnorm{\bm}$ to $\bm' / \tnorm{\bm'}$.
    Note that $\tnorm{T-I}_\op = o_{r_0}(1)$.
    Consider any $\dbv \in \{\tnorm{\dbv} = 1, \dbv \perp \bm\}$, and let $\dbv' = T \dbv$, so $\tnorm{\dbv - \dbv'} = o_{r_0}(1)$.
    Then
    \[
        \tnorm{\hbg_\Pl(\dbv') - \hbg_\Pl(\dbv)}
        \le O(1) \lt(\tnorm{\hbh'} + \tnorm{\hbxi'} + \tnorm{\hbxi}\rt) \tnorm{\dbv - \dbv'}.
    \]
    With high probability over $\hbxi,\hbxi'$, this is bounded by $o_{r_0}(1) \sqrt{N}$.
    Thus
    \beq
        \label{eq:hg-pl-dbvpr-goal}
        \bbW_2\lt(
            \fr1M \sum_{a=1}^M
            \delta(\hh''_a, \hg_\Pl(\dbv)_a),
            (\tq_\eps^{1/2} Z, Z')
        \rt)
        = o_{r_0}(1) + o_\ups(1).
    \eeq
    Note that
    \[
        \tnorm{\hbh' - \hbh''} = \tnorm{F_{\eps,\rho_\eps(q(\bm))}^{-1}(\bn') - F_{\eps,\varrho_\eps}^{-1}(\bn')}
        \le \err \sqrt{N}.
    \]
    Identically to \eqref{eq:hbhk-to-abh-bd} and \eqref{eq:bnk-to-bn-bd}, we can show
    \[
        \tnorm{\hbh' - \abh}, \tnorm{F_{\eps,\varrho_\eps}(\hbh'') - \bn}
        \le \err \sqrt{N}.
    \]
    Combined with \eqref{eq:hg-pl-dbvpr-goal}, this proves \eqref{eq:pl-emp-msr-approximation-clause}.
    The proof of \eqref{eq:pl-emp-msr-approximation-var} is analogous.
\end{proof}
The following two propositions are proved identically to Propositions~\ref{ppn:hbv-terms} and \ref{ppn:dbv-terms}, with $\hbg_\Pl$, $\dbg_\Pl$, and Lemma~\ref{lem:pl-emp-msr-approximations} playing the roles of $\hbg_\AMP$, $\dbg_\AMP$, and Lemma~\ref{lem:amp-emp-msr-approximations}.
\begin{ppn}
    \label{ppn:hbv-terms-pl}
    For all $(\bm',\bn') \in \cS_{\eps,\ups}$, the following holds with high probability.
    Uniformly over $(\bm,\bn) \in U(r_0)$, $\abh \in U'(r_0)$, $\dbv \in \{\tnorm{\dbv} = 1, \dbv \perp \bm\}$, we have
    \[
        \inf_{\substack{\tnorm{\hbv} = r_\eps, \\ \hbv \perp \bn}}
        \la \bD_2(\abh)^{-1} \hbv, \hbv \ra
        + \fr{2}{\sqrt{N}} \la \hbv, \hbg_\Pl(\dbv) \ra
        \le
        - \alpha_\star \EE \lt[
            \fr{\hf_\eps(\tq_\eps^{1/2} Z)}{1 + m_\eps(z_\eps) \hf_\eps(\tq_\eps^{1/2} Z)}
        \rt]
        - m_\eps(z_\eps) r_\eps^2
        + \err.
    \]
\end{ppn}

\begin{ppn}
    \label{ppn:dbv-terms-pl}
    For all $(\bm',\bn') \in \cS_{\eps,\ups}$, the following holds with high probability.
    Uniformly over $(\bm,\bn) \in U(r_0)$, $\hbv \in \{\tnorm{\hbv} = r_\eps, \hbv \perp \bn\}$, we have
    \[
        \sup_{\substack{\tnorm{\dbv} = 1 \\ \dbv \perp \bm}}
        -\la \bD_1 \dbv, \dbv \ra + \fr{2}{\sqrt{N}} \la \dbv, \dbg_\Pl(\hbv) \ra
        \le z_\eps + m_\eps(z_\eps) r_\eps^2 + \err.
    \]
\end{ppn}

\begin{proof}[Proof of Proposition~\ref{ppn:planted-concavity}]
    Adding Propositions~\ref{ppn:hbv-terms-pl} and \ref{ppn:dbv-terms-pl} shows that \eqref{eq:planted-concavity-goal2} holds with high probability.
    The result follows from the discussion leading to \eqref{eq:planted-concavity-goal2}.
\end{proof}

\begin{proof}[Proof of Proposition~\ref{ppn:amp-guarantees}\ref{itm:amp-guarantee-concave}, under $\bbP^{\bm,\bn}_{\eps,\Pl}$]
    By Proposition~\ref{ppn:amp-guarantees}\ref{itm:amp-guarantee-planted}, $\tnorm{(\bm^k,\bn^k) - (\bm,\bn)} = \ups_0 \sqrt{N}$ with high probability.
    We set $\ups_0 < r_0$.
    Since we defined
    \[
        U(r_0)
        = \{(\bm,\bn) : \tnorm{(\bm,\bn) - (\bm',\bn')} \le 2r_0 \sqrt{N}\}
        \supseteq
        \{(\bm,\bn) : \tnorm{(\bm,\bn) - (\bm^k,\bn^k)} \le r_0 \sqrt{N}\},
    \]
    the conclusion of Proposition~\ref{ppn:planted-concavity} holds for all $\tnorm{(\bm,\bn) - (\bm^k,\bn^k)} \le r_0 \sqrt{N}$.
    Identically to \eqref{eq:Ur0-in-cSeps}, we have
    \[
        \{(\bm,\bn) : \tnorm{(\bm,\bn) - (\bm^k,\bn^k)} \le r_0 \sqrt{N}\}
        \subseteq \cS_{\eps,2C_\eps r_0}
    \]
    for some $C_\eps = O_\eps(1)$.
    Since $\tnorm{\bG}_\op, \tnorm{\hbg} \le C\sqrt{N}$ holds with high probability under $\bbP^{\bm,\bn}_{\eps,\Pl}$, Lemma~\ref{lem:hessian-estimate} holds.
    Applying this lemma (with $2C_\eps r_0$ in place of $r_0$) gives that for all $\tnorm{(\bm,\bn) - (\bm^k,\bn^k)} \le r_0 \sqrt{N}$,
    \baln
        \nabla^2_\diamond \cF^\eps_\TAP(\bm,\bn)
        &\preceq \bR(\bm,\bn)
        + \lambda_\eps P_\bm
        + (o_{C_\cvx}(1) + o_{r_0}(1)) \bI_N \\
        &\preceq (\lambda_\eps + o_{C_\cvx}(1) + o_{r_0}(1) + o_k(1)) \bI_N \\
        &\preceq (\lambda_0 + o_\eps(1) + o_{C_\cvx}(1) + o_{r_0}(1) + o_k(1)) \bI_N.
    \ealn
    Under Condition~\ref{con:local-concavity}, $\lambda_0 < 0$, and the result follows by setting the error terms small.
\end{proof}

\subsection{Determinant concentration}

In this subsection, we prove Lemma~\ref{lem:det-concentration}.
We fix some $(\bm,\bn) \in \cS_{\eps,\ups}$ and work under the measure $\PP^{\bm,\bn}_{\eps,\Pl}$.
Define, as in Lemma~\ref{lem:tap-1deriv},
\baln
    \dbh &= \th_\eps^{-1}(\bm), &
    \hbh &= F_{\eps,\rho_\eps(\bm)}^{-1}(\bn), &
    \abh &= \fr{\bG \bm}{\sqrt{N}} + \eps^{1/2} \hbg - \rho_\eps(q(\bm)) \bn.
\ealn
Recall from Lemma~\ref{lem:tap-1deriv} that under $\PP^{\bm,\bn}_{\eps,\Pl}$, we have $\abh = \hbh$ deterministically.
We computed $\nabla^2 \cF_\TAP(\bm,\bn)$ in Fact~\ref{fac:tap-2deriv}, and under $\PP^{\bm,\bn}_{\eps,\Pl}$ the matrices $\bD_1, \tbD_2, \bD_3, \bD_4$ therein are all nonrandom.
By Schur's lemma,
\beq
    \label{eq:det-conc-step1}
    |\det \nabla^2 \cF_\TAP(\bm,\bn)|
    = |\det \nabla^2_{\bn,\bn} \cF_\TAP(\bm,\bn)|
    |\det \nabla^2_\diamond \cF_\TAP(\bm,\bn)|,
\eeq
and $\nabla^2_{\bn,\bn} \cF_\TAP(\bm,\bn)$ is nonrandom.
By Fact~\ref{fac:tap-2deriv},
\[
    \nabla^2_\diamond \cF_\TAP(\bm,\bn)
    = - \bD_1 - \fr{1}{N} \bG^\top \tbD_2 \bG + \rho'_\eps(q(\bm)) d_\eps(\bm,\bn) \bI_N
    + \fr{C}{N}\bm\bm^\top
    + \fr{1}{N}(\bG^\top \bv \bm^\top + \bm \bv^\top \bG)
\]
for some nonrandom $C\in \bbR$, $\bv \in \bbR^M$ depending on $(\bm,\bn)$.
By Lemma~\ref{lem:hessian-crude-estimates}, $|C|, \tnorm{\bv}$ are uniformly bounded over $(\bm,\bn) \in \cS_{\eps,\ups}$, with bound depending on $\eps,C_\cvx$.
Define for convenience the nonrandom matrix
\[
    \bA = \bD_1 - \rho'_\eps(q(\bm)) d_\eps(\bm,\bn) \bI_N - \fr{C}{N}\bm\bm^\top
\]
and note that $\tnorm{\bA}_\op$ is uniformly bounded (depending on $\eps,C_\cvx$) over $(\bm,\bn) \in \cS_{\eps,\ups}$.
Then let
\beq
    \label{eq:hermitianization}
    \bX = \begin{bmatrix}
        \bA & \fr{1}{\sqrt{N}} \bm \bv^\top & \fr{1}{\sqrt{N}} \bG^\top \\
        \fr{1}{\sqrt{N}} \bv \bm^\top & \tbD_2 & \bI_M \\
        \fr{1}{\sqrt{N}} \bG & \bI_M & \bzero
    \end{bmatrix}
    \in \bbR^{(N+2M) \times (N+2M)}.
\eeq
\begin{lem}
    \label{lem:det-equivalence}
    We have $|\det \nabla^2_\diamond \cF_\TAP(\bm,\bn)| = |\det \bX|$.
\end{lem}
\begin{proof}
    Let $\bY = \lt[\begin{smallmatrix} \tbD_2 & \bI_M \\ \bI_M & \bzero \end{smallmatrix}\rt]$.
    Note that $|\det \bY| = 1$ and $\bY^{-1} = \lt[\begin{smallmatrix} \bzero & \bI_M \\ \bI_M & -\tbD_2 \end{smallmatrix}\rt]$.
    By Schur's lemma,
    \[
        |\det \bX|
        = \lt|\det\lt(
            \bA
            - \fr1N
            \begin{bmatrix}
                \bm\bv^\top & \bG^\top
            \end{bmatrix}
            \bY^{-1}
            \begin{bmatrix}
                \bv\bm^\top \\
                \bG
            \end{bmatrix}
        \rt)\rt|
        = |\det \nabla^2_\diamond \cF_\TAP(\bm,\bn)|. \qedhere
    \]
\end{proof}
\noindent It therefore suffices to study $|\det \bX|$.
This formulation has the benefit that the only randomness in $\bX$ is from $\bG$, and by Lemma~\ref{lem:conditional-law} (in a suitable orthonormal basis) $\bG$ is a matrix of independent (noncentered) gaussians.
This structure will enable us to prove Lemma~\ref{lem:det-concentration} using the spectral concentration results of \cite{guionnet2000concentration}.
Before carrying out this argument, we first prove a preliminary lemma.
\begin{lem}
    \label{lem:bM-spec-gap}
    There exists $\tau > 0$ depending on $\eps,C_\cvx$ such that, for all $(\bm,\bn) \in \cS_{\eps,\ups}$, $\bX$ has no eigenvalues in $[-\tau,\tau]$ with high probability under $\PP^{\bm,\bn}_{\eps,\Pl}$.
\end{lem}
\begin{proof}
    We will show that $\det(z\bI_{N+2M} - \bX)$ has no zeros in $[-\tau,\tau]$.
    By Schur's lemma, for any $z\neq 0$,
    \[
        |\det(z\bI_{2M} - \bY)|
        = |\det(z\bI_M - \tbD_2)|
        |\det(z\bI_M - (z\bI_M - \tbD_2)^{-1})|
        = |\det(z(z\bI_M - \tbD_2) - \bI_M)|
    \]
    Let $\tau_1$ be the smallest positive solution to $\tau_1 |\max(\hf_\eps) + \tau| \le \fr12$.
    Note that $\tau_1$ depends only on $\eps$, and the above determinant is nonzero for any $|z| \le \tau_1$.
    Further, note that
    \[
        (z\bI_{2M} - \bY)^{-1}
        = \begin{bmatrix}
            -z(\bI_M - z(z\bI_M-\tbD_2))^{-1} & (\bI_M - z(z\bI_M-\tbD_2))^{-1} \\
            (\bI_M - z(z\bI_M-\tbD_2))^{-1} & -(z\bI_M - \tbD_2)(\bI_M - z(z\bI_M-\tbD_2))^{-1}
        \end{bmatrix}.
    \]
    From this, we see that there exists $C_\eps > 0$ such that for all $|z| \le \tau_1$,
    \[
        \tnorm{(z\bI_{2M} - \bY)^{-1} + \bY^{-1}}_\op
        \le C_\eps |z|.
    \]
    By Schur's lemma, for all $|z|\le \tau_1$,
    \[
        |\det(z\bI_{N+2M} - \bX)|
        = |\det(z\bI_{2M} - \bY)|
        |\det \bB(z)|,
    \]
    for
    \[
        \bB(z)
        =
        z\bI_N - \bA
        - \fr1N
        \begin{bmatrix}
            \bm\bv^\top & \bG^\top
        \end{bmatrix}
        (z\bI_{2M} - \bY)^{-1}
        \begin{bmatrix}
            \bv\bm^\top \\
            \bG
        \end{bmatrix}.
    \]
    It follows that for all $|z| \le \tau_1$,
    \[
        \tnorm{\bB(z) - \nabla^2_\diamond \cF^\eps_\TAP(\bm,\bn)}_\op
        \le |z| + C_\eps |z|\lt(\fr{\tnorm{\bv\bm^\top}_\op}{\sqrt{N}} + \fr{\tnorm{\bG}_\op}{\sqrt{N}}\rt)^2.
    \]
    As shown in Proposition~\ref{ppn:amp-guarantees}\ref{itm:amp-guarantee-concave}, $\nabla^2_\diamond \cF^\eps_\TAP(\bm,\bn) \preceq -C_{\spec} \bI_N$ with high probability under $\PP^{\bm,\bn}_{\eps,\Pl}$.
    Furthermore, $\fr{\tnorm{\bv\bm^\top}_\op}{\sqrt{N}} = \fr{1}{\sqrt{N}} \tnorm{\bv} \tnorm{\bm}$ is bounded, with bound depending on $\eps,C_\cvx$, and with high probability, $\fr{\tnorm{\bG}_\op}{\sqrt{N}}$ is bounded by an absolute constant.
    It follows that for $|z|$ small enough depending on $\eps,C_\cvx$, $\bB(z) \preceq -C_{\spec} \bI_N/2$, and thus $|\det \bB(z)| \neq 0$.
\end{proof}
The core of the proof of Lemma~\ref{lem:det-concentration} is the following spectral concentration inequality, which adapts \cite[Theorem 1.1(b)]{guionnet2000concentration}.
For any $f : \bbR \to \bbR$, let
\[
    \tr f(\bX) = \sum_{i=1}^{N+2M} f(\lambda_i(\bX)),
\]
where $\lambda_1(\bX), \ldots, \lambda_{N+2M}(\bX)$ are the eigenvalues of $\bX$.
\begin{lem}
    \label{lem:spec-conc}
    If $f$ is $L$-Lipschitz, then for any $t\ge 0$,
    \[
        \bbP^{\bm,\bn}_{\eps,\Pl}(|\tr f(\bX) - \bbE^{\bm,\bn}_{\eps,\Pl} \tr f(\bX)| \ge t) \le 2 e^{-t^2/8L^2}.
    \]
\end{lem}
\begin{proof}
    Let $\{\omega_{a,i} : a\in [M], i\in [N]\}$ be i.i.d. standard gaussians, and let $\dbe_1,\ldots,\dbe_N$ and $\hbe_1,\ldots,\hbe_M$ be orthonormal bases of $\bbR^N$ and $\bbR^M$ as in Lemma~\ref{lem:conditional-law}.
    By \eqref{eq:residual-variances}, we can sample $\tbG$ by
    \baln
        \tbG &= \sum_{a=1}^M \sum_{i=1}^N w_{a,i} \omega_{a,i} \hbe_a \dbe_i^\top, &
        w_{a,i} &= \begin{cases}
            \sqrt{\eps / (q(\bm) + \psi(\bn) + \eps)} & i=j=1, \\
            \sqrt{\eps / (q(\bm) + \eps)} & i=1, j\neq 1, \\
            \sqrt{\eps / (\psi(\bn) + \eps)} & i\neq 1, j=1, \\
            1 & i\neq 1, j\neq 1.
        \end{cases}
    \ealn
    By \cite[Lemma 1.2(b)]{guionnet2000concentration}, the map $\{\omega_{a,i} : a\in [M], i\in [N]\} \mapsto \tr f(\bX)$ is $2L$-Lipschitz.
    The result follows from the gaussian concentration inequality.
\end{proof}
\begin{proof}[Proof of Lemma~\ref{lem:det-concentration}]
    Define $f(x) = \log \max(|x|, \tau)$, which is $\tau^{-1}$-Lipschitz.
    Lemma~\ref{lem:spec-conc} implies that
    \beq
        \label{eq:spec-conc-refined}
        \bbP^{\bm,\bn}_{\eps,\Pl}(|\tr f(\bX) - \bbE^{\bm,\bn}_{\eps,\Pl} \tr f(\bX)| \ge t) \le 2e^{-\tau^2 t^2/8}.
    \eeq
    Let $\tdet(\bX) = \exp \tr f(\bX)$.
    Also let
    \[
        \cE_{\spec}(\bX) = \lt\{
            \spec(\bX) \cap [-\tau,\tau]
            = \emptyset
        \rt\},
    \]
    so that $\bbP(\cE_{\spec}) \ge 1-\iota$ for some $\iota = o_N(1)$ by Lemma~\ref{lem:bM-spec-gap}.
    Note that $|\det(\bX)| \le \tdet(\bX)$ for all $\bX$, with equality for all $\bX \in \sE_{\spec}$.
    Thus
    \balnn
        \label{eq:det-conc-prelim1}
        \bbE^{\bm,\bn}_{\eps,\Pl} [|\det(\bX)|^2] &\le \bbE^{\bm,\bn}_{\eps,\Pl} [\tdet(\bX)^2], &
        \bbE^{\bm,\bn}_{\eps,\Pl} [|\det(\bX)|] &\ge \bbE^{\bm,\bn}_{\eps,\Pl} [\tdet(\bX) \bone\{\cE_{\spec}\}].
    \ealnn
    By the concentration \eqref{eq:spec-conc-refined}, there exists $C$ depending on $\eps,C_\cvx$ such that
    \[
        \bbE^{\bm,\bn}_{\eps,\Pl}[\tdet(\bX)^2]
        \le C \exp (2\bbE^{\bm,\bn}_{\eps,\Pl} \tr f(\bX)).
    \]
    Furthermore, by Jensen's inequality $\bbE^{\bm,\bn}_{\eps,\Pl}[\tdet(\bX)] \ge \exp (\bbE^{\bm,\bn}_{\eps,\Pl} \tr f(\bX))$.
    Thus,
    \beq
        \label{eq:det-conc-prelim2}
        \bbE^{\bm,\bn}_{\eps,\Pl}[\tdet(\bX)^2]
        \le C \bbE^{\bm,\bn}_{\eps,\Pl}[\tdet(\bX)]^2.
    \eeq
    By Cauchy--Schwarz,
    \[
        \bbE^{\bm,\bn}_{\eps,\Pl} [\tdet(\bX) \bone\{\cE^c_{\spec}\}]
        \le \bbE^{\bm,\bn}_{\eps,\Pl} [\tdet(\bX)^2]^{1/2}
        \bbP^{\bm,\bn}_{\eps,\Pl}(\cE^c_{\spec})^{1/2}
        \le C^{1/2} \iota^{1/2} \bbE^{\bm,\bn}_{\eps,\Pl}[\tdet(\bX)].
    \]
    It follows that
    \[
        \bbE^{\bm,\bn}_{\eps,\Pl} [\tdet(\bX) \bone\{\cE_{\spec}\}]
        \ge (1 - C^{1/2} \iota^{1/2}) \bbE^{\bm,\bn}_{\eps,\Pl}[\tdet(\bX)].
    \]
    Combining with \eqref{eq:det-conc-prelim1}, \eqref{eq:det-conc-prelim2} shows that
    \[
        \bbE^{\bm,\bn}_{\eps,\Pl} [|\det(\bX)|^2]^{1/2}
        \le C^{1/2} (1 - C^{1/2} \iota^{1/2})^{-1}
        \bbE^{\bm,\bn}_{\eps,\Pl}[|\det(\bX)|],
    \]
    which implies the result after adjusting $C$.
\end{proof}

\section{First moment in planted model}
\label{sec:planted-mt}

In this section, we prove Proposition~\ref{ppn:planted-mt}, bounding the first moment of $Z_N(\bG)$ in the planted model.
The proof is structured as follows.
In \S\ref{subsec:planted-functional-optimization}, we show this moment is bounded by a optimization problem over $\bLam : \bbR \to \bbR$ encoding subsets of $\Sigma_N$ with a certain coordinate profile (heuristically described in \eqref{eq:SigmaN-subset-with-profile-heuristic}).
\S\ref{subsec:planted-to-two-params} reduces this optimization to two dimensions by showing the maximizer is attained in a two-parameter family.
For technical reasons, the functional in this optimization problem is not the $\sS_\star$ defined in \eqref{eq:def-sS-star}, but a variant $\sS_\star^{s_{\max}}$ where $s$ is minimized over $[0,s_{\max}]$ instead of $[0,+\infty)$ see \eqref{eq:sS-star-numax-bLam}.
\S\ref{subsec:smax-to-infty} and \S\ref{subsec:no-bdry-max} show that we recover the optimization of $\sS_\star$ when $s_{\max} \to \infty$, completing the proof of Proposition~\ref{ppn:planted-mt}.
\S\ref{subsec:first-mt-functional-local} proves Lemma~\ref{lem:sS-zero}, on the local behavior of the first moment functional $\sS_\star(\lambda_1,\lambda_2)$ near $(1,0)$.

\subsection{Reduction to functional optimization}
\label{subsec:planted-functional-optimization}

Recall that $(q_0,\psi_0)$ are given by Condition~\ref{con:km-well-defd}.
Let $\dbH \sim \cN(0,\psi_0)$, $\bM = \th(\dbH)$, and $\hbH \sim \cN(0, q_0)$, $\bN = F_{1-q_0}(\hbH)$, for $F_{1-q_0}$ given by \eqref{eq:def-cE-F}.
Let $\sL = L^2(\bbR,\cN(0,\psi_0))$ denote the space of measurable functions $\bLam : \bbR \to \bbR$, equipped with the inner product
\[
	\la \bLam_1, \bLam_2 \ra = \EE[\bLam_1(\dbH)\bLam_2(\dbH)]
\]
and square-integrable w.r.t. the associated norm.
Let $\sK \subseteq \sL$ denote the set of functions with image in $[-1,1]$.
For $s_{\max}>0$, define
\beq
	\label{eq:sS-star-numax-bLam}
	\sS_\star^{s_{\max}}(\bLam) = \inf_{0\le s \le s_{\max}} \sS_\star(\bLam,s),
\eeq
where $\sS_\star : \sK \times [0,+\infty) \to \bbR$ is defined by \eqref{eq:def-sS-star-bLam-nu}.
The following proposition bounds the first moment by the maximum of an optimization problem over functions $\bLam$, and is the starting point of the proof of Proposition~\ref{ppn:planted-mt}.
\begin{ppn}
	\label{ppn:functional-optimization}
	For any $s_{\max}>0$, $(\bm,\bn) \in \cS_{\eps,\ups}$, we have $\fr1N \log \EE^{\bm,\bn}_{\eps,\Pl} [Z_N(\bG)] \le \sup_{\bLam \in \sK} \sS^{s_{\max}}_\star(\bLam) + o_{\eps,\ups}(1)$.
\end{ppn}
\noindent Here $o_{\eps,\ups}(1)$ denotes a term vanishing as $\eps,\ups \to 0$, which can depend on $s_{\max}$; we send $s_{\max}\to\infty$ after $\eps,\ups \to 0$ in the end.

Before proving Proposition~\ref{ppn:functional-optimization}, we state a few facts that will be useful below.
Lemma~\ref{lem:denominator-doesnt-explode} ensures that the denominator of $\sS_\star(\bLam,s)$ is well-behaved, while Lemmas~\ref{lem:logPsi-pseudo-lipschitz} and \ref{lem:functional-continuity} are useful in approximation arguments.
\begin{lem}
	\label{lem:denominator-doesnt-explode}
	There exists $\iota > 0$ such that $\EE[\bM \bLam(\dbH)]^2 < (1-\iota) q_0$ for all $\bLam \in \sK$.
\end{lem}
\begin{proof}
	Since $|\bLam(\dbH)| \le 1$, by Cauchy--Schwarz,
	\[
		\EE[\bM \bLam(\dbH)]^2
		\le \EE[|\bM|]^2
		< \EE[\bM^2].
	\]
	The inequality is strict because $|\bM|$ has nonzero variance.
	Since $\EE[\bM^2] = P(\psi_0) = q_0$ (recall Condition~\ref{con:km-well-defd}), the result follows.
\end{proof}
\begin{lem}
	\label{lem:logPsi-pseudo-lipschitz}
	The function $\log \Psi(x)$ is $(2,1)$-pseudo-Lipschitz (recall Definition~\ref{dfn:pseudo-lipschitz}).
\end{lem}
\begin{proof}
	Note that $(\log \Psi)'(x) = -\cE(x)$.
	Recall from Lemma~\ref{lem:E-derivative-bds}\ref{itm:cE-bd} that $0\le \cE(x) \le 1 + |x|$.
	Thus,
	\[
		|\log \Psi(x) - \log \Psi(y)|
		= \lt|\int_x^y \cE(s) ~\de s\rt|
		\le |x-y| (1+|x|+|y|).
	\]
\end{proof}
\begin{lem}[Proved in Appendix~\ref{app:amp}]
	\label{lem:functional-continuity}
	There exists $C>0$ such that for all $a_1,a_2,b_1,b_2,c_1,c_2 > 0$,
	\baln
		&\lt|\EE \log \Psi \lt\{
			\fr{\kappa
				- a_1\hbH
				- b_1\bN
			}{
				c_1
			}
		\rt\}
		- \log \Psi \lt\{
			\fr{\kappa
				- a_2\hbH
				- b_2\bN
			}{
				c_2
			}
		\rt\}\rt| \\
		&\le \fr{C\max(a_1,a_2,b_1,b_2,c_1,c_2,1)^3}{\min(c_1,c_2)^2}
		\lt(|a_1-a_2|+|b_1-b_2|+|c_1-c_2|\rt).
	\ealn
\end{lem}

We turn to the proof of Proposition~\ref{ppn:functional-optimization}.
The main step will be Proposition~\ref{ppn:functional-optimization-elementary} below, where we show the bound in Proposition~\ref{ppn:functional-optimization} holds for piecewise-constant $\bLam$ with finitely many parts.
This case follows from a direct moment calculation, and Proposition~\ref{ppn:functional-optimization} follows by approximation.

For any $\vr = (r_1,\ldots,r_{n-1})$ with $-\infty < r_1 < r_2 < \cdots < r_{n-1} < +\infty$, let $\sK_\elt(\vr) \subseteq \sK$ denote the set of right-continuous functions which are constant on each interval $[r_{k-1},r_k)$, $1\le k\le n$.
Here we take as convention $r_0 = -\infty$, $r_n = +\infty$.
Define the quantiles $\vp = (p_0,\ldots,p_n)$ by $p_k = \PP(\dbH < r_k)$, and let
\[
	\mesh(\vp) = \min_{1\le k\le n} (p_k - p_{k-1}).
\]
Let $o_{\eps,\ups,\vp}(1)$ denote a term vanishing as $\eps,\ups,\mesh(\vp) \to 0$, where (like before) this limit is taken after $N\to\infty$ for fixed $s_{\max}$.
We will show the following.
\begin{ppn}
	\label{ppn:functional-optimization-elementary}
	Suppose $s_{\max}>0$, $(\bm,\bn) \in \cS_{\eps,\ups}$, and $\vr = (r_1,\ldots,r_{n-1})$ is as above.
	We have that $\fr1N \log \EE^{\bm,\bn}_{\eps,\Pl} [Z_N(\bG)] \le \sup_{\bLam \in \sK_\elt(\vr)} \sS^{s_{\max}}_\star(\bLam) + o_{\eps,\ups,\vp}(1)$.
\end{ppn}
For the rest of this subsection, fix $s_{\max},\eps,\ups,\vr$ and $(\bm,\bn)$ as in Proposition~\ref{ppn:functional-optimization-elementary}.
Let $\dbh = \th_\eps^{-1}(\bm)$ and $\hbh = F_{\eps,\varrho_\eps}^{-1}(\bn)$, so that $(\dbh,\hbh) \in \cT_{\eps,\ups}$.
Fix a partition $[N] = \cI_1 \cup \cdots \cup \cI_n$ satisfying
\baln
	|\cI_k| &= \lfloor p_k N\rfloor - \lfloor p_{k-1} N\rfloor, & \forall 1&\le k\le n, \\
	\max\{\dh_i : i\in \cI_k\} &\le \min\{\dh_i : i\in \cI_{k+1}\}, & \forall 1&\le k\le n-1.
\ealn
(In words, $\cI_k$ is the set of coordinates $i\in [N]$ such that the quantile of $\dh_i$ among the entries of $\dbh$, breaking ties in an arbitrary but fixed order, lies in $[p_{k-1},p_k)$.)
Then, partition $\Sigma_N$ into sets
\beq
	\label{eq:SigmaN-subset-with-profile}
	\Sigma_N(\va)
	= \lt\{
		\bx \in \Sigma_N :
		\sum_{i\in \cI_k} x_i = a_k,
		~\forall 1\le k\le n
	\rt\}.
\eeq
indexed by $\va = (a_1,\ldots,a_n) \in \bbZ^n$.
Let $\cJ$ be the set of $\va$ such that $\Sigma_N(\va)$ is nonempty, and note that $|\cJ| \le N^n$.
Thus
\balnn
	\notag
	\fr1N \log \bbE^{\bm,\bn}_{\eps,\Pl} [Z_N(\bG)]
	&= \fr1N \log \sum_{\va \in \cJ} \sum_{\bx \in \Sigma_N(\va)}
	\bbP^{\bm,\bn}_{\eps,\Pl} \lt(\fr{\bG \bx}{\sqrt{N}} \ge \kappa\rt) \\
	\label{eq:planted-mt-elementary-start}
	&= \sup_{\va \in \cJ} \lt\{
		\fr1N \log |\Sigma_N(\va)|
		+ \sup_{\bx \in \Sigma_N(\va)} \fr1N \log
		\bbP^{\bm,\bn}_{\eps,\Pl} \lt(\fr{\bG \bx}{\sqrt{N}} \ge \kappa\rt)
	\rt\} + o_N(1).
\ealnn
Associate to each $\va \in \cJ$ a function $\bLam^{\va} \in \sK_\elt(r_1,\ldots,r_{n-1})$ defined by
\[
	\bLam^{\va}(x) =
	\fr{a_k}{|\cI_k|},
	\qquad x\in [r_{k-1},r_k),
	1\le k\le n.
\]
Recall the function $\ent : \sK \to \bbR$ defined in \eqref{eq:def-ent}.
\begin{lem}
	\label{lem:planted-mt-enumeration}
	We have $\fr1N \log |\Sigma_N(\va)| = \ent(\bLam^\va) + o_N(1)$ for an error $o_N(1)$ uniform over $\va \in \cJ$.
\end{lem}
\begin{proof}
	By direct counting,
	\[
		|\Sigma_N(\va)|
		= \prod_{k=1}^n
		\binom{|\cI_k|}{\fr12(|\cI_k| + a_k)}.
	\]
	Stirling's approximation yields
	\[
		\fr1N \log |\Sigma_N(\va)|
		= \sum_{k=1}^n \lt\{
			(p_k - p_{k-1})
			\cH\lt(\fr{1 + \fr{a_k}{(p_k-p_{k-1})N}}{2}\rt)
		\rt\} + o_N(1)
		= \EE \cH\lt(\fr{1 + \bLam^{\va}(\dbH)}{2}\rt) + o_N(1),
	\]
	where the last equality holds because $\PP(\dbH \in [r_{k-1},r_k)) = p_k - p_{k-1}$.
\end{proof}
\begin{lem}
	\label{lem:planted-mt-inner-prod}
	For all $\va\in \cJ$ and $\bx \in \Sigma_N(\va)$,
	\baln
		\fr1N \la \dbh, \bx \ra &= \EE[\dbH \bLam^{\va}(\dbH)] + o_{\eps,\ups,\vp}(1), &
		\fr1N \la \bm, \bx \ra &= \EE[\bM \bLam^{\va}(\dbH)] + o_{\eps,\ups,\vp}(1),
	\ealn
	for error terms $o_{\eps,\ups,\vp}(1)$ uniform over $\va, \bx$.
\end{lem}
\begin{proof}
	We will only show the proof for $\fr1N \la \dbh, \bx \ra$, as the other estimate is analogous.
	Let $\bx \in \Sigma_N(\va)$ be fixed, and let $\by \in [-1,1]^N$ be defined by $y_i = \fr{a_k}{|\cI_k|}$ for all $i\in \cI_k$.
	We write $(\dbH',\bX,\bY,\bK)$ for the random variable with value $(\dh_i,x_i,y_i,k)$, where $i\sim \unif([N])$ and $k\in [n]$ is the index of the set $\cI_k$ containing $i$.
	Recall that $\dbH \sim \cN(0,\psi_0)$.
	Note that
	\[
		\bbW_2(\cL(\dbH'),\cL(\dbH))
		\le \bbW_2(\mu_{\dbh},\cN(0,\psi_\eps+\eps))
		+ \bbW_2(\cN(0,\psi_\eps+\eps),\cN(0,\psi_0))
		= o_{\eps,\ups}(1),
	\]
	where the latter two distances are bounded by definition of $\cT_\ups$ and Proposition~\ref{ppn:eps-perturb-fixed-point}, respectively.
	We couple $(\dbH', \dbH)$ monotonically (which is the $\bbW_2$-optimal coupling) and write
	\[
		\fr1N \la \dbh, \bx \ra
		= \EE[\dbH' \bX]
		= \EE[\dbH \bY]
		+ \EE[(\dbH' - \dbH) \bX]
		+ \EE[\dbH (\bX - \bY)].
	\]
	We now estimate each of these terms.
	Because $(\dbH', \dbH)$ are coupled monotonically, $\bK = k$ if and only if the quantile of $\dbH$ lies in $[p'_{k-1}, p'_k)$, where $p'_k = \fr1N \lfloor p_k N \rfloor = p_k + O(N^{-1})$.
	Thus, on an event with probability $1-O(N^{-1})$, $\bK = k$ if and only if $\dbH \in [r_{k-1},r_k)$.
	On this event, $\bY = \bLam^\va(\dbH)$.
	Thus
	\[
		\EE[\dbH \bY] = \EE[\dbH \bLam^\va(\dbH)] + o_N(1).
	\]
	Moreover,
	\[
		|\EE[(\dbH' - \dbH) \bX]|
		\le \EE[(\dbH' - \dbH)^2]^{1/2}
		= \bbW_2(\cL(\dbH'),\cL(\dbH))
		= o_{\eps,\ups}(1).
	\]
	Finally, note that $\bY = \EE[\bX | \bK]$, so
	\[
		\EE[\EE[\dbH | \bK] (\bX - \bY)]
		= \EE[\EE[\dbH | \bK] \EE[\bX - \bY | \bK]] = 0.
	\]
	Thus
	\[
		|\EE[\dbH (\bX - \bY)]|
		= |\EE[(\dbH - \EE[\dbH | \bK]) (\bX - \bY)]|
		\le \EE[(\dbH - \EE[\dbH | \bK])^2]^{1/2}.
	\]
	Recall from the above discussion that conditioning on $\bK$ reveals the interval $[p'_{k-1}, p'_k)$ containing the quantile of $\dbH$.
	It follows that $\EE[(\dbH - \EE[\dbH | \bK])^2] = o_{\eps,\ups,\vp}(1)$.
\end{proof}
\begin{lem}
	\label{lem:planted-mt-probability}
	For all $\va \in \cJ$, $\bx \in \Sigma_N(\va)$, and $s \in [0,s_{\max}]$,
	\[
		\fr1N \log \bbP^{\bm,\bn}_{\eps,\Pl} \lt(\fr{\bG \bx}{\sqrt{N}} \ge \kappa \rt)
		\le
		\fr12 s^2 \psi_0
		+ \alpha_\star \EE \log \Psi \lt\{
			\fr{\kappa
				- \fr{\EE[\bM \bLam^\va(\dbH)]}{q_0}\hbH
				- \fr{\EE[\dbH \bLam^\va(\dbH)]}{\psi_0}\bN
			}{
				\sqrt{1 - \fr{\EE[\bM \bLam^\va(\dbH)]^2}{q_0}}
			}
			+ s \bN
		\rt\}
		+ o_{\eps,\ups,\vp}(1),
	\]
	where the $o_{\eps,\ups,\vp}(1)$ is uniform over $\va, \bx, s$ (but can depend on $s_{\max}$).
\end{lem}
\begin{proof}
	Let $\tbG$ be defined in Corollary~\ref{cor:conditional-law-correct-profile}.
	By Corollary~\ref{cor:conditional-law-correct-profile} and Lemma~\ref{lem:planted-mt-inner-prod},
	\baln
		\fr{\bG \bx}{\sqrt{N}}
		&\stackrel{d}{=}
        \lt(
        	\fr{(1+o_{\eps,\ups}(1))}{q_0} \hbh
        	+ o_{\eps,\ups}(1) \bn
    	\rt)
    	\fr1N \la \bm, \bx \ra
        + \fr{(1+o_{\eps,\ups}(1))}{\psi_0} \bn
        \cdot \fr1N \la \dbh, \bx \ra
        + \fr{\tbG \bx}{\sqrt{N}} \\
        &=
        \fr{\EE[\bM \bLam^{\va}(\dbH)] + o_{\eps,\ups,\vp}(1)}{q_0} \hbh
        + \fr{\EE[\dbH \bLam^{\va}(\dbH)] + o_{\eps,\ups,\vp}(1)}{\psi_0} \bn
        + \fr{\tbG \bx}{\sqrt{N}}.
	\ealn
	Let $\hbn = \bn / \tnorm{\bn}$.
	By inspecting \eqref{eq:residual-variances}, we see that for independent $\tbg \sim \cN(0, P^\perp_\bn)$ and $Z \sim \cN(0,1)$,
	\[
		\fr{\tbG \bx}{\sqrt{N}}
		\stackrel{d}{=}
		\lt(\fr{\tnorm{P_\bm^\perp(\bx)}^2}{N} + o_\eps(1)\rt)^{1/2} \tbg
		+ o_\eps(1) Z \hbn
		= t^{1/2} \tbg
		+ \iota_1^{1/2} Z \hbn,
	\]
	where $t = 1 - \fr{\EE[\bM \bLam^\va(\dbH)]^2}{q_0} + \iota_2$ and $\iota_1,\iota_2 = o_{\eps,\ups,\vp}(1)$.
	For $Z' \sim \cN(0,1)$ independent of $\tbg,Z$, let
	\[
		\hbg = \tbg + Z' \hbn + s \bn
	\]
	so that $\hbg \sim \cN(s \bn, \bI_N)$.
	Then, for any measurable $S \subseteq \bbR^N$,
	\baln
		\fr{\PP(t^{1/2} \tbg + \iota_1^{1/2} Z \hbn \in S)}{\PP(t^{1/2} \hbg \in S)}
		&\le
		\sup_{T\subseteq \bbR}
		\fr{\PP(\iota_1^{1/2} Z \in T)}{\PP(s t^{1/2} \tnorm{\bn} + t^{1/2} Z' \in T)} \\
		&\le
		\sup_{x\in \bbR} \fr{\iota_1^{-1/2}\exp(-\fr{1}{2\iota_1}x^2)}{t^{-1/2}\exp(-\fr{1}{2t}(x-s t^{1/2} \tnorm{\bn})^2)}
		= \sqrt{\fr{t}{\iota_1}} \exp\lt(
			\fr{s^2 \tnorm{\bn}^2}{2(1-\iota_1/t)}
		\rt).
	\ealn
	Thus,
	\balnn
		\notag
		&\fr1N \log \bbP^{\bm,\bn}_{\eps,\Pl} \lt(\fr{\bG \bx}{\sqrt{N}} \ge \kappa \rt)
		\le \fr{s^2 \psi(\bn)}{2(1-\iota_1/t)} + o_N(1) \\
		\label{eq:planted-mt-probability-aux}
		&\qquad + \fr1N \log \bbP\lt\{
			\fr{\EE[\bM \bLam^{\va}(\dbH)] + o_{\eps,\ups,\vp}(1)}{q_0} \hbh
        	+ \fr{\EE[\dbH \bLam^{\va}(\dbH)] + o_{\eps,\ups,\vp}(1)}{\psi_0} \bn
        	+ t^{1/2} \hbg \ge \kappa
		\rt\}.
	\ealnn
	By Lemma~\ref{lem:denominator-doesnt-explode}, $t$ is bounded away from $0$.
	Since $\psi(\bn) = \psi_0 + o_\eps(1)$, we have
	\[
		\fr{s^2 \psi(\bn)}{2(1-\iota_1/t)}
		= (1 + o_{\eps,\ups,\vp}(1)) \fr12 s^2 \psi_0
		= \fr12 s^2 \psi_0 + o_{\eps,\ups,\vp}(1).
	\]
	The last estimate holds uniformly over $s \in [0,s_{\max}]$.
	The last term of \eqref{eq:planted-mt-probability-aux} equals
	\[
		\fr1N \sum_{a=1}^M \log \Psi\lt\{
			\fr{
				\kappa
				- \fr{\EE[\bM \bLam^{\va}(\dbH)] + o_{\eps,\ups,\vp}(1)}{q_0} \hh_a
				- \fr{\EE[\dbH \bLam^{\va}(\dbH)] + o_{\eps,\ups,\vp}(1)}{\psi_0} n_a
			}{
				\sqrt{1 - \fr{\EE[\bM \bLam^\va(\dbH)]^2}{q_0} + o_{\eps,\ups,\vp}(1)}
			} + s n_a
		\rt\}
		+ o_N(1).
	\]
	By Lemma~\ref{lem:logPsi-pseudo-lipschitz}, $\log \Psi$ is $(2,1)$-pseudo-Lipschitz.
	By Fact~\ref{fac:pseudo-lipschitz} and Lemma~\ref{lem:functional-continuity} (using again that the denominator is bounded away from $0$), the last display equals
	\[
		\alpha_\star \EE \log \Psi\lt\{
			\fr{
				\kappa
				- \fr{\EE[\bM \bLam^{\va}(\dbH)]}{q_0} \hbH
				- \fr{\EE[\dbH \bLam^{\va}(\dbH)]}{\psi_0} \bN
			}{\sqrt{1 - \fr{\EE[\bM \bLam^{\va}(\dbH)]^2}{q_0}}}
			+ s \bN
		\rt\} + o_{\eps,\ups,\vp}(1).
	\]
	Combining the above concludes the proof.
\end{proof}

\begin{proof}[Proof of Proposition~\ref{ppn:functional-optimization-elementary}]
	Follows from equation \eqref{eq:planted-mt-elementary-start} and Lemmas~\ref{lem:planted-mt-enumeration} and \ref{lem:planted-mt-probability}.
\end{proof}

\begin{proof}[Proof of Proposition~\ref{ppn:functional-optimization}]
	Set $\vr$ such that $\mesh(\vp)$ is suitably small depending on $(\eps,\ups)$.
	Then
	\[
		\fr1N \log \bbE^{\bm,\bn}_{\eps,\Pl} [Z_N(\bG)]
		\le \sup_{\bLam \in \sK_\elt(\vr)} \sS^{s_{\max}}_\star(\bLam) + o_{\eps,\ups}(1)
		\le \sup_{\bLam \in \sK} \sS^{s_{\max}}_\star(\bLam) + o_{\eps,\ups}(1).
	\]
\end{proof}

\subsection{Reduction to two parameters}
\label{subsec:planted-to-two-params}

Let $\sK_\ast \subseteq \sK$ denote the set of functions of the form $\bLam_{\lambda_1,\lambda_2}$ defined above \eqref{eq:def-sS-star}.
Let $\osK_\ast$ denote the closure of this set in the topology of $\sL$.
We next prove the following, which reduces the functional optimization problem in Proposition~\ref{ppn:functional-optimization} to an optimization over $\osK_\ast$.
\begin{ppn}
	\label{ppn:planted-mt-2param-reduction}
	For any $s_{\max} > 0$, we have $\sup_{\bLam \in \sK} \sS^{s_{\max}}_\star (\bLam) = \sup_{\bLam \in \osK_\ast} \sS^{s_{\max}}_\star (\bLam)$.
	Similarly, $\sup_{\bLam \in \sK} \sS_\star (\bLam) = \sup_{\bLam \in \osK_\ast} \sS_\star (\bLam)$ for $\sS_\star(\bLam)$ defined in \eqref{eq:def-sS-star}.
\end{ppn}
\begin{lem}
	\label{lem:planted-mt-lm}
	Let $a_1,a_2 \in \bbR$ be such that there exists $\bLam \in \sK$ with $\EE[\dbH \bLam(\dbH)] = a_1$, $\EE[\bM \bLam(\dbH)] = a_2$.
	Then, the concave optimization problem
	\[
		\text{maximize}~
		\ent(\bLam) \qquad
		\text{subject to} \quad
		\bLam \in \sK, \quad
		\EE[\dbH \bLam(\dbH))] = a_1, \quad
		\EE[\bM \bLam(\dbH))] = a_2
	\]
	has a maximizer in $\osK_\ast$.
\end{lem}
\begin{proof}
	Introduce Lagrange multipliers $\lambda_1,\lambda_2 \in \bbR$.
	The Lagrangian is
	\[
		L(\bLam;\lambda_1,\lambda_2)
		=
		\EE \lt\{
			\cH\lt(\fr{1+\bLam(\dbH)}{2}\rt)
			+ \lambda_1 \dbH \bLam(\dbH)
			+ \lambda_2 \bM \bLam(\dbH)
		\rt\}
		- \lambda_1 a_1
		- \lambda_2 a_2.
	\]
	The quantity inside the expectation is concave in $\bLam(\dbH)$, with derivative
	\[
		- \th^{-1}(\bLam(\dbH))
		+ \lambda_1 \dbH
		+ \lambda_2 \bM.
	\]
	This is pointwise maximized by $\bLam(\dbH) = \th(\lambda_1 \dbH + \lambda_2 \bM)$, i.e. $\bLam = \bLam_{\lambda_1,\lambda_2}$.
\end{proof}

\begin{proof}[Proof of Proposition~\ref{ppn:planted-mt-2param-reduction}]
	Note that $\sS^{s_{\max}}_\star (\bLam)$ is the sum of $\ent(\bLam)$ and a term depending on $\bLam$ only through $\EE [\dbH \bLam(\dbH)]$ and $\EE [\bM \bLam(\dbH)]$.
	Let $\bLam \in \sK$ be arbitrary.
	By Lemma~\ref{lem:planted-mt-lm}, the maximum of $\ent(\tbLam)$ subject to $\tbLam \in \sK$, $\EE [\dbH \tbLam(\dbH)] = \EE [\dbH \bLam(\dbH)]$, $\EE [\bM \tbLam(\dbH)] = \EE[\bM \bLam(\dbH)]$ is attained by some $\tbLam \in \osK_\ast$.
	Thus $\sS^{s_{\max}}_\star(\bLam) \le \sS^{s_{\max}}_\star(\tbLam)$, which implies the conclusion for $\sS^{s_{\max}}$.
	The proof for $\sS_\star$ is identical.
\end{proof}

\subsection{The $s_{\max} \to \infty$ limit}
\label{subsec:smax-to-infty}

In this subsection, we prove the following proposition, which shows that the optimization problem derived in Proposition~\ref{ppn:planted-mt-2param-reduction} has a well-behaved limit when we take $s_{\max} \to \infty$.
This allows us to remove the parameter $s_{\max}$, replacing the constrained optimization $\sS^{s_{\max}}_\star$ defined in \eqref{eq:sS-star-numax-bLam} with the $\sS_\star$ defined in \eqref{eq:def-sS-star}.
\begin{ppn}
	\label{ppn:nu-max-infty-limit}
	We have $\lim_{s_{\max} \to \infty} \sup_{\bLam \in \osK_\ast} \sS^{s_{\max}}_\star (\bLam) = \sup_{\bLam \in \osK_\ast} \sS_\star (\bLam)$, and moreover $\sS_\star$ attains its supremum on $\osK_\ast$.
\end{ppn}
\begin{lem}
	\label{lem:sS-star-cts}
	The function $\sS_\star : \sK \times \bbR \to \bbR$ (recall \eqref{eq:def-sS-star-bLam-nu}) is continuous.
\end{lem}
\begin{proof}
	Note that $s \mapsto \fr12 s^2 \psi_0$ is manifestly continuous.
	By concavity of $\cH$, $|\cH(x) - \cH(y)| \le \cH(|x-y|)$ for all $x,y \in [0,1]$.
	By concavity of $x \mapsto \cH(\sqrt{x}/2)$ and Jensen's inequality,
	\baln
		|\ent(\bLam) - \ent(\bLam')|
		&\le \EE \lt| \cH\lt(\fr{1+\bLam(\dbH)}{2}\rt) - \cH\lt(\fr{1+\bLam'(\dbH)}{2}\rt) \rt|
		\le \EE \lt| \cH\lt(\fr{|\bLam(\dbH) - \bLam'(\dbH)|}{2}\rt)\rt| \\
		&\le \cH\lt(\fr{\EE[|\bLam(\dbH) - \bLam'(\dbH)|^2]^{1/2}}{2}\rt)
		= \cH\lt(\fr{\tnorm{\bLam-\bLam'}}{2}\rt).
	\ealn
	Thus $\ent$ is continuous.
	By Cauchy--Schwarz,
	\[
		|\bbE[\dbH \bLam] - \bbE[\dbH \bLam']|
		\le \EE[\dbH^2]^{1/2} \tnorm{\bLam-\bLam'}
		= \psi_0^{1/2} \tnorm{\bLam-\bLam'}
	\]
	and similarly $|\bbE[\bM \bLam] - \bbE[\bM \bLam']| \le q_0^{1/2} \tnorm{\bLam-\bLam'}$.
	Since the denominator $1 - \fr{\bbE[\bM \bLam(\dbH)]^2}{q_0}$ is bounded away from $0$ by Lemma~\ref{lem:denominator-doesnt-explode}, the final term of $\sS_\star$ is continuous by Lemma~\ref{lem:functional-continuity}.
	Thus $\sS_\star$ is continous.
\end{proof}
We will need the following analytical lemma, which is a simple adaptation of Dini's Theorem \cite[Theorem 7.13]{rudin1976principles}.
We provide a proof for completeness.
\begin{lem}
	\label{lem:dini}
	Suppose $f_1,f_2,\ldots : K \to \bbR$ are a decreasing sequence of continuous functions on a compact space $K$.
	Let $f : K \to \bbR \cup \{-\infty\}$ denote their (not necessarily continuous) pointwise limit, which we assume is not $-\infty$ everywhere.
	Then $\lim_{n\to\infty} \sup f_n = \sup f$, and furthermore $f$ attains its supremum.
\end{lem}
\begin{proof}
	Without loss of generality assume $\sup f = 0$.
	For $\iota > 0$, let $E_n = \{x \in K : f_n(x) < \iota \}$.
	Then $E_n$ is open and $E_n \subseteq E_{n+1}$.
	Since the $f_n$ converge pointwise to $f$, $\cup_n E_n = K$.
	By compactness of $K$, $E_n = K$ for some finite $n$, and thus $\sup f_n < \iota$.
	As this holds for any $\iota$, $\lim_{n\to\infty} \sup f_n = 0$.
	Finally, $f$, as the decreasing limit of (upper-semi)continuous functions, is upper-semicontinuous.
	Therefore $f$ attains its supremum.
\end{proof}
To apply Lemma~\ref{lem:dini}, we verify that $\sS_\star$ is not $-\infty$ everywhere by calculating its value at $\bLam_{1,0}(x) = \th(x)$ in Lemma~\ref{lem:value-at-1-0} below.
Recalling \S\ref{subsec:heuristic-1mt}, we expect this to be the maximizer of $\sS_\star$.
\begin{lem}
	\label{lem:nu-strictly-convex}
	For any $\bLam \in \sK$, $s \ge 0$, we have $\fr{\partial^2}{\partial s^2} \sS_\star(\bLam,s) > 0$.
\end{lem}
\begin{proof}
	Since $(\log \Psi)' = -\cE$, we have
	\baln
		\fr{\partial^2}{\partial s^2} \sS_\star(\bLam,s)
		&= \psi_0 - \alpha_\star \EE \lt\{
			\cE' \lt(
				\fr{\kappa
					- \fr{\EE[\bM \bLam(\dbH)]}{q_0}\hbH
					- \fr{\EE[\dbH \bLam(\dbH)]}{\psi_0}\bN
				}{
					\sqrt{1 - \fr{\EE[\bM \bLam(\dbH)]^2}{q_0}}
				}
				+ s \bN
			\rt)\bN^2
		\rt\} \\
		&\stackrel{Lem.~\ref{lem:E-derivative-bds}\ref{itm:cEpr}}{>}
		\psi_0 - \alpha_\star \EE[\bN^2] = 0.
	\ealn
\end{proof}
\begin{lem}
	\label{lem:value-at-1-0}
	We have $\sS_\star(\bLam_{1,0}) = \sS_\star(\bLam_{1,0},\sqrt{1-q_0}) = 0$.
\end{lem}
\begin{proof}
	Let $\bLam = \bLam_{1,0}$.
	Note that $\bLam(\dbH) = \th(\dbH) = \bM$.
    Thus $\EE[\bM \bLam(\dbH)] = q_0$ and, by gaussian integration by parts, $\EE[\dbH \bLam(\dbH)] = (1-q_0) \psi_0$.
    So
    \[
        \fr{
            \kappa
            - \fr{\EE[\bM \bLam(\dbH)]}{q_0} \hbH
            - \fr{\EE[\dbH \bLam(\dbH)]}{\psi_0} \bN
        }{
            \sqrt{1 - \fr{\EE[\bM \bLam(\dbH)]^2}{q_0}}
        }
        + \sqrt{1-q_0} \bN
        = \fr{\kappa - \hbH}{\sqrt{1-q_0}}.
    \]
    By the identity $\cH(\fr{1+\th x}{2}) = \log (2 \ch x) - x \th x$,
    \[
        \EE \cH\lt(\fr{1 + \bLam}{2}\rt)
        = \EE \log (2 \ch \dbH) - \EE [\dbH \bLam]
        = \EE \log (2 \ch \dbH) - (1-q_0) \psi_0.
    \]
    Thus
    \[
        \sS_\star(\bLam,\sqrt{1-q_0})
        = - \fr12 (1-q_0) \psi_0
        + \EE \log (2 \ch \dbH)
        + \alpha \EE \log \Psi \lt(
            \fr{\kappa - \hbH}{\sqrt{1 - q_0}}
        \rt)
        = \sG (\alpha_\star,q_0,\psi_0),
    \]
    which equals $0$ by definition of $\alpha_\star$.
    Furthermore,
    \baln
    	\fr{\partial}{\partial s}
    	\sS_\star(\bLam,s)
    	\big|_{s = \sqrt{1-q_0}}
    	&= \sqrt{1-q_0} \psi_0 - \alpha_\star \EE \lt\{
			\cE \lt(
				\fr{\kappa - \hbH}{\sqrt{1 - q_0}}
			\rt)\bN
		\rt\} \\
		&= \sqrt{1-q_0} \lt(
			\psi_0 - \alpha_\star \EE [\bN^2]
		\rt)
		= 0.
    \ealn
    By Lemma~\ref{lem:nu-strictly-convex}, this implies $s = \sqrt{1-q_0}$ minimizes $\sS_\star(\bLam,s)$, and thus $\sS_\star(\bLam) = \sS_\star(\bLam,\sqrt{1-q_0})$.
\end{proof}

\begin{proof}[Proof of Proposition~\ref{ppn:nu-max-infty-limit}]
	The set $\osK_\ast$ is compact in the topology of $\sL$.
	The functions $\sS^{s_{\max}}_\star : \osK_\ast \to \bbR$ are continuous by Lemma~\ref{lem:sS-star-cts} and compactness of $[0,s_{\max}]$.
	On any sequence of $s_{\max}$ tending to $\infty$, the sequence of $\sS^{s_{\max}}_\star$ is decreasing with pointwise limit $\sS_\star$.
	Since Lemma~\ref{lem:value-at-1-0} implies $\sS_\star$ is not $-\infty$ everywhere, the result follows from Lemma~\ref{lem:dini}.
\end{proof}

\subsection{No boundary maximizers and conclusion}
\label{subsec:no-bdry-max}

The results proved so far imply that the exponential order of $\bbE^{\bm,\bn}_{\eps,\Pl} Z_N(\bG)$ is bounded up to vanishing error by $\sup_{\bLam \in \osK_\ast} \sS_\star(\bLam)$.
Condition~\ref{con:2varfn} provides a bound on $\sup_{\bLam \in \sK_\ast} \sS_\star(\bLam)$.
Since $\sS_\star$ (unlike $\sS_\star^{s_{\max}}$) is not a priori continuous, to complete the proof we verify in the following proposition that it is not maximized on the boundary.
\begin{ppn}
	\label{ppn:no-boundary}
	The maximum of $\sS_\star(\bLam)$ on $\osK_\ast$ (which exists by Proposition~\ref{ppn:nu-max-infty-limit}) is not attained on $\osK_\ast \setminus \sK_\ast$.
\end{ppn}
\begin{lem}
	\label{lem:when-inf-infty}
	Let $d_0 = \alpha_\star \EE[F'_{1-q_0}(q_0^{1/2} Z)]$, and
	\[
		\sO = \lt\{
			\bLam \in \sK :
			d_0 \EE[\bM \bLam(\dbH)] + \EE[\dbH \bLam(\dbH)] > \alpha_\star \kappa
		\rt\}.
	\]
	Then, for $\bLam \in \sK$,
	\[
		\lim_{s\to+\infty}
		\sS_\star(\bLam,s)
		= \begin{cases}
			+\infty & \bLam \in \sO, \\
			-\infty & \bLam \not\in \sO.
		\end{cases}
	\]
\end{lem}
\begin{proof}
	A well-known gaussian tail bound gives $\fr{\varphi(x)}{x} < \Psi(x) < \fr{x\varphi(x)}{1+x^2}$ for all $x > 0$.
	Thus, for large $x$,
	\beq
		\label{eq:logPsi-estimate}
		\log \Psi(x) = -\fr12 x^2 - \log x + O(1).
	\eeq
	Let $s$ be large and define
	\baln
		\xi(x) &= -\fr12 x^2 - \bone\{s^{1/2} \le x \le s^2\} \log x, &
		\bU &= \fr{
	        \kappa
	        - \fr{\EE[\bM \bLam(\dbH)]}{q_0} \hbH
	        - \fr{\EE[\dbH \bLam(\dbH)]}{\psi_0} \bN
	    }{
	        \sqrt{1 - \fr{\EE[\bM \bLam(\dbH)]^2}{q_0}}
	    }, &
	    \bV &= \bU + s \bN.
	\ealn
	Note that
	\baln
		\lt| \EE \log \Psi (\bV) - \EE \xi(\bV) \rt|
		&\le \lt| \EE \bone\{\bV \le \log \log s \} (\log \Psi (\bV) - \xi(\bV)) \rt| \\
		&+ \lt| \EE \bone\{\log \log s \le \bV \le s^{1/2} \} (\log \Psi (\bV) - \xi(\bV)) \rt| \\
		&+ \lt| \EE \bone\{s^{1/2} \le \bV \le s^2 \} (\log \Psi (\bV) - \xi(\bV)) \rt| \\
		&+ \lt| \EE \bone\{\bV \ge s^2 \} (\log \Psi (\bV) - \xi(\bV)) \rt|.
	\ealn
	We will show each of these terms is $o(\log s)$.
	Let $\bV_+ = \max(\bV,0)$, $\bV_- = -\min(\bV,0)$, and let $C > 0$ be a constant varying from line to line.
	Then,
	\baln
		&\lt| \EE \bone\{\bV \le \log \log s \} (\log \Psi (\bV) - \xi(\bV)) \rt| \\
		&\le \EE \bone\{\bV \le \log \log s \} |\log \Psi (\bV)|
		+ \EE \bone\{\bV \le \log \log s\} \bV_+^2
		+ \EE \bV_-^2 \\
		&\le
		C(\log \log s)^2
		+ \EE \bU_-^2
		\le
		C(\log \log s)^2.
	\ealn
	In the last line we used that $\bN > 0$ almost surely, and thus $\bU_- \ge \bV_-$.
	By the estimate \eqref{eq:logPsi-estimate}, if $\log \log s \le \bV < s^{1/2}$, then $|\log \Psi (\bV) - \xi(\bV)| \le C \log s$.
	Thus
	\baln
		\big| \EE \bone\{\log \log s \le \bV < s^{1/2} \} &(\log \Psi (\bV) - \xi(\bV)) \big|
		\le (C \log s) \PP(\bV \le s^{1/2}) \\
		&\le (C \log s) \lt(\PP(\bU \le -s^{1/2}) + \PP(s \bN \le 2s^{1/2})\rt)
		= o(\log s).
	\ealn
	The estimate \eqref{eq:logPsi-estimate} directly implies
	\[
		\lt| \EE \bone\{s^{1/2} \le \bV \le s^2 \} (\log \Psi (\bV) - \xi(\bV)) \rt| = O(1).
	\]
	Finally, Lemma~\ref{lem:E-derivative-bds}\ref{itm:cE-bd} gives $0\le \cE(x) \le |x|+1$.
	Thus
	\[
		|\bV|
		\le |\bU| + \fr{s}{\sqrt{1-q_0}} \cE \lt(
			\fr{\kappa - \hbH}{\sqrt{1-q_0}}
		\rt)
		\le Cs(|\hbH| + 1).
	\]
	It follows that for $t \ge s^2$, we have $\PP(|\bV| \ge t) \le \exp(-t^2/Cs^2)$.
	So, crudely
	\baln
		\lt| \EE \bone\{\bV \ge s^2 \} (\log \Psi (\bV) - \xi(\bV)) \rt|
		&\le C'\EE \bone\{\bV \ge s^2\} \bV^2 \\
		&\le C'\lt(
			s^2 \exp(-s^2/C)
			+ \int_{s^2}^\infty 2t \exp(-t^2/Cs^2) ~\de t
		\rt) \\
		&\le C' s^2 \exp(-s^2/C).
	\ealn
	Thus $\lt| \EE \log \Psi (\bV) - \EE \xi(\bV) \rt| = o(\log s)$.
	So,
	\[
		\sS_\star(\bLam,s)
		= \fr12 s^2 \psi_0
		+ \alpha_\star \EE \xi(\bV) + o(\log s).
	\]
	We now evaluate $\alpha_\star \EE \xi(\bV)$.
	First,
	\baln
		\fr12 \alpha_\star \EE \bV^2
		&= \fr12 \alpha_\star s^2 \EE[\bN^2]
		+ \alpha_\star s \EE[\bU \bN]
		+ O(1) \\
		&= \fr12 s^2 \psi_0
		+ \fr{s \lt(\alpha_\star \kappa - d_0 \EE[\bM \bLam(\dbH)] - \EE[\dbH \bLam(\dbH)]\rt)}{\sqrt{1 - \fr{\EE[\bM \bLam(\dbH)]^2}{q_0}}}
		+ O(1).
	\ealn
	Thus
	\[
		\sS_\star(\bLam,s)
		= \fr{s \lt(d_0 \EE[\bM \bLam(\dbH)] + \EE[\dbH \bLam(\dbH)] - \alpha_\star \kappa \rt)}{\sqrt{1 - \fr{\EE[\bM \bLam(\dbH)]^2}{q_0}}}
		- \EE \bone\{s^{1/2} \le \bV \le s^2\} \log \bV
		+ o(\log s).
	\]
	The logarithmic term clearly has magnitude $O(\log s)$.
	So, $\lim_{s\to+\infty} \sS_\star(\bLam,s) = +\infty$ if $\bLam \in \sO$, and $-\infty$ if $\bLam$ is in the interior of $\sK \setminus \sO$.
	Finally, we have shown above that $\PP(\bV < s^{1/2}), \PP(\bV > s^2) = o_s(1)$, so
	\[
		\EE \bone\{s^{1/2} \le \bV \le s^2\} \log \bV
		\ge \fr12 (1-o_s(1)) \log s.
	\]
	Thus $\lim_{s\to+\infty} \sS_\star(\bLam,s) = -\infty$ for $\bLam$ on the boundary of $\sK \setminus \sO$.
\end{proof}
\begin{proof}[Proof of Proposition~\ref{ppn:no-boundary}]
	Suppose for contradiction that $\bLam \in \osK_\ast \setminus \sK_\ast$ maximizes $\sS_\star(\bLam)$ in $\osK_\ast$.
	By Proposition~\ref{ppn:planted-mt-2param-reduction}, $\bLam$ is also a maximizer of $\sS_\star(\bLam)$ in $\sK$.

	By Lemma~\ref{lem:when-inf-infty}, if $\bLam \not\in \sO$, then $\sS_\star(\bLam) = -\infty$ is not a maximizer (recall Lemma~\ref{lem:value-at-1-0}).
	Thus $\bLam \in \sO$.
	Let $\bLam^t = (1-t)\bLam$.
	Since $\sO$ is open, $\bLam^t \in \sO$ for $t \in [0,t_+)$, for sufficiently small $t_+$.

	By Lemma~\ref{lem:when-inf-infty}, for $t \in [0,t_+)$, the infimum of $\sS(\bLam^t,s)$ is attained at some $s(\bLam^t) \in [0,+\infty)$.
	Note that
	\[
		\fr{\partial}{\partial s}
    	\sS_\star(\bLam^t,s)
    	\Bigg|_{s = 0}
     	= - \alpha_\star \EE \lt\{
			\cE \lt(
				\fr{\kappa
					- \fr{\EE[\bM \bLam^t(\dbH)]}{q_0}\hbH
					- \fr{\EE[\dbH \bLam^t(\dbH)]}{\psi_0}\bN
				}{
					\sqrt{1 - \fr{\EE[\bM \bLam^t(\dbH)]^2}{q_0}}
				}
				+ s \bN
			\rt)\bN
		\rt\}
		< 0
	\]
	because $\bN > 0$ almost surely and the image of $\cE$ is positive.
	Combined with Lemma~\ref{lem:nu-strictly-convex}, this implies $s(\bLam^t)$ is the unique solution to $\fr{\partial}{\partial s} \sS_\star(\bLam,s) = 0$, and $s(\bLam^t) > 0$.

	Note that $\fr{\partial}{\partial s} \sS_\star(\bLam^t,s)$ is differentiable in $t$, as the denominator $\sqrt{1 - \fr{\EE[\bM \bLam^t(\dbH)]^2}{q_0}}$ is bounded away from $0$ by Lemma~\ref{lem:denominator-doesnt-explode}.
	By Lemma~\ref{lem:nu-strictly-convex} and the implicit function theorem, $s(\bLam^t)$ is differentiable in $t$ for all $t\in [0,t_+)$.
	It follows that
	\[
		\fr{\de}{\de t} \lt\{
			\fr12 s(\bLam^t)^2 \psi_0
			+ \alpha_\star \EE \log \Psi \lt(
		        \fr{
		            \kappa
		            - \fr{\EE[\bM \bLam^t(\dbH)]}{q_0} \hbH
		            - \fr{\EE[\dbH \bLam^t(\dbH)]}{\psi_0} \bN
		        }{
		            \sqrt{1 - \fr{\EE[\bM \bLam^t(\dbH)]^2}{q_0}}
		        }
		        + s(\bLam^t) \bN
		    \rt)
		\rt\} \Bigg|_{t=0}
	\]
	exists and is finite.
	However, since $\bLam \in \osK_\ast \setminus \sK_\ast$, we have $\bLam(\dbH) \in \{-1,1\}$ $\dbH$-almost surely.
	Thus
	\[
		\fr{\de}{\de t} \ent(\bLam^t) \big|_{t=0}
		= \fr{\de}{\de t} \cH(t/2) \big|_{t=0} = +\infty.
	\]
	Hence $\fr{\de}{\de t} \sS_\star(\bLam^t) \big|_{t=0} = +\infty$, and $\bLam$ is not a maximizer of $\sS_\star(\bLam)$ in $\sK$.
\end{proof}

\begin{proof}[Proof of Proposition~\ref{ppn:planted-mt}]
	By Propositions~\ref{ppn:functional-optimization}, \ref{ppn:planted-mt-2param-reduction}, for any $s_{\max} > 0$,
	\beq
		\label{eq:planted-mt-final-bd}
		\fr1N \log \bbE^{\bm,\bn}_{\eps,\Pl} [Z_N(\bG)]
		\le \sup_{\bLam \in \sK} \sS^{s_{\max}}_\star(\bLam) + o_{\eps,\ups}(1)
		= \sup_{\bLam \in \osK_\ast} \sS^{s_{\max}}_\star(\bLam) + o_{\eps,\ups}(1).
	\eeq
	By Propositions~\ref{ppn:nu-max-infty-limit} and \ref{ppn:no-boundary} and Condition~\ref{con:2varfn},
	\[
		\lim_{s_{\max} \to \infty} \sup_{\bLam \in \osK_\ast} \sS^{s_{\max}}_\star(\bLam)
		= \sup_{\bLam \in \osK_\ast} \sS_\star(\bLam)
		= \sup_{\bLam \in \sK_\ast} \sS_\star(\bLam)
		= \sup_{\lambda_1,\lambda_2 \in \bbR} \sS_\star(\lambda_1,\lambda_2)
		\le 0.
	\]
	Thus, taking the limit $\eps,\ups \to 0$ followed by $s_{\max} \to \infty$ in \eqref{eq:planted-mt-final-bd} implies the result.
\end{proof}

\subsection{Local analysis of first moment functional at $(1,0)$}
\label{subsec:first-mt-functional-local}

We now prove Lemma~\ref{lem:sS-zero}.
Note that part \ref{itm:sS-zero-attain-sup} follows from Proposition~\ref{ppn:no-boundary}, and part \ref{itm:sS-zero-val} was already proved in Lemma~\ref{lem:value-at-1-0}.
We turn to the proofs of the remaining parts.
\begin{proof}[Proof of Lemma~\ref{lem:sS-zero}\ref{itm:sS-zero-1deriv}]
    Let $\sS_\star(\lambda_1,\lambda_2,s) = \sS_\star(\bLam_{\lambda_1,\lambda_2},s)$, and let $s(\lambda_1,\lambda_2)$ minimize $\sS_\star(\lambda_1,\lambda_2,s)$.
    Lemma~\ref{lem:value-at-1-0} shows $s(1,0) = \sqrt{1-q_0}$, and the proof of Proposition~\ref{ppn:no-boundary} shows that for $(\lambda_1,\lambda_2)$ in a neighborhood of $(1,0)$, $s(\lambda_1,\lambda_2)$ is the unique solution to $\partial_s \sS_\star(\lambda_1,\lambda_2,s) = 0$.
    By Lemma~\ref{lem:nu-strictly-convex} and the implicit function theorem, $s(\lambda_1,\lambda_2)$ is differentiable in this neighborhood.
    So,
    \balnn
        \notag
        \nabla \sS_\star(\lambda_1,\lambda_2)
        &= \nabla_{\lambda_1,\lambda_2} \sS_\star(\lambda_1,\lambda_2,s(\lambda_1,\lambda_2))
        + \partial_s \sS_\star(\lambda_1,\lambda_2,s(\lambda_1,\lambda_2)) \nabla s(\lambda_1,\lambda_2) \\
        \label{eq:sS-star-deriv1-formula}
        &= \nabla_{\lambda_1,\lambda_2} \sS_\star(\lambda_1,\lambda_2,s(\lambda_1,\lambda_2)),
    \ealnn
    and in particular $\nabla \sS_\star(1,0) = \nabla \osS_\star(1,0)$.
    To calculate the latter gradient, let $u_1,u_2 \in \bbR$ be arbitrary and
    \[
        \bDel
        \equiv (u_1 \partial_{\lambda_1} + u_2 \partial_{\lambda_2}) \bLam
        = (1 - \bLam^2)(u_1 \dbH + u_2 \bM).
    \]
    Then
    \balnn
        \label{eq:osS-1deriv}
        \la \nabla \osS_\star(\lambda_1,\lambda_2), (u_1,u_2) \ra
        &= - \EE [\th^{-1}(\bLam) \bDel]
        - \alpha_\star \EE \Bigg\{
            \cE\lt(\fr{
                \kappa - \fr{\EE[\bM\bLam]}{q_0} \hbH - \fr{\EE[\dbH\bLam]}{\psi_0} \bN
            }{
                \sqrt{1-\fr{\EE[\bM\bLam]^2}{q_0}}
            } + \sqrt{1-q_0} \bN \rt) \\
            \notag
            &\times
            \lt(
                \fr{
                    - \fr{\EE[\bM\bDel]}{q_0} \hbH
                    - \fr{\EE[\dbH\bDel]}{\psi_0} \bN
                }{\sqrt{1-\fr{\EE[\bM\bLam]^2}{q_0}}}
                + \fr{
                    \kappa - \fr{\EE[\bM\bLam]}{q_0} \hbH - \fr{\EE[\dbH\bLam]}{\psi_0} \bN
                }{\lt(1-\fr{\EE[\bM\bLam]^2}{q_0}\rt)^{3/2}}
                \cdot \fr{\EE[\bM\bLam]\EE[\bM\bDel]}{q_0}
            \rt)
        \Bigg\}.
    \ealnn
    Specializing to $(\lambda_1,\lambda_2) = (1,0)$,
    \baln
        &\la \nabla \osS_\star(1,0), (u_1,u_2) \ra \\
        &= - \EE [\th^{-1}(\bM) \bDel]
        - \alpha_\star \EE \lt\{
            \cE\lt(\fr{\kappa - \hbH}{\sqrt{1-q_0}}\rt)
            \lt(
                \fr{
                    - \fr{\EE[\bM\bDel]}{q_0} \hbH
                    - \fr{\EE[\dbH\bDel]}{\psi_0} \bN
                }{\sqrt{1-q_0}}
                + \fr{
                    \kappa - \hbH - (1-q_0)\bN
                }{(1-q_0)^{3/2}} \EE[\bM\bDel]
            \rt)
        \rt\} \\
        &= - \EE [\dbH \bDel]
        - \alpha_\star \EE \lt\{
            F_{1-q_0}(\hbH)
            \lt(
                - \fr{\EE[\bM\bDel]}{q_0} \hbH
                - \fr{\EE[\dbH\bDel]}{\psi_0} \bN
                + \fr{
                    \kappa - \hbH - (1-q_0)\bN
                }{1-q_0} \EE[\bM\bDel]
            \rt)
        \rt\} \\
        &= - \EE[\dbH \bDel]
        + \fr{\alpha_\star \EE[\bN^2]}{\psi_0} \EE[\dbH \bDel]
        + \alpha_\star \lt(
            \fr{\EE[\bN \hbH]}{q_0}
            + \EE \lt[\bN \lt(\bN - \fr{\kappa - \hbH}{1-q_0}\rt)\rt]
        \rt) \EE[\bM \bDel].
    \ealn
    The first two terms cancel because $\alpha_\star \EE[\bN^2] = \psi_0$.
    Finally, note the identity
    \[
        F'_{1-q_0}(x) = - F_{1-q_0}(x)\lt(F_{1-q_0}(x) - \fr{x}{1-q_0}\rt).
    \]
    By gaussian integration by parts,
    \[
        \EE[\bN\hbH]
        = \EE[\hbH F_{1-q_0}(\hbH)]
        = \EE[\hbH^2] \EE[F'_{1-q_0}(\hbH)]
        = - q_0 \EE \lt[\bN \lt(\bN - \fr{\kappa - \hbH}{1-q_0}\rt)\rt].
    \]
    It follows that $\la \nabla \sS_\star(1,0), (u_1,u_2) \ra = 0$.
    Since $u_1,u_2$ were arbitrary, $\nabla \sS_\star(1,0) = 0$.
\end{proof}
\begin{proof}[Proof of Lemma~\ref{lem:sS-zero}\ref{itm:sS-zero-2deriv-bd}]
    Differentiating \eqref{eq:sS-star-deriv1-formula} and applying the implicit function theorem yields
    \baln
        \nabla^2 \sS_\star(\lambda_1,\lambda_2)
        &= \nabla^2_{\lambda_1,\lambda_2} \sS_\star(\lambda_1,\lambda_2,s(\lambda_1,\lambda_2))
        + \nabla_{\lambda_1,\lambda_2} \partial_s \sS_\star(\lambda_1,\lambda_2,s(\lambda_1,\lambda_2)) (\nabla s(\lambda_1,\lambda_2))^\top \\
        &= \nabla^2_{\lambda_1,\lambda_2} \sS_\star(\lambda_1,\lambda_2,s(\lambda_1,\lambda_2))
        - \fr{(\nabla_{\lambda_1,\lambda_2} \partial_s \sS_\star(\lambda_1,\lambda_2,s(\lambda_1,\lambda_2)))^{\otimes 2}} {\partial_s^2 \sS_\star(\lambda_1,\lambda_2,s(\lambda_1,\lambda_2))} \\
        &\preceq \nabla^2_{\lambda_1,\lambda_2} \sS_\star(\lambda_1,\lambda_2,s(\lambda_1,\lambda_2)).
    \ealn
    Specializing to $(\lambda_1,\lambda_2) = (1,0)$ yields the result.
\end{proof}

\bibliographystyle{alpha}
\bibliography{bib}

\appendix

\section{Deferred proofs}
\label{app:amp}

In this appendix, we provide proofs of various results deferred from the paper.

\subsection{Well definedness and $\eps \downarrow 0$ limit of $(q_\eps,\psi_\eps,\varrho_\eps)$}

\begin{proof}[Proof of Proposition~\ref{ppn:eps-perturb-fixed-point}]
    Let $\iota_0$ be small enough that $[q_0 - 3\iota_0, q_0 + 3\iota_0] \subseteq [0,1]$.
    Note that $\zeta_0(\psi) = (R_{\alpha_\star} \circ P)(\psi)$.
    By Condition~\ref{con:km-well-defd}, $\zeta_0(\psi_0) = \psi_0$ and
    \[
        \zeta'_0(\psi_0) = R'_{\alpha_\star}(q_0) P'(\psi_0) = (P \circ R_{\alpha_\star})'(q_0) < 1.
    \]
    By continuity of $\zeta_0$ and $\zeta'_0$, we can find $\iota > 0$ such that for all $\psi \in [\psi_0 - \iota, \psi_0 + \iota]$, $P(\psi) \in [q_0 - \iota_0, q_0 + \iota_0]$ and $\zeta'_0(\psi) < 1$.
    Set $\iota_1$ small enough that
    \baln
        \zeta_0(\psi_0 - \iota) &\ge \psi_0 - \iota + 2\iota_1, &
        \zeta_0(\psi_0 + \iota) &\le \psi_0 + \iota - 2\iota_1, &
        \sup_{\psi \in [\psi_0 - \iota, \psi_0 + \iota]}
        \zeta'_0(\psi) \le 1-2\iota_1.
    \ealn
    We will show that for sufficiently small $\eps$,
    \beq
        \label{eq:perturb-fixed-point-goal}
        \sup_{\psi \in [\psi_0 - \iota, \psi_0 + \iota]}
        |\zeta_\eps(\psi) - \zeta_0(\psi)|,
        \sup_{\psi \in [\psi_0 - \iota, \psi_0 + \iota]}
        |\zeta'_\eps(\psi) - \zeta'_0(\psi)|
        = o_\eps(1).
    \eeq
    We first explain why this implies the result.
    First, \eqref{eq:perturb-fixed-point-goal} implies that for sufficiently small $\eps$,
    \baln
        \zeta_\eps(\psi_0 - \iota) &\ge \psi_0 - \iota + \iota_1, &
        \zeta_\eps(\psi_0 + \iota) &\le \psi_0 + \iota - \iota_1, &
        \sup_{\psi \in [\psi_0 - \iota, \psi_0 + \iota]}
        \zeta'_\eps(\psi) &\le 1-\iota_1.
    \ealn
    This implies that $\zeta_\eps$ has a unique fixed point $\psi_\eps$ in $[\psi_0 - \iota, \psi_0 + \iota]$.
    Furthermore, it implies $|\zeta_\eps(\psi_0) - \psi_0| = o_\eps(1)$, which combined with the above derivative estimate gives
    \[
        |\psi_\eps - \psi_0| \le |\zeta_\eps(\psi_0) - \psi_0| / \iota_1 = o_\eps(1).
    \]
    Continuity considerations then imply $(q_\eps,\psi_\eps,\varrho_\eps) \to (q_0,\psi_0,1-q_0)$ as $\eps \downarrow 0$.
    We now turn to the proof of \eqref{eq:perturb-fixed-point-goal}.
    Let $\psi \in [\psi_0 - \iota, \psi_0 + \iota]$.
    Below, $o_\eps(1)$ is an error uniform over $\psi$.
    Let $q = P^\eps(\psi)$ and $\tq = P(\psi)$.
    Note that
    \[
        |q - \tq|
        \le \EE \lt[|
            (\th((\psi + \eps)^{1/2} Z) + \eps (\psi + \eps)^{1/2} Z)^2
            - \th^2(\psi^{1/2} Z)
        |\rt]
        \le o_\eps(1).
    \]
    Let $\varrho = \varrho_\eps(q,\psi)$, and note that
    \[
        |\varrho - (1-q)| = o_\eps(1).
    \]
    Thus
    \[
        \varrho
        \ge (1-\tq) - |\tq - q| - |\varrho - (1-q)|
        \ge 2\iota_0 - o_\eps(1)
        \ge \iota_0,
    \]
    so $\varrho$ is bounded away from $0$.
    By Cauchy-Schwarz,
    \baln
        &|\zeta_\eps(\psi) - \zeta_0(\psi)|
        = |R^\eps(q,\psi) - R_{\alpha_\star}(\tq)| \\
        &= \alpha_\star \EE \lt[
            |F_{\eps,\varrho}((q+\eps)^{1/2} Z)
            - F_{1-q_0}(q^{1/2} Z)|
            |F_{\eps,\varrho}((q+\eps)^{1/2} Z)
            + F_{1-q_0}(q^{1/2} Z)|
        \rt] \\
        &\le
        \alpha_\star \EE \lt[
            (F_{\eps,\varrho}((q+\eps)^{1/2} Z)
            - F_{1-\tq}(\tq^{1/2} Z))^2
        \rt]^{1/2}
        \EE \lt[
            (F_{\eps,\varrho}((q+\eps)^{1/2} Z)
            + F_{1-\tq}(\tq^{1/2} Z))^2
        \rt]^{1/2}.
    \ealn
    Expanding $F_{\eps,\varrho}$ using \eqref{eq:F-explicit} shows the first expectation is $o_\eps(1)$, while the second is bounded by Lemma~\ref{lem:E-derivative-bds}\ref{itm:cE-bd}.
    Thus $|\zeta_\eps(\psi) - \zeta_0(\psi)| = o_\eps(1)$ uniformly in $\psi \in [\psi_0 - \iota, \psi_0 + \iota]$.
    Furthermore,
    \baln
        \zeta'_\eps(\psi)
        &= \fr{\partial R^\eps}{\partial q}(q,\psi) (P^\eps)'(\psi)
        + \fr{\partial R^\eps}{\partial \psi}(q,\psi), &
        \zeta'_0(\psi)
        &= R'_{\alpha_\star}(\tq) P'(\psi).
    \ealn
    Similar computations to above show
    \[
        \lt|\fr{\partial R^\eps}{\partial q}(q,\psi) - R'_{\alpha_\star}(\tq)\rt|,
        |(P^\eps)'(\psi) - P'(\psi)|,
        \lt|\fr{\partial R^\eps}{\partial \psi}(q,\psi)\rt|
        = o_\eps(1),
    \]
    and thus $|\zeta'_\eps(\psi) - \zeta'_0(\psi)| = o_\eps(1)$ uniformly in $\psi$.
    This proves \eqref{eq:perturb-fixed-point-goal}.
\end{proof}

\subsection{Approximation for (pseudo)-Lipschitz functions}

\begin{proof}[Proof of Fact~\ref{fac:pseudo-lipschitz}]
    Let $(x,y)$ be a sample from the optimal coupling of $(\mu,\mu')$.
    Then
    \baln
        |\bbE_\mu[f] - \bbE_{\mu'}[f]|
        &\le \EE |f(x)-f(y)|
        \le L \EE \lt[|x-y|(|x|+|y|+1)\rt] \\
        &\le L \EE[|x-y|^2]^{1/2} \EE[3(|x|^2+|y|^2+1)]^{1/2} \\
        &\le L \EE[|x-y|^2]^{1/2} \EE[3(3|x|^2+2|x-y|^2+1)]^{1/2} \\
        &\le 3L \bbW_2(\mu,\mu') (\mu_2 + \bbW_2(\mu,\mu') + 1),
    \ealn
    where we have used the estimate $|y|^2 \le 2|x|^2 + 2|x-y|^2$.
\end{proof}

\begin{proof}[Proof of Fact~\ref{fac:lipschitz-3prod-bound}]
    Couple $(x,y,z)\sim \mu$ and $(x',y',z') \sim \mu'$ in the $\bbW_2$-optimal way.
    Then, the left-hand side of \eqref{eq:lipschitz-3prod-bound} is bounded by the sum of:
    \baln
        \bbE |f_1(x)| |f_2(y)| |f_3(z)-f_3(z')| &\le L (\bbE f_1(x)^4)^{1/4} (\bbE f_2(y)^4)^{1/4} (\bbE |z-z'|^2)^{1/2} \\
        &\le L (\bbE f_1(x)^4)^{1/4} (\bbE f_2(y)^4)^{1/4} \bbW_2(\mu,\mu'), \\
        \bbE |f_1(x)| |f_3(z')| |f_2(y)-f_2(y')| &\le L^2 (\bbE f_1(x)^2)^{1/2} (\bbE |y-y'|^2)^{1/2}
        \le L^2 (\bbE f_1(x)^2)^{1/2} \bbW_2(\mu,\mu') \\
        \bbE |f_2(y')| |f_3(z')| |f_1(x)-f_1(x')| &\le L^2 (\bbE f_2(y')^2)^{1/2} (\bbE |x-x'|^2)^{1/2}
        \le L^2 (\bbE f_2(y')^2)^{1/2} \bbW_2(\mu,\mu').
    \ealn
    Finally, by Fact~\ref{fac:pseudo-lipschitz},
    \[
        \bbE f_2(y')^2
        \le \bbE f_2(y)^2 + 3\bbW_2(\mu,\mu') (\bbE f_2(y)^2 + \bbW_2(\mu,\mu') + 1).
    \]
    Combining gives the conclusion.
\end{proof}

\subsection{Gradient and Hessian formulas for $\cF^\eps_\TAP$, and regularity estimates}

\begin{proof}[Proof of Lemma~\ref{lem:tap-1deriv}]
    By standard properties of convex duals,
    \[
        (V^\ast_\eps)'(m)
        = -\argmin_{\dh} \lt\{
            - m \dh + V_\eps(\dh)
        \rt\}
        = -\th_\eps^{-1}(m).
    \]
    We differentiate the interaction term in $\cF^\eps_\TAP$ by gaussian integration by parts.
    For each $i\in [N]$, $a\in [M]$,
    \baln
        &\fr{\partial}{\partial m_i}
        \oF_{\eps,\rho_\eps(q(\bm))}
        \lt(
            \fr{\la \bg^a, \bm \ra}{\sqrt{N}}
            + \eps^{1/2} \hg_a
            - \rho_\eps(q(\bm)) n_a
        \rt) \\
        &= \fr{\partial}{\partial m_i} \log \EE \chi_\eps\lt(
            \fr{\la \bg^a, \bm \ra}{\sqrt{N}}
            + \eps^{1/2} \hg_a
            - \rho_\eps(q(\bm)) n_a
            + \rho_\eps(q(\bm))^{1/2} Z
        \rt) \\
        &= \fr{
            \EE \chi'_\eps\lt(
                \fr{\la \bg^a, \bm \ra}{\sqrt{N}}
                + \eps^{1/2} \hg_a
                - \rho_\eps(q(\bm)) n_a
                + \rho_\eps(q(\bm))^{1/2} Z
            \rt) \lt(
                \fr{g^a_i}{\sqrt{N}}
                - \rho'_\eps(q(\bm)) \fr{2m_in_a}{N}
                + \fr{\rho'_\eps(q(\bm))}{\rho_\eps(q(\bm))^{1/2}} \fr{m_i}{N} Z
            \rt)
        }{
            \EE \chi'_\eps\lt(
                \fr{\la \bg^a, \bm \ra}{\sqrt{N}}
                + \eps^{1/2} \hg_a
                - \rho_\eps(q(\bm)) n_a
                + \rho_\eps(q(\bm))^{1/2} Z
            \rt)
        } \\
        &= F_{\eps,\rho_\eps(q(\bm))} (\ah_a) \lt(
            \fr{g^a_i}{\sqrt{N}}
            - \rho'_\eps(q(\bm)) \fr{2m_in_a}{N}
        \rt)
        + \fr{\EE \chi''_\eps(\ah_a + \rho_\eps(q(\bm))^{1/2} Z)}{\EE \chi_\eps(\ah_a + \rho_\eps(q(\bm))^{1/2} Z)}
        \cdot \fr{\rho'_\eps(q(\bm)) m_i}{N}. \\
        &= \fr{g^a_i}{\sqrt{N}} F_{\eps,\rho_\eps(q(\bm))} (\ah_a)
        + \fr{\rho'_\eps(q(\bm)) m_i}{N} \lt(
            -2 F_{\eps,\rho_\eps(q(\bm))} (\ah_a) n_a
            + F_{\eps,\rho_\eps(q(\bm))} (\ah_a)^2
            + F'_{\eps,\rho_\eps(q(\bm))} (\ah_a)
        \rt).
    \ealn
    Thus
    \baln
        \fr{\partial}{\partial m_i} \cF^\eps_\TAP(\bm,\bn)
        &= -\th^{-1}_\eps(m_i)
        + \eps^{1/2} \dg_i
        + \fr{(\bG^\top F_{\eps,\rho_\eps(q(\bm))} (\ah_a))_i}{\sqrt{N}} \\
        &+ \fr{\rho'_\eps(q(\bm)) m_i}{N} \sum_{a=1}^M \lt(
            (n_a - F_{\eps,\rho_\eps(q(\bm))} (\ah_a))^2
            + F'_{\eps,\rho_\eps(q(\bm))} (\ah_a)
        \rt),
    \ealn
    which implies \eqref{eq:tap-deriv-m}.
    The formula \eqref{eq:tap-deriv-n} follows by directly differentiating $\cF^\eps_\TAP$.
    Setting \eqref{eq:tap-deriv-n} to zero shows that $\nabla_\bn \cF^\eps_\TAP(\bm,\bn) = 0$ if and only if $\abh = \hbh$, which rearranges to \eqref{eq:TAP-stationarity-m}.
    This implies $F_{\eps,\rho_\eps(q(\bm))}(\abh) = \bn$, so setting \eqref{eq:tap-deriv-m} to zero yields \eqref{eq:TAP-stationarity-n}.
\end{proof}

\begin{proof}[Proof of Fact~\ref{fac:tap-2deriv}]
    Note that
    \[
        \fr{\partial}{\partial m_i}
        \th^{-1}_\eps(m_i) =
        \fr{1}{\th'_\eps(\dh_i)}
        = \fr{1}{1 + \eps - \th^2(\dh_i)}
        = \fr{\ch^2 (\dh_i)}{1 + \eps \ch^2(\dh_i)}.
    \]
    The functions $F_{\eps,\varrho}, F'_{\eps,\varrho}$ can be differentated in $\varrho$ as follows.
    By gaussian integration by parts (or It\^o's formula),
    \[
        \fr{\de}{\de \varrho} \EE \chi_\eps (x + \varrho^{1/2} Z)
        = \fr12 \EE \chi''_\eps (x + \varrho^{1/2} Z),
    \]
    and similarly for $\chi'_\eps$.
    Thus, abbreviating $\chi_{\eps,\varrho}(x) = \EE \chi_\eps(x + \varrho^{1/2} Z)$,
    \[
        \fr{\de}{\de \varrho} F_{\eps,\varrho}(x)
        = \fr{\de}{\de \varrho} \fr{\chi_{\eps,\varrho}(x)}{\chi'_{\eps,\varrho}(x)}
        = \fr12 \lt(
            \fr{\chi^{(3)}_{\eps,\varrho}(x)}{\chi_{\eps,\varrho}(x)}
            - \fr{\chi'_{\eps,\varrho}(x)\chi''_{\eps,\varrho}(x)}{\chi_{\eps,\varrho}(x)^2}
        \rt).
    \]
    We also have
    \baln
        F'_{\eps,\varrho}(x) &= \fr{\chi''_{\eps,\varrho}(x)}{\chi_{\eps,\varrho}(x)}
        - \fr{(\chi'_{\eps,\varrho}(x))^2}{\chi_{\eps,\varrho}(x)^2}, &
        F''_{\eps,\varrho}(x) &= \fr{\chi'''_{\eps,\varrho}(x)}{\chi_{\eps,\varrho}(x)}
        - \fr{3(\chi'_{\eps,\varrho}(x))(\chi''_{\eps,\varrho}(x))}{\chi_{\eps,\varrho}(x)^2}
        + \fr{2(\chi'_{\eps,\varrho}(x))^3}{\chi_{\eps,\varrho}(x)^3}.
    \ealn
    Thus
    \[
        \fr{\de}{\de \varrho} F_{\eps,\varrho}(x) = \fr12 \lt(
            2F_{\eps,\varrho}(x)F'_{\eps,\varrho}(x)
            + F''_{\eps,\varrho}(x)
        \rt).
    \]
    A similar calculation shows
    \[
        \fr{\de}{\de \varrho} F'_{\eps,\varrho}(x) = \fr12 \lt(
            2F_{\eps,\varrho}(x)F''_{\eps,\varrho}(x)
            + 2F'_{\eps,\varrho}(x)^2
            + F^{(3)}_{\eps,\varrho}(x)
        \rt).
    \]
    The result follows by directly differentiating \eqref{eq:tap-deriv-m} and \eqref{eq:tap-deriv-n} using the above formulas.
\end{proof}

\begin{proof}[Proof of Lemma~\ref{lem:hessian-crude-estimates}]
    As $(\bm,\bn) \in \cS_{\eps,r_0}$, approximation arguments identical to the proof of Corollary~\ref{cor:conditional-law-correct-profile} show the estimates for $q(\bm), \psi(\bn), d_\eps(\bm,\bn)$ in part \ref{itm:hessian-estimates-approx}.
    The regularity estimate \eqref{eq:rho-tau-regularity} of $\rho_\eps$ and its derivatives proves the rest of part \ref{itm:hessian-estimates-approx}.
    Differentiating \eqref{eq:F-explicit} yields
    \[
        F'_{\eps,\varrho}(x)
        = -\fr{\eps}{1+\eps \varrho}
        - \fr{1}{(\varrho + \eps(1 + \eps\varrho))(1+\eps\varrho)}
        \cE'\lt(
            \fr{\kappa (1 + \eps \varrho) - x}{\sqrt{(\varrho + \eps(1 + \eps\varrho))(1+\eps\varrho)}},
        \rt).
    \]
    By Lemma~\ref{lem:E-derivative-bds}, we see that for $\varrho$ in a neighborhood of $\varrho_\eps$, $\sup_{x\in \bbR} \lt|\fr{\de}{\de \varrho} F'_{\eps,\varrho}(x)\rt|$ is bounded by an absolute constant.
    Note that
    \beq
        \label{eq:uniform-diff-in-varrho}
        \sup_{x\in \bbR} \lt|\fr{\de}{\de \varrho} \fr{F'_{\eps,\varrho}(x)}{1+\varrho F'_{\eps,\varrho}(x)}\rt|
        \le \sup_{x\in \bbR} \lt|\fr{F'_{\eps,\varrho}(x)}{(1+\varrho F'_{\eps,\varrho}(x))^2}\rt|
        + \sup_{x\in \bbR} \lt|\fr{1}{(1+\varrho F'_{\eps,\varrho}(x))^2}\rt|
        \cdot \sup_{x\in \bbR} \lt|\fr{\de}{\de \varrho} F'_{\eps,\varrho}(x)\rt|.
    \eeq
    By \eqref{eq:d2n-bound},
    \[
        \fr{1}{1+\varrho F'_{\eps,\varrho}(x)}
        \ge \fr{\varrho + \eps(1+\eps\varrho)}{\eps},
    \]
    which for $\varrho$ in a neighborhood of $\varrho_\eps$ is bounded depending only on $\eps$.
    It follows that \eqref{eq:uniform-diff-in-varrho} is is bounded depending only on $\eps$.
    So,
    \[
        \tnorm{\bD_2 - \tbD_2}_\op
        \le \lt|
            \fr{F'_{\eps,\varrho_\eps}(x)}{1+\varrho_\eps F'_{\eps,\varrho_\eps}(x)}
            - \fr{F'_{\eps,\rho_\eps(q(\bm))}(x)}{1+\rho_\eps(q(\bm)) F'_{\eps,\rho_\eps(q(\bm))}(x)}
        \rt|
        = o_{r_0}(1).
    \]
    This proves part \ref{itm:hessian-estimates-approx-bD2}.
    Part \ref{itm:hessian-estimates-approx-mm} follows from Fact~\ref{fac:Fp-bounded}, as (for $\rho_\eps(q(\bm))$ in a neighborhood of $\varrho_\eps > 0$) the images of $F'_{\eps,\rho_\eps(q(\bm))}$ and $F^{(3)}_{\eps,\rho_\eps(q(\bm))}$ are bounded.
    Similarly,
    \[
        \fr{1}{\sqrt{N}} \tnorm{\bD_4^{-1} F''(\abh)}
        \le \tnorm{\bD_4^{-1}}_\op \tnorm{F''(\abh)}_\infty
        \stackrel{\eqref{eq:d2n-bound}}{\le} \fr{\rho_\eps(q(\bm)) + \eps(1+\eps \rho_\eps(q(\bm)))}{\eps} \tnorm{F''(\abh)}_\infty.
    \]
    Since the image of $F''_{\eps,\rho_\eps(q(\bm))}$ is bounded by Fact~\ref{fac:Fp-bounded}, this proves part \ref{itm:hessian-estimates-approx-cross}.
\end{proof}

\begin{proof}[Proof of Proposition~\ref{ppn:regularity}]
    We will show that the matrices $\nabla^2_{\bm,\bm} \cF^\eps_\TAP$, $\nabla^2_{\bm,\bn} \cF^\eps_\TAP$, $\nabla^2_{\bn,\bn} \cF^\eps_\TAP$ in Fact~\ref{fac:tap-2deriv} have bounded operator norm (with bound depending on $\eps,C_\cvx,C_\bd,D$).
    Throughout this proof, $C$ is a constant depending on $\eps,C_\cvx,C_\bd,D$, which may change from line to line.

    Under $\bbP$, we have $\tnorm{\bG}_\op, \tnorm{\hbg} \le C\sqrt{N}$ with high probability.
    Under $\bbP^{\bm',\bn'}_{\eps,\Pl}$, we may write $\bG = \bbE^{\bm',\bn'}_{\eps,\Pl} \bG + \tbG$ for $\tbG$ as in Lemma~\ref{lem:conditional-law}.
    Then $\tnorm{\tbG}_\op \le C\sqrt{N}$ with high probability, and by Lemma~\ref{lem:conditional-law}, $\tnorm{\bbE^{\bm',\bn'}_{\eps,\Pl} \bG} \le C\sqrt{N}$.
    On this event, $\tnorm{\bG}_\op \le C\sqrt{N}$.
    Since $\rho_\eps(q(\bm')) \in [C_\bd^{-1}, C_\bd]$, $\hbh' = F^{-1}_{\eps,\rho_\eps(q(\bm'))}(\bn)$ satisfies $\tnorm{\hbh'} \le C\sqrt{N}$.
    Then, \eqref{eq:TAP-stationarity-n} implies $\tnorm{\hbg} \le C\sqrt{N}$.
    So, under both $\bbP$ and $\bbP^{\bm',\bn'}_{\eps,\Pl}$, we have $\tnorm{\bG}_\op, \tnorm{\hbg} \le C\sqrt{N}$ with high probability.
    For the remainder of this proof, we assume this event holds.

    Consider any $\tnorm{\bm}^2, \tnorm{\bn}^2 \le DN$.
    The above bounds on $\tnorm{\bG}_\op, \tnorm{\hbg}$ imply $\tnorm{\abh} \le C\sqrt{N}$.
    By \eqref{eq:rho-tau-regularity}, $C_\bd^{-1} \le \rho_\eps(q(\bm)) \le C_\bd$ and $|\rho'_\eps(q(\bm))|, |\rho''_\eps(q(\bm))| \le C_\bd$.
    Abbreviate $F = F_{\eps,\rho_\eps(q(\bm))}$ as above.
    By Fact~\ref{fac:Fp-bounded},
    \beq
        \label{eq:regularity-F-derivatives}
        \sup_{x\in \bbR} |F'(x)|,
        \sup_{x\in \bbR} |F''(x)|,
        \sup_{x\in \bbR} |F^{(3)}(x)|
        \le C.
    \eeq
    Thus $F$ is $C$-Lipschitz.
    By \eqref{eq:F-explicit},
    \[
        F(0) =
        \fr{1}{\sqrt{(\rho_\eps(q(\bm)) + \eps(1 + \eps \rho_\eps(q(\bm))))(1+\eps \rho_\eps(q(\bm)))}}
        \cE\lt(\fr{
            \kappa \sqrt{1+\eps \rho_\eps(q(\bm))}
        }{
            \sqrt{\rho_\eps(q(\bm)) + \eps(1 + \eps \rho_\eps(q(\bm)))}
        }\rt)
    \]
    is bounded, and thus
    \[
        \tnorm{F(\abh)} \le \tnorm{F(\bzero)} + C \tnorm{\abh}
        \le C\sqrt{N}.
    \]
    By \eqref{eq:regularity-F-derivatives} we also have $\tnorm{F'(\abh)}, \tnorm{F''(\abh)}, \tnorm{F^{(3)}(\abh)} \le C\sqrt{N}$.
    This also implies $d_\eps(\bm,\bn) \le C$.

    Since $\df_\eps$ is bounded, $\tnorm{\bD_1}_\op \le C$.
    Since $F'$ is bounded, $\tnorm{\bD_3}_\op, \tnorm{\bD_4}_\op \le C$.
    The estimate \eqref{eq:d2n-bound} also implies $\tnorm{\tbD_2}_\op, \tnorm{\bD_4^{-1}}_\op \le C$.
    Combining these estimates shows $\tnorm{\nabla^2_{\bm,\bm} \cF^\eps_\TAP(\bm,\bn)}_\op$, $\tnorm{\nabla^2_{\bm,\bn} \cF^\eps_\TAP(\bm,\bn)}_\op$, $\tnorm{\nabla^2_{\bn,\bn} \cF^\eps_\TAP(\bm,\bn)}_\op \le C$.
\end{proof}

\subsection{Analysis of AMP iteration in planted model}

\begin{proof}[Proof of Proposition~\ref{ppn:planted-state-evolution-aux}]
    The state evolution \cite[Theorem 1]{berthier2020state} implies that
    \baln
        \fr1N \sum_{i=1}^N \delta(\dh_i,\dxi_i,\dh^{(1),1}_i,\ldots,\dh^{(1),k}_i)
        &\stackrel{\bbW_2}{\to}
        \cN(0, \dSig^{(1)}_{\le k}), &
        \fr1M \sum_{a=1}^M \delta(\hh_a,\hxi_a,\hh^{(1),0}_a,\ldots,\hh^{(1),k}_a)
        &\stackrel{\bbW_2}{\to}
        \cN(0, \hSig^{(1)}_{\le k}),
    \ealn
    for the following arrays $\dSig^{(1)}, \hSig^{(1)}$.
    First, $\hSig^{(1)}$ agrees with $\hSig^+$ on indices $(i,j)$ where $\{(i,j)\} \cap \{\diamond,\bowtie\} \neq \emptyset$, and $\dSig^{(1)}$ agrees with $\dSig^+$ on $(i,j)$ where $\{(i,j)\} \cap \{\diamond,\bowtie,0\} \neq \emptyset$.
    The remaining entries are defined by the following recursion.
    For $(\dH,\dXi,\dH_1,\ldots,\dH_k) \sim \cN(0,\dSig^{(1)}_{\le k})$ and $0\le i\le k$,
    \beq
        \label{eq:hSig-1-recursion}
        \hSig^{(1)}_{i,k}
        = \EE \lt[
            \lt(\th_\eps(\dH_i) - \fr{\oq_i}{q_\eps} \th_\eps(\dH)\rt)
            \lt(\th_\eps(\dH_k) - \fr{\oq_k}{q_\eps} \th_\eps(\dH)\rt)
        \rt]
        + \fr{\eps (q_\eps - \oq_i)(q_\eps - \oq_k)}{q_\eps(q_\eps + \eps)}
        + \fr{(\oq_i + \eps)(\oq_k + \eps)}{q_\eps + \eps}.
    \eeq
    For $(\hH,\hXi,\hH_0,\ldots,\hH_k) \sim \cN(0,\hSig^{(1)}_{\le k})$ and $0\le i\le k$, we have
    \balnn
        \notag
        \dSig^{(1)}_{i+1,k+1}
        &= \alpha_\star \EE \lt[
            \lt(F_{\eps,\varrho_\eps}(\hH_i) - \fr{\opsi_{i+1}}{\psi_\eps} F_{\eps,\varrho_\eps}(\hH)\rt)
            \lt(F_{\eps,\varrho_\eps}(\hH_k) - \fr{\opsi_{k+1}}{\psi_\eps} F_{\eps,\varrho_\eps}(\hH)\rt)
        \rt] \\
        \label{eq:dSig-1-recursion}
        &\qquad + \fr{\eps (\psi_\eps - \opsi_{i+1})(\psi_\eps - \opsi_{k+1})}{\psi_\eps(\psi_\eps + \eps)}
        + \fr{(\opsi_{i+1} + \eps)(\opsi_{k+1} + \eps)}{\psi_\eps + \eps}.
    \ealnn
    We now verify by induction that $\hSig^{(1)}$ and $\dSig^{(1)}$ coincide with $\hSig^+$ and $\dSig^+$.
    Suppose $\dSig^{(1)}_{\le k} = \dSig^+_{\le k}$.
    Then,
    \baln
        \EE[\th_\eps(\dH_i) \th_\eps(\dH_k)] &= \dSig_{i,k}, &
        \EE[\th_\eps(\dH_i) \th_\eps(\dH)] &= \oq_i, &
        \EE[\th_\eps(\dH)^2] &= q_\eps,
    \ealn
    so the right-hand side of \eqref{eq:hSig-1-recursion} simplifies as
    \[
        \dSig_{i,k} - \fr{\oq_i \oq_k}{q_\eps}
        + \fr{\eps (q_\eps - \oq_i)(q_\eps - \oq_k)}{q_\eps(q_\eps + \eps)}
        + \fr{(\oq_i + \eps)(\oq_k + \eps)}{q_\eps + \eps}
        = \dSig_{i,k} + \eps
        = \dSig^+_{i,k}.
    \]
    Now, suppose $\hSig^{(1)}_{\le k} = \hSig^+_{\le k}$.
    Then,
    \baln
        \alpha_\star \EE [F_{\eps,\varrho_\eps}(\hH_i) F_{\eps,\varrho_\eps}(\hH_k)] &= \hSig_{i+1,k+1}, &
        \alpha_\star \EE [F_{\eps,\varrho_\eps}(\hH_i) F_{\eps,\varrho_\eps}(\hH)] &= \opsi_{i+1}, &
        \alpha_\star \EE[F_{\eps,\varrho_\eps}(\hH)^2] &= \psi_\eps,
    \ealn
    so the right-hand side of \eqref{eq:dSig-1-recursion} simplifies as
    \[
        \hSig_{i+1,k+1} - \fr{\opsi_{i+1} \opsi_{k+1}}{\psi_\eps}
        + \fr{\eps (\psi_\eps - \opsi_{i+1})(\psi_\eps - \opsi_{k+1})}{\psi_\eps(\psi_\eps + \eps)}
        + \fr{(\opsi_{i+1} + \eps)(\opsi_{k+1} + \eps)}{\psi_\eps + \eps}
        = \hSig_{i+1,k+1} + \eps
        = \hSig^+_{i+1,k+1}.
    \]
    This completes the induction.
\end{proof}
To prove Proposition~\ref{ppn:amp-approximation-main}, we introduce two additional auxiliary AMP iterations.
They are initialized at $\bn^{(2),-1} = \bn^{(3),-1} = \bzero$, $\bm^{(2),0} = \bm^{(3),0} = q_\eps^{1/2} \bone$, with iteration
\baln
    \bm^{(i),k} &= \th_\eps(\dbh^{(i),k}), &
    \bn^{(i),k} &= F_{\eps,\varrho_\eps}(\hbh^{(i),k}),
\ealn
for $i \in \{2,3\}$ and $\dbh^{(i),k}, \hbh^{(i),k}$ as follows.
Recall that $\obG$ is the matrix \eqref{eq:amp-def-obG}, and $\opsi_0 = 0$.
Then,
\balnn
    \label{eq:hbh2-iteration}
    \hbh^{(2),k}
    &=
    \fr{1}{\sqrt{N}} \obG \lt( \bm^{(2),k} - \fr{\oq_k}{q_\eps} \bm \rt)
    + \fr{\sqrt{\eps} (q_\eps - \oq_k)}{\sqrt{q_\eps(q_\eps + \eps)}} \hbxi
    + \fr{\oq_k + \eps}{q_\eps + \eps} \hbh
    - \varrho_\eps \lt( \bn^{(2),k-1} - \fr{\opsi_k}{\psi_\eps} \bn \rt) \\
    \notag
    \dbh^{(2),k+1}
    &= \fr{1}{\sqrt{N}} \obG^\top \lt( \bn^{(2),k} - \fr{\opsi_{k+1}}{\psi_\eps} \bn \rt)
    + \fr{\sqrt{\eps} (\psi_\eps - \psi_{k+1})}{\sqrt{\psi_\eps(\psi_\eps + \eps)}} \dbxi
    + \fr{\opsi_{k+1} + \eps}{\psi_\eps + \eps} \dbh
    - d_\eps \lt( \bm^{(2),k} - \fr{\oq_k}{q_\eps} \bm \rt) \\
    \label{eq:hbh3-iteration}
    \hbh^{(3),k}
    &=
    \fr{1}{\sqrt{N}} \tbG \lt( \bm^{(3),k} - \bm \rt)
    + \fr{\oq_k + \eps}{q_\eps + \eps} \hbh
    - \varrho_\eps \lt( \bn^{(3),k-1} - \fr{\opsi_k + \bone\{k\ge 1\} \eps}{\psi_\eps + \eps} \bn \rt) \\
    \notag
    \dbh^{(3),k+1}
    &= \fr{1}{\sqrt{N}} \tbG^\top \lt( \bn^{(3),k} - \bn \rt)
    + \fr{\opsi_{k+1} + \eps}{\psi_\eps + \eps} \dbh
    - d_\eps \lt( \bm^{(3),k} - \fr{\oq_k + \eps}{q_\eps + \eps} \bm \rt).
\ealnn
The following proposition shows that all these AMP iterations approximate each other.
\begin{ppn}
    \label{ppn:amp-approximations}
    For any $k\ge 0$, as $N\to\infty$ we have the following convergences in probability under $\bbP^{\bm,\bn}_{\eps,\Pl}$.
    \begin{enumerate}[label=(\alph*),ref=(\alph*)]
        \item \label{itm:amp-approximation-1} $\tnorm{\hbh^{(1),k} - \hbh^{(2),k}} / \sqrt{N} \to 0$, and if $k\ge 1$, $\tnorm{\dbh^{(1),k} - \dbh^{(2),k}} / \sqrt{N} \to 0$.
        \item \label{itm:amp-approximation-2} $\tnorm{\hbh^{(2),k} - \hbh^{(3),k}} / \sqrt{N} \to 0$, and if $k\ge 1$, $\tnorm{\dbh^{(2),k} - \dbh^{(3),k}} / \sqrt{N} \to 0$.
        \item \label{itm:amp-approximation-3} $\tnorm{\hbh^{(3),k} - \hbh^k} / \sqrt{N} \to 0$, and if $k\ge 1$, $\tnorm{\dbh^{(3),k} - \dbh^k} / \sqrt{N} \to 0$.
    \end{enumerate}
\end{ppn}
\begin{proof}[Proof of Proposition~\ref{ppn:amp-approximations}\ref{itm:amp-approximation-1}]
    Similarly to \eqref{eq:tbG-decomp}, we can sample $Z' \sim \cN(0,1)$, $\dbxi' \sim \cN(0,\bI_N)$, $\hbxi' \sim \cN(0,\bI_M)$ coupled to $\hbG$ such that
    \balnn
        \label{eq:hbG-decomp}
        \hbG + \bDel'
        &= \obG - \fr{\hbxi' \bm^\top}{\tnorm{\bm}} - \fr{\bn (\dbxi')^\top}{\tnorm{\bn}}, &
        \bDel' = \fr{\bn\bm^\top}{\tnorm{\bn}\tnorm{\bm}} Z'
    \ealnn
    Note that $\tnorm{\bDel'}_\op = o(\sqrt{N})$ with high probability.
    Let $\simeq$ denote equality up to additive $o_N(1)$.
    By Proposition~\ref{ppn:planted-state-evolution-aux}, for $(\dH,\dXi,\dH_1,\ldots,\dH_k) \sim \cN(0,\dSig^{(1)}_{\le k})$ and $(\hH,\hXi,\hH_0,\ldots,\hH_k) \sim \cN(0,\hSig^{(1)}_{\le k})$,
    \baln
        \fr1N \la \bm, \dbh^{(1),k} \ra
        \simeq \EE [\th_\eps(\dH)\dH_k]
        &= \varrho_\eps (\opsi_k + \eps), &
        \fr1N \la \bn, \hbh^{(1),k} \ra
        \simeq \alpha_\star \EE [F_{\eps,\varrho_\eps}(\hH)\hH_k]
        &= d_\eps (\oq_k + \eps), \\
        \fr1N \la \bm, \dbh \ra
        \simeq \EE [\th_\eps(\hH)\hH]
        &= \varrho_\eps (\psi_\eps + \eps), &
        \fr1N \la \bn, \hbh \ra
        \simeq \alpha_\star \EE [F_{\eps,\varrho_\eps}(\hH)\hH]
        &= d_\eps (q_\eps + \eps).
    \ealn
    Also,
    \balnn
        \label{eq:bm-bn-no-correlation-with-approximation1}
        \fr1N \lt\la \bm, \bm^{(1),k} - \fr{\oq_k}{q_\eps} \bm\rt\ra
        \simeq \oq_k - \fr{\oq_k}{q_\eps} \cdot q_\eps
        &= 0, &
        \fr1N \lt\la \bn, \bn^{(1),k-1} - \fr{\opsi_k}{\psi_\eps} \bn\rt\ra
        \simeq \opsi_k - \fr{\opsi_k}{\psi_\eps} \cdot \psi_\eps
        &= 0.
    \ealnn
    Finally $\fr1N \la \dbxi, \bm \ra \simeq \fr1N \la \hbxi, \bn \ra \simeq 0$.
    Considering the inner product of \eqref{eq:hbh1-iteration} with $\bn$ shows
    \[
        0 \simeq
        \fr1N \lt\la
            \bn, \fr{1}{\sqrt{N}} \hbG \lt( \bm^{(1),k} - \fr{\oq_k}{q_\eps} \bm \rt)
        \rt\ra.
    \]
    We can expand $\hbG$ using \eqref{eq:hbG-decomp}.
    Since $\bn^\top \obG = \bzero$, $\fr1N \la \bn, \hbxi' \ra \simeq 0$ in probability, and $\tnorm{\bDel'}_\op = o(\sqrt{N})$,
    \[
        0 \simeq
        \fr1N \lt\la
            \bn, \fr{1}{\sqrt{N}} \lt(
                \obG
                - \fr{\hbxi' \bm^\top}{\tnorm{\bm}}
                - \fr{\bn (\dbxi')^\top}{\tnorm{\bn}}
                - \bDel'
            \rt) \lt( \bm^{(1),k} - \fr{\oq_k}{q_\eps} \bm \rt)
        \rt\ra
        \simeq \fr{\tnorm{\bn}}{N^{3/2}} \lt\la \dbxi', \bm^{(1),k} - \fr{\oq_k}{q_\eps} \bm \rt\ra.
    \]
    Thus,
    \beq
        \label{eq:dbxi-no-correlation-with-approximation1}
        \fr1N \lt\la \dbxi', \bm^{(1),k} - \fr{\oq_k}{q_\eps} \bm \rt\ra \simeq 0
    \eeq
    in probability for all $k$.
    An analogous computation shows
    \[
        \fr1N \lt\la \hbxi', \bn^{(1),k-1} - \fr{\opsi_k}{\psi_\eps} \bn \rt\ra
        \simeq 0.
    \]
    By \eqref{eq:hbG-decomp},
    \baln
        \fr{1}{\sqrt{N}} (\hbG - \obG) \lt( \bm^{(1),k} - \fr{\oq_k}{q_\eps} \bm \rt)
        &=
        \fr{\hbxi'}{\sqrt{N} \tnorm{\bm}} \lt\la \bm^\top, \bm^{(1),k} - \fr{\oq_k}{q_\eps} \bm \rt\ra
        + \fr{\bn}{\sqrt{N} \tnorm{\bn}} \lt\la \dbxi', \bm^{(1),k} - \fr{\oq_k}{q_\eps} \bm \rt\ra \\
        &- \fr{1}{\sqrt{N}} \bDel' \lt( \bm^{(1),k} - \fr{\oq_k}{q_\eps} \bm \rt),
    \ealn
    and this has norm $o(\sqrt{N})$ by \eqref{eq:bm-bn-no-correlation-with-approximation1}, \eqref{eq:dbxi-no-correlation-with-approximation1}.
    Subtracting \eqref{eq:hbh1-iteration} and \eqref{eq:hbh2-iteration} yields
    \baln
        \hbh^{(1),k} - \hbh^{(2),k}
        &= \fr{1}{\sqrt{N}} (\hbG - \obG) \lt( \bm^{(1),k} - \fr{\oq_k}{q_\eps} \bm \rt)
        + \fr{1}{\sqrt{N}} \obG (\bm^{(1),k} - \bm^{(2),k})
        - \varrho_\eps (\bn^{(1),k-1} - \bn^{(2),k-1}) \\
        &= \fr{1}{\sqrt{N}} \obG (\bm^{(1),k} - \bm^{(2),k})
        - \varrho_\eps (\bn^{(1),k-1} - \bn^{(2),k-1})
        + o(\sqrt{N}),
    \ealn
    where $o(\sqrt{N})$ denotes a vector with this norm.
    Analogously,
    \[
        \dbh^{(1),k+1} - \dbh^{(2),k+1}
        = \fr{1}{\sqrt{N}} \obG^\top (\bn^{(1),k} - \bn^{(2),k})
        - d_\eps (\bm^{(1),k} - \bm^{(2),k})
        + o(\sqrt{N}).
    \]
    On the high probability event that $\tnorm{\obG}_\op = O(\sqrt{N})$, we have
    \baln
        \tnorm{\hbh^{(1),k} - \hbh^{(2),k}}
        &\le O(1) \tnorm{\bm^{(1),k} - \bm^{(2),k}}
        + \varrho_\eps \tnorm{\bn^{(1),k-1} - \bn^{(2),k-1}}
        + o(\sqrt{N}), \\
        \tnorm{\dbh^{(1),k+1} - \dbh^{(2),k+1}}
        &\le O(1) \tnorm{\bn^{(1),k} - \bn^{(2),k}}
        + |d_\eps| \tnorm{\bm^{(1),k} - \bm^{(2),k}}
        + o(\sqrt{N}).
    \ealn
    The claim now follows by induction on $k$: $\tnorm{\bm^{(1),0} - \bm^{(2),0}} = \tnorm{\bn^{(1),-1} - \bn^{(2),-1}} = 0$ by initialization, and because $\th_\eps$ and $F_{\eps,\varrho_\eps}$ are $O(1)$-Lipschitz,
    \baln
        \tnorm{\bm^{(1),k} - \bm^{(2),k}} &\le O(1) \tnorm{\dbh^{(1),k} - \dbh^{(2),k}}, &
        \tnorm{\bn^{(1),k} - \bn^{(2),k}} &\le O(1) \tnorm{\hbh^{(1),k} - \hbh^{(2),k}},
    \ealn
    for all $k\ge 1$, $k\ge 0$ respectively.
\end{proof}
\begin{proof}[Proof of Proposition~\ref{ppn:amp-approximations}\ref{itm:amp-approximation-2}]
    Note that $\bDel$ defined in \eqref{eq:def-bDel-amp-approximation} w.h.p. satisfies $\tnorm{\bDel}_\op = o(\sqrt{N})$.
    We write \eqref{eq:hbh3-iteration} as
    \baln
        \hbh^{(3),k}
        &=
        \fr{1}{\sqrt{N}} \tbG ( \bm^{(2),k} - \bm )
        + \fr{\oq_k + \eps}{q_\eps + \eps} \hbh
        - \varrho_\eps \lt( \bn^{(2),k-1} - \fr{\opsi_k + \bone\{k\ge 1\} \eps}{\psi_\eps + \eps} \bn \rt) \\
        &\qquad + \fr{1}{\sqrt{N}} \tbG ( \bm^{(3),k} - \bm^{(2),k} )
        - \varrho_\eps ( \bn^{(3),k-1} - \bn^{(2),k-1} ).
    \ealn
    By Proposition~\ref{ppn:amp-approximations}\ref{itm:amp-approximation-1}, $\bbW_2(\mu_{\dbh^{(2),k}}, \mu_{\dbh^{(1),k}}) = o_N(1)$.
    So, Fact~\ref{fac:pseudo-lipschitz} and Proposition~\ref{ppn:planted-state-evolution-aux} imply
    \balnn
        \label{eq:bm-bm2-overlap}
        \fr1N \la \bm, \bm^{(2),k} \ra
        \simeq \fr1N \la \bm, \bm^{(1),k} \ra
        &\simeq \oq_k, \\
        \notag
        \fr1N \la \dbxi, \bm^{(2),k} \ra
        \simeq \fr1N \la \dbxi, \bm^{(1),k} \ra
        &\simeq
        \bone\{k\ge 1\}
        \varrho_\eps
        \fr{\sqrt{\eps} (\psi_\eps - \opsi_k)}{\sqrt{\psi_\eps (\psi_\eps + \eps)}}.
    \ealnn
    By \eqref{eq:tbG-decomp},
    \baln
        \fr{1}{\sqrt{N}} \tbG ( \bm^{(2),k} - \bm )
        &= \fr{1}{\sqrt{N}} \lt(
            \obG
            - \sqrt{\fr{\eps}{q(\bm) + \eps}} \cdot \fr{\hbxi \bm^\top}{\tnorm{\bm}}
            - \sqrt{\fr{\eps}{\psi(\bn) + \eps}} \cdot \fr{\bn \dbxi^\top}{\tnorm{\bn}}
            - \bDel
        \rt) ( \bm^{(2),k} - \bm ) \\
        &= \fr{1}{\sqrt{N}} \obG ( \bm^{(2),k} - \bm )
        + \fr{\sqrt{\eps} (q_\eps - \oq_k)}{\sqrt{q_\eps(q_\eps + \eps)}} \hbxi
        - \bone\{k\ge 1\} \varrho_\eps \fr{\eps (\psi_\eps - \opsi_k)}{\psi_\eps (\psi_\eps + \eps)} \bn
        + o(\sqrt{N}).
    \ealn
    Since $\obG \bm = \bzero$, we have $\obG ( \bm^{(2),k} - \bm ) = \obG ( \bm^{(2),k} - \fr{\oq_k}{q_\eps} \bm )$.
    Moreover,
    \[
        \fr{\opsi_k + \bone\{k\ge 1\} \eps}{\psi_\eps + \eps}
        - \bone\{k\ge 1\} \fr{\eps (\psi_\eps - \opsi_k)}{\psi_\eps (\psi_\eps + \eps)}
        = \fr{\opsi_k}{\psi_\eps}.
    \]
    Combining the above and comparing with \eqref{eq:hbh2-iteration} shows
    \[
        \hbh^{(3),k}
        = \hbh^{(2),k}
        + \fr{1}{\sqrt{N}} \tbG ( \bm^{(3),k} - \bm^{(2),k} )
        - \varrho_\eps ( \bn^{(3),k-1} - \bn^{(2),k-1} )
        + o(\sqrt{N}).
    \]
    Similarly,
    \[
        \dbh^{(3),k+1}
        = \hbh^{(2),k+1}
        + \fr{1}{\sqrt{N}} \tbG^\top ( \bn^{(3),k} - \bn^{(2),k} )
        - d_\eps ( \bm^{(3),k} - \bm^{(2),k} )
        + o(\sqrt{N}).
    \]
    On the high-probability event that $\tnorm{\tbG}_\op = O(\sqrt{N})$, this implies
    \baln
        \tnorm{\hbh^{(3),k} - \hbh^{(2),k}}
        &\le O(1) \tnorm{\bm^{(3),k} - \bm^{(2),k}}
        + \varrho_\eps \tnorm{\bn^{(3),k-1} - \bn^{(2),k-1}}
        + o(\sqrt{N}), \\
        \tnorm{\dbh^{(3),k+1} - \hbh^{(2),k+1}}
        &\le O(1) \tnorm{\bn^{(3),k} - \bn^{(2),k}}
        + |d_\eps| \tnorm{\bm^{(3),k} - \bm^{(2),k}}
        + o(\sqrt{N}).
    \ealn
    The result follows by induction on $k$, like above.
\end{proof}
\begin{proof}[Proof of Proposition~\ref{ppn:amp-approximations}\ref{itm:amp-approximation-3}]
    By Corollary~\ref{cor:conditional-law-correct-profile}, we have
    \balnn
        \notag
        \fr{\bG}{\sqrt{N}}
        &\stackrel{d}{=}
        \fr{(1+o_N(1)) \hbh \bm^\top}{N(q_\eps + \eps)}
        + \fr{(1+o_N(1)) \bn \dbh^\top}{N(\psi_\eps + \eps)}
        + \fr{o_N(1) \bn \bm^\top}{N}
        + \fr{\tbG}{\sqrt{N}} \\
        \label{eq:conditional-law-approximation-expansion}
        &= \fr{\hbh \bm^\top}{N(q_\eps + \eps)}
        + \fr{\bn \dbh^\top}{N(\psi_\eps + \eps)}
        + \fr{\tbG}{\sqrt{N}}
        + o_N(1),
    \ealnn
    for $\tbG$ as above and $o_N(1)$ a matrix with this operator norm.
    Since $q(\bm) \simeq q_\eps$, $\psi(\bn) \simeq \psi_\eps$, and under $\PP^{\bm,\bn}_{\eps,\Pl}$ we have a.s. $\abh = F^{-1}_{\eps,\rho_\eps(q(\bm))}(\bn)$,
    the following terms appearing in \eqref{eq:TAP-stationarity-m}, \eqref{eq:TAP-stationarity-n} satisfy
    \baln
        \rho_\eps(q(\bm)) &\simeq \varrho_\eps, &
        \rho'_\eps(q(\bm)) &\simeq -1, &
        d_\eps(\bm,\bn) &\simeq d_\eps.
    \ealn
    Combining the AMP iteration \eqref{eq:amp-iterates-n} with \eqref{eq:TAP-stationarity-n} yields
    \baln
        \hbh^k &= \fr{1}{\sqrt{N}} \bG (\bm^k - \bm) + \hbh + \varrho_\eps (\bn - \bn^{k-1}) \\
        &= \fr{1}{\sqrt{N}} \bG (\bm^{(3),k} - \bm) + \hbh - \varrho_\eps (\bn^{(3),k-1} - \bn)
        + \fr{1}{\sqrt{N}} \bG (\bm^{k} - \bm^{(3),k})
        - \varrho_\eps (\bn^{k-1} - \bn^{(3),k-1}).
    \ealn
    By Proposition~\ref{ppn:amp-approximations}\ref{itm:amp-approximation-1}\ref{itm:amp-approximation-2}, $\bbW_2(\mu_{\dbh^{(3),k}}, \mu_{\dbh^{(1),k}}) = o_N(1)$.
    So, Fact~\ref{fac:pseudo-lipschitz} and Proposition~\ref{ppn:planted-state-evolution-aux} imply $\fr1N \la \bm, \bm^{(3),k} \ra \simeq \oq_k$ (similarly to \eqref{eq:bm-bm2-overlap}) and
    \[
        \fr1N \la \dbh, \bm^{(3),k} \ra
        \simeq \fr1N \la \dbh, \bm^{(1),k} \ra
        \simeq (\opsi_k + \bone\{k\ge 1\} \eps) \varrho_\eps.
    \]
    Expanding $\bG$ using \eqref{eq:conditional-law-approximation-expansion} then yields
    \baln
        \hbh^k
        &= \fr{1}{\sqrt{N}} \tbG (\bm^{(3),k} - \bm)
        + \fr{\oq_k + \eps}{q_\eps + \eps} \hbh
        - \varrho_\eps \lt(\bn^{(3),k-1} - \fr{\opsi_k + \eps}{\psi_\eps + \eps} \bn\rt) \\
        &\qquad + \fr{1}{\sqrt{N}} \bG (\bm^{k} - \bm^{(3),k})
        - \varrho_\eps (\bn^{k-1} - \bn^{(3),k-1})
        + o(\sqrt{N}) \\
        &= \hbh^{(3),k} + \fr{1}{\sqrt{N}} \bG (\bm^{k} - \bm^{(3),k})
        - \varrho_\eps (\bn^{k-1} - \bn^{(3),k-1})
        + o(\sqrt{N}).
    \ealn
    Analogously,
    \[
        \dbh^{k+1}
        = \dbh^{(3),k+1}
        + \fr{1}{\sqrt{N}} \bG^\top (\bn^{k} - \bn^{(3),k})
        - d_\eps (\bm^{k-1} - \bm^{(3),k-1})
        + o(\sqrt{N}).
    \]
    So, on the high probability event that $\tnorm{\bG}_\op = O(\sqrt{N})$,
    \baln
        \tnorm{\hbh^k - \hbh^{(3),k}}
        &= O(1) \tnorm{\bm^k - \bm^{(3),k}}
        + \varrho_\eps \tnorm{\bn^{k-1} - \bn^{(3),k-1}}
        + o(\sqrt{N}), \\
        \tnorm{\dbh^{k+1} - \dbh^{(3),k+1}}
        &= O(1) \tnorm{\bn^k - \bn^{(3),k}}
        + |d_\eps| \tnorm{\bm^k - \bm^{(3),k}}
        + o(\sqrt{N}).
    \ealn
    The result follows by induction on $k$, like above.
\end{proof}
\begin{proof}[Proof of Proposition~\ref{ppn:amp-approximation-main}]
	Immediate from Proposition~\ref{ppn:amp-approximations}.
\end{proof}

\subsection{Continuity of first moment functional term}

\begin{proof}[Proof of Lemma~\ref{lem:functional-continuity}]
    Let $C$ denote an absolute constant, which may change from line by line.
    By Lemma~\ref{lem:logPsi-pseudo-lipschitz}, $\log \Psi$ is $(2,1)$-pseudo-Lipschitz.
    By Cauchy--Schwarz (similarly to the proof of Fact~\ref{fac:pseudo-lipschitz}),
    \[
        \lt|\EE \log \Psi \lt\{
            \fr{\kappa
                - a_1\hbH
                - b_1\bN
            }{
                c_1
            }
        \rt\}
        - \log \Psi \lt\{
            \fr{\kappa
                - a_2\hbH
                - b_2\bN
            }{
                c_2
            }
        \rt\}\rt|
        \le C\sqrt{T_1T_2},
    \]
    where
    \baln
        T_1 &=
        \EE \lt[\lt(
            \fr{\kappa
                - a_1\hbH
                - b_1\bN
            }{c_1}
            - \fr{\kappa
                - a_2\hbH
                - b_2\bN
            }{c_2}
        \rt)^2\rt] \\
        &\le C\lt(\fr{\max(a_1,a_2,b_1,b_2,c_1,c_2,1)(|a_1-a_2|+|b_1-b_2|+|c_1-c_2|)}{\min(c_1,c_2)^2}\rt)^2
    \ealn
    and
    \[
        T_2 =
        \EE \lt[
            \lt(\fr{\kappa
                - a_1\hbH
                - b_1\bN
            }{c_1}\rt)^2
            + \lt(\fr{\kappa
                - a_2\hbH
                - b_2\bN
            }{c_2}\rt)^2
            + 1
        \rt]
        \le C \lt(\fr{\max(a_1,a_2,b_1,b_2,c_1,c_2,1)}{\min(c_1,c_2)}\rt)^4.
    \]
\end{proof}

\section{Verification of numerical conditions for $\kappa = 0$}
\label{app:numerics}

In this appendix, we use rigorous interval arithmetic (implemented in the attached Python 3 file using \verb|python-flint|) to verify the conditions in Theorem~\ref{thm:main}, other than Condition~\ref{con:2varfn}, at $\kappa = 0$.
This proves Theorem~\ref{thm:main-kappa0}.
We also verify Claim~\ref{clm:sS-zero-2deriv} using interval arithmetic.

Throughout this section we take $\kappa = 0$, $\alpha_\star = \alpha_\star(0)$, $q_0 = q_\star(\alpha_\star,0)$, and $\psi_0 = \psi_\star(\alpha_\star,0)$.
We will use Claims to denote statements whose proofs require interval arithmetic.

\subsection{Numerical estimates of parameters and special functions}

By \cite[\S7]{ding2018capacity}, the following are lower and upper bounds for $\alpha_\star$, $q_0$, $\psi_0$:
\baln
    \alpha_{\lb} &= 0.833078599, &
    q_{\lb} &= 0.56394907949, &
    \psi_{\lb} &= 2.5763513100, & \\
    \alpha_{\ub} &= 0.833078600, &
    q_{\ub} &= 0.56394908030, &
    \psi_{\ub} &= 2.5763513224.
\ealn
Let $\gamma_0 = \fr{q_0}{1-q_0}$, $\gamma_{\lb} = \fr{q_{\lb}}{1-q_{\lb}}$ and $\gamma_{\ub} = \fr{q_{\ub}}{1-q_{\ub}}$.
Note that Condition~\ref{con:local-concavity} only requires us to exhibit a value of $z > -1$ such that $\lambda(z) < 0$.
In the verification below we will use the value
\[
    \hz = -0.669316.
\]
For $k \in \{2,4\}$, define
\baln
    p_k(\psi) &= \EE[\th(\psi^{1/2} Z)^k], &
    r_k(\gamma) &= \EE[\cE(\gamma^{1/2} Z)^k].
\ealn
Note that the fixed-point condition in Condition~\ref{con:km-well-defd} defining $(q_0,\psi_0)$ implies (for $\kappa = 0$)
\balnn
    \label{eq:p2-r2-trivial}
    p_2(\psi_0) &= q_0, &
    r_2(\gamma_0) &= \fr{(1-q_0) \psi_0}{\alpha_\star}.
\ealnn
Let
\beq
    \label{eq:special-fn-m}
    m(z,\psi) = \bbE[(z + \ch^2(\psi^{1/2} Z))^{-1}].
\eeq
Finally, define
\beq
    \label{eq:special-fn-g}
    g(m,q,\gamma) = \bbE \lt\{
        \fr{\cE'(\gamma^{1/2} Z)}{(1-q)(1-\cE'(\gamma^{1/2} Z)) + m\cE'(\gamma^{1/2} Z)}
    \rt\}.
\eeq
We now collect the main estimates in the verification whose proofs require computer assistance.
The proofs of these claims are deferred to \S\ref{subapp:interval-arithmetic}, with computer-assisted parts carried out in the attached Python file.

\begin{clm}
    \label{clm:p}
    We have $p_4(\psi_0) \in [p_{4,\lb},p_{4,\ub}] \equiv [0.4405902310,0.4405902320]$.
\end{clm}
\begin{clm}
    \label{clm:r}
    We have $r_4(\gamma_0) \in [r_{4,\lb},r_{4,\ub}] \equiv [5.297,5.317]$.
\end{clm}
\begin{clm}
    \label{clm:m}
    We have $m(\hz) \le m_{\ub} \equiv 0.9309695$, where $m(z) = m(z,\psi_0)$ is defined in Condition~\ref{con:local-concavity}.
\end{clm}
\begin{clm}
    \label{clm:g}
    We have $g(m(\hz),q_0,\gamma_0) \ge g_{\lb} \equiv 0.7739$.
\end{clm}
We conclude this preparatory subsection with a few useful lemmas.
First, we reduce several integrals that will appear below to the functions $p_2,p_4,r_2,r_4$.
\begin{lem}
    \label{lem:special-fn-identities}
    The following identities hold.
    \balnn
        \label{eq:frp}
        \frt(\psi) &\equiv \EE[\th'(\psi_0^{1/2} Z)^2]
        = 1 - 2p_2(\psi) + p_4(\psi), \\
        \label{eq:frr1}
        \frs_1(\gamma) &\equiv \EE \lt\{\cE'(\gamma^{1/2} Z)\rt\}
        = \fr{r_2(\gamma)}{1+\gamma}, \\
        \label{eq:frr2}
        \frs_2(\gamma) &\equiv \EE \lt\{\cE(\gamma^{1/2} Z)^2 \cE'(\gamma^{1/2} Z)\rt\}
        = \fr{r_4(\gamma)}{1+3\gamma}, \\
        \label{eq:frr3}
        \frs_3(\gamma) &\equiv \EE \lt\{\gamma^{1/2} Z \cE(\gamma^{1/2} Z) \cE'(\gamma^{1/2} Z)\rt\}
        = - \fr{\gamma}{1+2\gamma} r_2(\gamma)
        + \fr{3\gamma}{(1+2\gamma)(1+3\gamma)} r_4(\gamma), \\
        \label{eq:frr4}
        \frs_4(\gamma) &\equiv \EE \lt\{(\gamma^{1/2} Z)^2 \cE'(\gamma^{1/2} Z)\rt\}
        = - \fr{\gamma(4\gamma^2 + \gamma - 1)}{(1+\gamma)^2(1+2\gamma)} r_2(\gamma)
        + \fr{6\gamma^2}{(1+\gamma)(1+2\gamma)(1+3\gamma)} r_4(\gamma), \\
        \label{eq:frr5}
        \frs_5(\gamma) &\equiv \EE \lt\{\cE'(\gamma^{1/2} Z)^2\rt\}
        = \fr{\gamma}{1+2\gamma} r_2(\gamma)
        + \fr{1-\gamma}{(1+2\gamma)(1+3\gamma)} r_4(\gamma).
    \ealnn
\end{lem}
\begin{proof}
    Equation \eqref{eq:frp} follows directly from the identity
    \[
        \th'(x)^2 = (1-\th^2(x))^2 = 1 - 2\th^2(x) + \th^4(x).
    \]
    For the remaining parts, we apply the identity $\cE'(x) = \cE(x)(\cE(x)-x)$ (Lemma~\ref{lem:E-derivative-bds}\ref{itm:cEpr}) and integrate by parts.
    First,
    \baln
        \frs_1(\gamma)
        &= \EE \lt\{\cE(\gamma^{1/2} Z)^2\rt\}
        - \EE \lt\{\cE(\gamma^{1/2} Z) \gamma^{1/2} Z\rt\} \\
        &= \EE \lt\{\cE(\gamma^{1/2} Z)^2\rt\}
        - \gamma \EE \lt\{\cE'(\gamma^{1/2} Z)\rt\}
        = r_2(\gamma) - \gamma \frs_1(\gamma),
    \ealn
    which proves \eqref{eq:frr1}.
    Similarly,
    \[
        \frs_2(\gamma) = \EE \lt\{\cE(\gamma^{1/2} Z)^4\rt\}
        - \EE \lt\{\cE(\gamma^{1/2} Z)^3 \gamma^{1/2} Z\rt\}
        = r_4(\gamma) - 3\gamma \frs_2(\gamma),
    \]
    which proves \eqref{eq:frr2}.
    Then,
    \baln
        \frs_3(\gamma) &= \EE \lt\{\gamma^{1/2} Z \cE(\gamma^{1/2} Z)^3\rt\}
        - \EE \lt\{(\gamma^{1/2} Z)^2 \cE(\gamma^{1/2} Z)^2\rt\} \\
        &= 3\gamma \EE \lt\{\cE(\gamma^{1/2} Z)^2 \cE'(\gamma^{1/2} Z)\rt\}
        - \gamma \EE \lt\{\cE(\gamma^{1/2} Z)^2\rt\}
        - 2\gamma \EE \lt\{(\gamma^{1/2} Z) \cE(\gamma^{1/2} Z) \cE'(\gamma^{1/2} Z)\rt\} \\
        &= 3\gamma \frs_2(\gamma) - \gamma r_2(\gamma) - 2\gamma \frs_3(\gamma).
    \ealn
    Rearranging proves \eqref{eq:frr3}.
    Further,
    \baln
        \frs_4(\gamma) &= \EE \lt\{(\gamma^{1/2} Z)^2 \cE(\gamma^{1/2} Z)^2\rt\}
        - \EE \lt\{(\gamma^{1/2} Z)^3 \cE(\gamma^{1/2} Z)\rt\} \\
        &= \gamma \EE \lt\{\cE(\gamma^{1/2} Z)^2\rt\}
        + 2\gamma \EE \lt\{(\gamma^{1/2} Z) \cE(\gamma^{1/2} Z) \cE'(\gamma^{1/2} Z)\rt\} \\
        &\qquad - 2\gamma \EE \lt\{(\gamma^{1/2} Z) \cE(\gamma^{1/2} Z)\rt\}
        - \gamma \EE \lt\{(\gamma^{1/2} Z)^2 \cE'(\gamma^{1/2} Z)\rt\}.
    \ealn
    Integrating by parts again yields
    \[
        \EE \lt\{(\gamma^{1/2} Z) \cE(\gamma^{1/2} Z)\rt\}
        = \gamma \EE \lt\{\cE'(\gamma^{1/2} Z)\rt\}
        = \gamma \frs_1(\gamma).
    \]
    So
    \[
        \frs_4(\gamma)
        = \gamma r_2(\gamma)
        + 2\gamma \frs_3(\gamma)
        - 2\gamma^2 \frs_1(\gamma)
        - \gamma \frs_4(\gamma).
    \]
    Rearranging proves \eqref{eq:frr4}.
    Finally,
    \baln
        \frs_5(\gamma) &= \EE \lt\{ \cE(\gamma^{1/2} Z)^4 \rt\}
        - 2 \EE \lt\{ (\gamma^{1/2} Z) \cE(\gamma^{1/2} Z)^3 \rt\}
        + \EE \lt\{ (\gamma^{1/2} Z)^2 \cE(\gamma^{1/2} Z)^2 \rt\} \\
        &= \EE \lt\{ \cE(\gamma^{1/2} Z)^4 \rt\}
        - 6\gamma \EE \lt\{ \cE(\gamma^{1/2} Z)^2 \cE'(\gamma^{1/2} Z) \rt\}
        + \gamma \EE \lt\{ \cE(\gamma^{1/2} Z)^2 \rt\} \\
        &\qquad + 2\gamma \EE \lt\{ (\gamma^{1/2} Z) \cE(\gamma^{1/2} Z) \cE'(\gamma^{1/2} Z) \rt\} \\
        &= r_4(\gamma) - 6\gamma \frs_2(\gamma) + \gamma r_2(\gamma) + 2\gamma \frs_3(\gamma).
    \ealn
    Rearranging proves \eqref{eq:frr5}.
\end{proof}
\noindent Recall from Condition~\ref{con:local-concavity} that $d_0 = \alpha_\star \EE[F'_{1-q_0}(q_0^{1/2} Z)]$.
As a consequence of \eqref{eq:p2-r2-trivial} and \eqref{eq:frr1}, we have
\beq
    \label{eq:d0-formula}
    d_0 = -\fr{\alpha_\star}{1-q_0} \frs_1(\gamma_0)
    = -\fr{\alpha_\star}{1-q_0} \cdot \fr{r_2(\gamma_0)}{1+\gamma_0}
    = -(1-q_0) \psi_0,
\eeq
where we have used that $(1-q_0)(1+\gamma_0) = 1$.
\begin{lem}
    \label{lem:special-fn-monotonicity}
    The functions $p_4$ and $r_4$ are increasing.
    Moreover, for any $z > -1$, and $m$ defined in \eqref{eq:special-fn-m}, the function $\psi \mapsto m(z,\psi)$ is decreasing.
\end{lem}
\begin{proof}
    The function $p_4$ is increasing simply because the maps $\psi \mapsto \th(\psi^{1/2} x)^4$ are pointwise increasing for all $x\in \bbR$.
    Similarly, since the maps $\psi \mapsto (z + \ch^2(\psi^{1/2} x))^{-1}$ are pointwise increasing for all $x \in \bbR$, $z > -1$, the function $\psi \mapsto m(z,\psi)$ is decreasing.
    Finally,
    \[
        r'_4(\gamma)
        = \EE \lt\{
            6\cE(\gamma^{1/2} Z)^2 \cE'(\gamma^{1/2} Z)^2
            + 2 \cE(\gamma^{1/2} Z)^3 \cE''(\gamma^{1/2} Z)
        \rt\} \ge 0,
    \]
    as Lemma~\ref{lem:E-derivative-bds}\ref{itm:cE2} implies $\cE'' > 0$.
    Thus $r_4$ is increasing.
\end{proof}

\subsection{Verification of numerical conditions in Theorem~\ref{thm:main}}

Condition~\ref{con:km-well-defd} was proved in \cite[Proposition 1.3]{ding2018capacity} (recorded as Proposition~\ref{ppn:km-well-defd-kappa0}).
We now verify Conditions \ref{con:amp-works} and \ref{con:local-concavity} by proving the following.
\begin{clm}
    \label{clm:amp-works-kappa0}
    Condition~\ref{con:amp-works} holds for $\kappa = 0$, with $\alpha_\star \EE[\th'(\psi_0^{1/2}Z)^2] \EE [F'_{1-q_0}(q_0^{1/2} Z)^2] \le a_{\ub} \equiv 0.5446$.
\end{clm}
\begin{proof}
    We calculate:
    \baln
        &\alpha_\star \EE[\th'(\psi_0^{1/2}Z)^2] \EE [F'_{1-q_0}(q_0^{1/2} Z)^2]
        = \fr{\alpha_\star}{(1-q_0)^2} \frt(\psi_0) \frs_5(\gamma_0) \\
        &\stackrel{Lem.~\ref{lem:special-fn-identities}}{=} \fr{\alpha_\star}{(1-q_0)^2} (1 - 2p_2(\psi) + p_4(\psi))
        \lt(\fr{\gamma_0}{1+2\gamma_0} r_2(\gamma_0) + \fr{1-\gamma_0}{(1+2\gamma_0)(1+3\gamma_0)} r_4(\gamma_0) \rt) \\
        &\stackrel{\eqref{eq:p2-r2-trivial}}{=} \fr{\alpha_\star}{(1-q_0)^2} (1 - 2q_0 + p_4(\psi))
        \lt(\fr{\gamma_0}{1+2\gamma_0} \cdot \fr{(1-q_0) \psi_0}{\alpha_\star} + \fr{1-\gamma_0}{(1+2\gamma_0)(1+3\gamma_0)} r_4(\gamma_0) \rt) \\
        &= (1 - 2q_0 + p_4(\psi))
        \lt(\fr{\gamma_0\psi_0}{1+q_0} + \fr{\alpha_\star(1-\gamma_0)}{(1+q_0)(1+2q_0)} r_4(\gamma_0) \rt) \\
        &\le (1 - 2q_{\lb} + p_{4,\ub})
        \lt(\fr{\gamma_{\ub}\psi_{\ub}}{1+q_{\lb}} + \fr{\alpha_{\lb}(1-\gamma_{\lb})}{(1+q_{\ub})(1+2q_{\ub})} r_{4,\lb} \rt)
        \stackrel{(*)}{\le} a_{\ub}.
    \ealn
    The estimate $(*)$ is verified in the attached Python file.
    We note that this is a simple arithmetic comparison, as all terms are explicitly defined decimal numbers.
\end{proof}
\begin{clm}
    \label{clm:local-concavity-kappa0}
    Condition~\ref{con:local-concavity} holds for $\kappa = 0$, with $\lambda(\hz) \le \lambda_{\ub} \equiv -0.1906$.
\end{clm}
\begin{proof}
    Note that for $g$ defined in \eqref{eq:special-fn-g},
    \baln
        \lambda(\hz)
        &= \hz - \alpha_\star g(m(\hz),q_0,\gamma_0) - d_0
        \stackrel{\eqref{eq:d0-formula}}{=}
        \hz - \alpha_\star g(m(\hz),q_0,\gamma_0) + (1-q_0) \psi_0 \\
        &\le \hz - \alpha_{\lb} g_{\lb} + (1-q_{\lb}) \psi_{\ub}
        \stackrel{(*)}{\le} \lambda_{\ub}.
    \ealn
    The step $(*)$ is verified in the attached Python file, and is a simple arithmetic comparison of explicitly defined decimal numbers.
\end{proof}
\begin{proof}[Proof of Theorem~\ref{thm:main-kappa0}]
    Follows from Theorem~\ref{thm:main}, Proposition~\ref{ppn:km-well-defd-kappa0}, and Claims \ref{clm:amp-works-kappa0} and \ref{clm:local-concavity-kappa0}.
\end{proof}

\subsection{Local maximality of first moment functional at $(1,0)$}

We next verify Claim~\ref{clm:sS-zero-2deriv}.
\begin{lem}
    \label{lem:sS-2deriv-formula}
    For $\kappa = 0$, we have
    \baln
        \la \nabla^2 \osS_\star(1,0), (u_1,u_2)^{\otimes 2} \ra
        &= - \EE[(1-\bM^2) (u_1 \dbH + u_2 \bM)^2]
        + C_1 \EE[(1-\bM^2) (u_1 \dbH + u_2 \bM)\dbH]^2 \\
        &+ C_2 \EE[(1-\bM^2) (u_1 \dbH + u_2 \bM)\bM]\EE[(1-\bM^2) (u_1 \dbH + u_2 \bM)\dbH] \\
        &+ C_3 \EE[(1-\bM^2) (u_1 \dbH + u_2 \bM)\bM]^2,
    \ealn
    where
    \baln
        C_1 &= \fr{\alpha_\star}{\psi_0^2} \EE \lt\{F'_{1-q_0}(\hbH) \bN^2\rt\}, &
        C_2 &= \fr{2\alpha_\star}{\psi_0} \EE \lt\{ F'_{1-q_0}(\hbH) \lt(\fr{1}{q_0(1-q_0)} \hbH + \bN\rt) \bN \rt\} + \fr{2}{1-q_0}, \\
        && C_3 &= \alpha_\star \EE \lt\{F'_{1-q_0}(\hbH) \lt(\fr{1}{q_0(1-q_0)} \hbH + \bN\rt)^2\rt\}
        + \fr{\psi_0}{q_0}.
    \ealn
\end{lem}
\begin{proof}
    Analogously to the proof of Lemma~\ref{lem:sS-zero}\ref{itm:sS-zero-1deriv}, define $\bDel_2 = (u_1 \partial_{\lambda_1} + u_2 \partial_{\lambda_2})^2 \bLam$.
    Also abbreviate
    \[
        \bV = \fr{
            \kappa - \fr{\EE[\bM\bLam]}{q_0} \hbH - \fr{\EE[\dbH\bLam]}{\psi_0} \bN
        }{
            \sqrt{1-\fr{\EE[\bM\bLam]^2}{q_0}}
        } + \sqrt{1-q_0} \bN.
    \]
    We differentiate \eqref{eq:osS-1deriv} to obtain
    \baln
        &\la \osS_\star(\lambda_1,\lambda_2), (u_1,u_2)^{\otimes 2} \ra
        = - \EE [(u_1 \dbH + u_2 \bM) \bDel] \\
        &- \alpha_\star \EE \lt\{
            \cE'(\bV)
            \lt(
                \fr{
                    - \fr{\EE[\bM\bDel]}{q_0} \hbH
                    - \fr{\EE[\dbH\bDel]}{\psi_0} \bN
                }{\sqrt{1-\fr{\EE[\bM\bLam]^2}{q_0}}}
                + \fr{
                    \kappa - \fr{\EE[\bM\bLam]}{q_0} \hbH - \fr{\EE[\dbH\bLam]}{\psi_0} \bN
                }{\lt(1-\fr{\EE[\bM\bLam]^2}{q_0}\rt)^{3/2}}
                \cdot \fr{\EE[\bM\bLam]\EE[\bM\bDel]}{q_0}
            \rt)^2
        \rt\} \\
        &- \alpha_\star \EE \Bigg\{
            \cE(\bV)
            \Bigg(
                \fr{
                    - \fr{2\EE[\bM\bDel]}{q_0} \hbH - \fr{2\EE[\dbH\bDel]}{\psi_0} \bN
                }{\lt(1-\fr{\EE[\bM\bLam]^2}{q_0}\rt)^{3/2}}
                \cdot \fr{\EE[\bM\bLam]\EE[\bM\bDel]}{q_0} \\
                &+ \fr{
                    \kappa - \fr{\EE[\bM\bLam]}{q_0} \hbH - \fr{\EE[\dbH\bLam]}{\psi_0} \bN
                }{\lt(1-\fr{\EE[\bM\bLam]^2}{q_0}\rt)^{5/2}}
                \cdot \fr{3\EE[\bM\bLam]^2\EE[\bM\bDel]^2}{q_0^2}
                + \fr{
                    \kappa - \fr{\EE[\bM\bLam]}{q_0} \hbH - \fr{\EE[\dbH\bLam]}{\psi_0} \bN
                }{\lt(1-\fr{\EE[\bM\bLam]^2}{q_0}\rt)^{3/2}}
                \cdot \fr{\EE[\bM\bDel]^2}{q_0}
            \Bigg)
        \Bigg\}
        + f(\bDel_2),
    \ealn
    where $f(\bDel_2)$ is \eqref{eq:osS-1deriv} with $\bDel$ replaced by $\bDel_2$.
    We now specialize to $(\lambda_1,\lambda_2) = (1,0)$.
    As argued in the proof of Lemma~\ref{lem:sS-zero}\ref{itm:sS-zero-1deriv}, at $(\lambda_1,\lambda_2) = (1,0)$ we have $f(\bDel_2) = 0$.
    So,
    \baln
        &\la \osS_\star(1,0), (u_1,u_2)^{\otimes 2} \ra
        = - \EE [(u_1 \dbH + u_2 \bM) \bDel] \\
        &+\alpha_\star \EE \lt\{
            F'_{1-q_0}(\hbH)
            \lt(
                - \fr{\EE[\bM\bDel]}{q_0} \hbH
                - \fr{\EE[\dbH\bDel]}{\psi_0} \bN
                + \fr{
                    \kappa - \hbH - (1-q_0) \bN
                }{1-q_0}
                \cdot \EE[\bM\bDel]
            \rt)^2
        \rt\} \\
        &- \alpha_\star \EE \Bigg\{
            F_{1-q_0}(\hbH)
            \Bigg(
                \fr{
                    - \fr{2\EE[\bM\bDel]}{q_0} \hbH - \fr{2\EE[\dbH\bDel]}{\psi_0} \bN
                }{1-q_0}
                \cdot \EE[\bM\bDel] \\
                &+ \fr{
                    \kappa - \hbH - (1-q_0) \bN
                }{(1-q_0)^{2}}
                \cdot 3\EE[\bM\bDel]^2
                + \fr{
                    \kappa - \hbH - (1-q_0) \bN
                }{1-q_0}
                \cdot \fr{\EE[\bM\bDel]^2}{q_0}
            \Bigg)
        \Bigg\}.
    \ealn
    Specializing further to $\kappa = 0$ (which was not used up to here),
    \baln
        &\la \osS_\star(1,0), (u_1,u_2)^{\otimes 2} \ra \\
        &= - \EE [(u_1 \dbH + u_2 \bM) \bDel]
        +\alpha_\star \EE \lt\{
            F'_{1-q_0}(\hbH)
            \lt(
                \lt(\fr{1}{q_0(1-q_0)} \hbH + \bN\rt) \EE[\bM\bDel]
                + \fr{\bN}{\psi_0} \EE[\dbH \bDel]
            \rt)^2
        \rt\} \\
        &+ \alpha_\star \EE \lt\{
            \bN
            \lt(
                \lt(\fr{3}{q_0(1-q_0)^2} \hbH
                + \fr{1 + 2q_0}{q_0(1-q_0)} \bN\rt) \EE[\bM\bDel]^2
                + \fr{2}{\psi_0 (1-q_0)} \bN \EE[\bM \bDel] \EE[\dbH \bDel]
            \rt)
        \rt\}
    \ealn
    Finally, as $\alpha_\star \EE[\bN \hbH] = q_0 d_0 = -q_0 (1-q_0) \psi_0$ (by \eqref{eq:d0-formula}) and $\alpha_\star \EE[\bN^2] = \psi_0$, the last term simplifies to
    \[
        \fr{\psi_0}{q_0} \EE[\bM\bDel]^2
        + \fr{2}{1-q_0} \EE[\bM \bDel] \EE[\dbH \bDel].
    \]
    Expanding $\bDel = (1-\bM^2) (u_1 \dbH + u_2 \bM)$ concludes the proof.
\end{proof}
\begin{clm}
    \label{clm:numeric-C-estimates}
    The following estimates hold.
    \begin{enumerate}[label=(\alph*)]
        \item $C_1 \in [C_{1,\lb}, C_{1,\ub}] \equiv [-0.7193,-0.7165]$.
        \item $C_2 \in [C_{2,\lb}, C_{2,\ub}] \equiv [5.0439,5.0568]$.
        \item $C_3 \in [C_{3,\lb}, C_{3,\ub}] \equiv [1.1345,1.1526]$.
    \end{enumerate}
\end{clm}
\begin{proof}
    We compute using Lemma~\ref{lem:special-fn-identities} and \eqref{eq:p2-r2-trivial}:
    \baln
        C_1 &= \fr{\alpha_\star}{\psi_0^2} \cdot \fr{-\frs_2(\gamma_0)}{(1-q_0)^2}
        = - \fr{\alpha_\star r_4(\gamma_0)}{\psi_0^2(1-q_0)^2 (1+3\gamma_0)}
        = - \fr{\alpha_\star r_4(\gamma_0)}{\psi_0^2(1-q_0)(1+2q_0)}, \\
        C_2 &= \fr{2\alpha_\star}{\psi_0 (1-q_0)^2}
        \lt(- \frs_2(\gamma_0) + \fr{\frs_3(\gamma_0)}{q_0} \rt)
        + \fr{2}{1-q_0} \\
        &= \fr{2\alpha_\star}{\psi_0 (1-q_0)^2} \lt(
            \fr{(2-q_0)(1-q_0) r_4(\gamma_0)}{(1+q_0)(1+2q_0)}
            - \fr{r_2(\gamma_0)}{1+q_0}
        \rt)
        + \fr{2}{1-q_0}
        = \fr{2(2-q_0) \alpha_\star r_4(\gamma_0)}{\psi_0 (1-q_0^2)(1+2q_0)}
        + \fr{2q_0}{1-q_0^2}, \\
        C_3 &= - \fr{\alpha_\star}{(1-q_0)^2} \lt(
            \frs_2(\gamma_0) - \fr{2\frs_3(\gamma_0)}{q_0} + \fr{\frs_4(\gamma_0)}{q_0^2}
        \rt) + \fr{\psi_0}{q_0} \\
        &= - \fr{\alpha_\star}{(1-q_0)^2} \lt(
            \fr{1-q_0}{1+2q_0} r_4(\gamma_0)
            - \fr{2q_0-1}{q_0} r_2(\gamma_0)
        \rt) + \fr{\psi_0}{q_0}
        = - \fr{\alpha_\star r_4(\gamma_0)}{(1-q_0)(1+2q_0)} + \fr{\psi_0}{1-q_0}.
    \ealn
    So
    \baln
        C_{1,\lb}
        \stackrel{(*)}{\le} -\fr{\alpha_{\ub} r_{4,\ub}}{\psi_{\lb}^2 (1-q_{\ub}) (1 + 2q_{\lb})}
        &\le C_1
        \le -\fr{\alpha_{\lb} r_{4,\lb}}{\psi_{\ub}^2 (1-q_{\lb}) (1 + 2q_{\ub})}
        \stackrel{(*)}{\le} C_{1,\ub}, \\
        C_{2,\lb}
        \stackrel{(*)}{\le} \fr{2(2-q_{\ub}) \alpha_{\lb} r_{4,\lb}}{\psi_{\ub} (1-q_{\lb}^2)(1+2q_{\ub})}
        + \fr{2q_{\lb}}{1-q_{\lb}^2}
        &\le C_2
        \le \fr{2(2-q_{\lb}) \alpha_{\ub} r_{4,\ub}}{\psi_{\lb} (1-q_{\ub}^2)(1+2q_{\lb})}
        + \fr{2q_{\ub}}{1-q_{\ub}^2}
        \stackrel{(*)}{\le} C_{2,\ub}, \\
        C_{3,\lb}
        \stackrel{(*)}{\le} - \fr{\alpha_{\ub} r_{4,\ub}}{(1-q_{\ub})(1+2q_{\lb})} + \fr{\psi_{\lb}}{1-q_{\lb}}
        &\le C_3
        \le - \fr{\alpha_{\lb} r_{4,\lb}}{(1-q_{\lb})(1+2q_{\ub})} + \fr{\psi_{\ub}}{1-q_{\ub}}
        \stackrel{(*)}{\le} C_{3,\ub}.
    \ealn
    The steps marked $(*)$ are verified in the attached Python file, and are simple arithmetic comparisons of explicitly defined decimal numbers.
\end{proof}
\begin{clm}
    \label{clm:numeric-I-estimates}
    Define $I_1 = \EE[(1-\bM^2) \dbH^2]$, $I_2 = \EE[(1-\bM^2) \dbH \bM]$, $I_3 = \EE[(1-\bM^2) \bM^2]$.
    Then,
    \begin{enumerate}[label=(\alph*)]
        \item $I_1 \in [I_{1,\lb}, I_{1,\ub}] \equiv [0.24759912, 0.24759923]$.
        \item $I_2 \in [I_{2,\lb}, I_{2,\ub}] \equiv [0.16997315, 0.16997318]$.
        \item $I_3 \in [I_{3,\lb}, I_{3,\ub}] \equiv [0.12335884, 0.12335885]$.
    \end{enumerate}
\end{clm}
\begin{proof}
    By repeated integration by parts and \eqref{eq:p2-r2-trivial}:
    \baln
        I_1 &= \psi_0 (1-q_0) - 2\psi_0^2 (1 - 4q_0 + 3p_4(\psi_0)), &
        I_2 &= \psi_0 (1 - 4q_0 + 3p_4(\psi_0)), &
        I_3 &= q_0 - p_4(\psi_0).
    \ealn
    Thus
    \baln
        I_{1,\lb}
        \stackrel{(*)}{\le} \psi_{\lb} (1-q_{\ub}) - 2\psi_{\ub}^2 (1 - 4q_{\lb} + 3 p_{4,\ub})
        &\le I_1
        \le \psi_{\ub} (1-q_{\lb}) - 2\psi_{\lb}^2 (1 - 4q_{\ub} + 3 p_{4,\lb})
        \stackrel{(*)}{\le} I_{1,\ub}, \\
        I_{2,\lb}
        \stackrel{(*)}{\le} \psi_{\lb} (1 - 4q_{\ub} + 3 p_{4,\lb})
        &\le I_2
        \le \psi_{\ub} (1 - 4q_{\lb} + 3 p_{4,\ub})
        \stackrel{(*)}{\le} I_{2,\ub}, \\
        I_{3,\lb}
        \stackrel{(*)}{\le} q_{\lb} - p_{4,\ub}
        &\le I_3
        \le q_{\ub} - p_{4,\lb}
        \stackrel{(*)}{\le} I_{3,\ub}.
    \ealn
    The steps marked $(*)$ are verified in the attached Python file, and are simple arithmetic comparisons of explicitly defined decimal numbers.
\end{proof}

\begin{clm}
    \label{clm:numeric-M-estimates}
    Let $M = \nabla^2 \osS_\star(1,0)$. The following estimates hold.
    \begin{enumerate}[label=(\alph*),ref=(\alph*)]
        \item \label{itm:osS-hessian-11} $M_{1,1} \le M_{1,1,\ub} \equiv -0.045408$.
        \item \label{itm:osS-hessian-22} $M_{2,2} \le M_{2,2,\ub} \equiv -0.020490$.
        \item \label{itm:osS-hessian-12} $M_{1,2} \in [M_{1,2,\lb}, M_{1,2,\ub}] \equiv [-0.025685,-0.026567]$.
        \item \label{itm:osS-hessian-det} $\det(M) \ge M_{\det,\lb} \equiv 0.0002246$.
    \end{enumerate}
\end{clm}
\begin{proof}
    By Lemma~\ref{lem:sS-2deriv-formula},
    \baln
        M_{1,1} &= -I_1 + C_1 I_1^2 + C_2 I_1 I_2 + C_3 I_2^2, \\
        M_{1,2} &= -I_2 + C_1 I_1 I_2 + \fr12 C_2 (I_2^2 + I_1I_3) + C_3 I_2I_3, \\
        M_{2,2} &= -I_3 + C_1 I_2^2 + C_2 I_2 I_3 + C_3 I_3^2.
    \ealn
    Estimating with Claims~\ref{clm:numeric-C-estimates} and \ref{clm:numeric-I-estimates}, we find
    \baln
        M_{1,1} &\le - I_{1,\lb} + C_{1,\ub} I_{1,\lb}^2 + C_{2,\ub} I_{1,\ub} I_{2,\ub} + C_{3,\ub} I_{2,\ub}^2
        \stackrel{(*)}{\le} M_{1,1,\ub}, \\
        M_{2,2} &\le - I_{3,\lb} + C_{1,\ub} I_{2,\lb}^2 + C_{2,\ub} I_{2,\ub} I_{3,\ub} + C_{3,\ub} I_{3,\ub}^2
        \stackrel{(*)}{\le} M_{2,2,\ub}, \\
        M_{1,2} &\le - I_{2,\lb} + C_{1,\ub} I_{1,\lb} I_{2,\lb} + \fr12 C_{2,\ub} (I_{2,\ub}^2 + I_{1,\ub} I_{3,\ub}) + C_{3,\ub} I_{2,\ub} I_{3,\ub}
        \stackrel{(*)}{\le} M_{1,2,\ub}, \\
        M_{1,2} &\ge - I_{2,\ub} + C_{1,\lb} I_{1,\ub} I_{2,\ub} + \fr12 C_{2,\lb} (I_{2,\lb}^2 + I_{1,\lb} I_{3,\lb}) + C_{3,\lb} I_{2,\lb} I_{3,\lb}
        \stackrel{(*)}{\ge} M_{1,2,\lb}.
    \ealn
    The steps marked $(*)$ are verified in the attached Python file, and are simple arithmetic comparisons of explicitly defined decimal numbers.
    This proves parts \ref{itm:osS-hessian-11}, \ref{itm:osS-hessian-22}, and \ref{itm:osS-hessian-12}.
    Finally,
    \[
        \det(M) = M_{1,1} M_{2,2} - M_{1,2}^2
        \ge M_{1,1,\ub} M_{2,2,\ub} - M_{1,2,\lb}^2
        \stackrel{(*)}{\ge} M_{\det,\lb},
    \]
    where the step $(*)$ is verified in the attached Python file.
    This proves part \ref{itm:osS-hessian-det}.
\end{proof}
\begin{proof}[Proof of Claim~\ref{clm:sS-zero-2deriv}]
    Follows from Claim~\ref{clm:numeric-M-estimates}, which implies $M_{1,1}, M_{2,2} < 0$ and $\det(M) > 0$.
\end{proof}

\subsection{Interval arithmetic estimates}
\label{subapp:interval-arithmetic}

We now describe the computer-assisted proofs of Claims~\ref{clm:p}, \ref{clm:r}, \ref{clm:m}, and \ref{clm:g}.
We begin with the more straightforward Claims~\ref{clm:p} and \ref{clm:m}.
\begin{proof}[Proof of Claim~\ref{clm:p}]
    We first show the upper bound.
    Set $L = 10$.
    Since $\th^4$ takes values in $[0,1]$,
    \baln
        p_4(\psi_0)
        \stackrel{Lem.~\ref{lem:special-fn-monotonicity}}{\le}
        p_4(\psi_{\ub})
        &\le \EE [\th^4(\psi_{\ub}^{1/2} Z) \bone\{|Z| \le L\}]
        + \PP[|Z| \ge L] \\
        &\le \int_{-L}^{L} \th^4(\psi_{\ub}^{1/2} x) \varphi(x)~\de x
        + 2e^{-L^2/2}
        \stackrel{(*)}{\le} p_{4,\ub},
    \ealn
    where the step $(*)$ is verified in the attached Python file.
    Similarly,
    \[
        p_4(\psi_0)
        \stackrel{Lem.~\ref{lem:special-fn-monotonicity}}{\ge}
        p_4(\psi_{\lb})
        \ge \EE [\th^4(\psi_{\lb}^{1/2} Z) \bone\{|Z| \le L\}]
        = \int_{-L}^{L} \th^4(\psi_{\lb}^{1/2} x) \varphi(x)~\de x
        \stackrel{(*)}{\ge} p_{4,\lb},
    \]
    where the step $(*)$ is verified in the attached Python file.
\end{proof}

\begin{proof}[Proof of Claim~\ref{clm:m}]
    Let $L=10$.
    Note that for any $x\in \bbR$, $(\hz + \ch^2(x))^{-1} \le (1+\hz)^{-1}$.
    Then,
    \baln
        m(\hz) = m(\hz,\psi_0)
        \stackrel{Lem.~\ref{lem:special-fn-monotonicity}}{\le} m(\hz,\psi_{\lb})
        &\le \EE [(\hz + \ch^2(\psi_{\lb}^{1/2} Z))^{-1} \bone\{|Z| \le L\}]
        + (1+\hz)^{-1} \PP[|Z| \ge L] \\
        &\le \int_{-L}^{L} (\hz + \ch^2(\psi_{\lb}^{1/2} x))^{-1} \varphi(x)~\de x
        + 2(1+\hz)^{-1} e^{-L^2/2}
        \stackrel{(*)}{\le} m_{\ub},
    \ealn
    where the step $(*)$ is verified in the attached Python file.
\end{proof}
Claims~\ref{clm:r} and \ref{clm:g} will involve integrating functions that involve $\cE$ against the gaussian measure.
This is more challenging because $\cE$ is itself defined in terms of an integral, which makes these claims less amenable to numerical integration.
We take a cruder approach of discretizing these integrals into small intervals, and bounding the integral on each small interval using monotonicity properties of $\cE$ and $\cE'$.
\begin{proof}[Proof of Claim~\ref{clm:r}]
    Let $L=8$, $\delta=10^{-3}$, and $J = L/\delta$.
    For integer $j \in [-J,J]$, let $x_j = j\delta$.
    Then,
    \[
        r_4(\gamma_0)
        \stackrel{Lem.~\ref{lem:special-fn-monotonicity}}{\le} r_4(\gamma_{\ub})
        = \sum_{j=-J}^{J-1} \EE\lt\{\cE(\gamma_{\ub}^{1/2} Z)^4 \bone\{Z\in [x_j,x_{j+1}]\} \rt\}
        + \EE\lt\{\cE(\gamma_{\ub}^{1/2} Z)^4 \bone\{|Z| \ge L\} \rt\}.
    \]
    These terms can be bounded as follows.
    Since $\cE$ is nonnegative and increasing (by Lemma~\ref{lem:E-derivative-bds}\ref{itm:cE-bd}\ref{itm:cEpr}),
    \[
        \EE\lt\{\cE(\gamma_{\ub}^{1/2} Z)^4 \bone\{Z\in [x_j,x_{j+1}]\} \rt\}
        \le \cE(\gamma_{\ub}^{1/2} x_{j+1})^4 \PP[Z\in [x_j,x_{j+1}]],
    \]
    and this probability is bounded above by $\delta \varphi(x_{j+1})$ if $j\le -1$ and $\delta \varphi(x_j)$ if $j\ge 0$.
    We estimate the tail term using Cauchy--Schwarz:
    \[
        \EE\lt\{\cE(\gamma_{\ub}^{1/2} Z)^4 \bone\{|Z| \ge L\} \rt\}
        \le \EE \lt\{\cE(\gamma_{\ub}^{1/2} Z)^8 \rt\}^{1/2} \PP[|Z| \ge L]^{1/2}.
    \]
    The probability is bounded by $2e^{-L^2/2}$.
    For the remaining expectation, recall from Lemma~\ref{lem:E-derivative-bds}\ref{itm:cE-bd} that $0\le \cE(x) \le |x|+1$.
    So,
    \[
        \EE \lt\{\cE(\gamma_{\ub}^{1/2} Z)^8\rt\}
        \le \EE \lt\{(1 + \gamma_{\ub}^{1/2} |Z|)^8\rt\}
        \le 2^7 \EE \lt\{1 + \gamma_{\ub}^4 Z^8\rt\}
        = 2^7 (1 + 105\gamma_{\ub}^2).
    \]
    Combining these estimates yields
    \[
        \EE\lt\{\cE(\gamma_{\ub}^{1/2} Z)^4 \bone\{|Z| \ge L\} \rt\}
        \le 2^4 (1 + 105\gamma_{\ub}^2)^{1/2} e^{-L^2/4}
        \le 2^4 (1 + 11\gamma_{\ub}) e^{-L^2/4}.
    \]
    All in all,
    \[
        r_4(\gamma_0)
        \le \delta \sum_{j=-J}^{-1} \cE(\gamma_{\ub}^{1/2} x_{j+1})^4 \varphi(x_{j+1})
        + \delta \sum_{j=0}^{J-1} \cE(\gamma_{\ub}^{1/2} x_{j+1})^4 \varphi(x_j)
        + 2^4 (1 + 11\gamma_{\ub}) e^{-L^2/4}
        \stackrel{(*)}{\le} r_{4,\ub},
    \]
    where the step $(*)$ is verified in the attached Python file.
    (See Remark~\ref{rmk:eval-cE} below for how the function $\cE$ is evaluated numerically).
    For the lower bound, we similarly have
    \baln
        r_4(\gamma_0)
        \stackrel{Lem.~\ref{lem:special-fn-monotonicity}}{\ge} r_4(\gamma_{\lb})
        &= \sum_{j=-J}^{J-1} \EE\lt\{\cE(\gamma_{\lb}^{1/2} Z)^4 \bone\{Z\in [x_j,x_{j+1}]\} \rt\} \\
        &\ge \delta \sum_{j=-J}^{-1} \cE(\gamma_{\lb}^{1/2} x_j)^4 \varphi(x_j)
        + \delta \sum_{j=0}^{J-1} \cE(\gamma_{\lb}^{1/2} x_j)^4 \varphi(x_{j+1})
        \stackrel{(*)}{\ge} r_{4,\lb},
    \ealn
    where the step $(*)$ is verified in the attached Python file.
\end{proof}
\begin{rmk}
    \label{rmk:eval-cE}
    The above computer-assisted proof requires evaluating the function $\cE(x) = \varphi(x) / \Psi(x)$, where $\Psi(x) = \PP[Z \ge x]$ is itself an integral.
    We evaluate this as follows.
    Note that the inputs $x$ on which we numerically evaluate $\cE$ are bounded above by $\gamma_{\ub}^{1/2} L \le 10$.
    Define $L_+ = 12$.
    We estimate
    \[
        \cE(x)^{-1}
        = \int_x^{L_+} \fr{\varphi(y)}{\varphi(x)} ~\de y
        + \fr{\PP[Z \ge L_+]}{\varphi(x)}
        \le \int_x^{L_+} e^{-(y^2-x^2)/2} ~\de y
        + \fr{e^{-(L_+^2-x^2)/2}}{\sqrt{2\pi}}
    \]
    and
    \[
        \cE(x)^{-1}
        \ge \int_x^{L_+} \fr{\varphi(y)}{\varphi(x)} ~\de y
        = \int_x^{L_+} e^{-(y^2-x^2)/2} ~\de y.
    \]
    The remaining integral can be rigorously bounded by numerical integration, and for $x \le 10$ the term $e^{-(L_+^2-x^2)/2} / \sqrt{2\pi}$ will contribute an error that is multiplicatively small.
\end{rmk}
Finally, we turn to Claim~\ref{clm:g}.
By Lemma~\ref{lem:E-derivative-bds}\ref{itm:cEpr}, $\cE'$ takes values in $(0,1)$.
Thus the function $g$ defined in \eqref{eq:special-fn-g} is decreasing in $m$ and increasing in $q$, and
\beq
    \label{eq:g-monotonicity}
    g(m(\hz),q_0,\gamma_0) \ge g(m_{\ub},q_{\lb},\gamma_0).
\eeq
However, $g$ is not clearly monotone in $\gamma$, so we instead control the derivative of $g$ in $\gamma$.
\begin{lem}
    \label{lem:g-deriv}
    Let $\tg(\gamma) = g(m_{\ub},q_{\lb},\gamma)$.
    Then, for all $\gamma \ge 0$, $|\tg'(\gamma)| \le 20$.
\end{lem}
\begin{proof}
    We write $\tg(\gamma) = \EE[\hg(\gamma^{1/2} Z)]$, where
    \beq
        \label{eq:special-fn-hg}
        \hg(x) = \fr{\cE'(x)}{(1-q_{\lb})(1-\cE'(x)) + m_{\ub}\cE'(x)}.
    \eeq
    A straightforward calculation shows that
    \[
        \hg''(x) = \fr{(1-q_{\lb})\cE^{(3)}(x)}{((1-q_{\lb})(1-\cE'(x)) + m_{\ub}\cE'(x))^2}
        - \fr{2(1-q_{\lb})(m_{\ub}+q_{\lb}-1) \cE''(x)^2}{(((1-q_{\lb})(1-\cE'(x)) + m_{\ub}\cE'(x))^2)^3}.
    \]
    Since $\cE'(x) \in (0,1)$ by Lemma~\ref{lem:E-derivative-bds}\ref{itm:cEpr},
    \[
        (1-q_{\lb})(1-\cE'(x)) + m_{\ub}\cE'(x)
        \ge \min(1-q_{\lb}, m_{\ub})
        = 1-q_{\lb}.
    \]
    Lemma~\ref{lem:E-derivative-bds}\ref{itm:cE2}\ref{itm:cE3} yields $|\cE''(x)| \le 1$, $|\cE^{(3)}(x)| \le 13$.
    Thus
    \[
        |\hg''(x)|
        \le \fr{13}{1-q_{\lb}}
        + \fr{2(m_{\ub} + q_{\lb} - 1)}{(1-q_{\lb})^2}
        \le 40,
    \]
    where the final estimate follows from the simple bounds $q_{\lb} \le 3/5$, $m_{\ub} \le 1$.
    Finally, a gaussian integration by parts calculation yields
    \[
        \tg'(\gamma) = \fr12 \EE[\hg''(\gamma^{1/2} Z)],
    \]
    which implies the result.
\end{proof}
\begin{proof}[Proof of Claim~\ref{clm:g}]
    In light of \eqref{eq:g-monotonicity} and Lemma~\ref{lem:g-deriv}, we will estimate
    \[
        g(m(\hz),q_0,\gamma_0)
        \ge g(m_{\ub},q_{\lb},\gamma_{\lb}) - 20 |\gamma_{\ub} - \gamma_{\lb}|.
    \]
    We will estimate $g(m_{\ub},q_{\lb},\gamma_{\lb})$ by discretization, like in the proof of CLaim~\ref{clm:r}.
    Let $L=8$, $\delta=10^{-3}$, and $J = L/\delta$.
    For integer $j\in [-J,J]$, let $x_j = j\delta$.

    Note that $\hg(x)$ defined in \eqref{eq:special-fn-hg} takes positive values, and is an increasing function of $\cE'(x)$.
    Moreover, by Lemma~\ref{lem:E-derivative-bds}\ref{itm:cE2}, $\cE'(x)$ is an increasing function of $x$.
    Thus $\hg(x)$ is an increasing function of $x$.
    Hence,
    \baln
        g(m_{\ub},q_{\lb},\gamma_{\lb})
        &= \EE[\hg(\gamma_{\lb}^{1/2} Z)]
        \ge \sum_{j=-J}^{J-1} \EE[\hg(\gamma_{\lb}^{1/2} Z) \bone\{Z \in [x_j,x_{j+1}]\}] \\
        &\ge \delta \sum_{j=-J}^{-1} \hg(\gamma_{\lb}^{1/2} x_j) \varphi(x_j)
        + \delta \sum_{j=0}^{J-1} \hg(\gamma_{\lb}^{1/2} x_j) \varphi(x_{j+1}).
    \ealn
    Combining the above,
    \[
        g(m(\hz),q_0,\gamma_0)
        \ge \delta \sum_{j=-J}^{-1} \hg(\gamma_{\lb}^{1/2} x_j) \varphi(x_j)
        + \delta \sum_{j=0}^{J-1} \hg(\gamma_{\lb}^{1/2} x_j) \varphi(x_{j+1})
        - 20 |\gamma_{\ub} - \gamma_{\lb}|
        \stackrel{(*)}{\ge} g_{\lb},
    \]
    where the step $(*)$ is verified in the attached Python file.
    We numerically evaluate $\hg$ using the identity $\cE'(x) = \cE(x)(\cE(x) - x)$ (Lemma~\ref{lem:E-derivative-bds}\ref{itm:cEpr}), evaluating $\cE$ as in Remark~\ref{rmk:eval-cE}.
\end{proof}

\end{document}